\documentclass[11pt, a4paper, UKenglish]{article}
\usepackage[UKenglish]{babel}
\usepackage[utf8]{inputenc}
\usepackage{amssymb}
\usepackage{amsmath}
\usepackage{amsthm}
\usepackage{tikz}
\usepackage{tikz-cd}

\newcommand{\nospacepunct}[1]{\makebox[0pt][l]{\,,}}
\usepackage{quiver}
\usepackage{hyperref}
\newtheorem{thm}{Theorem}[subsection]
\newtheorem{prop}[thm]{Proposition}
\newtheorem{lem}[thm]{Lemma}
\newtheorem{cor}[thm]{Corollary}
\theoremstyle{definition}
\newtheorem{definition}[thm]{Definition}
\theoremstyle{remark}
\newtheorem{remark}[thm]{Remark}
\theoremstyle{construction}
\newtheorem{construction}[thm]{Construction}
\usepackage{atbegshi,picture}

\begin{document}
		\title{
		{Motivic Toda brackets}\\
	}
	\author{Xiaowen Dong\footnote{\href{mailto:xiaowen.dong@uni-osnabrueck.de}{xiaowen.dong@uni-osnabrueck.de}}}
		\maketitle
		\begin{abstract}
			We construct Toda brackets in unstable motivic homotopy theory and prove some fundamental properties of them. Furthermore we construct some examples of motivic Toda brackets.
			\end{abstract}
	\tableofcontents

\section{Introduction}
	\
	\
	
Toda brackets were originally introduced by Toda in \cite{Toda+1963} to compute homotopy groups of spheres and are important computational tools in homotopy theory. We recall how Toda brackets are defined in topology.
\begin{definition}
	\label{def1}
	Let \[\begin{tikzcd}
		W & X & Y & Z
		\arrow["\gamma", from=1-1, to=1-2]
		\arrow["\beta", from=1-2, to=1-3]
		\arrow["\alpha", from=1-3, to=1-4]
	\end{tikzcd}\] be a sequence of pointed topological spaces such that $\alpha\circ \beta$ and $\beta\circ \gamma$ are nullhomotopic. Let $A: X\times I\rightarrow Z$ be a homotopy with $A(-,0)=\alpha\circ \beta$ and $A(-,1)=\ast$. Let $B: W\times I\rightarrow Y$ be a homotopy with $B(-,0)=\beta\circ \gamma$ and $B(-,1)=\ast$. Then we can construct a pointed continuous map $H$ from the reduced suspension $W\wedge I/\partial I$ to $Z$ by the formula \begin{align*}
		H(w\wedge t)=
		\begin{cases}
			\alpha(B(w, 2t-1)) & \text{if} \  \frac{1}{2}\leq t\leq 1\\
			A(\gamma(w), 1-2t) & \text{if} \ 0\leq t \leq\frac{1}{2}
		\end{cases}
	\end{align*}
	for every $w\in W$ and $t\in I/ \partial I$. The Toda bracket $\{ \alpha, \beta,\gamma\}$ is defined to be the set of the homotopy classes of all such maps obtained in the way described above by choosing all possible nullhomotopies for $\alpha\circ \beta$ and $\beta\circ \gamma$.\\
\end{definition}

Using such brackets and the EHP-sequence Toda was able to compute $\pi_{n+k}(S^n)$ for $k\leq 19$ and $n\leq 20$. By the same methods Mimura and Toda \cite{Mimura1963TheH} computed the groups $\pi_{n+20}(S^n) $ for all $n$. The paper \cite{mimura1965generalized} by Mimura contains the computations for $\pi_{n+21}(S^n) $ and $\pi_{n+22}(S^n) $ for all $n$. In \cite{1975} Mimura, Mori and Oda determined the 2-primary components of the groups $\pi_{n+23}(S^n) $ and $\pi_{n+24}(S^n) $ for all $n$. Furthermore Oda was able to compute the 2-primary components of the groups $\pi_{n+k}(S^n) $ for $25\leq k\leq30$ in \cite{oda19772}.\\

Since the stable motivic homotopy  category $\mathcal{SH}(k) $ over a field $k$ is a triangulated category, one can naturally define Toda brackets in this category via distinguished triangles. In particular we can use Toda brackets for computations of stable motivic homotopy groups of spheres. For example they are used in the paper \cite{stsblehomo} by Isaksen, Wang and Xu for computations of the $\mathbb{C}$-motivic stable homotopy groups from dimension 0 to 90. Toda brackets also appear in the paper \cite{inverted} by Guillou and Isaksen.\\

We would like to have such useful computational tools in unstable motivic homotopy theory, too.  Since the unstable motivic homotopy category is not a triangulated category, we have to construct Toda brackets in a different way; namely we follow the original approach by Toda. We hope that we can use them to compute unstable homotopy groups of motivic spheres. In particular we would like to have a hands-on approach for computations as in classical algebraic topology where one constructs explicit elements and relations using Toda brackets.\\
	
In this paper we construct the motivic analogue of Toda brackets in the unstable motivic homotopy theory. Let $S$ be a noetherian base scheme of finite dimension. We consider the category of pointed simplicial presheaves on the Nisnevich site $\mathcal{S}\mathrm{m}_{S}$. We denote this category by $\mathrm{sPre}(S)_{\ast}$. Let
\[\begin{tikzcd}
	\mathcal{W} & \mathcal{X} & \mathcal{Y} & \mathcal{Z}
	\arrow["\gamma", from=1-1, to=1-2]
	\arrow["\beta", from=1-2, to=1-3]
	\arrow["\alpha", from=1-3, to=1-4]
\end{tikzcd}\]
be a sequence of composable morphisms in $\mathrm{sPre}(S)_{\ast}$ such that the compositions $\alpha\circ\beta$ and $\beta\circ\gamma$ are 0 in the pointed motivic homotopy category $\mathcal{H}_{\bullet}(S)$. Then we would like to define the Toda bracket $\{\alpha,\beta,\gamma\}\subseteq \mathcal{H}_{\bullet}(S)(\mathcal{W}\wedge (\Delta^1/\partial\Delta^1), \mathcal{Z})$. The construction is given in section~\ref{subsec:Toda}. For the construction we use the  $\mathbb{A}^1$-local injective model structure introduced by Morel and Voevodsky \cite{Morel1999A1homotopyTO}. An infinity-categorical approach is definitely possible; but it turns out that it is not really necessary for our purpose. In the same section we also show that our notion of motivic Toda brackets depends only on the homotopy classes of the morphisms $\alpha, \beta$ and $\gamma$ (see Proposition~\ref{prop:propostion 2.2.10}).\\

Topological Toda brackets satisfy many relations (see \cite{Toda+1963}). A natural question is whether our motivic Toda brackets also satisfy some of these relations. The answer is yes. The key observation is that we can construct an $\mathbb{A}^1$-local model structure on the category of presheaves with values in pointed $\Delta$-generated topological spaces. This category is denoted by $\mathrm{Pre}_{\Delta}(S)_{\ast}$. The notion of $\Delta$-generated spaces was introduced by Jeff Smith (see \cite{delta}). Similar ideas can be found already in \cite{Vogt1971ConvenientCO}. In particular, the category of $\Delta$-generated spaces $\mathcal{T}op_\Delta $ is locally presentable. Using this fact we can construct an $\mathbb{A}^1$-local injective model structure on $\mathrm{Pre}_{\Delta}(S)_{\ast}$ such that there is a Quillen equivalence \begin{center}  
	$|\cdot|: \mathrm{sPre}(S)_{_\ast \  \mathbb{A}^1-local\ inj}\rightleftarrows \mathrm{Pre}_{\Delta}(S)_{_\ast \ \mathbb{A}^1-local\ inj}: Sing$
\end{center}
where $|\cdot|$ is the geometric realization functor and $Sing$ the singular complex functor. Both are defined levelwise. We also construct an $\mathbb{A}^1$-local projective model structure on $\mathrm{Pre}_{\Delta}(S)_{\ast}$ such that we get the following diagram of Quillen equivalences 
\[\begin{tikzcd}
	{\mathrm{sPre}(S)_{_\ast \  \mathbb{A}^1-local\ proj}} & {\mathrm{Pre}_{\Delta}(S)_{_\ast \  \mathbb{A}^1-local\ proj}} \\
	{\mathrm{sPre}(S)_{_\ast\  \mathbb{A}^1-local\ inj}} & {\mathrm{Pre}_{\Delta}(S)_{_\ast \  \mathbb{A}^1-local\ inj}} & \cdot
	\arrow["{{|\cdot|}}", from=1-1, to=1-2]
	\arrow["Sing", shift left=2, from=1-2, to=1-1]
	\arrow["{{|\cdot|}}", from=2-1, to=2-2]
	\arrow["Sing", shift left=2, from=2-2, to=2-1]
	\arrow["{{\mathrm{id}}}"', from=1-1, to=2-1]
	\arrow[shift right=2, from=2-1, to=1-1]
	\arrow[from=1-2, to=2-2]
	\arrow["{{\mathrm{id}}}"', shift right=2, from=2-2, to=1-2]
\end{tikzcd}\]
The complete construction is in section~\ref{Quillen}. Via the Quillen equivalence \begin{center}  
	$|\cdot|: \mathrm{sPre}(S)_{_\ast \ \mathbb{A}^1-local\ inj}\rightleftarrows \mathrm{Pre}_{\Delta}(S)_{_\ast \ \mathbb{A}^1-local\ inj}: Sing$
\end{center}
we can now easily transform the topological proofs for the relations into proofs in the motivic setting, since presheaves with values in $\Delta$-generated spaces behave in many cases like topological spaces.\\

Let
\[\begin{tikzcd}
	W & X & Y & Z
	\arrow["\gamma", from=1-1, to=1-2]
	\arrow["\beta", from=1-2, to=1-3]
	\arrow["\alpha", from=1-3, to=1-4]
\end{tikzcd}\]
be a sequence of composable morphisms in $\mathrm{sPre}_{\ast}(S)$ such that the compositions $\alpha\circ\beta$ and $\beta\circ\gamma$ are nullhomotopic.
Precisely, we can prove for example the following relations for motivic Toda brackets:
\begin{itemize}
	\item[1.] (Proposition~\ref{prop:proposition 2.3.1}) The Toda bracket $\{\alpha,\beta, \gamma\}$ is a double coset of certain subgroups of $\mathcal{H}_{\bullet}(S)(W\wedge (\Delta^1/\partial\Delta^1), Z)$.
	\item[2.] (Proposition~\ref{Prop2.3.3}) Let $\delta:\mathcal{V}\rightarrow \mathcal{W}$ be a pointed morphism such that $\gamma\circ\delta$ is nullhomotopic. Then we have $\{\alpha,\beta, \gamma\}\circ\Sigma\delta=-(\alpha\circ \{\beta,\gamma, \delta\}) $.
	\item[3.] (Corollary~\ref{cor2.3.5}) Let $\gamma': \mathcal{W}\rightarrow \mathcal{X}$ be another morphism such that $\beta\circ\gamma'$ is nullhomotopic. Then we have $\{\alpha,\beta, \gamma\}+\{\alpha,\beta, \gamma'\}\supseteq \{\alpha,\beta, \gamma+\gamma'\} $, if $\mathcal{W}$ is a simplicial suspension of a space. 
	\item[4.] (Corollary~\ref{cor2.3.5}) Let $\beta': \mathcal{X}\rightarrow \mathcal{Y}$ be another morphism such that $\beta'\circ \gamma$ and $\alpha\circ\beta'$ are nullhomotopic. If $\mathcal{W}=\Sigma\mathcal{W}', \mathcal{X}=\Sigma\mathcal{X}'$ and $\gamma=\Sigma\gamma'$ for some $\gamma'\in \mathrm{sPre}(S)_{\ast}(\mathcal{W}', \mathcal{X}')$, then we have $\{\alpha,\beta, \Sigma\gamma'\}+\{\alpha,\beta', \Sigma\gamma'\}= \{\alpha,\beta+\beta', \Sigma\gamma'\} $.
	\item[5.] (Proposition~\ref{prop2.3.6}) Let $\alpha'$ be another morphism from $\mathcal{Y} $ to $\mathcal{Z}$. Then we have $\{\alpha, \Sigma\beta', \Sigma\gamma'\}+\{\alpha', \Sigma\beta', \Sigma\gamma'\}\supseteq \{\alpha+\alpha', \Sigma\beta', \Sigma\gamma'\}$ if $\mathcal{W}=\Sigma \mathcal{W'} $, $\mathcal{X}=\Sigma \mathcal{X'} $, $\mathcal{Y}=\Sigma \mathcal{Y'} $, $\gamma=\Sigma \gamma'$ and $\beta=\Sigma \beta'$ for some $\mathcal{W'}, \mathcal{X'},\mathcal{Y'} \in \mathrm{sPre}(S)_{\ast}$ and $\gamma'\in \mathrm{sPre}(S)_{\ast}(\mathcal{W'}, \mathcal{X'} )$ and $\beta'\in \mathrm{sPre}(S)_{\ast}(\mathcal{X'}, \mathcal{Y'} )$.
\end{itemize}
Finally we show in Lemma~\ref{lemma 2.3.9} and Proposition~\ref{proposition 2.3.10} that we can also get Toda brackets by using extensions and coextensions (see Definition~\ref{def:def2.3.7}).\\

The geometric realization functor turns out to be a very useful tool. It allows us to treat simplicial presheaves in some cases just like topological spaces. We also use it in section 3 where we construct examples for motivic Toda brackets.\\

In section~\ref{example} we work over the base $\mathrm{Spec}\mathbb{Z}$. Let $S^{\alpha+(\beta)}$ denote the motivic sphere $S^{\alpha}\wedge\mathbb{G}_{m}^{\beta}$. Maps from spheres to spheres are indexed by the bidegree of the target. Suspension from the right with $\mathbb{G}_{m} $ increases the weight $(\beta)$ by 1. Suspension from the right by the simplicial circle $S^1$ increases the degree $\alpha$ by 1. Let $\eta_{1+(1)}$ denote the first algebraic Hopf map from $ S^{1+(2)}$ to $S^{1+(1)}$. Let $\epsilon:\mathbb{G}_{m}\rightarrow \mathbb{G}_{m} $ be given by $x\mapsto x^{-1}$. Then we define the hyperbolic plane $h_{1+(1)}$ to be $1_{1+(1)}-\epsilon_{1+(1)}$. Let $q_{(1)}:\mathbb{G}_{m}\rightarrow \mathbb{G}_{m} $ denote the map defined by $x\mapsto x^{2}$.\\

Using the methods in \cite{ASENS_2012_4_45_4_511_0} we can show in particular that the Toda brackets $\{h_{1+(2)}, \eta_{1+(2)}, h_{1+(3)}\}$ and $\{\eta_{1+(2)}, h_{1+(3)}, \eta_{1+(3)}\}$ are defined. The key point for the proof is to show that the relation $ q_{1+(1)}=1_{1+(1)}-\epsilon_{1+(1)}$ holds (Proposition~\ref{Proposition 3.4.3}).\\

In the following we give briefly the idea of the proof. The two morphisms are endomorphisms of $S^1\wedge\mathbb{G}_{m} $. We consider now the projective line $\mathbb{P}^1_{\mathbb{Z}}$ equipped with the base point $\infty:=[1:0]$. Using a suitable \hyperlink{iso}{isomorphism} between $S^1\wedge\mathbb{G}_{m} $ and $(\mathbb{P}^1_{\mathbb{Z}}, \infty)$ in the pointed $\mathbb{A}^1$-homotopy category, we can transform these two morphisms into two endomorphisms of $(\mathbb{P}^1_{\mathbb{Z}}, \infty) $. Let $[\mathbb{P}^1_{\mathbb{Z}}, \mathbb{P}^1_{\mathbb{Z}}]^{\mathrm{N}}$ be the set of pointed naive $\mathbb{A}^1$-homotopy classes of pointed scheme endomorphisms of $(\mathbb{P}^1_{\mathbb{Z}}, \infty)$ (Definition~\ref{definition3.1.1}). In Proposition~\ref{Proposition 3.1.9} we give a characterization of pointed scheme endomorphisms of $(\mathbb{P}^1_{\mathbb{Z}}, \infty)$. In Proposition~\ref{Proposition 3.1.10} we give a characterization of pointed $\mathbb{A}^1$-homotopies of pointed scheme endomorphisms of $(\mathbb{P}^1_{\mathbb{Z}}, \infty)$. Let $\mathcal{H}_{\bullet}(\mathbb{Z})(\mathbb{P}^1_{\mathbb{Z}}, \mathbb{P}^1_{\mathbb{Z}})$
be the set of endomorphisms of $(\mathbb{P}^1_{\mathbb{Z}}, \infty)$ in the pointed $\mathbb{A}^1$-homotopy category $\mathcal{H}_{\bullet}(\mathbb{Z})$. We can equip $[\mathbb{P}^1_{\mathbb{Z}}, \mathbb{P}^1_{\mathbb{Z}}]^{\mathrm{N}}$ with a monoid structure and denote its monoid operation by $\oplus^{\mathrm{N}}$ (Definition~\ref{Definition 3.1.11}). Furthermore via the chosen \hyperlink{iso}{isomorphism} between $S^1\wedge\mathbb{G}_{m} $ and $(\mathbb{P}^1_{\mathbb{Z}}, \infty)$ we can equip $(\mathbb{P}^1_{\mathbb{Z}}, \infty)$ with a cogroup structure such that $\mathcal{H}_{\bullet}(\mathbb{Z})(\mathbb{P}^1_{\mathbb{Z}}, \mathbb{P}^1_{\mathbb{Z}})$ is a group. We denote the induced group operation by $\oplus^{\mathbb{A}^1}$. In \cite[Appendix B]{ASENS_2012_4_45_4_511_0} Cazanave shows that the canonical map\begin{align*}
	[\mathbb{P}^1_{k}, \mathbb{P}^1_{k}]^{\mathrm{N}}\rightarrow \mathcal{H}_{\bullet}(k)(\mathbb{P}^1_{k}, \mathbb{P}^1_{k})
\end{align*} is a homomorphism of monoids for any field $k$. In this paper we extend this result partially to the base $\mathrm{Spec}\mathbb{Z}$. We show that for certain pointed $\mathbb{A}^1$-homotopy classes of pointed scheme endomorphisms of $(\mathbb{P}^1_{\mathbb{Z}}, \infty)$ their $\oplus^{\mathrm{N}}$-sums coincide with the $\oplus^{\mathbb{A}^1}$-sums, This result can be found in the proof of Proposition~\ref{Proposition 3.4.3}.\\

In order to get Proposition~\ref{Proposition 3.4.3} we also have to fix a gap in Cazanave's paper on the cogroup structure on $\mathbb{P}^1$. In his paper Cazanave gives only a codiagonal morphism for $\mathbb{P}^1$ using some geometry for the projective line (see \cite[Lemma B.4]{ASENS_2012_4_45_4_511_0}). We are able to show that his codiagonal morphism comes actually from the chosen isomorphism with $S^1\wedge\mathbb{G}_{m} $ and therefore defines really a cogroup structure (Proposition~\ref{Proposition 3.2.1}).\\

We consider now the pointed endomorphisms $f$ and $g$ of $(\mathbb{P}^1_{\mathbb{Z}}, \infty)$ which correspond to $q_{1+(1)}$ and $-\epsilon_{1+(1)}$, respectively. In particular, $1_{1+(1)}-\epsilon_{1+(1)}$ corresponds to $\mathrm{id}\oplus^{\mathbb{A}^1}g$. And we can apply the extended result to $\mathrm{id}\oplus^{\mathbb{A}^1}g$ and get $\mathrm{id}\oplus^{\mathbb{A}^1}g= \mathrm{id}\oplus^{\mathrm{N}}g$. Thus we can determine the $\mathbb{A}^1$-sum in this case explicitly. Then we can give an explicit sequence of pointed naive $\mathbb{A}^1$-homotopies between $f$ and $\mathrm{id}\oplus^{\mathrm{N}}g$. Therefore we get $ q_{1+(1)}=1_{1+(1)}-\epsilon_{1+(1)}$. Using this relation we can then show that $h_{1+(2)}\circ\eta_{1+(2)}$ and $ \eta_{1+(2)}\circ h_{1+(3)}$ are nulhomotopic (Proposition~\ref{Proposition 3.4.9}). Thus the Toda bracket $\{h_{1+(2)}, \eta_{1+(2)}, h_{1+(3)}\}$ is defined. Using similar methods we can show that $\{\eta_{1+(2)}, h_{1+(3)}, \eta_{1+(3)}\}$ is defined, too.\\

Moreover these two Toda brackets \begin{align*}
	\{h_{1+(2)}, \eta_{1+(2)}, h_{1+(3)}\}
\end{align*} and \begin{align*}\{\eta_{1+(2)}, h_{1+(3)}, \eta_{1+(3)}\}\end{align*} are not trivial, in the sense that they do not contain the homotopy classes of constant morphisms (Proposition~\ref{Proposition 3.4.10}). This can be proved by using the complex realization.\\
	
Furthermore we give an application of motivic Toda brackets. The Toda bracket \begin{align*}	
	\{\eta_{1+(2)}, h_{1+(3)}, \eta_{1+(3)}\}
\end{align*} is also defined over fields, since the arguments described above work for fields, too.  Let $k$ be a field of characteristic zero. We define $\pi_{s+(w)}\mathcal{E}$ to be the group $\mathcal{H}_{\bullet}(k)(S^{s+(w)}, \mathcal{E})$ for any pointed motivic spaces $\mathcal{E}$, $s>0 $ and $w\geq 0$. Let $\nu'\in \{\eta_{1+(2)}, h_{1+(3)}, \eta_{1+(3)}\}\subseteq \pi_{2+(4)}S^{1+(2)}$ be any element contained in this Toda bracket. Since $S^{1+(2)}$ is isomorphic to $\mathrm{SL}_2$, we can also consider $\nu'$ as an element of $\pi_{2+(4)}\mathrm{SL}_2$. Recall that there is an inclusion $\mathrm{SL}_2\hookrightarrow \mathrm{SL}_3 $. We can show that the image of $\nu'$ under the inclusion in $\pi_{2+(4)}\mathrm{SL}_3$ generates the $2$-primary component of this group (Proposition~\ref{Proposition 3.4.11}).\\	
	
At the end of the paper we also  construct another Toda bracket\begin{align*}
 \{\Delta_{1+(3)}, h_{1+(2)}, \eta_{1+(2)}\} 	
 \end{align*} over the base $\mathrm{Spec}\mathbb{Z} $, where $\Delta_{1+(3)} $ is a suspension of  the diagonal map $\Delta_{(2)}:\mathbb{G}_{m}\rightarrow \mathbb{G}_{m}\wedge \mathbb{G}_{m} $ defined by $ x\mapsto x\wedge x$. The interesting point is that the complex realization of this Toda bracket is trivial, but the Toda bracket itself is actually not trivial.\\

\section*{Acknowledgement}
I would like to thank my advisor Oliver R\"{o}ndigs for suggesting the topic and his comments on drafts of this paper. I also want to thank Paul Catala and Federico Mocchetti for proofreading parts of this paper.
	
	\newpage
\section{Motivic Toda brackets}
\
\

	\subsection{Quillen equivalences}\label{Quillen}
	\
	In this section we would like to establish Quillen equivalences between some model structures on $\mathrm{sPre}(S)$ and some appropriate model structures on the category of presheaves on $\mathcal{S}\mathrm{m}_{S}$ with values in $\Delta$-generated topological spaces (see Section~\ref{delta}).\\
	
	Let $\mathrm{Pre}_{\Delta}(S)$ denote the category of presheaves on $\mathcal{S}\mathrm{m}_{S}$ with values in $\Delta$-generated topological spaces.
	By Propostion~\ref{Proposition A.1.2} the geometric realization of a simplicial set is $\Delta$-generated. Therefore by applying the usual geometric realization functor sectionwise we get a functor \begin{align*}
		|\cdot|: \mathrm{sPre}(S)\rightarrow \mathrm{Pre}_{\Delta}(S)\ \ \cdot
	\end{align*}
	Recall that the geometric realization functor for simplicial sets has a right adjoint $Sing: \mathcal{T}op \rightarrow \mathrm{s}\mathcal{S}\mathrm{et} $. Then if we apply this functor again sectionwise, we obtain a right adjoint for $|\cdot| $ and we denote the right adjoint still by $Sing$. Hence we have the following adjoint pair\begin{align*}
		|\cdot|: \mathrm{sPre}(S)\rightleftarrows \mathrm{Pre}_{\Delta}(S): Sing \ \ \cdot
	\end{align*}
	As for $ \mathrm{sPre}(S)$ we can also define the global projective model structure on $\mathrm{Pre}_{\Delta}(S)$. First recall that there is a cofibrantly generated model structure on $\mathcal{T}op_{\Delta}$ with the same weak equivalences, fibrations and cofibrations as for the classical Quillen model structure on $\mathcal{T}op$. In particular, these two model structures have the same generating cofibrations and the same generating acyclic cofibrations; they are both proper, too. The existence of the global projective model structure on $\mathrm{Pre}_{\Delta}(S)$ follows from \cite[Proposition A.2.8.2]{lurie2009higher}. The weak equivalences and fibrations are defined sectionwise. The cofibrations are the morphisms which have the left lifting property with respect to acyclic fibrations. This model structure is combinatiorial and proper. In particular, every object in $\mathrm{Pre}_{\Delta}(S)$ is fibrant.\\
	
	Using \cite[Lemma 2.4]{goerss2009simplicial} we can make $\mathrm{Pre}_{\Delta}(S)$ into a simplicial category. We define a functor $\cdot\otimes\cdot: \mathrm{Pre}_{\Delta}(S)\times\mathrm{s}\mathcal{S}\mathrm{et}\rightarrow \mathrm{Pre}_{\Delta}(S)$ by $\mathcal{X}\otimes K:= \mathcal{X}\times |K|$. This functor satisfies all conditions in \cite[Lemma 2.4]{goerss2009simplicial}. Therefore the simplicial mapping space is given by \begin{align*}
		\mathbf{Map}_{\Delta}(\mathcal{X},\mathcal{Y})_n=\mathrm{Pre}_{\Delta}(S)(\mathcal{X}\times |\Delta^n|,\mathcal{Y}) 
	\end{align*}for all $\mathcal{X},\mathcal{Y}\in \mathrm{Pre}_{\Delta}(S)$ and $n\geq 0$. Now we show that $\mathrm{Pre}_{\Delta}(S)_{global\ proj}$ is actually a simplicial model category.\\
	
	\begin{prop} The global projective model structure on $\mathrm{Pre}_{\Delta}(S)$ is simplicial.
		\end{prop}
	
\begin{proof} We want to show that the axiom $\mathbf{SM7}$ (see \cite[3.1]{goerss2009simplicial}) holds. By \cite[Proposition 3.11]{goerss2009simplicial} the axiom is equivalent to the following condition\begin{itemize}
		\item Let $j:\mathcal{X}\rightarrow\mathcal{Y}$ be a cofibration in $\mathrm{Pre}_{\Delta}(S)_{global\ proj}$ and $i:K\rightarrow L$ be a cofibration of simplicial sets. Then\begin{align*}
			j\otimes i: \mathcal{X}\otimes L\sqcup_{\mathcal{X}\otimes K}\mathcal{Y}\otimes K\rightarrow\mathcal{Y}\otimes L
		\end{align*}is a cofibration in $\mathrm{Pre}_{\Delta}(S)_{global\ proj}$ which is a weak equivalence if $i$ or $j$ is.\\
	\end{itemize} 
	\
	Since $\mathrm{Pre}_{\Delta}(S)_{global\ proj}$ is cofibrantly generated, it suffices to check the above condition on 
	$j\otimes i$ where $j$ and $i$ run over the generating cofibrations (or acyclic cofibrations, as needed) for $\mathrm{Pre}_{\Delta}(S)_{global\ proj}$ and for $\mathrm{s}\mathcal{S}\mathrm{et}$. By \cite[Theorem 11.6.1]{hirschhorn2003model} the generating cofibrations consist of the following morphisms\begin{align*}
		U\times S^{n-1}\rightarrow U\times D^n 
	\end{align*}
	for $n\geq 0$ and all $U\in \mathcal{S}\mathrm{m}_{S}$. The generating acyclic cofibrations are the morphisms\begin{align*}
		U\times D^n\rightarrow U\times D^n\times I
	\end{align*} for all $n\geq0$ and all $U\in \mathcal{S}\mathrm{m}_{S}$ where $D^n\rightarrow D^n\times I $ is the composition $D^n\cong D^n\times\{0\}\hookrightarrow D^n\times I$.\\
	
	Now let $i:K\rightarrow L$ be a cofibration of simplicial sets and $j:U\times S^{n-1}\rightarrow U\times D^n $ an arbitrary generating cofibration. Suppose we have a commutative diagram 
	\[\begin{tikzcd}
		{U\times S^{n-1}\times|L|\sqcup_{U\times S^{n-1}\times|K|}U\times D^n\times|K|} & {\mathcal{X}} & {} \\
		{U\times D^n\times|L|} & {\mathcal{Y}} & {}
		\arrow[from=1-1, to=1-2]
		\arrow[from=2-1, to=2-2]
		\arrow["{{j\otimes i}}"', from=1-1, to=2-1]
		\arrow["f", from=1-2, to=2-2]
		\arrow["{(\ast)}", draw=none, from=1-3, to=2-3]
	\end{tikzcd}\] where $f$ is an acyclic fibration. Since $U$ is the presheaf represented by the corresponding scheme, the above diagram is uniquely determined by the sections at $U$ and the values at $\mathrm{id}_{U}$:
	\[\begin{tikzcd}
		{ S^{n-1}\times|L|\sqcup_{ S^{n-1}\times|K|} D^n\times|K|} & {\mathcal{X}(U)} & {} \\
		{ D^n\times|L|} & {\mathcal{Y}(U)} & {} & {}
		\arrow[from=1-1, to=1-2]
		\arrow[from=1-1, to=2-1]
		\arrow[from=2-1, to=2-2]
		\arrow["{{f(U)}}", from=1-2, to=2-2]
		\arrow["{(\ast\ast)}", draw=none, from=1-3, to=2-3]
	\end{tikzcd}\]
	Now $S^{n-1}\times|L|\sqcup_{ S^{n-1}\times|K|} D^n\times|K|\rightarrow  D^n\times|L|$ is a cofibration in $\mathcal{T}op_{\Delta}$ because it is a subcomplex inclusion. Hence there exists a lift in $(\ast\ast)$
	\[\begin{tikzcd}
		{ S^{n-1}\times|L|\sqcup_{ S^{n-1}\times|K|} D^n\times|K|} & {\mathcal{X}(U)} \\
		{ D^n\times|L|} & {\mathcal{Y}(U)}
		\arrow[from=1-1, to=1-2]
		\arrow[from=1-1, to=2-1]
		\arrow[from=2-1, to=2-2]
		\arrow["{f(U)}", from=1-2, to=2-2]
		\arrow[dashed, from=2-1, to=1-2]
	\end{tikzcd}\]
	Correspondingly, there is also a lift in the original diagram $(\ast)$
	\[\begin{tikzcd}
		{ U\times S^{n-1}\times|L|\sqcup_{U\times S^{n-1}\times|K|}U\times D^n\times|K|} & {\mathcal{X}} \\
		{ U\times D^n\times|L|} & {\mathcal{Y}}
		\arrow[from=1-1, to=1-2]
		\arrow[from=1-1, to=2-1]
		\arrow[from=2-1, to=2-2]
		\arrow["f", from=1-2, to=2-2]
		\arrow[dashed, from=2-1, to=1-2]
	\end{tikzcd}\]
	The same argument works also for the generating acyclic cofibrations.
	Note that if $i$ is in addition also a weak equivalence, then $S^{n-1}\times|L|\sqcup_{ S^{n-1}\times|K|} D^n\times|K|\rightarrow  D^n\times|L|$ is a weak equivalence, too. In particular, it has the left lifting property with respect to all Serre fibrations. In this case $j\otimes i$ is an acyclic cofibration. If $j:U\times D^n\rightarrow U\times D^n\times I $ is a generating cofibration, then the cofibration
	$ U\times D^n\times|L|\sqcup_{U\times D^n\times|K|}U\times D^n\times I\times|K|\rightarrow U\times D^n\times I\times|L|$ is also a weak equivalence, therefore $j\otimes i$ is an acyclic cofibration.\\
	\end{proof}
	
	\begin{prop}The adjunction\begin{align*}	|\cdot|: \mathrm{sPre}(S)_{global\ proj}\rightleftarrows \mathrm{Pre}_{\Delta}(S)_{global\ proj}: Sing
		\end{align*} is a Quillen equivalence.\\
		\end{prop}
	
	\begin{proof}The singular complex functor for topological spaces preserves weak equivalences and Serre fibrations. Since the functor $Sing$ is defined sectionwise and the weak equivalences and fibrations for both model structures are also defined sectionwise, we see that $Sing$ is a right Quillen functor. Let $\mathcal{X}$ be a simplicial presheaf on $\mathcal{S}\mathrm{m}_{S}$ and $\mathcal{Y}\in \mathrm{Pre}_{\Delta}(S) $. Then $\mathcal{X}\rightarrow Sing(|\mathcal{X}|)$ and $|Sing(\mathcal{Y})|\rightarrow\mathcal{Y}$ are sectionwise weak equivalences. Hence the above adjunction is a Quillen equivalence.\\
		\end{proof}
	
	According to \cite[Lemma 4.3]{Blander2001LocalPM} the local projective model structure on $ \mathrm{sPre}(S)$ is a left Bousfield localization of the global projective localization with respect to the morphisms $P(\alpha)\rightarrow X$, where $X$ is a smooth scheme over $S$ and $\alpha$ is the homotopy pushout of an elementary distinguished square of $X$ in the global projective model structure. Note that in order to ensure that such a left Bousfield localization exists, we implicitly choose a skeleton of $\mathcal{S}\mathrm{m}_{S}$ so that we can localize at a set. Correspondingly, we can localize at the morphisms $|P(\alpha)|\rightarrow |X|$ and we call the resulting model structure on $\mathrm{Pre}_{\Delta}(S)$ the local projective model structure. It follows from Proposition \cite[Proposition A.3.7.3]{lurie2009higher} that this model structure is again simplicial, combinatorial and left proper. By \cite[Proposition 3.3.20]{hirschhorn2003model} we get the following statement.\\
	
	\begin{prop} The adjunction\begin{align*}	|\cdot|: \mathrm{sPre}(S)_{local\ proj}\rightleftarrows \mathrm{Pre}_{\Delta}(S)_{local\ proj}: Sing
		\end{align*} is a Quillen equivalence.\\
		\end{prop}
	
	Now we characterize the fibrant objects in $\mathrm{Pre}_{\Delta}(S)_{local\ proj}$.\\
	
	\begin{lem}
		\label{lemma2.1.4}
		An object $\mathcal{X}\in\mathrm{Pre}_{\Delta}(S) $ is local projective fibrant if and only if $Sing(\mathcal{X})$ is fibrant in $\mathrm{sPre}(S)_{local\ proj}$.\\
		\end{lem}

\begin{proof} If $\mathcal{X}$ is local projective fibrant, then it is clear the $Sing(\mathcal{X})$ is also local projective fibrant. Suppose now that $Sing(\mathcal{X})$ is a fibrant object in $\mathrm{sPre}(S)_{local\ proj}$. We want to show that $\mathcal{X}$ is local projective fibrant. First recall that by \cite[Proposition 3.4.1]{hirschhorn2003model} the fibrant objects in $\mathrm{sPre}(S)_{local\ proj}$ are precisely the objects $\mathcal{Y}$ such that \begin{align*}
		\mathbf{Map}_{\Delta}(|X|,\mathcal{Y})\rightarrow \mathbf{Map}_{\Delta}(|P(\alpha)|,\mathcal{Y})
	\end{align*} is a weak equivalence of simplicial sets for all morphisms $|P(\alpha)|\rightarrow |X|$.\\
	
	Since $Sing(\mathcal{X})$ is fibrant, we have that by \cite[Theorem 4.1]{Blander2001LocalPM} the square
	\[\begin{tikzcd}
		{Sing(\mathcal{X})(X)} & {Sing(\mathcal{X})(V)} \\
		{Sing(\mathcal{X})(U)} & {Sing(\mathcal{X})(U\times_{X}V)}
		\arrow[from=1-1, to=1-2]
		\arrow[from=1-1, to=2-1]
		\arrow[from=2-1, to=2-2]
		\arrow[from=1-2, to=2-2]
	\end{tikzcd}\]
	is a homotopy pullback of simplicial sets for every elementary distinguished square \[\begin{tikzcd}
		{U\times_{X}V} & V \\
		U & X
		\arrow[from=1-1, to=1-2]
		\arrow[from=1-1, to=2-1]
		\arrow[from=2-1, to=2-2]
		\arrow[from=1-2, to=2-2]
	\end{tikzcd}\](including the degenerate sqaure).\\
	
	By construction, the simplicial presheaf $P(\alpha)$ is a homotopy pushout for the upper half of an elementary distinguished square $\alpha$
	\[\begin{tikzcd}
		{U\times_{X}V} & V \\
		U 
		\arrow[from=1-1, to=1-2]
		\arrow[from=1-1, to=2-1]
	\end{tikzcd}\]
	in the global projective model structure. Since the geometric realization functor $|\cdot|$ preserves homotopy pushouts, $|P(\alpha)|$ is a homotopy pushout of 
	\[\begin{tikzcd}
		{|U\times_{X}V|} & {|V|} \\
		{|U|} && \cdot
		\arrow[from=1-1, to=1-2]
		\arrow[from=1-1, to=2-1]
	\end{tikzcd}\]
	Note that the contravariant functor $\mathbf{Map}_{\Delta}(-,\mathcal{X})$ preserves limits; hence it suffices to show that the following diagram 
	\[\begin{tikzcd}
		{\mathbf{Map}_{\Delta}(|X|,\mathcal{X})} & {\mathbf{Map}_{\Delta}(|Y|,\mathcal{X})} & {} \\
		{\mathbf{Map}_{\Delta}(|U|,\mathcal{X})} & {\mathbf{Map}_{\Delta}(|U\times_{X}V|,\mathcal{X})} & {}
		\arrow[from=1-1, to=1-2]
		\arrow[from=1-1, to=2-1]
		\arrow[from=2-1, to=2-2]
		\arrow[from=1-2, to=2-2]
		\arrow["{(\ast)}", draw=none, from=1-3, to=2-3]
	\end{tikzcd}\]
	is a homotopy pullback of simplicial sets. Let $T$ be a smooth scheme over $S$. By definition we have $\mathbf{Map}_{\Delta}(|T|,\mathcal{X})_n=\mathrm{Pre}_{\Delta}(S)(|T|\times|\Delta^n|, \mathcal{X}) $ for all $n\geq 0$. By adjunction and the fact that $|\cdot|$ preserves finite products we get $\mathrm{Pre}_{\Delta}(S)(|T|\times|\Delta^n|)\cong\mathrm{sPre}(S)(T\times \Delta^n, Sing(\mathcal{X}))=\mathbf{Map}(T, Sing(\mathcal{X}))_n\cong Sing(\mathcal{X}(T))_n $. Therefore the diagram $(\ast)$ is canonically isomorphic to the diagram
	\[\begin{tikzcd}
		{Sing(\mathcal{X})(X)} & {Sing(\mathcal{X})(Y)} \\
		{Sing(\mathcal{X})(U)} & {Sing(\mathcal{X})(U\times_{X}V)}
		\arrow[from=1-1, to=1-2]
		\arrow[from=1-1, to=2-1]
		\arrow[from=2-1, to=2-2]
		\arrow[from=1-2, to=2-2]
	\end{tikzcd}\] which is a homotopy pullback of simplicial sets.\\
	\end{proof}
	
	Next we characterize the weak equivalences in $\mathrm{Pre}_{\Delta}(S)$.\\
	
	\begin{lem}
		\label{lemma 2.1.5} Let $f:\mathcal{X}\rightarrow \mathcal{Y}$ be a morphism in $\mathrm{Pre}_{\Delta}(S)$. Then $f$ is a weak equivalence in $\mathrm{Pre}_{\Delta}(S)_{local\ proj}$ if and only if $Sing(f)$ is a local weak equivalence in $\mathrm{sPre}(S)_{local\ proj}$.\\
		\end{lem}
	
	\begin{proof} One direction is easy. We assume that $Sing(f)$ is a local weak equivalence. Let $Q(-)$ be a functorial cofibrant replacement functor for the global projective model structure. Then $Q(Sing(\mathcal{X}))\rightarrow Q(Sing(\mathcal{Y}))$ is a local weak equivalence. By construction the morphism $Q(Sing(\mathcal{Z}))\rightarrow\mathcal{Z}$ is a sectionwise weak equivalence for any $\mathcal{Z}\in\mathrm{Pre}_{\Delta}(S) $. In particular, $|Q(Sing(\mathcal{Z}))|\rightarrow|\mathcal{Z}| $  is a sectionwise weak equivalence. Now we have the following commutative diagram
	\[\begin{tikzcd}
		{|Q(Sing(\mathcal{X}))|} & {|Q(Sing(\mathcal{Y}))|} \\
		{|Sing(\mathcal{X})|} & {|Sing(\mathcal{Y})|} \\
		{\mathcal{X}} & {\mathcal{Y}} & \cdot
		\arrow["\sim", from=1-1, to=1-2]
		\arrow["f", from=3-1, to=3-2]
		\arrow["\sim", from=1-1, to=2-1]
		\arrow["\sim", from=2-1, to=3-1]
		\arrow["\sim", from=1-2, to=2-2]
		\arrow["\sim", from=2-2, to=3-2]
		\arrow[from=2-1, to=2-2]
	\end{tikzcd}\]
	It follows that $f$ is a weak equivalence.\\
	
	Assume now that $f$ is a weak equivalence in $\mathrm{Pre}_{\Delta}(S)_{local\ proj}$ . Then $|Sing(f)|$ is also a weak equivalence. We would like to show that \begin{align*}
		Q(Sing(f))^{\ast}:\mathbf{Map}(Q(Sing(\mathcal{Y})),\mathcal{A})\rightarrow \mathbf{Map}(Q(Sing(\mathcal{X})),\mathcal{A})
	\end{align*} is a weak equivalence of simplicial sets for any fibrant object $\mathcal{A}\in\mathrm{sPre}(S)_{local\ proj} $. In fact, it suffices to show this for all local injective fibrant objects. Let $\mathcal{B}\in \mathrm{sPre}(S)$ be a local injective fibrant object and $\mathcal{Z}\in\mathrm{Pre}_{\Delta}(S) $. Since the local injective model structure is simplicial and $Q(Sing(\mathcal{Z}))$$\rightarrow$$\mathcal{Z}$ is a local weak equivalence between local injective cofibrant objects, the morphism \begin{align*}
		\mathbf{Map}(Sing(\mathcal{Z}),\mathcal{B})\rightarrow \mathbf{Map}(Q(Sing(\mathcal{Z})),\mathcal{B})
	\end{align*} is a weak equivalence. Hence we only need to show that \begin{align*}
		Sing(f)^{\ast}:\mathbf{Map}(Sing(\mathcal{Y}),\mathcal{B})\rightarrow \mathbf{Map}(Sing(\mathcal{X}),\mathcal{B})
	\end{align*} is a weak equivalence for all local injective fibrant $\mathcal{B}\in \mathrm{sPre}(S)$. Take a local injective fibrant $\mathcal{B}\in \mathrm{sPre}(S)$. Let $Sing(|\mathcal{B}|)\rightarrow RSing(|\mathcal{B}|)$ be a local injective fibrant replacement for $Sing(|\mathcal{B}|)$. We consider the following composite \begin{align*}
		\mathcal{B}\rightarrow Sing(|\mathcal{B}|)\rightarrow RSing(|\mathcal{B}|).
	\end{align*} The composite is a local weak equivalence between local injective fibrant objects. By \cite[Corollary 5.13 ]{jardine2015local} the composite is a sectionwise weak equivalence. Since the unit $\mathcal{B}\rightarrow Sing(|\mathcal{B}|) $ is a sectionwise weak equivalence, it follows that $ Sing(|\mathcal{B}|)\rightarrow RSing(|\mathcal{B}|)$ is a sectionwise weak equivalence. By construction, $Sing(|\mathcal{B}|)$ is sectionwise fibrant and $Sing(|\mathcal{B}|)$ is also local projective fibrant. Therefore $Sing(|\mathcal{B}|)$ sends every elementary distinguished square to a homotopy pullback. Since $Sing(|\mathcal{B}|)\rightarrow RSing(|\mathcal{B}|) $ is a sectionwise weak equivalence, we see that $Sing(|\mathcal{B}|)$ is local projective fibrant. It follows from Lemma~\ref{lemma2.1.4} that $|\mathcal{B}|$ is fibrant in $ \mathrm{Pre}_{\Delta}(S)_{local\ proj}$.\\
	
	Considering the following commutative diagram\[\begin{tikzcd}
		{\mathbf{Map}(Sing(\mathcal{Y}),\mathcal{B})} & {\mathbf{Map}(Sing(\mathcal{X}),\mathcal{B})} \\
		{\mathbf{Map}(Sing(\mathcal{Y}),Sing(|\mathcal{B}|))} & {\mathbf{Map}(Sing(\mathcal{X}),Sing(|\mathcal{B}|))} \\
		{\mathbf{Map}_{\Delta}(|Sing(\mathcal{Y})|,|\mathcal{B}|)} & {\mathbf{Map}_{\Delta}(|Sing(\mathcal{X})|,|\mathcal{B}|)} & {,}
		\arrow["\sim", from=1-1, to=2-1]
		\arrow["{{Sing(f)^\ast}}", from=1-1, to=1-2]
		\arrow["\sim", from=1-2, to=2-2]
		\arrow["\cong", from=2-1, to=3-1]
		\arrow["\cong", from=2-2, to=3-2]
		\arrow["\sim", from=3-1, to=3-2]
		\arrow[from=2-1, to=2-2]
	\end{tikzcd}\]
	we see that $Sing(f)^\ast $ is a weak equivalence of simplicial sets.\\
	\end{proof}
	
	Analogously, we can also construct the global injective model structure on  $\mathrm{Pre}_{\Delta}(S)$. Again by \cite[Proposition A.2.8.2]{lurie2009higher} there is a combinatorial and left proper model structure on $\mathrm{Pre}_{\Delta}(S)$ called the global injective model structure with weak equivalences and cofibrations defined objectwise. We also get a Quillen equivalence
	\begin{center}
		$\mathrm{id}: \mathrm{Pre}_{\Delta}(S)_{global\ proj}\rightleftarrows \mathrm{Pre}_{\Delta}(S)_{global\ inj}:\mathrm{id}$ \ \ .\\
	\end{center}
	\
	
	\begin{lem} The global injective model structure on $\mathrm{Pre}_{\Delta}(S)$ is simplicial.\\
		\end{lem}
	
\begin{proof} As in the proof of Proposition 2.1.1 we only need to show that $	j\otimes i$ is a sectionwise cofibration for any $j:\mathcal{X}\rightarrow\mathcal{Y}$ sectionwise cofibration in $\mathrm{Pre}_{\Delta}(S)$ and any cofibration $i: K\rightarrow L$ of simplicial sets. Furthermore $	j\otimes i$ is a weak equivalence if $i $
	or $j$ is.\\
	
	Let $U\in \mathcal{S}\mathrm{m}_{S} $. Then we want to show that \begin{align*}
		(j\otimes i)(U): \mathcal{X}(U)\times|L|\sqcup_{\mathcal{X}(U)\times|K|}\mathcal{Y}(U)\times|K|\rightarrow\mathcal{Y}(U)\times|L|
	\end{align*}
	is a cofibration of $\Delta$-generated spaces. Now we use the fact that the Quillen model structure on $\mathcal{T}op_{\Delta}$ is cofibrantly generated. The generating cofibrations are the maps $S^{n-1}\rightarrow D^n$ for all $n\geq 0$. The generating acyclic cofibrations consist of the maps $D^n\rightarrow D^n\times I$ for all $n\geq 0$. Therefore we can use the same argument as in the proof of Proposition 2.1.1 to show that $(j\otimes i)(U)$ is a cofibration of $\Delta$-generated spaces and it is a weak equivalence if $j(U)$ or $i$ is.\\
	\end{proof}
	
	In particular, we have the following Quillen equivalence \begin{align*}
		|\cdot|: \mathrm{sPre}(S)_{global\ inj}\rightleftarrows \mathrm{Pre}_{\Delta}(S)_{global\ inj}: Sing \ \ \cdot
	\end{align*} By \cite[Theorem 6.2]{dugger_hollander_isaksen_2004} the local injective model structure on $\mathrm{sPre}(S)$ is the left Bousfield localization of the global injective model structure at a certain set consisting of hypercovers. Now we can apply the geometric realization functor at this set and then localize $\mathrm{Pre}_{\Delta}(S)_{global\ inj}$ at it. The resulting new model structure on $\mathrm{Pre}_{\Delta}(S)$ is called the local injective model structure. It is again simplicial and left proper. By \cite[Proposition 3.3.20]{hirschhorn2003model} we get the Quillen equivalence	\begin{align*}
		|\cdot|: \mathrm{sPre}(S)_{local\ inj}\rightleftarrows \mathrm{Pre}_{\Delta}(S)_{local\ inj}: Sing \ \ \cdot
	\end{align*}	 
	Next we consider the identity functor \begin{align*}
		\mathrm{id}: \mathrm{Pre}_{\Delta}(S)_{local\ proj} \rightarrow \mathrm{Pre}_{\Delta}(S)_{local\ inj}\ \ \cdot \\
	\end{align*}
	\begin{lem}The identity functor is a left Quillen functor. \\
	\end{lem}
	
	\begin{proof} Note that the identity functor is a left Quillen functor for the global model structures. Let $f:\mathcal{X}\rightarrow\mathcal{Y}$ be a local projective cofibration. Therefore $f$ is a global injective cofibration and hence also a local injective cofibration. Next we show that the identity functor preserves local weak equivalences. Let $g:\mathcal{A}\rightarrow\mathcal{B}$ be a local projective weak equivalence. Then by Lemma~\ref{lemma 2.1.5} $Sing(g)$ is a local weak equivalence of simplicial presheaves. Note that the local model structures on $\mathrm{sPre}(S)$ have the same weak equivalences. Since every object is cofibrant in $\mathrm{sPre}(S)$ and $|\cdot|$ is a left Quillen functor, we have that $|Sing(g)|$ is a local weak equivalence in $\mathrm{Pre}_{\Delta}(S)_{local\ inj}$. Then it follows that $g$ is a local weak equivalence in $\mathrm{Pre}_{\Delta}(S)_{local\ inj}$.\\
		\end{proof}
	
	Altogether we have the following diagram of Quillen equivalences
	\[\begin{tikzcd}
		{\mathrm{sPre}(S)_{local\ proj}} & {\mathrm{Pre}_{\Delta}(S)_{local\ proj}} \\
		{\mathrm{sPre}(S)_{local\ inj}} & {\mathrm{Pre}_{\Delta}(S)_{local\ inj}}
		\arrow["{|\cdot|}", from=1-1, to=1-2]
		\arrow["Sing", shift left=2, from=1-2, to=1-1]
		\arrow["{|\cdot|}", from=2-1, to=2-2]
		\arrow["{\mathrm{id}}"', from=1-1, to=2-1]
		\arrow[shift right=2, from=2-1, to=1-1]
		\arrow[from=1-2, to=2-2]
		\arrow["{\mathrm{id}}"', shift right=2, from=2-2, to=1-2]
		\arrow["Sing", shift left=2, from=2-2, to=2-1]
	\end{tikzcd}\]

	\begin{remark} 
		Note that the identity functor $\mathrm{id}: \mathrm{Pre}_{\Delta}(S)_{local\ inj} \rightarrow \mathrm{Pre}_{\Delta}(S)_{local\ proj}$  preserves local weak equivalences, too. Let $f:\mathcal{X}\rightarrow\mathcal{Y}$ be a local weak equivalence in $\mathrm{Pre}_{\Delta}(S)_{local\ inj}$. Taking a functorial fibrant replacement $R(-)$, we get a commutative diagram
	\[\begin{tikzcd}
		{\mathcal{X}} & {\mathcal{Y}} \\
		{R\mathcal{X}} & {R\mathcal{Y}} & \cdot
		\arrow["f", from=1-1, to=1-2]
		\arrow[from=1-1, to=2-1]
		\arrow[from=1-2, to=2-2]
		\arrow["Rf"', from=2-1, to=2-2]
	\end{tikzcd}\]
	The vertical arrows are local injective weak equivalences. The morphism $Rf$ is a local injective weak equivalence between local injective fibrant objects. Therefore it is a sectionwise weak equivalence. Take a functorial local projective cofibrant replacement $Q(-)$. Then $Q\mathcal{X}\rightarrow \mathcal{X}\rightarrow R\mathcal{X} $ is a local projective weak equivalence, since\begin{align*}
		\mathrm{id}: \mathrm{Pre}_{\Delta}(S)_{local\ proj} \rightarrow \mathrm{Pre}_{\Delta}(S)_{local\ inj}:\mathrm{id}
	\end{align*} is a Quillen equivalence. It follows that $\mathcal{X}\rightarrow R\mathcal{X}$ is a local projective weak equivalence. By a similar argument we can show that $\mathcal{Y}\rightarrow R\mathcal{Y}$ is also  a local projective weak equivalence. It follows now that $f$ is a local projective weak equivalence. We conclude that the local model structures on $\mathrm{Pre}_{\Delta}(S)$ have the same weak equivalences as expected.\\
	\end{remark}
	
	Finally we can consider $\mathbb{A}^1$-localizations. We localize both local model structures on $\mathrm{Pre}_{\Delta}(S)$ at the class of morphisms $\{|X\times\mathbb{A}^1|\rightarrow |X|; X\in \mathcal{S}\mathrm{m}_{S}\}$. We call the new model structures the $\mathbb{A}^1$-local projective/injective model structure. It follows directly from the construction that these two $\mathbb{A}^1$-local model structures have the same weak equivalences. Especially, we have again a diagram of Quillen equivalences
	\[\begin{tikzcd}
		{\mathrm{sPre}(S)_{\mathbb{A}^1-local\ proj}} & {\mathrm{Pre}_{\Delta}(S)_{\mathbb{A}^1-local\ proj}} \\
		{\mathrm{sPre}(S)_{\mathbb{A}^1-local\ inj}} & {\mathrm{Pre}_{\Delta}(S)_{\mathbb{A}^1-local\ inj}} & \cdot
		\arrow["{{|\cdot|}}", from=1-1, to=1-2]
		\arrow["Sing", shift left=2, from=1-2, to=1-1]
		\arrow["{{|\cdot|}}", from=2-1, to=2-2]
		\arrow["Sing", shift left=2, from=2-2, to=2-1]
		\arrow["{{\mathrm{id}}}"', from=1-1, to=2-1]
		\arrow[shift right=2, from=2-1, to=1-1]
		\arrow[from=1-2, to=2-2]
		\arrow["{{\mathrm{id}}}"', shift right=2, from=2-2, to=1-2]
	\end{tikzcd}\]
	
	\begin{remark}  Using \cite[Proposition 1.1.8, 1.3.5, 1.3.17]{hovey2007model} we obtain pointed versions of the model structures above on $\mathrm{Pre}_{\Delta}(S)_{\ast}$. Analogously, we have two diagrams of Quillen equivalences
	\[\begin{tikzcd}
		{\mathrm{sPre}(S)_{\ast \ local\ proj}} & {\mathrm{Pre}_{\Delta}(S)_{\ast\ local\ proj}} \\
		{\mathrm{sPre}(S)_{\ast \ local\ inj}} & {\mathrm{Pre}_{\Delta}(S)_{\ast\ local\ inj}}
		\arrow["{|\cdot|}", from=1-1, to=1-2]
		\arrow["Sing", shift left=2, from=1-2, to=1-1]
		\arrow["{|\cdot|}", from=2-1, to=2-2]
		\arrow["Sing", shift left=2, from=2-2, to=2-1]
		\arrow["{\mathrm{id}}"', from=1-1, to=2-1]
		\arrow[shift right=2, from=2-1, to=1-1]
		\arrow[from=1-2, to=2-2]
		\arrow["{\mathrm{id}}"', shift right=2, from=2-2, to=1-2]
	\end{tikzcd}\]
	and
	\[\begin{tikzcd}
		{\mathrm{sPre}(S)_{\ast \ \mathbb{A}^1-local\ proj}} & {\mathrm{Pre}_{\Delta}(S)_{\ast\ \mathbb{A}^1-local\ proj}} \\
		{\mathrm{sPre}(S)_{\ast \ \mathbb{A}^1-local\ inj}} & {\mathrm{Pre}_{\Delta}(S)_{\ast\ \mathbb{A}^1-local\ inj}} & \cdot
		\arrow["{{|\cdot|}}", from=1-1, to=1-2]
		\arrow["Sing", shift left=2, from=1-2, to=1-1]
		\arrow["{{|\cdot|}}", from=2-1, to=2-2]
		\arrow["Sing", shift left=2, from=2-2, to=2-1]
		\arrow["{{\mathrm{id}}}"', from=1-1, to=2-1]
		\arrow[shift right=2, from=2-1, to=1-1]
		\arrow[from=1-2, to=2-2]
		\arrow["{{\mathrm{id}}}"', shift right=2, from=2-2, to=1-2]
	\end{tikzcd}\]\\[1cm]
	\end{remark}
	
	\subsection {Toda brackets}\label{subsec:Toda}
	\
	
	We can now construct Toda brackets in unstable motivic homotopy theory. For the construction we use the $\mathbb{A}^1$-local injective model structure on $\mathrm{sPre}(S)_{\ast}$. First we want to recall some notions. A functorial cylinder object $cyl(-)$ is an endofunctor of $\mathrm{sPre}(S)_{\ast}$ together with a cofibration $i_0+i_1: \mathcal{X}\vee \mathcal{X}\rightarrow cyl(\mathcal{X})$ and a weak equivalence $s:cyl(\mathcal{X})\rightarrow \mathcal{X}$ for each $\mathcal{X}\in \mathrm{sPre}(S)_{\ast} $ such that the composition $s\circ i_0+i_1 $ is the fold map $\nabla:\mathcal{X}\vee \mathcal{X}\rightarrow \mathcal{X} $ and such that for every morphism $f:\mathcal{X}\rightarrow \mathcal{Y}$ the following diagram \[\begin{tikzcd}
		{\mathcal{X}\vee \mathcal{X}} & {cyl(\mathcal{X})} & {\mathcal{X}} \\
		{\mathcal{Y}\vee \mathcal{Y}} & {cyl(\mathcal{Y})} & {\mathcal{Y}}
		\arrow["{i_0+i_1}", from=1-1, to=1-2]
		\arrow["s", from=1-2, to=1-3]
		\arrow["{f\vee f}"', from=1-1, to=2-1]
		\arrow["{cyl(f)}"', from=1-2, to=2-2]
		\arrow["f", from=1-3, to=2-3]
		\arrow["s", from=2-2, to=2-3]
		\arrow["{i_0+i_1}", from=2-1, to=2-2]
	\end{tikzcd}\]
	commutes. The pushout of the diagram 
	\[\begin{tikzcd}
		{\mathcal{X}} & {cyl(\mathcal{X})} \\
		\ast
		\arrow["{i_1}", from=1-1, to=1-2]
		\arrow[from=1-1, to=2-1]
	\end{tikzcd}\]
	is a cone on $\mathcal{X}$. We denote it by $C(cyl(\mathcal{X}))$. For the construction of Toda brackets we will always use cones defined in the previous way. Note that the morphism $i_0$ induces a cofibration $\mathcal{X}\rightarrow C(cyl(\mathcal{X}))$. We abuse here the notation and denote this monomorphism also by $i_0$.\\
	
	\begin{remark} Naturally we could also consider cones coming from the morphisms $i_0$, but for clarity we do not want to mix different types of cones, otherwise we could not control signs. Therefore we restrict ourselves to cones defined using the morphisms $i_1$.\\
		\end{remark}
	
	Since the $\mathbb{A}^1$-local injective model structure is simplicial and every object is cofibrant, there is a canonical functorial cylinder object, namely the functor $(-)\wedge \Delta_{+}^1$. The corresponding cone functor is $(-)\wedge \Delta^1$ where $\Delta^1$ is based at 1. Using a functorial factorization for the $\mathbb{A}^1$-local injective model structure we get another functorial cylinder object $F(-)$ such that the weak equivalence $s_F(\mathcal{X}): F(\mathcal{X})\rightarrow \mathcal{X}$ is an acyclic fibration for every $\mathcal{X}\in\mathrm{sPre}(S)_{\ast} $. In general, functorial cylinder objects with acyclic fibrations are called good functorial cylinder objects.\\
	
	Let $cyl(-)$ be an arbitrary functorial cylinder object. Then for every $\mathcal{X}\in\mathrm{sPre}(S)_{\ast} $ the diagram \[\begin{tikzcd}
		{\mathcal{X}\vee\mathcal{X}} && {F(\mathcal{X})} \\
		{cyl(\mathcal{X})} && {\mathcal{X}}
		\arrow["{i_0+i_1}"', from=1-1, to=2-1]
		\arrow["{(i_0)_F+(i_1)_F}", from=1-1, to=1-3]
		\arrow["{s(\mathcal{X})}"', from=2-1, to=2-3]
		\arrow["{s_F(\mathcal{X})}", from=1-3, to=2-3]
		\arrow["f"{description}, dashed, from=2-1, to=1-3]
	\end{tikzcd}\]
	has a lift $f$ and it is a motivic weak equivalence. In particular, the homotopy class of $f$ in the motivic homotopy category $\mathcal{H}_{\bullet}(S)$ is unique. The lift $f$ induces a morphism between the cones $Cf: C(cyl(\mathcal{X}))\rightarrow CF(\mathcal{X})$ such that the diagram \[\begin{tikzcd}
		{\mathcal{X}} & {C(cyl(\mathcal{X}))} \\
		& {F(\mathcal{X})}
		\arrow["{i_0}", from=1-1, to=1-2]
		\arrow["{(i_0)_F}"', from=1-1, to=2-2]
		\arrow["Cf", from=1-2, to=2-2]
	\end{tikzcd}\]
	commutes. Moreover we obtain the following commutative diagram 
	\[\begin{tikzcd}
		{C(cyl(\mathcal{X}))_-} & {\mathcal{X}} & {C(cyl(\mathcal{X}))_+} \\
		{C(F(\mathcal{X}))_-} & {\mathcal{X}} & {C(F(\mathcal{X}))_+} & \cdot
		\arrow["{{\mathrm{id}}}", from=1-2, to=2-2]
		\arrow["{{i_0}}", from=1-2, to=1-3]
		\arrow["{{i_0}}"', from=1-2, to=1-1]
		\arrow["{{(i_0)_F}}", from=2-2, to=2-3]
		\arrow["{{(i_0)_F}}"', from=2-2, to=2-1]
		\arrow["Cf"', from=1-1, to=2-1]
		\arrow["Cf", from=1-3, to=2-3]
	\end{tikzcd}\]
	\
	Especially, we get an induced morphism \[\begin{tikzcd}
		{C(cyl(\mathcal{X}))_+\sqcup_{\mathcal{X}}C(cyl(\mathcal{X}))_-} && {C(F(\mathcal{X}))_+\sqcup_{\mathcal{X}}C(F(\mathcal{X}))_-}
		\arrow["{C(f)_+\sqcup_{\mathrm{id}}C(f)_-}", from=1-1, to=1-3]
	\end{tikzcd}\] between the pushouts of the two rows.
	The pushout of the top row (respectively also the bottom row) is a homotopy pushout of $\ast \leftarrow \mathcal{X}\rightarrow \ast $. Hence it is a model for the suspension of $\mathcal{X}$.\\

	\begin{remark} If we have two cones on $\mathcal{X}$ and would like to glue them together to get a model of suspension for $\mathcal{X}$, then it is important to have the two cones distinguished, since mixing the roles of the cones would change maps by a factor of -1. Therefore we call $C(\mathcal{X})_+$ the top cone and $C(\mathcal{X})_- $ the bottom cone.\\
		\end{remark}
	
	\begin{lem}The homotopy class of the morphism $C(f)_+\sqcup_{\mathrm{id}}C(f)_-$ in $\mathcal{H}_{\bullet}(S)$ is independent of the choice of the lift $f$.\\
	\end{lem}
	
	\begin{proof} Let $g$ be another lift. Correpondingly, we get also a commutative diagram of the following form 
	\[\begin{tikzcd}
		{C(cyl(\mathcal{X}))_-} & {\mathcal{X}} & {C(cyl(\mathcal{X}))_+} \\
		{C(F(\mathcal{X}))_-} & {\mathcal{X}} & {C(F(\mathcal{X}))_+} & \cdot
		\arrow["{{\mathrm{id}}}", from=1-2, to=2-2]
		\arrow["{{i_0}}", from=1-2, to=1-3]
		\arrow["{{i_0}}"', from=1-2, to=1-1]
		\arrow["{{(i_0)_F}}", from=2-2, to=2-3]
		\arrow["{{(i_0)_F}}"', from=2-2, to=2-1]
		\arrow["Cg"', from=1-1, to=2-1]
		\arrow["Cg", from=1-3, to=2-3]
	\end{tikzcd}\]
	Now we denote the top row by $\mathcal{A} $ and the bottom row by $\mathcal{B}$. Let $\mathcal{I}$ denote the pushout category. Then $\mathcal{A}$ and $\mathcal{B}$ are $\mathcal{I}$-diagrams. Furthermore the lifts $f$ and $g$ induce two morphisms from $\mathcal{A}$ to $\mathcal{B}$. We abuse here again the notation and denote these two morphisms by $Cf$ and $Cg$, respectively. We take a cofibrant replacement $Q(\ast \leftarrow \mathcal{X}\rightarrow \ast )$ for $\ast \leftarrow \mathcal{X}\rightarrow \ast $ in the global projective model structure on the functor category $\mathcal{F}\mathrm{un}(\mathcal{I},\mathrm{sPre}(S)_{\ast}) $ such that $ Q(\ast \leftarrow \mathcal{X}\rightarrow \ast )\rightarrow (\ast \leftarrow \mathcal{X}\rightarrow \ast )$ is a levelwise acyclic fibration. The diagrams $\mathcal{A}$ and $\mathcal{B}$ are cofibrant in the the global projective model structure. Therefore the two diagrams 
	\[\begin{tikzcd}
		\ast && {Q(\ast \leftarrow \mathcal{X}\rightarrow \ast )} \\
		{\mathcal {A}} && {\ast \leftarrow \mathcal{X}\rightarrow \ast }
		\arrow[from=1-1, to=1-3]
		\arrow[from=2-1, to=2-3]
		\arrow[from=1-1, to=2-1]
		\arrow[from=1-3, to=2-3]
		\arrow[dashed, from=2-1, to=1-3]
	\end{tikzcd}\]
	and 
	\[\begin{tikzcd}
		\ast && {Q(\ast \leftarrow \mathcal{X}\rightarrow \ast )} \\
		{\mathcal {B}} && {\ast \leftarrow \mathcal{X}\rightarrow \ast }
		\arrow[from=1-1, to=1-3]
		\arrow[from=2-1, to=2-3]
		\arrow[from=1-1, to=2-1]
		\arrow[from=1-3, to=2-3]
		\arrow[dashed, from=2-1, to=1-3]
	\end{tikzcd}\]
	have lifts. The lifts are levelwise motivic weak equivalences. Note that the homotopy classes of the lifts are unique in $\mathrm{Ho}(\mathcal{F}\mathrm{un}(\mathcal{I},\mathrm{sPre}(S)_{\ast}))$. Furthermore we have the commutative diagram
	\[\begin{tikzcd}
		\ast && \ast && {Q(\ast \leftarrow \mathcal{X}\rightarrow \ast )} \\
		{\mathcal{A}} && {\mathcal {B}} && {\ast \leftarrow \mathcal{X}\rightarrow \ast } & \cdot
		\arrow[from=1-3, to=1-5]
		\arrow[from=2-3, to=2-5]
		\arrow[from=1-3, to=2-3]
		\arrow[from=1-5, to=2-5]
		\arrow[from=1-1, to=1-3]
		\arrow[from=1-1, to=2-1]
		\arrow["{{Cf\  \mathrm{or} \ Cg}}", from=2-1, to=2-3]
	\end{tikzcd}\]
	Therefore $Cf$ is equal to $Cg$ in $\mathrm{Ho}(\mathcal{F}\mathrm{un}(\mathcal{I},\mathrm{sPre}(S)_{\ast}))$. It follows that $C(f)_+\sqcup_{\mathrm{id}}C(f)_-$ is equal to $C(g)_+\sqcup_{\mathrm{id}}C(g)_-$ in $\mathcal{H}_{\bullet}(S)$.\\
	\end{proof}
	
	In this sense there is a unique isomorphism in $\mathcal{H}_{\bullet}(S)$ from $C(cyl(\mathcal{X}))_+\sqcup_{\mathcal{X}}C(cyl(\mathcal{X}))_-$ to  $C(F(\mathcal{X}))_+\sqcup_{\mathcal{X}}C(F(\mathcal{X}))_-$. In particular, we get a canonical isomorphism from $C(cyl(\mathcal{X}))_+\sqcup_{\mathcal{X}}C(cyl(\mathcal{X}))_-$ to $(\mathcal{X}\wedge\Delta^1)_+\sqcup_{\mathcal{X}}(\mathcal{X}\wedge \Delta^1)_-$ in $\mathcal{H}_{\bullet}(S)$.
	Now collapsing the bottom cone to a point $(\mathcal{X}\wedge \Delta^1)_-$ induces a weak equivalence from $(\mathcal{X}\wedge\Delta^1)_+\sqcup_{\mathcal{X}}(\mathcal{X}\wedge \Delta^1)_-$ to $X\wedge S^1$. Hence we also get a canonical isomorphism from $C(cyl(\mathcal{X}))_+\sqcup_{\mathcal{X}}C(cyl(\mathcal{X}))_-$ to $X\wedge S^1$. For the construction of Toda brackets we will use these canonical isomorphisms.\\
	
	\begin{remark} Actually the canonical isomorphism above is also independent of the choice of good functorial cylinder objects. Let  $F'$ be another functorial cylinder object  such that $s'(\mathcal{X}): F'(\mathcal{X})\rightarrow \mathcal{X}$ is an acyclic fibration for all $X\in\mathrm{sPre}(S)_{\ast} $. In particular, there is a  motivic weak equivalence $ f: F(\mathcal{X})\rightarrow F'(\mathcal{X})$ such that the diagram 
	\[\begin{tikzcd}
		{\mathcal{X}\vee\mathcal{X}} & {F(\mathcal{X})} & {\mathcal{X}} \\
		{\mathcal{X}\vee\mathcal{X}} & {F'(\mathcal{X})} & {\mathcal{X}}
		\arrow[from=1-1, to=1-2]
		\arrow[from=1-2, to=1-3]
		\arrow["f", from=1-2, to=2-2]
		\arrow[from=2-2, to=2-3]
		\arrow[from=1-3, to=2-3]
		\arrow[from=1-1, to=2-1]
		\arrow[from=2-1, to=2-2]
	\end{tikzcd}\]commutes. Now we take a functorial cylinder object $cyl(-)$. First we have a weak equivalence $\alpha: cyl(\mathcal{X})\rightarrow F'(\mathcal{X})$ such that the diagram 
	\[\begin{tikzcd}
		{\mathcal{X}\vee\mathcal{X}} & {cyl(\mathcal{X})} & {\mathcal{X}} \\
		{\mathcal{X}\vee\mathcal{X}} & {F'(\mathcal{X})} & {\mathcal{X}}
		\arrow[from=1-1, to=1-2]
		\arrow[from=1-2, to=1-3]
		\arrow[from=2-1, to=2-2]
		\arrow[from=2-2, to=2-3]
		\arrow["{\mathrm{id}}"', from=1-1, to=2-1]
		\arrow["\alpha", from=1-2, to=2-2]
		\arrow["{\mathrm{id}}", from=1-3, to=2-3]
	\end{tikzcd}\]
	commutes. Analogously, there is a weak equivalence $\beta: \mathcal{X}\wedge\Delta^1_+\rightarrow F'(\mathcal{X}) $ with a commutative diagram
	\[\begin{tikzcd}
		{\mathcal{X}\vee\mathcal{X}} & {\mathcal{X}\wedge\Delta^1_+} & {\mathcal{X}} \\
		{\mathcal{X}\vee\mathcal{X}} & {F'(\mathcal{X})} & {\mathcal{X}} & \cdot
		\arrow[from=1-1, to=1-2]
		\arrow[from=1-2, to=1-3]
		\arrow[from=2-1, to=2-2]
		\arrow[from=2-2, to=2-3]
		\arrow["{{\mathrm{id}}}"', from=1-1, to=2-1]
		\arrow["\beta", from=1-2, to=2-2]
		\arrow["{{\mathrm{id}}}", from=1-3, to=2-3]
	\end{tikzcd}\]
	Using $\alpha$ and $\beta$ we obtain morphisms $C(\alpha)_+\sqcup_{\mathrm{id}}C(\alpha)_-:C(cyl(\mathcal{X}))_+\sqcup_{\mathcal{X}}C(cyl(\mathcal{X}))_-\longrightarrow C(F'(\mathcal{X}))_+\sqcup_{\mathcal{X}}C(F'(\mathcal{X}))_-$ and $C(\beta)_+\sqcup_{\mathrm{id}}C(\beta)_-:(\mathcal{X}\wedge\Delta^1)_+\sqcup_{\mathcal{X}}(\mathcal{X}\wedge \Delta^1)_-\longrightarrow C(F'(\mathcal{X}))_+\sqcup_{\mathcal{X}}C(F'(\mathcal{X}))_-$. Again we apply the same arguments as in Lemma 2.2.3 to $F'(-)$ and get that the homotopy class of $C(\alpha)_+\sqcup_{\mathrm{id}}C(\alpha)_-$ is equal to the homotopy class of $C(f)_+\sqcup_{\mathrm{id}}C(f)_-\circ C(\alpha')_+\sqcup_{\mathrm{id}}C(\alpha')_-$ for an arbitrary lift $\alpha'$ of the diagram
	\[\begin{tikzcd}
		{\mathcal{X}\vee\mathcal{X}} & {F(\mathcal{X})} \\
		{cyl(\mathcal{X})} & {\mathcal{X}} & \cdot
		\arrow[from=1-1, to=1-2]
		\arrow[from=1-1, to=2-1]
		\arrow[from=1-2, to=2-2]
		\arrow[from=2-1, to=2-2]
		\arrow["{{\alpha'}}"{description}, dashed, from=2-1, to=1-2]
	\end{tikzcd}\]
	Correspondingly, we have that the homotopy class of $C(\beta)_+\sqcup_{\mathrm{id}}C(\beta)_-$is equal to the homotopy class of $C(f)_+\sqcup_{\mathrm{id}}C(f)_-\circ C(\beta')_+\sqcup_{\mathrm{id}}C(\beta')_-$ for an arbitrary lift $\beta'$ of the diagram
	\[\begin{tikzcd}
		{\mathcal{X}\vee\mathcal{X}} & {F(\mathcal{X})} \\
		{\mathcal{X}\wedge\Delta^1_+} & {\mathcal{X}} & \cdot
		\arrow[from=1-1, to=1-2]
		\arrow[from=1-1, to=2-1]
		\arrow[from=1-2, to=2-2]
		\arrow[from=2-1, to=2-2]
		\arrow["{{\beta'}}"{description}, dashed, from=2-1, to=1-2]
	\end{tikzcd}\]
	Thus the morphism $[C(\beta)_+\sqcup_{\mathrm{id}}C(\beta)_-]^{-1}\circ [C(\alpha)_+\sqcup_{\mathrm{id}}C(\alpha)_-]$ is equal to $[C(\beta')_+\sqcup_{\mathrm{id}}C(\beta')_-]^{-1}\circ [C(f)_+\sqcup_{\mathrm{id}}C(f)_-]^{-1}\circ [C(f)_+\sqcup_{\mathrm{id}}C(f)_-]\circ [C(\alpha')_+\sqcup_{\mathrm{id}}C(\alpha')_-] = [C(\beta')_+\sqcup_{\mathrm{id}}C(\beta')_-]^{-1}\circ [C(\alpha')_+\sqcup_{\mathrm{id}}C(\alpha')_-]$ which is the canonical isomorphism defined above with respect to $F(-)$ (see p.26). As a result the canonical isomorphisms are independent of the choice of good functorial cylinder objects.\\
	\end{remark}
	
	\begin{construction} 
		\label{construction 5.5}
		Now let \[\begin{tikzcd}
		{\mathcal{W}} & {\mathcal{X}} & {\mathcal{Y}} & {\mathcal{Z}}
		\arrow["\gamma", from=1-1, to=1-2]
		\arrow["\beta", from=1-2, to=1-3]
		\arrow["\alpha", from=1-3, to=1-4]
	\end{tikzcd}\]
	be a sequence of composable morphisms of pointed simplicial presheaves on $\mathcal{S}\mathrm{m}_{S}$ such that $\mathcal{Y}$ and $\mathcal{Z}$ are $\mathbb{A}^1$-local injective fibrant. Furthermore we require that the compositions $\alpha\circ \beta$ and $\beta\circ \gamma$ are nullhomotopic. We take an arbitrary functorial cylinder object $cyl(-)$. Since $\alpha\circ \beta$ is nullhomotopic, there is a nullhomotopy $A: cyl(\mathcal{X})\rightarrow Z$ with $A\circ i_0(\mathcal{X})= \alpha\circ\beta$ and $A\circ i_1(\mathcal{X})=\ast$. Hence we get a morphism $CA: C(cyl(\mathcal{X}))\rightarrow \mathcal{Z}$ such that $CA\circ i_0(\mathcal{X})=\alpha\circ \beta$. Analogously, we can choose a nullhomotopy $B:cyl(\mathcal{W})\rightarrow \mathcal{Y}$ such that $B\circ i_0(\mathcal{W})= \beta\circ\gamma$ and $B\circ i_1(\mathcal{W})=\ast$. Therefore we obtain a morphism $CB:C(cyl(\mathcal{W}))\rightarrow \mathcal{Y} $ with $CB\circ i_0(\mathcal{W})=\beta\circ \gamma$. Using the two morphisms $CA$ and $CB$ we can construct a morphism $H$ from $C(cyl(\mathcal{W}))_+\sqcup_{\mathcal{W}}C(cyl(\mathcal{W}))_-$ to $\mathcal{Z}$ via the commutative diagram
	\[\begin{tikzcd}
		{\mathcal{W}} && {C(cyl(\mathcal{W}))_+} \\
		\\
		{C(cyl(\mathcal{W}))_-} && {\mathcal{Z}} & \cdot
		\arrow["{{i_0(\mathcal{W})}}", from=1-1, to=1-3]
		\arrow["{{i_0(\mathcal{W})}}"', from=1-1, to=3-1]
		\arrow["{{\alpha\circ CB}}", from=1-3, to=3-3]
		\arrow["{{CA\circ C(\gamma)}}"', from=3-1, to=3-3]
	\end{tikzcd}\]
	If we compose the canonical isomorphism discussed above from $C(cyl(\mathcal{W}))_+\sqcup_{\mathcal{W}}C(cyl(\mathcal{W}))_-$ to $ \mathcal{W}\wedge S^1$ with $H$, then we get a morphism $\mathcal{W}\wedge S^1\rightarrow \mathcal{Z} $ in $\mathcal{H}_{\bullet}(S)$.\\
	\end{construction}
	
	\emph{In the following we will often slightly abuse the notation and denote morphisms from $\mathcal{W}\wedge S^1$ to $ \mathcal{Z}  $ in $\mathcal{H}_{\bullet}(S)$ obtained in the previous way just by the same symbols which we use for morphisms from $C(cyl(\mathcal{W}))_+\sqcup_{\mathcal{W}}C(cyl(\mathcal{W}))_-$ to $ \mathcal{Z}$}. \\

	\begin{definition} We define the Toda bracket $\{\alpha, \beta, \gamma\}$ of the sequence of composable morphisms above \[\begin{tikzcd}
		{\mathcal{W}} & {\mathcal{X}} & {\mathcal{Y}} & {\mathcal{Z}}
		\arrow["\gamma", from=1-1, to=1-2]
		\arrow["\beta", from=1-2, to=1-3]
		\arrow["\alpha", from=1-3, to=1-4]
	\end{tikzcd}\] to be the subset of $\mathcal{H}_{\bullet}(S)(\mathcal{W}\wedge S^1, \mathcal{Z})$ consisting of all morphisms $H$ obtained in the way described in Construction~\ref{construction 5.5} by choosing all possible functorial cylinder objects and all possible nullhomotopies for $\alpha\circ\beta$ and $\beta\circ\gamma$.\\
	\end{definition}
	
	Next we show that we can actually get the Toda bracket $\{\alpha, \beta, \gamma\}$ by using only the simplicial cone $(-)\wedge \Delta^1$.\\
	
	\begin{lem} All elements in  $\{\alpha, \beta, \gamma\}$ can be obtained by using only the simplicial cone $(-)\wedge \Delta^1$.\\
	\end{lem}
	
\begin{proof} Let $cyl(-)$ be an arbitrary functorial cylinder object. Let $\mathcal{X}\in \mathrm{sPre}(S)_{\ast}$. Then we consider the following pushout diagram \[\begin{tikzcd}
		{\mathcal{X}\vee\mathcal{X}} & {cyl(\mathcal{X})} \\
		{\mathcal{X}\wedge\Delta^1_+} & {D(\mathcal{X})} & \cdot
		\arrow["{{i_0+i_1}}", from=1-1, to=1-2]
		\arrow["{{d^1\vee d^0}}"', from=1-1, to=2-1]
		\arrow["{{\psi(\mathcal{X})}}"', from=2-1, to=2-2]
		\arrow["{{\phi(\mathcal{X})}}", from=1-2, to=2-2]
	\end{tikzcd}\]
	Furthermore there is a unique morphism $s_D(\mathcal{X}): D(\mathcal{X})\rightarrow \mathcal{X}$ such that the diagram commutes 
	\[\begin{tikzcd}
		{\mathcal{X}\vee\mathcal{X}} & {cyl(\mathcal{X})} \\
		{\mathcal{X}\wedge\Delta^1_+} & {D(\mathcal{X})} \\
		&& {\mathcal{X}} & \cdot
		\arrow["{{i_0+i_1}}", from=1-1, to=1-2]
		\arrow["{{d^1\vee d^0}}"', from=1-1, to=2-1]
		\arrow["{{\psi(\mathcal{X})}}"', from=2-1, to=2-2]
		\arrow["{{\phi(\mathcal{X})}}", from=1-2, to=2-2]
		\arrow["{{s_D(\mathcal{X})}}"{description}, from=2-2, to=3-3]
		\arrow["{{pr(\mathcal{X})}}"', curve={height=12pt}, from=2-1, to=3-3]
		\arrow["{{s(\mathcal{X})}}", curve={height=-18pt}, from=1-2, to=3-3]
	\end{tikzcd}\]
	By construction we get a functor $D(-)$. Now we apply again the functorial factorization for the morphisms $ s_D(\mathcal{X})\rightarrow \mathcal{X}$ to get a new functor $\widetilde{D}$ such that $s_D(\mathcal{X}) $ factors as follows 
	\[\begin{tikzcd}
		{D(\mathcal{X})} & {\widetilde{D}(\mathcal{X})} & {\mathcal{X}}
		\arrow["{\epsilon(\mathcal{X})}", from=1-1, to=1-2]
		\arrow["{s_{\widetilde{D}}(\mathcal{X})}", from=1-2, to=1-3]
	\end{tikzcd}\]
	where $s_{\widetilde{D}}(\mathcal{X}) $ is an acyclic fibration and $\epsilon(\mathcal{X}) $ is a cofibration.
	We claim that $\widetilde{D}(-)$ is also a functorial cylinder object. We set $((i_0)_{\widetilde{D}}+(i_1)_{\widetilde{D}})(\mathcal{X})$ to be $\epsilon(\mathcal{X})\circ \psi(\mathcal{X})\circ d^1\vee d^0= \epsilon(\mathcal{X})\circ \phi(\mathcal{X})\circ i_0+i_1 $. It is a cofibration and $s_{\widetilde{D}}(\mathcal{X})\circ((i_0)_{\widetilde{D}}+(i_1)_{\widetilde{D}}(\mathcal{X}))$ is the fold map on $\mathcal{X}$. Hence it is a functorial cylinder object. Furthermore by construction the composite $\epsilon(\mathcal{X})\circ \phi(\mathcal{X})$ is an acyclic cofibration.\\
	
	Next we take a nullhomotopy $A: cyl(\mathcal{X})\rightarrow Z$ with $A\circ i_0(\mathcal{X})= \alpha\circ\beta$ and $A\circ i_1(\mathcal{X})=\ast$. Let $B$ be a nullhomotopy $B:cyl(\mathcal{W})\rightarrow \mathcal{Y}$ such that $B\circ i_0(\mathcal{W})= \beta\circ\gamma$ and $B\circ i_1(\mathcal{W})=\ast$. Since $\epsilon(\mathcal{X})\circ \phi(\mathcal{X})$ and $\epsilon(\mathcal{W})\circ \phi(\mathcal{W})$ are acyclic cofibrations, the two diagrams below both admit a lift, namely
	\[\begin{tikzcd}
		{cyl(\mathcal{X})} & {\mathcal{Z}} \\
		{\widetilde{D}(\mathcal{X})} && {,}
		\arrow["A", from=1-1, to=1-2]
		\arrow["{{\epsilon(\mathcal{X})\circ \phi(\mathcal{X})}}"', from=1-1, to=2-1]
		\arrow["{{\widetilde{A}}}"', dashed, from=2-1, to=1-2]
	\end{tikzcd}\]
	and
	\[\begin{tikzcd}
		{cyl(\mathcal{W})} & {\mathcal{Y}} \\
		{\widetilde{D}(\mathcal{W})} && \cdot
		\arrow["B", from=1-1, to=1-2]
		\arrow["{{\epsilon(\mathcal{B})\circ \phi(\mathcal{B})}}"', from=1-1, to=2-1]
		\arrow["{{\widetilde{B}}}"', dashed, from=2-1, to=1-2]
	\end{tikzcd}\]
	Then $\widetilde{A}$ is a nullhomotopy of $\alpha\circ \beta$ with respect to the cylinder object $\widetilde{D}(\mathcal{X})$. Similarly, $\widetilde{B}$ is also a nullhomotopy of $\beta\circ\gamma$ with respect to the cylinder object $\widetilde{D}(\mathcal{X})$. If we precompose $\widetilde{A}$ with $\psi(\mathcal{X})$ and $\widetilde{B}$ with $\psi(\mathcal{W})$, then we get nullhomotopies with respect to the simplicial cylinder objects. Using all these nullhomotopies we can construct as described in Construction~\ref{construction 5.5} three morphisms via the following commutative diagrams 
	\[\begin{tikzcd}
		{\mathcal{W}} && {C(cyl(\mathcal{W}))_+} \\
		\\
		{C(cyl(\mathcal{W}))_-} && {C(cyl(\mathcal{W}))_+\sqcup_{\mathcal{W}}C(cyl(\mathcal{W}))_-} \\
		&&&& {\mathcal{Z}} & {,}
		\arrow["{{i_0(\mathcal{W})}}", from=1-1, to=1-3]
		\arrow["{{i_0(\mathcal{W})}}"', from=1-1, to=3-1]
		\arrow[from=3-1, to=3-3]
		\arrow[from=1-3, to=3-3]
		\arrow["{{CA\circ C(\gamma)}}"', curve={height=18pt}, from=3-1, to=4-5]
		\arrow["{{\alpha\circ CB}}", curve={height=-12pt}, from=1-3, to=4-5]
		\arrow["{{\exists!H}}"', from=3-3, to=4-5]
	\end{tikzcd}\]
	\[\begin{tikzcd}
		{\mathcal{W}} && {C(\widetilde{D}(\mathcal{W}))_+} \\
		{} \\
		{C(\widetilde{D}(\mathcal{W}))_-} && {C(\widetilde{D}(\mathcal{W}))_+\sqcup_{\mathcal{W}}C(\widetilde{D}(\mathcal{W}))_-} \\
		&&&& {\mathcal{Z}} & {,}
		\arrow["{{(i_0)_{\widetilde{D}}(\mathcal{W})}}", from=1-1, to=1-3]
		\arrow["{{(i_0)_{\widetilde{D}}(\mathcal{W})}}"', from=1-1, to=3-1]
		\arrow[from=3-1, to=3-3]
		\arrow[from=1-3, to=3-3]
		\arrow["{{C\widetilde{A}\circ C(\gamma)}}"', curve={height=18pt}, from=3-1, to=4-5]
		\arrow["{{\alpha\circ C\widetilde{B}}}", curve={height=-12pt}, from=1-3, to=4-5]
		\arrow["{{\exists!\widetilde{H}}}"', from=3-3, to=4-5]
	\end{tikzcd}\]
	\[\begin{tikzcd}
		{\mathcal{W}} && {(\mathcal{W}\wedge\Delta^1)_+} \\
		{} \\
		{(\mathcal{W}\wedge\Delta^1)_-} && {(\mathcal{W}\wedge\Delta^1)_+\sqcup_{\mathcal{W}}(\mathcal{W}\wedge\Delta^1)_-} \\
		&&&& {\mathcal{Z}} & \cdot
		\arrow["{{d^1}}", from=1-1, to=1-3]
		\arrow["{{d^1}}"', from=1-1, to=3-1]
		\arrow[from=3-1, to=3-3]
		\arrow[from=1-3, to=3-3]
		\arrow["{{C(\widetilde{A}\circ \psi(\mathcal{X}))\circ C(\gamma)}}"', curve={height=18pt}, from=3-1, to=4-5]
		\arrow["{{\alpha\circ C(\widetilde{B}\circ\psi(\mathcal{W}))}}", curve={height=-12pt}, from=1-3, to=4-5]
		\arrow["{{\exists!\hat{H}}}"', from=3-3, to=4-5]
	\end{tikzcd}\]
	\
	
	Moreover we also have the following commutative diagram 
	\[\begin{tikzcd}
		{C(cyl(\mathcal{W}))_+\sqcup_{\mathcal{W}}C(cyl(\mathcal{W}))_-} \\
		\\
		{C(\widetilde{D}(\mathcal{W}))_+\sqcup_{\mathcal{W}}C(\widetilde{D}(\mathcal{W}))_-} &&& {\mathcal{Z}} \\
		\\
		{(\mathcal{W}\wedge\Delta^1)_+\sqcup_{\mathcal{W}}(\mathcal{W}\wedge\Delta^1)_-} &&&& \cdot
		\arrow["{{C(\epsilon(\mathcal{W})\circ \phi(\mathcal{W}))_+\sqcup_{\mathrm{id}}C(\epsilon(\mathcal{W})\circ \phi(\mathcal{W}))_-}}"{description}, from=1-1, to=3-1]
		\arrow["{{C(\epsilon(\mathcal{W})\circ \psi(\mathcal{W}))_+\sqcup_{\mathrm{id}}C(\epsilon(\mathcal{W})\circ \psi(\mathcal{W}))_-}}"{description}, from=5-1, to=3-1]
		\arrow["{{\widetilde{H}}}"{description}, from=3-1, to=3-4]
		\arrow["{{\hat{H}}}"{description}, curve={height=6pt}, from=5-1, to=3-4]
		\arrow["H"{description}, curve={height=-12pt}, from=1-1, to=3-4]
	\end{tikzcd}\]
	By construction the isomorphism $(C(\epsilon(\mathcal{W})\circ \psi(\mathcal{W}))_+\sqcup_{\mathrm{id}}C(\epsilon(\mathcal{W})\circ \psi(\mathcal{W}))_-)^{-1}\circ (C(\epsilon(\mathcal{W})\circ \phi(\mathcal{W}))_+\sqcup_{\mathrm{id}}C(\epsilon(\mathcal{W})\circ \phi(\mathcal{W}))_-)$ is the canonical isomorphism in $\mathcal{H}_{\bullet}(S)$ from $C(cyl(\mathcal{W}))_+\sqcup_{\mathcal{W}}C(cyl(\mathcal{W}))_-$ to  $(\mathcal{W}\wedge\Delta^1)_+\sqcup_{\mathcal{W}}(\mathcal{W}\wedge \Delta^1)_-$. Therefore we can get all elements in  $\{\alpha, \beta, \gamma\}$ by using only the simplicial cone $(-)\wedge \Delta^1$.\\
	\end{proof}
	
	\begin{remark} Although we could theoretically obtain all elements in a Toda bracket only by the simplicial cylinder object, it is still very useful in practice to consider general functorial cylinder objects. For example we use cones coming from functorial cylinder objects in the proof of Proposition~\ref{prop:propostion 2.2.10}.\\
		\end{remark}
	
	We take now a functorial fibrant replacement $R(-)$ such that the morphism $f(\mathcal{S}):\mathcal{S}\rightarrow R(\mathcal{S})$ is an acyclic cofibration for every $\mathcal{S}\in \mathrm{sPre}(S)_{\ast} $. Then the diagram 
	\[\begin{tikzcd}
		{\mathcal{X}} & {\mathcal{Y}} \\
		{R\mathcal{X}}
		\arrow["{f(\mathcal{X})}"', from=1-1, to=2-1]
		\arrow["\beta", from=1-1, to=1-2]
		\arrow["{\beta'}"', dashed, from=2-1, to=1-2]
	\end{tikzcd}\]
	admits a lift $\beta'$ such that the diagram \[\begin{tikzcd}
		{\mathcal{W}} & {\mathcal{X}} & {\mathcal{Y}} & {\mathcal{Z}} \\
		{\mathcal{W}} & {R\mathcal{X}} & {\mathcal{Y}} & {\mathcal{Z}}
		\arrow["{f(\mathcal{X})\circ\gamma}", from=2-1, to=2-2]
		\arrow["{\beta'}", from=2-2, to=2-3]
		\arrow["\alpha", from=2-3, to=2-4]
		\arrow["\gamma", from=1-1, to=1-2]
		\arrow["\beta", from=1-2, to=1-3]
		\arrow["\alpha", from=1-3, to=1-4]
		\arrow["{\mathrm{id}}"{description}, from=1-1, to=2-1]
		\arrow["{\mathrm{id}}"{description}, from=1-3, to=2-3]
		\arrow["{\mathrm{id}}"{description}, from=1-4, to=2-4]
		\arrow["{f(\mathcal{X})}"{description}, from=1-2, to=2-2]
	\end{tikzcd}\]
	commutes. In particular we can also consider the Toda bracket of the sequence 
	\[\begin{tikzcd}
		{\mathcal{W}} & {R\mathcal{X}} & {\mathcal{Y}} & {\mathcal{Z}} & \cdot
		\arrow["{{f(\mathcal{X})\circ\gamma}}", from=1-1, to=1-2]
		\arrow["{{\beta'}}", from=1-2, to=1-3]
		\arrow["\alpha", from=1-3, to=1-4]
	\end{tikzcd}\]
	\
	
	\begin{lem} \label{lemma 2.2.9}
		We have $\{\alpha, \beta, \gamma\}=\{\alpha, \beta', f(\mathcal{X})\circ\gamma\}$.\\
		\end{lem}
	
\begin{proof} It is easy to see that $\{\alpha, \beta', f(\mathcal{X})\circ\gamma\}\subseteq \{\alpha, \beta, \gamma\}$. By construction $\beta'\circ f(\mathcal{X})\circ\gamma=\beta\circ \gamma$. Take a nullhomotopy $B: \mathcal{W}\wedge\Delta^1_+\rightarrow \mathcal{Y}$ for $\beta'\circ f(\mathcal{X})\circ\gamma$. It is therefore also a nullhomotopy for $\beta\circ \gamma $. Let $A:R\mathcal{X}\wedge \Delta^1_+\rightarrow\mathcal{Z}$ be a nullhomotopy of $\alpha\circ \beta'$. Then $A\circ (f(\mathcal{X})\wedge \mathrm{id})$ is a nullhomotopy for $\alpha\circ \beta$. Now we get an element $H$ in $\{\alpha, \beta', f(\mathcal{X})\circ\gamma\}$ by the commutative diagram 
	\[\begin{tikzcd}
		{\mathcal{W}} & {(\mathcal{W}\wedge\Delta^1)_+} \\
		{(\mathcal{W}\wedge\Delta^1)_-} & {(\mathcal{W}\wedge\Delta^1)_+\sqcup_{\mathcal{W}}(\mathcal{W}\wedge\Delta^1)_-} \\
		&& {\mathcal{Z}} & \cdot
		\arrow["{{d^1}}", from=1-1, to=1-2]
		\arrow["{{d^1}}"', from=1-1, to=2-1]
		\arrow[from=2-1, to=2-2]
		\arrow[from=1-2, to=2-2]
		\arrow["{{\exists !H}}"{description}, from=2-2, to=3-3]
		\arrow["{{CA\circ C(f(\mathcal{X})\circ\gamma)}}"', curve={height=12pt}, from=2-1, to=3-3]
		\arrow["{{\alpha\circ CB}}", curve={height=-18pt}, from=1-2, to=3-3]
	\end{tikzcd}\]
	By definition we have $CA\circ C(f(\mathcal{X})\circ\gamma)=C(A\circ (f(\mathcal{X})\wedge \mathrm{id}))\circ C(\gamma)$. Hence $H$ is an element of $\{\alpha, \beta, \gamma\}$.\\
	
	Next we show that $\{\alpha, \beta, \gamma\}\subseteq \{\alpha, \beta', f(\mathcal{X})\circ\gamma\}$. Again any simplicial nullhomotopy $B'$ of $\beta\circ \gamma $ is a nullhomotopy of $\beta'\circ f(\mathcal{X})\circ\gamma$. Let $A': \mathcal{X}\wedge \Delta^1_+\rightarrow\mathcal{Z}$ be a nullhomotopy of $\alpha\circ \beta$. We consider now the pushout diagram
	\[\begin{tikzcd}
		{\mathcal{X}} & {\mathcal{X}\wedge\Delta^1} \\
		{R\mathcal{X}} & {R\mathcal{X}\sqcup_{\mathcal{X}}\mathcal{X}\wedge\Delta^1} & \cdot
		\arrow["{{d^1}}", from=1-1, to=1-2]
		\arrow["{{f(\mathcal{X})}}"', from=1-1, to=2-1]
		\arrow[from=2-1, to=2-2]
		\arrow[from=1-2, to=2-2]
	\end{tikzcd}\]
	Since $f(\mathcal{X})$ is an acyclic cofibration, the morphism $\mathcal{X}\wedge\Delta^1\rightarrow R\mathcal{X}\sqcup_{\mathcal{X}}\mathcal{X}\wedge\Delta^1$ is also an acyclic cofibration. Moreover we have a morphism $i: R\mathcal{X}\sqcup_{\mathcal{X}}\mathcal{X}\rightarrow R\mathcal{X}\wedge\Delta^1 $ induced by the commutative diagram 
	\[\begin{tikzcd}
		{\mathcal{X}} & {\mathcal{X}\wedge\Delta^1} \\
		{R\mathcal{X}} & {R\mathcal{X}\sqcup_{\mathcal{X}}\mathcal{X}\wedge\Delta^1} \\
		&& {R\mathcal{X}\wedge\Delta^1} & \cdot
		\arrow["{{d^1}}", from=1-1, to=1-2]
		\arrow["{{f(\mathcal{X})}}"', from=1-1, to=2-1]
		\arrow[from=2-1, to=2-2]
		\arrow[from=1-2, to=2-2]
		\arrow["{{\exists !i}}", from=2-2, to=3-3]
		\arrow["{{d^1}}"', curve={height=12pt}, from=2-1, to=3-3]
		\arrow["{{f(\mathcal{X})\wedge\Delta^1}}", curve={height=-12pt}, from=1-2, to=3-3]
	\end{tikzcd}\]
	It follows from the construction that $i$ is an acyclic cofibration. Therefore the diagram
	\[\begin{tikzcd}
		{R\mathcal{X}\sqcup_{\mathcal{X}}\mathcal{X}\wedge\Delta^1} && {\mathcal{Z}} \\
		{R\mathcal{X}\wedge\Delta^1}
		\arrow["i"', from=1-1, to=2-1]
		\arrow["{\alpha\circ\beta'\sqcup_{\alpha\circ\beta}A'}", from=1-1, to=1-3]
		\arrow["{A''}"', dashed, from=2-1, to=1-3]
	\end{tikzcd}\]
	admits a lift $A''$. The morphism $A''$ is a nullhomotopy of $ \alpha\circ\beta'$ such that $A''\circ (f(\mathcal{X})\wedge\Delta^1)=A'$. Using the nullhomotopies $B'$ and $A'$ we get an element of $\{\alpha, \beta, \gamma\}$ via the commutative diagram
	\[\begin{tikzcd}
		{\mathcal{W}} & {(\mathcal{W}\wedge\Delta^1)_+} \\
		{(\mathcal{W}\wedge\Delta^1)_-} & {(\mathcal{W}\wedge\Delta^1)_+\sqcup_{\mathcal{W}}(\mathcal{W}\wedge\Delta^1)_-} \\
		&& {\mathcal{Z}} & \cdot
		\arrow["{{d^1}}", from=1-1, to=1-2]
		\arrow["{{d^1}}"', from=1-1, to=2-1]
		\arrow[from=2-1, to=2-2]
		\arrow[from=1-2, to=2-2]
		\arrow["{{\exists !H'}}"{description}, from=2-2, to=3-3]
		\arrow["{{CA'\circ C(\gamma)}}"', curve={height=12pt}, from=2-1, to=3-3]
		\arrow["{{\alpha\circ CB'}}", curve={height=-18pt}, from=1-2, to=3-3]
	\end{tikzcd}\]
	On the other hand we have $CA'\circ C(\gamma)=C(A''\circ (f(\mathcal{X})\wedge\Delta^1))\circ C(\gamma) = CA''\circ C(f(\mathcal{X})\circ \gamma)$. Now we see that $H'$ is also an element in $\{\alpha, \beta', f(\mathcal{X})\circ\gamma\}$.\\
	\end{proof}
	
	\begin{prop}
		\label{prop:propostion 2.2.10}
		The Toda bracket $\{\alpha, \beta, \gamma\}$ only depends of the homotopy classes of the morphisms $\alpha, \beta$ and $\gamma$.\\
		\end{prop}
	
	\begin{proof} By Lemma 2.2.9 we can assume that $\mathcal{X}$ is $\mathbb{A}^1$-local injective fibrant. Let $\gamma': \mathcal{W}\rightarrow \mathcal{X}$ be a morphism such that $[\gamma]=[\gamma']$ in $\mathcal{H}_{\bullet}(S)$. Then we would like to show that $\{\alpha, \beta, \gamma\}=\{\alpha, \beta, \gamma'\}$. We first show that $\{\alpha, \beta, \gamma\}\subseteq \{\alpha, \beta, \gamma'\}$.\\
	
	Let $B:\mathcal{W}\wedge \Delta^1_+\rightarrow \mathcal{Y}$ be a nullhomotopy of $\beta\circ \gamma$ and $A:\mathcal{X}\wedge \Delta^1_+\rightarrow \mathcal{Z}$ a nullhomotopy of $\alpha\circ \beta$. Since $[ \gamma]=[\gamma']$ in $\mathcal{H}_{\bullet}(S)$, there is a simplicial homotopy 
	\[\begin{tikzcd}
		{\mathcal{W}} \\
		{\mathcal{W}\wedge \Delta^1_+} && {\mathcal{X}} \\
		{\mathcal{W}} &&& \cdot
		\arrow["\Psi", from=2-1, to=2-3]
		\arrow["{{\mathrm{id}\times d^1}}"', from=1-1, to=2-1]
		\arrow["{{\mathrm{id}\times d^0}}", from=3-1, to=2-1]
		\arrow["{{\gamma}}"', from=3-1, to=2-3]
		\arrow["{\gamma'}", from=1-1, to=2-3]
	\end{tikzcd}\]
	Now we consider the functorial cylinder object $cyl(\mathcal{W})$ defined as the pushout of the following diagram 
	\[\begin{tikzcd}
		{\mathcal{W}} & {\mathcal{W}\wedge \Delta^1_+} \\
		{\mathcal{W}\wedge \Delta^1_+} & {cyl(\mathcal{W})} & \cdot
		\arrow["{{\mathrm{id}\times d^1}}", from=1-1, to=1-2]
		\arrow["{{\mathrm{id}\times d^0}}"', from=1-1, to=2-1]
		\arrow["{{j_-}}"', from=2-1, to=2-2]
		\arrow["{{j_+}}", from=1-2, to=2-2]
		\arrow["\lrcorner"{anchor=center, pos=0.125, rotate=180}, draw=none, from=2-2, to=1-1]
	\end{tikzcd}\]
	We set $i_0(\mathcal{W})=j_-\circ(\mathrm{id}\times d^1) $ and $i_1(\mathcal{W})=j_+\circ(\mathrm{id}\times d^0) $. Furthermore we define $s(\mathcal{W}): cyl(\mathcal{W})\rightarrow \mathcal{W}$ to be the projection.\\
	
	Since $B$ is a nullhomotopy of $\beta\circ \gamma$ and $\Psi$ is a homotopy from $ \gamma'$ to $ \gamma$, we get a nullhomotopy $B'$ of $\beta\circ \gamma'$ using the cylinder object $cyl(\mathcal{W})$ via the commutative diagram
	\[\begin{tikzcd}
		{\mathcal{W}} & {\mathcal{W}\wedge\Delta^1_+} \\
		{\mathcal{W}\wedge\Delta^1_+} & {cyl(\mathcal{W})} \\
		&& {\mathcal{Y}} & \cdot
		\arrow["{{\mathrm{id}\times d^1}}", from=1-1, to=1-2]
		\arrow["{{\mathrm{id}\times d^0}}"', from=1-1, to=2-1]
		\arrow["{{j_+(\mathcal{W})}}", from=1-2, to=2-2]
		\arrow["{{j_-(\mathcal{W})}}"', from=2-1, to=2-2]
		\arrow["{{B'}}"', from=2-2, to=3-3]
		\arrow["B", curve={height=-24pt}, from=1-2, to=3-3]
		\arrow["\beta\circ\Psi"', curve={height=12pt}, from=2-1, to=3-3]
	\end{tikzcd}\]
	We illustrate this nullhomotopy as follows
	
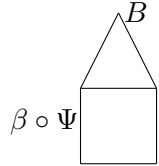
\begin{center}
		\begin{tikzpicture}
			\draw[draw=black, thin, solid] (-0.50,2.00) -- (-1.00,1.00);
			\draw[draw=black, thin, solid] (-0.50,2.00) -- (0.00,1.00);
			\draw[draw=black, thin, solid] (-1.00,1.00) -- (0.00,1.00);
			\draw[draw=black, thin, solid] (-1.00,1.00) -- (-1.00,0.00);
			\draw[draw=black, thin, solid] (0.00,1.00) -- (0.00,0.00);
			\draw[draw=black, thin, solid] (-1.00,0.00) -- (0.00,0.00);
			\node[black, anchor=south west] at (-0.56,1.75) {$B$};
			\node[black, anchor=south west] at (-2.06,0.25) {$\beta\circ \Psi$};
		\end{tikzpicture}
		
	\end{center}
	Correspondingly, there is a nullhomotopy $A'$ of $\alpha \circ \beta$ with respect to the cylinder object $cyl(\mathcal{X})$ defined by the commutative diagram
	\[\begin{tikzcd}
		{\mathcal{X}} & {\mathcal{X}\wedge\Delta^1_+} \\
		{\mathcal{X}\wedge\Delta^1_+} & {cyl(\mathcal{X})} \\
		&& {\mathcal{Z}} & \cdot
		\arrow["{{\mathrm{id}\times d^1}}", from=1-1, to=1-2]
		\arrow["{{\mathrm{id}\times d^0}}"', from=1-1, to=2-1]
		\arrow["{{j_-(\mathcal{X})}}"', from=2-1, to=2-2]
		\arrow["{{j_+(\mathcal{X})}}", from=1-2, to=2-2]
		\arrow["A", curve={height=-12pt}, from=1-2, to=3-3]
		\arrow["\alpha\circ\beta"', curve={height=12pt}, from=2-1, to=3-3]
		\arrow["{{A'}}"', from=2-2, to=3-3]
	\end{tikzcd}\]
	Hence we get a morphism $H'$ from  $C(cyl(\mathcal{W}))_+\sqcup_{\mathcal{W}}C(cyl(\mathcal{W}))_-$ to $\mathcal{Z}$ induced by the commutative diagram 
	\[\begin{tikzcd}
		{\mathcal{W}} && {C(cyl(\mathcal{W}))_+} \\
		\\
		{C(cyl(\mathcal{W}))_-} && {\mathcal{Z}} & \cdot
		\arrow["{{i_0(\mathcal{W})}}", from=1-1, to=1-3]
		\arrow["{{i_0(\mathcal{W})}}"', from=1-1, to=3-1]
		\arrow["{{\alpha\circ CB'}}", from=1-3, to=3-3]
		\arrow["{{CA'\circ C(\gamma')}}"', from=3-1, to=3-3]
	\end{tikzcd}\]
	We illustrate $H'$ in $\{\alpha, \beta, \gamma'\}$ in the following way
	\begin{center}
		\begin{tikzpicture}
			\draw[draw=black, thin, solid] (-0.50,2.00) -- (-1.00,1.00);
			\draw[draw=black, thin, solid] (-0.50,2.00) -- (0.00,1.00);
			\draw[draw=black, thin, solid] (-1.00,1.00) -- (-1.00,0.00);
			\draw[draw=black, thin, solid] (0.00,1.00) -- (0.00,0.00);
			\draw[draw=black, thin, solid] (-1.00,0.00) -- (0.00,0.00);
			\draw[draw=black, thin, solid] (-1.00,1.00) -- (0.00,1.00);
			\draw[draw=black, thin, solid] (-1.00,0.00) -- (-1.00,-1.00);
			\draw[draw=black, thin, solid] (0.00,0.00) -- (0.00,-1.00);
			\draw[draw=black, thin, solid] (-1.00,-1.00) -- (0.00,-1.00);
			\draw[draw=black, thin, solid] (-1.00,-1.00) -- (-0.50,-2.00);
			\draw[draw=black, thin, solid] (0.00,-1.00) -- (-0.50,-2.00);
			\node[black, anchor=south west] at (-0.06,1.25) {$\alpha\circ B$};
			\node[black, anchor=south west] at (-0.06,0.25) {$\alpha\circ \beta\circ\Psi$};
			\node[black, anchor=south west] at (-0.06,-0.75) {$\alpha\circ \beta\circ\gamma'$};
			\node[black, anchor=south west] at (-0.06,-1.75) {$CA\circ C(\gamma')$};
			\draw[draw=black, thin, solid] (2.50,0.00) .. controls (3.00, 0.00) and (3.50, 0.00) .. (4.00,0.00);
			\draw[draw=black, thin, solid] (2.50,0.00) parabola (4.00,0.00);
			\draw[draw=black, thin, solid] (2.50,0.00) parabola (4.00,0.00);
			\draw[draw=black, thin, solid] (4.00,0.00) -- (3.50,0.50);
			\draw[draw=black, thin, solid] (4.00,0.00) -- (3.50,-0.50);
			\node[black, anchor=south west] at (4.44,-0.25) {$\mathcal{Z}$};
		\end{tikzpicture}
	\end{center}
	Analogously, we get an element $H$ in $\{\alpha, \beta, \gamma\}$ with respect to the cylinder object $cyl(\mathcal{W})$
	\begin{center}
		\begin{tikzpicture}
			\draw[draw=black, thin, solid] (-2.50,1.00) -- (-3.00,0.00);
			\draw[draw=black, thin, solid] (-2.50,1.00) -- (-2.00,0.00);
			\draw[draw=black, thin, solid] (-3.00,0.00) -- (-2.00,0.00);
			\draw[draw=black, thin, solid] (-3.00,0.00) -- (-3.00,-1.00);
			\draw[draw=black, thin, solid] (-2.00,0.00) -- (-2.00,-1.00);
			\draw[draw=black, thin, solid] (-3.00,-1.00) -- (-2.00,-1.00);
			\draw[draw=black, thin, solid] (-3.00,-1.00) -- (-3.00,-2.00);
			\draw[draw=black, thin, solid] (-2.00,-1.00) -- (-2.00,-2.00);
			\draw[draw=black, thin, solid] (-3.00,-2.00) -- (-2.00,-2.00);
			\draw[draw=black, thin, solid] (-3.00,-2.00) -- (-2.50,-3.00);
			\draw[draw=black, thin, solid] (-2.00,-2.00) -- (-2.50,-3.00);
			\node[black, anchor=south west] at (-2.06,-2.75) {$CA\circ C(\gamma)$};
			\node[black, anchor=south west] at (-2.06,-1.75) {$\alpha\circ\beta\circ\gamma$};
			\node[black, anchor=south west] at (-2.06,-0.75) {$\alpha\circ\beta\circ\gamma$};
			\node[black, anchor=south west] at (-2.06,0.25) {$\alpha\circ B$};
			\draw[draw=black, thin, solid] (0.50,-1.00) -- (1.50,-1.00);
			\draw[draw=black, thin, solid] (1.50,-1.00) -- (1.00,-0.50);
			\draw[draw=black, thin, solid] (1.50,-1.00) -- (1.00,-1.50);
			\node[black, anchor=south west] at (0.44,-0.75) {$H$};
			\node[black, anchor=south west] at (2.44,-1.25) {$\mathcal{Z}$};
		\end{tikzpicture}
	\end{center}
	Next we apply the geometric realization functor $|\cdot|$ to $H'$ and $H$. First note that $|C(cyl(\mathcal{W}))_+\sqcup_{\mathcal{W}}C(cyl(\mathcal{W}))_-|$ is the presheaf 
	
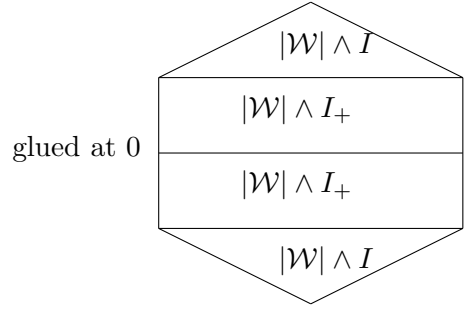
\begin{center}
		\begin{tikzpicture}
			\draw[draw=black, thin, solid] (3.00,2.00) -- (1.00,1.00);
			\draw[draw=black, thin, solid] (3.00,2.00) -- (5.00,1.00);
			\draw[draw=black, thin, solid] (1.00,1.00) -- (5.00,1.00);
			\draw[draw=black, thin, solid] (1.00,1.00) -- (1.00,0.00);
			\draw[draw=black, thin, solid] (1.00,0.00) -- (5.00,0.00);
			\draw[draw=black, thin, solid] (5.00,1.00) -- (5.00,0.00);
			\draw[draw=black, thin, solid] (1.00,0.00) -- (1.00,-1.00);
			\draw[draw=black, thin, solid] (1.00,-1.00) -- (5.00,-1.00);
			\draw[draw=black, thin, solid] (5.00,0.00) -- (5.00,-1.00);
			\draw[draw=black, thin, solid] (1.00,-1.00) -- (3.00,-2.00);
			\draw[draw=black, thin, solid] (5.00,-1.00) -- (3.00,-2.00);
			\node[black, anchor=south west] at (2.44,1.10) {$|\mathcal{W}|\wedge I$};
			\node[black, anchor=south west] at (1.94,0.25) {$|\mathcal{W}|\wedge I_+$};
			\node[black, anchor=south west] at (1.94,-0.75) {$|\mathcal{W}|\wedge I_+$};
			\node[black, anchor=south west] at (2.44,-1.75) {$|\mathcal{W}|\wedge I$};
			\node[black, anchor=south west] at (-1.06,-0.25) {glued at 0};
		\end{tikzpicture}
	\end{center}
	where $I$ is based at $1$. Using $|\Psi|: |\mathcal{W}|\wedge I_+\rightarrow |\mathcal{X}|$ we obtain a homotopy  $G: |C(cyl(\mathcal{W}))_+\sqcup_{\mathcal{W}}C(cyl(\mathcal{W}))_-|\wedge I_+\rightarrow |\mathcal{Z}|$ defined as follows: 
	\begin{itemize}
		\item $G(w\wedge t\wedge s)=|\alpha\circ B|(w\wedge t)$ for all $w\wedge t$ in the top cone and $s\in I$.
		\item $G(w\wedge t\wedge s)= |\alpha\circ \beta|\circ |\Psi|(w\wedge(1-s+ts))$ for all $w\wedge t$ in the first cylinder from above and all $s\in I$.
		\item  $G(w\wedge t\wedge s)=|\alpha\circ \beta|\circ |\Psi|(w\wedge(1-s))  $ for all $w\wedge t$ in the second cylinder from above and all $s\in I$.
		\item  $G(w\wedge t\wedge s)=|CA|(|\Psi|(w\wedge(1-s))\wedge t) $ for all $w\wedge t$ in the bottom cone and all $s\in I$.
	\end{itemize}
	Then we have $G(-, 0)=|H| $ and $G(-, 1)=|H'| $. By Remark 2.1.8 there is a Quillen equivalence\begin{align*}
		|\cdot|	:\mathrm{sPre}(S)_{\ast \ \mathbb{A}^1-local\ inj}\rightleftarrows\mathrm{Pre}_{\Delta}(S)_{\ast\ \mathbb{A}^1-local\ inj}: Sing
	\end{align*}
	The two morphisms $|H| $ and $|H'|$ are equal in the homotopy category associated to the $\mathbb{A}^1$-local injective model structure on $\mathrm{Pre}_{\Delta}(S)_{\ast}$, therefore $H$ and $H'$ are equal in $\mathcal{H}_{\bullet}(S)$.\\
	
	Next we have a morphism $\Phi$ from $C(cyl(\mathcal{W}))_+\sqcup_{\mathcal{W}}C(cyl(\mathcal{W}))_-$ to $(\mathcal{W}\wedge\Delta^1)_+\sqcup_{\mathcal{W}}(\mathcal{W}\wedge\Delta^1)_-$ by mapping the two cylinders in the middle to $\mathcal{W}$ via the canonical projections. The morphism $\Phi$ is a motivic weak equivalence. Using the nullhomotopies $B$ and $A$ we obtain an element $\widetilde{H}$ in $\{\alpha, \beta, \gamma\}$ with respect to the simplicial cylinder object $\mathcal{W}\wedge\Delta^1_+$. Furthermore we have the commutative diagram below
	\[\begin{tikzcd}
		{C(cyl(\mathcal{W}))_+\sqcup_{\mathcal{W}}C(cyl(\mathcal{W}))_-} && {\mathcal{Z}} \\
		\\
		{(\mathcal{W}\wedge\Delta^1)_+\sqcup_{\mathcal{W}}(\mathcal{W}\wedge\Delta^1)_-}
		\arrow["\Phi"', from=1-1, to=3-1]
		\arrow["H", from=1-1, to=1-3]
		\arrow["{\widetilde{H}}"', from=3-1, to=1-3]
	\end{tikzcd}\]
	in $\mathrm{Pre}_{\Delta}(S)_{\ast}$. It follows that  $\{\alpha, \beta, \gamma\}\subseteq\{\alpha, \beta, \gamma'\}$, because every element in $\{\alpha, \beta, \gamma\}$ can be obtained like $\widetilde{H}$. The other inclusion $\{\alpha, \beta, \gamma'\}\subseteq \{\alpha, \beta, \gamma\}$ follows by symmetry.\\
	\end{proof}
	
	So far we defined Toda brackets for sequences \[\begin{tikzcd}
		{\mathcal{W}} & {\mathcal{X}} & {\mathcal{Y}} & {\mathcal{Z}}
		\arrow["\gamma", from=1-1, to=1-2]
		\arrow["\beta", from=1-2, to=1-3]
		\arrow["\alpha", from=1-3, to=1-4]
	\end{tikzcd}\]
	with $\mathcal{Y}$ and $\mathcal{Z}$ being fibrant. Next we define Toda brackets for general sequences.\\
	
	\begin{definition} 
		\label{def:definition 2.2.11} Let 
	\[\begin{tikzcd}
		{\mathcal{W}} & {\mathcal{X}} & {\mathcal{Y'}} & {\mathcal{Z'}}
		\arrow["\gamma", from=1-1, to=1-2]
		\arrow["\beta", from=1-2, to=1-3]
		\arrow["\alpha", from=1-3, to=1-4]
	\end{tikzcd}\]
	be a sequence of composable morphisms such that $\alpha\circ\beta$ and $\beta\circ\gamma$ are nullhomotopic. Let $R(-)$ be an arbitrary functorial fibrant replacement. We denote the morphisms $\mathcal{Y'}\rightarrow R\mathcal{Y'} $ by $f(\mathcal{Y'})$ and $\mathcal{Z'}\rightarrow R\mathcal{Z'} $ by $f(\mathcal{Z'})$. We have the sequence 
	\[\begin{tikzcd}
		{\mathcal{W}} & {\mathcal{X}} & {\mathcal{RY'}} & {\mathcal{RZ'}}
		\arrow["\gamma", from=1-1, to=1-2]
		\arrow["f(\mathcal{Y'})\circ\beta", from=1-2, to=1-3]
		\arrow["R\alpha", from=1-3, to=1-4]
	\end{tikzcd}\]
	Then we define the Toda bracket of 
	\[\begin{tikzcd}
		{\mathcal{W}} & {\mathcal{X}} & {\mathcal{Y'}} & {\mathcal{Z'}}
		\arrow["\gamma", from=1-1, to=1-2]
		\arrow["\beta", from=1-2, to=1-3]
		\arrow["\alpha", from=1-3, to=1-4]
	\end{tikzcd}\]
	to be the subset $f(\mathcal{Z'})^{-1}\circ \{ R\alpha, f(\mathcal{Y'})\circ\beta, \gamma\} $.\\
	\end{definition}
	
	\begin{lem} The Toda bracket of \[\begin{tikzcd}
		{\mathcal{W}} & {\mathcal{X}} & {\mathcal{Y'}} & {\mathcal{Z'}}
		\arrow["\gamma", from=1-1, to=1-2]
		\arrow["\beta", from=1-2, to=1-3]
		\arrow["\alpha", from=1-3, to=1-4]
	\end{tikzcd}\] is independent of the choice of functorial fibrant replacements.\\
	\end{lem}
	
\begin{proof} Let $R(-)$ be a functorial fibrant replacement together with morphisms $f(\mathcal{S}): \mathcal{S}\rightarrow R(\mathcal{S})$ for every $\mathcal{S}\in \mathrm{sPre}(S)_{\ast} $ such that $f(-) $ is an acyclic cofibration. Let $R'(-)$ be another functorial fibrant replacement together with morphisms  $f'(\mathcal{S}): \mathcal{S}\rightarrow R'(\mathcal{S})$ for every $\mathcal{S}\in \mathrm{sPre}(S)_{\ast} $. We would like to show that for the following two sequences 
	\[\begin{tikzcd}
		{\mathcal{W}} & {\mathcal{X}} & {\mathcal{RY'}} & {\mathcal{RZ'}}
		\arrow["\gamma", from=1-1, to=1-2]
		\arrow["f(\mathcal{Y'})\circ \beta", from=1-2, to=1-3]
		\arrow["R\alpha", from=1-3, to=1-4]
	\end{tikzcd}\]
	and 
	\[\begin{tikzcd}
		{\mathcal{W}} & {\mathcal{X}} & {\mathcal{R'Y'}} & {\mathcal{R'Z'}}
		\arrow["\gamma", from=1-1, to=1-2]
		\arrow["f'(\mathcal{Y'})\circ \beta", from=1-2, to=1-3]
		\arrow["R'\alpha", from=1-3, to=1-4]
	\end{tikzcd}\]
	we have $f(\mathcal{Z'})^{-1}\circ \{ R\alpha, f(\mathcal{Y'})\circ\beta, \gamma\}=f'(\mathcal{Z'})^{-1}\circ \{ R'\alpha, f'(\mathcal{Y'})\circ\beta, \gamma\}  $.\\
	
	We first note that the diagram has a lift $\psi$
	\[\begin{tikzcd}
		{\mathcal{Z'}} & {\mathcal{R'Z'}} \\
		{\mathcal{RZ'}}
		\arrow["{f(\mathcal{Z})}"', from=1-1, to=2-1]
		\arrow["{f'(\mathcal{Z})}", from=1-1, to=1-2]
		\arrow["\psi"', dashed, from=2-1, to=1-2]
	\end{tikzcd}\] which is a motivic weak equivalence. We now show $f'(\mathcal{Z'})^{-1}\circ\{\psi\circ R\alpha, f(\mathcal{Y'})\circ\beta, \gamma\}=f(\mathcal{Z'})^{-1}\circ\{ R\alpha, f(\mathcal{Y'})\circ\beta, \gamma\}$.\\
	
	Let $B$ be a simplicial nullhomotopy of $f(\mathcal{Y'})\circ\beta\circ \gamma$ and $A$ be a simplicial nullhomotopy of $R\alpha\circ f(\mathcal{Y'})\circ\beta $. Then $\psi \circ A$ is a nullhomotopy of $\psi\circ R\alpha\circ f(\mathcal{Y'})\circ\beta $. Using $A$ and $B$ we get an element $H\in\{ R\alpha, f(\mathcal{Y'})\circ\beta, \gamma\} $ by the following commutative diagram
	\[\begin{tikzcd}
		{\mathcal{W}} && {(\mathcal{W}\wedge\Delta^1)_+} \\
		{} \\
		{(\mathcal{W}\wedge\Delta^1)_-} && {(\mathcal{W}\wedge\Delta^1)_+\sqcup_{\mathcal{W}}(\mathcal{W}\wedge\Delta^1)_-} \\
		&&&& {\mathcal{RZ'}} & \cdot
		\arrow["{{d^1}}", from=1-1, to=1-3]
		\arrow["{{d^1}}"', from=1-1, to=3-1]
		\arrow[from=3-1, to=3-3]
		\arrow[from=1-3, to=3-3]
		\arrow["{{CA\circ C(\gamma)}}"', curve={height=18pt}, from=3-1, to=4-5]
		\arrow["{{R\alpha\circ CB}}", curve={height=-12pt}, from=1-3, to=4-5]
		\arrow["{{\exists!H}}"', from=3-3, to=4-5]
	\end{tikzcd}\]
	Using $\psi\circ A$ and $B$ we get an element $H'\in\{\psi\circ R\alpha, f(\mathcal{Y'})\circ\beta, \gamma\} $ via the commutative diagram
	\[\begin{tikzcd}
		{\mathcal{W}} && {(\mathcal{W}\wedge\Delta^1)_+} \\
		{} \\
		{(\mathcal{W}\wedge\Delta^1)_-} && {(\mathcal{W}\wedge\Delta^1)_+\sqcup_{\mathcal{W}}(\mathcal{W}\wedge\Delta^1)_-} \\
		&&&& {\mathcal{R'Z'}} & \cdot
		\arrow["{{d^1}}", from=1-1, to=1-3]
		\arrow["{{d^1}}"', from=1-1, to=3-1]
		\arrow[from=3-1, to=3-3]
		\arrow[from=1-3, to=3-3]
		\arrow["{{C(\psi\circ A)\circ C(\gamma)}}"', curve={height=18pt}, from=3-1, to=4-5]
		\arrow["{{\psi\circ R\alpha\circ CB}}", curve={height=-12pt}, from=1-3, to=4-5]
		\arrow["{{\exists!H'}}"', from=3-3, to=4-5]
	\end{tikzcd}\]
	Therefore $H'$ is equal to $\psi\circ H$. Thus $f'(\mathcal{Z'})^{-1}\circ H'= f(\mathcal{Z'})^{-1}\circ\psi^{-1}\circ H'= f(\mathcal{Z'})^{-1}\circ\psi^{-1}\circ\psi\circ H= f(\mathcal{Z'})^{-1}\circ H $. It follows that $f(\mathcal{Z'})^{-1}\circ\{ R\alpha, f(\mathcal{Y'})\circ\beta, \gamma\}\subseteq f'(\mathcal{Z'})^{-1}\circ\{\psi\circ R\alpha, f(\mathcal{Y'})\circ\beta, \gamma\} $.\\
	
	Since $\mathcal{RZ'} $ and $\mathcal{R'Z'}$ are both bifibrant, the weak equivalence $ \psi$ is a homotopy equivalence. Let $\phi:\mathcal{R'Z'}\rightarrow \mathcal{RZ'}$ be a homotopy inverse for $\psi$. Then by Proposition~\ref{prop:propostion 2.2.10} we have $ \{ R\alpha, f(\mathcal{Y'})\circ\beta, \gamma\}=\{ \phi \circ \psi\circ R\alpha, f(\mathcal{Y'})\circ\beta, \gamma\}$. By the same argument as before we can show that\begin{align*}
		f'(\mathcal{Z'})^{-1}\circ\{\psi\circ R\alpha, f(\mathcal{Y'})\circ\beta, \gamma\}\subseteq  f(\mathcal{Z'})^{-1}\circ\{ \phi \circ \psi\circ R\alpha, f(\mathcal{Y'})\circ\beta, \gamma\}\end{align*} is equal to $f(\mathcal{Z'})^{-1}\circ \{ R\alpha, f(\mathcal{Y'})\circ\beta, \gamma\}$. As a consequence we get $f'(\mathcal{Z'})^{-1}\circ\{\psi\circ R\alpha, f(\mathcal{Y'})\circ\beta, \gamma\}=f(\mathcal{Z'})^{-1}\circ\{ R\alpha, f(\mathcal{Y'})\circ\beta, \gamma\}$.\\
	
	In the next step we can therefore consider the two sequences 
	\[\begin{tikzcd}
		{\mathcal{W}} & {\mathcal{X}} & {\mathcal{RY'}} & {\mathcal{R'Z'}}
		\arrow["\gamma", from=1-1, to=1-2]
		\arrow["f(\mathcal{Y'})\circ \beta", from=1-2, to=1-3]
		\arrow["\psi\circ R\alpha", from=1-3, to=1-4]
	\end{tikzcd}\]
	and 
	\[\begin{tikzcd}
		{\mathcal{W}} & {\mathcal{X}} & {\mathcal{R'Y'}} & {\mathcal{R'Z'}} & \cdot
		\arrow["\gamma", from=1-1, to=1-2]
		\arrow["{f'(\mathcal{Y'})\circ \beta}", from=1-2, to=1-3]
		\arrow["{R'\alpha}", from=1-3, to=1-4]
	\end{tikzcd}\]
	The diagram 
	\[\begin{tikzcd}
		{\mathcal{Y'}} & {\mathcal{R'Y'}} \\
		{\mathcal{RY'}}
		\arrow["{f(\mathcal{Y})}"', from=1-1, to=2-1]
		\arrow["{f'(\mathcal{Y})}", from=1-1, to=1-2]
		\arrow["\theta"', dashed, from=2-1, to=1-2]
	\end{tikzcd}\]
	has a lift $\theta$, since $f(\mathcal{Y})$ is an acyclic cofibration. Moreover we also have the commutative diagram in $\mathrm{sPre}(S)_{\ast} $ 
	\[\begin{tikzcd}
		{\mathcal{Y'}} & {\mathcal{R'Y'}} & {\mathcal{R'Z'}} && {\mathcal{Y'}} & {\mathcal{Z'}} & {\mathcal{R'Z'}} \\
		{\mathcal{RY'}} && {\mathcal{RZ'}} && {\mathcal{RY'}} && {\mathcal{RZ'}} & \cdot
		\arrow["{{f(\mathcal{Y'})}}"', from=1-1, to=2-1]
		\arrow["{{f'(\mathcal{Y'})}}", from=1-1, to=1-2]
		\arrow["{{R'\alpha}}", from=1-2, to=1-3]
		\arrow[""{name=0, anchor=center, inner sep=0}, "\psi"', from=2-3, to=1-3]
		\arrow["R\alpha"', from=2-1, to=2-3]
		\arrow["\alpha", from=1-5, to=1-6]
		\arrow["{{f'(\mathcal{Z'})}}", from=1-6, to=1-7]
		\arrow[""{name=1, anchor=center, inner sep=0}, "{{f(\mathcal{Y'})}}"{description}, from=1-5, to=2-5]
		\arrow["R\alpha"', from=2-5, to=2-7]
		\arrow["\psi"', from=2-7, to=1-7]
		\arrow["{{f(\mathcal{Z'})}}"{description}, from=1-6, to=2-7]
		\arrow[shorten <=13pt, shorten >=13pt, Rightarrow, no head, from=0, to=1]
	\end{tikzcd}\]
	Thus we get the following commutative diagram in $\mathcal{H}_{\bullet}(S)$
	\[\begin{tikzcd}
		{\mathcal{R'Y'}} && {\mathcal{R'Z'}} \\
		{\mathcal{RY'}} && {\mathcal{RZ'}} & \cdot
		\arrow["\psi"', from=2-3, to=1-3]
		\arrow["R\alpha"', from=2-1, to=2-3]
		\arrow["\theta", from=2-1, to=1-1]
		\arrow["{{R'\alpha}}", from=1-1, to=1-3]
	\end{tikzcd}\]
	Hence by Proposition~\ref{prop:propostion 2.2.10} we have $\{\psi\circ R\alpha, f(\mathcal{Y'})\circ\beta, \gamma\}= \{R'\alpha\circ\theta, f(\mathcal{Y'})\circ\beta, \gamma\}$. Now we have the following commutative diagram
	\[\begin{tikzcd}
		{\mathcal{W}} & {\mathcal{X}} & {\mathcal{RY'}} & {\mathcal{R'Z'}} \\
		{\mathcal{W}} & {\mathcal{X}} & {\mathcal{R'Y'}} & {\mathcal{R'Z'}} & \cdot
		\arrow["\gamma", from=1-1, to=1-2]
		\arrow["{{f(\mathcal{Y'})\circ \beta}}", from=1-2, to=1-3]
		\arrow["{{R'\alpha\circ\theta}}", from=1-3, to=1-4]
		\arrow["\gamma", from=2-1, to=2-2]
		\arrow["{{f'(\mathcal{Y'})\circ \beta}}", from=2-2, to=2-3]
		\arrow["{{R'\alpha}}", from=2-3, to=2-4]
		\arrow["{{\mathrm{id}}}"{description}, from=1-1, to=2-1]
		\arrow["{{\mathrm{id}}}"{description}, from=1-2, to=2-2]
		\arrow["\theta"{description}, from=1-3, to=2-3]
		\arrow["{{\mathrm{id}}}"{description}, from=1-4, to=2-4]
	\end{tikzcd}\]
	Let $B'$ be a simplicial nullhomotopy of $f(\mathcal{Y'})\circ\beta\circ \gamma$ and $A'$ be a simplicial nullhomotopy of $R'\alpha\circ\theta\circ f(\mathcal{Y'})\circ\beta=R'\alpha\ f'(\mathcal{Y'})\circ\beta $. Then $\theta\circ B'$ is a nullhomotopy for $ f'(\mathcal{Y'})\circ\beta\circ \gamma$. Via the nullhomotopies $B'$ and $A'$ we get an element $G\in \{R'\alpha\circ\theta, f(\mathcal{Y'})\circ\beta, \gamma\} $ induced by the commutative diagram 
	\[\begin{tikzcd}
		{\mathcal{W}} && {(\mathcal{W}\wedge\Delta^1)_+} \\
		{} \\
		{(\mathcal{W}\wedge\Delta^1)_-} && {(\mathcal{W}\wedge\Delta^1)_+\sqcup_{\mathcal{W}}(\mathcal{W}\wedge\Delta^1)_-} \\
		&&&& {\mathcal{R'Z'}} & \cdot
		\arrow["{{d^1}}", from=1-1, to=1-3]
		\arrow["{{d^1}}"', from=1-1, to=3-1]
		\arrow[from=3-1, to=3-3]
		\arrow[from=1-3, to=3-3]
		\arrow["{{C(A)\circ C(\gamma)}}"', curve={height=18pt}, from=3-1, to=4-5]
		\arrow["{{R'\alpha\circ\theta\circ CB}}", curve={height=-12pt}, from=1-3, to=4-5]
		\arrow["{{\exists!G}}"', from=3-3, to=4-5]
	\end{tikzcd}\]
	On the other hand we have $R'\alpha\circ\theta\circ CB= R'\alpha\circ C(\theta\circ B)$. It follows that $G\in\{ R'\alpha, f'(\mathcal{Y'})\circ\beta, \gamma\} $. Therefore we get $\{\psi\circ R\alpha, f(\mathcal{Y'})\circ\beta, \gamma\}=\{R'\alpha\circ\theta, f(\mathcal{Y'})\circ\beta, \gamma\}\subseteq \{ R'\alpha, f'(\mathcal{Y'})\circ\beta, \gamma\} $. Since $ R\mathcal{Y'}$ and $R'\mathcal{Y'}$ are again both bifibrant, the weak equivalence $\theta$ is a homotopy equivalence. In particular we can take a homotopy inverse $\epsilon:R'\mathcal{Y'}\rightarrow R\mathcal{Y'} $ for $\theta$ and apply exactly the same arguments above for $\theta$ to $\epsilon $. As a result we get then $\{ R'\alpha, f'(\mathcal{Y'})\circ\beta, \gamma\}\subseteq \{\psi\circ R\alpha, f(\mathcal{Y'})\circ\beta, \gamma\} $. In summary, the two Toda brackets $\{ R'\alpha, f'(\mathcal{Y'})\circ\beta, \gamma\} $ and $ \{\psi\circ R\alpha, f(\mathcal{Y'})\circ\beta, \gamma\}$ are equal. Thus we have $f(\mathcal{Z'})^{-1}\circ\{ R\alpha, f(\mathcal{Y'})\circ\beta, \gamma\}=f'(\mathcal{Z'})^{-1}\circ\{ R'\alpha, f'(\mathcal{Y'})\circ\beta, \gamma\} $.\\[1cm]
	\end{proof}
	
	\subsection{Properties of Toda brackets }
	\
	
	In this section we would like to show that motivic Toda brackets have many properties of topological Toda brackets. The key observation is that we can adapt Toda's proofs in \cite{Toda+1963} via the Quillen equivalence\begin{align*}
		|\cdot|	:\mathrm{sPre}(S)_{\ast \ \mathbb{A}^1-local\ inj}\rightleftarrows\mathrm{Pre}_{\Delta}(S)_{\ast\ \mathbb{A}^1-local\ inj}: Sing\ \ \ (\ast)
	\end{align*}
	to motivic Toda brackets. We denote suspension from the right by the simplicial circle by $\Sigma(-)$. Let $\mathcal{E}$ and $\mathcal{F}$ be in $\mathrm{sPre}(S)_{\ast}$. Then the Quillen equivalence $(\ast)$ induces a bijection \begin{align*}
		\mathcal{H}_{\bullet}(S)(\Sigma\mathcal{E}, \mathcal{F})\cong \mathcal{H}o_{\Delta}(S)(|\mathcal{E}|\wedge|\Delta^1/ \partial\Delta^1|,|\mathcal{F}|)
	\end{align*}
	where $\mathcal{H}o_{\Delta}(S)$ is the homotopy category associated to the $\mathbb{A}^1$-local injective model structure on $\mathrm{Pre}_{\Delta}(S)_{\ast}$. Note that this is not only a bijection, but actually an isomorphism of groups. Let $S^1$ be the topological circle and $\psi_1:|\Delta^1/ \partial\Delta^1|\rightarrow S^1$ be the canonical homeomorphism induced by the exponential function.\\

	\begin{prop}[{cf. \cite[Lemma 1.1]{Toda+1963}}] 
		\label{prop:proposition 2.3.1}
		Let
	\[\begin{tikzcd}
		{\mathcal{W}} & {\mathcal{X}} & {\mathcal{Y}} & {\mathcal{Z}}
		\arrow["\gamma", from=1-1, to=1-2]
		\arrow["\beta", from=1-2, to=1-3]
		\arrow["\alpha", from=1-3, to=1-4]
	\end{tikzcd}\] be a sequence of composable morphisms in $\mathrm{sPre}(S)_{\ast}$. Furthermore we require as usual that $\alpha\circ\beta$ and $\beta\circ\gamma$ are nullhomotopic. The Toda bracket $\{\alpha, \beta, \gamma\}$ is a double coset of the subgroups $\mathcal{H}_{\bullet}(S)(\Sigma\mathcal{X}, \mathcal{Z})\circ \Sigma\gamma $ and $\alpha\circ \mathcal{H}_{\bullet}(S)(\Sigma\mathcal{W}, \mathcal{Y})$ in $\mathcal{H}_{\bullet}(S)(\Sigma\mathcal{W}, \mathcal{Z}) $. If the group $\mathcal{H}_{\bullet}(S)(\Sigma\mathcal{W}, \mathcal{Z})$ is abelian, then $\{\alpha, \beta, \gamma\}$ is a coset of the subgroup $\mathcal{H}_{\bullet}(S)(\Sigma\mathcal{X}, \mathcal{Z})\circ \Sigma\gamma +\alpha\circ \mathcal{H}_{\bullet}(S)(\Sigma\mathcal{W}, \mathcal{Y}) $.\\
	\end{prop}
	
\begin{proof} Without loss of generality we may assume that $\mathcal{Y}$ and $\mathcal{Z}$ are $\mathbb{A}^1$-local injective fibrant.  We first take here again a simplicial nullhomotopy $B$ of $\beta\circ \gamma$ and a simplicial nullhomotopy $A$ of $\alpha\circ \beta$. Then we get a morphism from $(\mathcal{W}\wedge\Delta^1)_+\sqcup_{\mathcal{W}}(\mathcal{W}\wedge\Delta^1)_-$ to $\mathcal{Z}$ via the following commutative diagram 
	\[\begin{tikzcd}
		{\mathcal{W}} && {(\mathcal{W}\wedge\Delta^1)_+} \\
		{} \\
		{(\mathcal{W}\wedge\Delta^1)_-} && {(\mathcal{W}\wedge\Delta^1)_+\sqcup_{\mathcal{W}}(\mathcal{W}\wedge\Delta^1)_-} \\
		&&&& {\mathcal{Z}} & \cdot
		\arrow["{{d^1}}", from=1-1, to=1-3]
		\arrow["{{d^1}}"', from=1-1, to=3-1]
		\arrow[from=3-1, to=3-3]
		\arrow[from=1-3, to=3-3]
		\arrow["{{C(A)\circ C(\gamma)}}"', curve={height=18pt}, from=3-1, to=4-5]
		\arrow["{{\alpha\circ CB}}", curve={height=-12pt}, from=1-3, to=4-5]
		\arrow["{{\exists!}}"', from=3-3, to=4-5]
	\end{tikzcd}\]
	If we precompose this morphism with the canonical isomorphism defined in the previous section from $\Sigma\mathcal{W}$ to $(\mathcal{W}\wedge\Delta^1)_+\sqcup_{\mathcal{W}}(\mathcal{W}\wedge\Delta^1)_-$, then we get an element $H$ in $\{\alpha, \beta, \gamma\}$.\\
	
	Let $A'$ be another nullhomotopy of $\alpha\circ \beta$ and $B'$ be another nullhomotopy of $\beta\circ \gamma$. Using $A'$ and $B'$ we obtain an element $H'$ in $\{\alpha, \beta, \gamma\}$. Now we apply the geometric realization functor. First we get nullhomotopies $|A|$ and $|A'|$ for $|\alpha\circ\beta|$. Moreover, $|B|$ and $|B'|$ are nullhomotopies for $|\beta\circ\gamma|$. Then we observe that we can apply the classical construction for topological Toda brackets (see Definition~\ref{def1}) to these nullhomotopies. If we apply the construction to $|A|$ and $|B|$, then we get a morphism $G: |\mathcal{W}|\wedge |\Delta^1/ \partial\Delta^1|\rightarrow \mathcal{Z} $ which is equal to $|H|$ in $\mathcal{H}o_{\Delta}(S)$. Similarly we get a morphism $G': |\mathcal{W}|\wedge|\Delta^1/ \partial\Delta^1| \rightarrow \mathcal{Z} $ which is equal to $|H'|$ in $\mathcal{H}o_{\Delta}(S)$ if we apply the construction to $|A'|$ and $|B'|$. Now we can use the arguments in \cite[Lemma 1.1]{Toda+1963}.\\
	
	We construct morphisms $F:|\mathcal{X}|\wedge |\Delta^1/ \partial\Delta^1|\rightarrow |\mathcal{Z}|$ and $E:|\mathcal{W}|\wedge |\Delta^1/ \partial\Delta^1|\rightarrow |\mathcal{Y}| $. $F$ is defined in the following way \begin{align*}
		F(U)(x\wedge t)=
		\begin{cases}
			|A|(U)(x, 2t-1) & \text{if} \  \frac{1}{2}\leq t\leq 1\\
			|A'|(U)(x, 1-2t) & \text{if} \ 0\leq t \leq\frac{1}{2}
		\end{cases}
	\end{align*}
	for every $U\in \mathcal{S}\mathrm{m}_{S}$, $x\in|\mathcal{X}|(U) $ and $t\in |\Delta^1/ \partial\Delta^1|$. Similarly we define $E$ to be the morphism  
	\begin{align*}
		E(U)(w\wedge t)=
		\begin{cases}
			|B'|(U)(w, 2t-1) & \text{if} \  \frac{1}{2}\leq t\leq 1\\
			|B|(U)(w, 1-2t) & \text{if} \ 0\leq t \leq\frac{1}{2}
		\end{cases}
	\end{align*}
	for every $U\in \mathcal{S}\mathrm{m}_{S}$, $w\in|\mathcal{W}|(U) $ and $t\in |\Delta^1/ \partial\Delta^1|$.\\
	
	Just like for topological spaces we can define sums of morphisms in $\mathrm{Pre}_{\Delta}(S)_{\ast}(|\mathcal{W}|\wedge|\Delta^1/ \partial\Delta^1|, |\mathcal{Z}|)$. In particular we have the following sum $F\circ (|\gamma|\wedge \mathrm{id}_{|\Delta^1/ \partial\Delta^1|})+ G+ |\alpha|\circ E$ which is defined as follows:
	\begin{align*}
		w\wedge t\mapsto
		\begin{cases}
			|A'|(U)(|\gamma|(U)(w), 1-6t) & \text{if} \  0\leq t\leq \frac{1}{6}\\
			|A|(U)(|\gamma|(U)(w), 6t-1) & \text{if} \ \frac{1}{6}\leq t \leq\frac{1}{3}\\
			|A|(U)(|\gamma|(U)(w), 3-6t) & \text{if} \ \frac{1}{3}\leq t \leq\frac{1}{2}\\
			|\alpha|(U)\circ|B|(U)(w, 6t-3) & \text{if} \ \frac{1}{2}\leq t \leq\frac{2}{3}\\
			|\alpha|(U)\circ|B|(U)(w, 5-6t) & \text{if} \ \frac{2}{3}\leq t \leq\frac{5}{6}\\
			|\alpha|(U)\circ|B'|(U)(w, 6t-5) & \text{if} \ \frac{5}{6}\leq t \leq 1\\		
		\end{cases}
	\end{align*}
	for every $U\in \mathcal{S}\mathrm{m}_{S}$, $w\in|\mathcal{W}|(U) $ and $t\in |\Delta^1/ \partial\Delta^1|$.\\
	
	In the next step we can construct a pointed homotopy $|\mathcal{W}|\wedge|\Delta^1/ \partial\Delta^1|\wedge|\Delta^1|_{+} \rightarrow |\mathcal{Z}|  $ for  $F\circ (|\gamma|\wedge \mathrm{id}_{|\Delta^1/ \partial\Delta^1|})+ G+ |\alpha|\circ E$ which is defined in the following way
	\begin{align*} 
		w\wedge t\wedge s\mapsto
		\begin{cases}
			|A'|(U)(|\gamma|(U)(w), 1-6t) & \text{if} \  0\leq t\leq \frac{1}{6}\\
			|A|(U)(|\gamma|(U)(w), (6t-1)s) & \text{if} \ \frac{1}{6}\leq t \leq\frac{1}{3}\\
			|A|(U)(|\gamma|(U)(w), (3-6t)s) & \text{if} \ \frac{1}{3}\leq t \leq\frac{1}{2}\\
			|\alpha|(U)\circ|B|(U)(w, (6t-3)s) & \text{if} \ \frac{1}{2}\leq t \leq\frac{2}{3}\\
			|\alpha|(U)\circ|B|(U)(w, (5-6t)s) & \text{if} \ \frac{2}{3}\leq t \leq\frac{5}{6}\\
			|\alpha|(U)\circ|B'|(U)(w, 6t-5) & \text{if} \ \frac{5}{6}\leq t \leq 1\\		
		\end{cases}
	\end{align*}
	for every $U\in \mathcal{S}\mathrm{m}_{S}$, $w\in|\mathcal{W}|(U) $, $t\in |\Delta^1/ \partial\Delta^1|$ and $s\in |\Delta^1|$. Therefore $F\circ (|\gamma|\wedge \mathrm{id}_{|\Delta^1/ \partial\Delta^1|})+ G+ |\alpha|\circ E$ is homotopic to the morphism 
	\begin{align*}
		w\wedge t\mapsto
		\begin{cases}
			|A'|(U)(|\gamma|(U)(w), 1-6t) & \text{if} \  0\leq t\leq \frac{1}{6}\\
			|\alpha|\circ|\beta|\circ |\gamma|(U)(w) & \text{if} \ \frac{1}{6}\leq t \leq \frac{5}{6}\\
			|\alpha|(U)\circ|B'|(U)(w, 6t-5) & \text{if} \ \frac{5}{6}\leq t \leq 1\\		
		\end{cases}
	\end{align*}
	for every $U\in \mathcal{S}\mathrm{m}_{S}$, $w\in|\mathcal{W}|(U) $ and $t\in |\Delta^1/ \partial\Delta^1|$. It is easy to see that this morphism is homotopic to $G'$. Hence $G'$ is homotopic to the sum $F\circ (|\gamma|\wedge \mathrm{id}_{|\Delta^1/ \partial\Delta^1|})+ G+ |\alpha|\circ E$. It follows that in $\mathcal{H}o_{\Delta}(S)$ we have $
	|H'|=G'= F\circ (|\gamma|\wedge \mathrm{id}_{|\Delta^1/ \partial\Delta^1|})+ G+ |\alpha|\circ E=  F\circ (|\gamma|\wedge \mathrm{id}_{|\Delta^1/ \partial\Delta^1|})+|H|+|\alpha|\circ E= F\circ |\Sigma \gamma|+G+ |\alpha|\circ E
	$. Therefore $H'$ is contained in the double coset $\mathcal{H}_{\bullet}(S)(\Sigma\mathcal{X}, \mathcal{Z})\circ \Sigma\gamma + H+ \alpha\circ \mathcal{H}_{\bullet}(S)(\Sigma\mathcal{W}, \mathcal{Y})$. It follows that $\{\alpha, \beta, \gamma\}\subseteq \mathcal{H}_{\bullet}(S)(\Sigma\mathcal{X}, \mathcal{Z})\circ \Sigma\gamma + H+ \alpha\circ \mathcal{H}_{\bullet}(S)(\Sigma\mathcal{W}, \mathcal{Y})$.\\
	
	Next we prove that $\mathcal{H}_{\bullet}(S)(\Sigma\mathcal{X}, \mathcal{Z})\circ \Sigma\gamma + H+ \alpha\circ \mathcal{H}_{\bullet}(S)(\Sigma\mathcal{W}, \mathcal{Y})\subseteq \{\alpha, \beta, \gamma\}  $. Since $\mathcal{Y}$ and $\mathcal{Z}$ are $\mathbb{A}^1$-local injective fibrant, elements in $\mathcal{H}_{\bullet}(S)(\Sigma\mathcal{X}, \mathcal{Z})$ and $\mathcal{H}_{\bullet}(S)(\Sigma\mathcal{W}, \mathcal{Y}) $ can be represented by morphisms in $\mathrm{sPre}(S)_{\ast}$. Let $F': \Sigma\mathcal{X}\rightarrow \mathcal{Z}$ be a morphism in $\mathrm{sPre}(S)_{\ast}$. We take also a morphism $E': \Sigma\mathcal{W}\rightarrow \mathcal{Y}$ in $\mathrm{sPre}(S)_{\ast}$. Let $\theta_{\mathcal{W}}: \mathcal{W}\wedge \Delta^1_{+}\rightarrow \Sigma\mathcal{W}$ and $\theta_{\mathcal{X}}: \mathcal{X}\wedge \Delta^1_{+}\rightarrow \Sigma\mathcal{X} $ denote the canonical quotient morphisms. Let $cyl(\mathcal{S})$ be again the functorial cylinder object defined as the pushout of the diagram 
	\[\begin{tikzcd}
		{\mathcal{S}} & {\mathcal{S}\wedge \Delta^1_+} \\
		{\mathcal{S}\wedge \Delta^1_+} & {cyl(\mathcal{S})}
		\arrow["{\mathrm{id}\times d^1}", from=1-1, to=1-2]
		\arrow["{\mathrm{id}\times d^0}"', from=1-1, to=2-1]
		\arrow["{j_-}"', from=2-1, to=2-2]
		\arrow["{j_+}", from=1-2, to=2-2]
	\end{tikzcd}\]
	for every $\mathcal{S}\in \mathrm{sPre}(S)_{\ast}$. Using the nullhomotopies $A$ and $B$ at the beginning of the proof we get the following nullhomotopies $\widehat{A}$ and $\widehat{B}$ with respect to the functorial cylinder object $cyl(-)$ as follows:
	\[\begin{tikzcd}
		{\mathcal{W}} & {\mathcal{W}\wedge \Delta^1_+} \\
		{\mathcal{W}\wedge \Delta^1_+} & {cyl(\mathcal{W})} \\
		&& {\mathcal{Y}}
		\arrow["{\mathrm{id}\times d^1}", from=1-1, to=1-2]
		\arrow["{\mathrm{id}\times d^0}"', from=1-1, to=2-1]
		\arrow[from=2-1, to=2-2]
		\arrow[from=1-2, to=2-2]
		\arrow["{\exists ! \widehat{B}}", from=2-2, to=3-3]
		\arrow["B"', curve={height=18pt}, from=2-1, to=3-3]
		\arrow["{E'\circ \theta_{\mathcal{W}}}", curve={height=-24pt}, from=1-2, to=3-3]
	\end{tikzcd}\]
	\[\begin{tikzcd}
		{\mathcal{X}} & {\mathcal{X}\wedge \Delta^1_+} \\
		{\mathcal{X}\wedge \Delta^1_+} & {cyl(\mathcal{X})} \\
		&& {\mathcal{Z}}
		\arrow["{\mathrm{id}\times d^1}", from=1-1, to=1-2]
		\arrow["{\mathrm{id}\times d^0}"', from=1-1, to=2-1]
		\arrow[from=2-1, to=2-2]
		\arrow[from=1-2, to=2-2]
		\arrow["{\exists ! \widehat{A}}", from=2-2, to=3-3]
		\arrow["A"', curve={height=18pt}, from=2-1, to=3-3]
		\arrow["{F'\circ \theta_{\mathcal{X}}}", curve={height=-24pt}, from=1-2, to=3-3]
	\end{tikzcd}\]
	So $\widehat{A}$ is a nullhomotopy for $\alpha\circ \beta$ and $\widehat{B}$ is a nullhomotopy for $\beta\circ \gamma$. Therefore we get also a morphism $\widehat{H}$ from $C(cyl(\mathcal{W}))_+\sqcup_{\mathcal{W}}C(cyl(\mathcal{W}))_-$ to $\mathcal{Z}$ induced by the commutative diagram
	\[\begin{tikzcd}
		{\mathcal{W}} && {C(cyl(\mathcal{W}))_+} \\
		\\
		{C(cyl(\mathcal{W}))_-} && {\mathcal{Z}} & \cdot
		\arrow["{{i_0(\mathcal{W})}}", from=1-1, to=1-3]
		\arrow["{{i_0(\mathcal{W})}}"', from=1-1, to=3-1]
		\arrow["{{\alpha\circ C\widehat{B}}}", from=1-3, to=3-3]
		\arrow["{{C\widehat{A}\circ C(\gamma)}}"', from=3-1, to=3-3]
	\end{tikzcd}\]
	We have a motivic weak equivalence $\chi$ from $C(cyl(\mathcal{W}))_+\sqcup_{\mathcal{W}}C(cyl(\mathcal{W}))_-$ to $\Sigma\mathcal{W}$ defined by collapsing the two cylinders in the middle and the lower cone to the base point. Then $\widehat{H}\circ \chi^{-1}$ is an element of $\{\alpha, \beta, \gamma\}$.\\
	
	In the next step we apply again the geometric realization functor. First note that the morphism $|\chi|$ has a canonical homotopy inverse $\lambda: |\mathcal{W}|\wedge |\Delta^1/ \partial\Delta^1|\rightarrow |C(cyl(\mathcal{W}))_+\sqcup_{\mathcal{W}}C(cyl(\mathcal{W}))_-| $. We illustrate this homotopy inverse as follows:
	
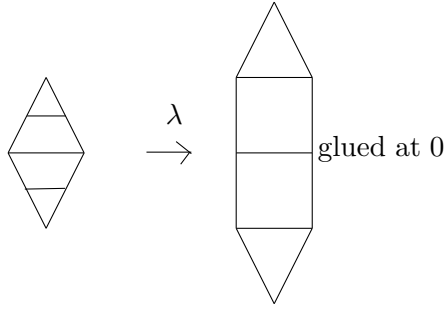
\begin{center}
		\begin{tikzpicture}
			\draw[draw=black, thin, solid] (-1.50,1.00) -- (-2.00,0.00);
			\draw[draw=black, thin, solid] (-1.50,1.00) -- (-1.00,0.00);
			\draw[draw=black, thin, solid] (-2.00,0.00) -- (-1.50,-1.00);
			\draw[draw=black, thin, solid] (-1.00,0.00) -- (-1.50,-1.00);
			\draw[draw=black, thin, solid] (-2.00,0.00) -- (-1.00,0.00);
			\draw[draw=black, thin, solid] (-1.78,-0.48) -- (-1.25,-0.47);
			\draw[draw=black, thin, solid] (-1.75,0.49) -- (-1.24,0.49);
			\draw[draw=black, thin, solid] (-0.18,0.01) -- (0.41,0.01);
			\draw[draw=black, thin, solid] (0.41,-0.01) -- (0.21,0.15);
			\draw[draw=black, thin, solid] (0.39,0.01) -- (0.27,-0.15);
			\draw[draw=black, thin, solid] (1.50,2.00) -- (1.00,1.00);
			\draw[draw=black, thin, solid] (1.50,2.00) -- (2.00,1.00);
			\draw[draw=black, thin, solid] (1.00,1.00) -- (2.00,1.00);
			\draw[draw=black, thin, solid] (1.00,1.00) -- (1.00,0.00);
			\draw[draw=black, thin, solid] (2.00,1.00) -- (2.00,0.00);
			\draw[draw=black, thin, solid] (1.00,0.00) -- (2.00,0.00);
			\draw[draw=black, thin, solid] (1.00,0.00) -- (1.00,-1.00);
			\draw[draw=black, thin, solid] (2.00,0.00) -- (2.00,-1.00);
			\draw[draw=black, thin, solid] (1.00,-1.00) -- (2.00,-1.00);
			\draw[draw=black, thin, solid] (1.00,-1.00) -- (1.50,-2.00);
			\draw[draw=black, thin, solid] (2.00,-1.00) -- (1.50,-2.00);
			\node[black, anchor=south west] at (1.94,-0.25) {$\small{\mathrm{glued \ at \ 0}}$};
			\node[black, anchor=south west] at (-0.06,0.25) {$\lambda$};
		\end{tikzpicture}
	\end{center}
	Therefore $|\widehat{H}|\circ \lambda$ is equal to $|\widehat{H}\circ \chi^{-1}| $ in $\mathcal{H}o_{\Delta}(S)$. Especially, $|\widehat{H}|\circ \lambda:|\mathcal{W}|\wedge |\Delta^1/ \partial\Delta^1|\rightarrow |\mathcal{Z}| $ is the morphism 
	\begin{align*}
		w\wedge t\mapsto
		\begin{cases}
			|F'\circ \theta_{\mathcal{X}}|(U)(|\gamma|(U)(w), 1-4t) & \text{if} \  0\leq t\leq \frac{1}{4}\\
			|A|(U)(|\gamma|(U)(w), 2-4t) & \text{if} \ \frac{1}{4}\leq t \leq \frac{1}{2}\\
			|\alpha|(U)\circ|B|(U)(w, 4t-2) & \text{if} \ \frac{1}{2}\leq t \leq \frac{3}{4}\\	
			|\alpha|(U)\circ|E'\circ \theta_{\mathcal{W}}|(U)(w, 4t-3) & \text{if} \  \frac{3}{4}\leq t\leq 1			
		\end{cases}
	\end{align*}
	for every $U\in \mathcal{S}\mathrm{m}_{S}$, $w\in|\mathcal{W}|(U) $ and $t\in |\Delta^1/ \partial\Delta^1|$.\\
	
	Let $G: |\mathcal{W}|\wedge |\Delta^1/ \partial\Delta^1|\rightarrow |\mathcal{Z}|$ again denote the morphism which we obtained by applying the classical construction of Toda brackets to $A$ and $B$. Recall that $G=|H|$ in $\mathcal{H}o_{\Delta}(S)$. Then we see that in $\mathcal{H}o_{\Delta}(S)$ the morphism $|\widehat{H}|\circ \lambda$ is equal to $	-(|F'\circ |\Sigma\gamma|) + G+ |\alpha|\circ|E'| $. Thus we have $-(F'\circ \Sigma\gamma)+H+ \alpha\circ E'=(-F' )\circ \Sigma\gamma+H+ \alpha\circ E' =\widehat{H}\circ \chi^{-1}\in \{\alpha, \beta, \gamma\} $ in $\mathcal{H}o_{\bullet}(S)$. It follows that $\mathcal{H}_{\bullet}(S)(\Sigma\mathcal{X}, \mathcal{Z})\circ \Sigma\gamma + H+ \alpha\circ \mathcal{H}_{\bullet}(S)(\Sigma\mathcal{W}, \mathcal{Y})\subseteq \{\alpha, \beta, \gamma\} $.\\
	\end{proof}
	
	\begin{lem}[{cf. \cite[Lemma 1.2]{Toda+1963}}]
		\label{lem:lemma 2.3.2}
		Let\[\begin{tikzcd}
		{\mathcal{W}} & {\mathcal{X}} & {\mathcal{Y}} & {\mathcal{Z}}
		\arrow["\gamma", from=1-1, to=1-2]
		\arrow["\beta", from=1-2, to=1-3]
		\arrow["\alpha", from=1-3, to=1-4]
	\end{tikzcd}\]
	be a sequence of composable morphisms.\end{lem}
	\begin{itemize}
		\item[0.] \textit{If one of $\alpha, \beta$ or $\gamma$ is constant and $\alpha\circ\beta= \beta\circ\gamma=0$, then $\{\alpha, \beta, \gamma\}$ contains $0 $.}
	\end{itemize}
	\textit{Let $\delta: \mathcal{V}\rightarrow \mathcal{W}$ be a morphism.}
	\begin{itemize}
		\item[1.] \textit{If $\alpha\circ\beta=\beta\circ\gamma=0$, then $\{\alpha, \beta, \gamma\}\circ \Sigma\delta\subseteq \{\alpha, \beta, \gamma\circ\delta\}. $}
		\item[2.] \textit{If $\alpha\circ\beta=\beta\circ\gamma\circ\delta=0$, then $\{\alpha, \beta, \gamma\circ\delta\} \subseteq \{\alpha, \beta\circ\gamma, \delta\}. $ }
		\item[3.]  \textit{If $\alpha\circ\beta=\beta\circ\gamma=\gamma\circ\delta=0$, then $\{\alpha\circ \beta, \gamma, \delta\} \subseteq \{\alpha, \beta\circ\gamma, \delta\}. $ }
		\item[4.] \textit{If $\beta\circ\gamma=\gamma\circ\delta=0$, then $\alpha\circ \{\beta, \gamma, \delta\}\subseteq \{\alpha\circ\beta, \gamma, \delta\}. $\\}
	\end{itemize}
	
\begin{proof} Without loss of generality we assume that $\mathcal{Y}$ and $\mathcal{Z}$ are $\mathbb{A}^1$-fibrant. We first prove statement 0. If $\alpha$ is constant, then we can choose the constant morphism $A: \mathcal{X}\wedge\Delta^1_+\rightarrow \mathcal{Z}$ as nullhomotopy. Take a nullhomotopy $B:\mathcal{W}\wedge\Delta^1_+\rightarrow \mathcal{Y} $. Then we get a morphism from $(\mathcal{W}\wedge\Delta^1)_+\sqcup_{\mathcal{W}}(\mathcal{W}\wedge\Delta^1)_-$ to $\mathcal{Z}$ via the following commutative diagram 
	\[\begin{tikzcd}
		{\mathcal{W}} && {(\mathcal{W}\wedge\Delta^1)_+} \\
		{} \\
		{(\mathcal{W}\wedge\Delta^1)_-} && {(\mathcal{W}\wedge\Delta^1)_+\sqcup_{\mathcal{W}}(\mathcal{W}\wedge\Delta^1)_-} \\
		&&&& {\mathcal{Z}} & \cdot
		\arrow["{{d^1}}", from=1-1, to=1-3]
		\arrow["{{d^1}}"', from=1-1, to=3-1]
		\arrow[from=3-1, to=3-3]
		\arrow[from=1-3, to=3-3]
		\arrow["{{C(A)\circ C(\gamma)}}"', curve={height=18pt}, from=3-1, to=4-5]
		\arrow["{{\alpha\circ CB}}", curve={height=-12pt}, from=1-3, to=4-5]
		\arrow["{{\exists!}}"', from=3-3, to=4-5]
	\end{tikzcd}\]
	But $C(A)\circ C(\gamma)$ and $\alpha\circ CB $ are constant, therefore we get the constant morphism from $(\mathcal{W}\wedge\Delta^1)_+\sqcup_{\mathcal{W}}(\mathcal{W}\wedge\Delta^1)_-$ to $\mathcal{Z}$. It follows that we obtain 0 in $\{\alpha, \beta, \gamma\}$. The other two cases are proved similarly.\\
	
	Now we prove statement 1. Let $A$ be a simplicial nullhomotopy of $\alpha\circ\beta$ and $B$ be a nullhomotopy of $\beta\circ \gamma$. Then $B\circ(\delta\wedge\mathrm{id}_{\Delta^1_{+}})$ is a nullhomotopy for $\beta\circ \gamma\circ \delta $. Using $A$ and $B$ we obtain a morphism $H$ induced by the commutative diagram
	\[\begin{tikzcd}
		{\mathcal{W}} && {(\mathcal{W}\wedge\Delta^1)_+} \\
		{} \\
		{(\mathcal{W}\wedge\Delta^1)_-} && {(\mathcal{W}\wedge\Delta^1)_+\sqcup_{\mathcal{W}}(\mathcal{W}\wedge\Delta^1)_-} \\
		&&&& {\mathcal{Z}} & \cdot
		\arrow["{{d^1}}", from=1-1, to=1-3]
		\arrow["{{d^1}}"', from=1-1, to=3-1]
		\arrow[from=3-1, to=3-3]
		\arrow[from=1-3, to=3-3]
		\arrow["{{C(A)\circ C(\gamma)}}"', curve={height=18pt}, from=3-1, to=4-5]
		\arrow["{{\alpha\circ CB}}", curve={height=-12pt}, from=1-3, to=4-5]
		\arrow["{{\exists! H}}"', from=3-3, to=4-5]
	\end{tikzcd}\]
	Using $A$ and $B\circ\delta$ we get a morphism $H'$ by the commutative diagram
	\[\begin{tikzcd}
		{\mathcal{V}} && {(\mathcal{V}\wedge\Delta^1)_+} \\
		{} \\
		{(\mathcal{V}\wedge\Delta^1)_-} && {(\mathcal{V}\wedge\Delta^1)_+\sqcup_{\mathcal{V}}(\mathcal{V}\wedge\Delta^1)_-} \\
		&&&& {\mathcal{Z}} & \cdot
		\arrow["{{d^1}}", from=1-1, to=1-3]
		\arrow["{{d^1}}"', from=1-1, to=3-1]
		\arrow[from=3-1, to=3-3]
		\arrow[from=1-3, to=3-3]
		\arrow["{{C(A)\circ C(\gamma\circ\delta)}}"', curve={height=18pt}, from=3-1, to=4-5]
		\arrow["{{\alpha\circ C(B\circ(\delta\wedge\mathrm{id}_{\Delta^1_{+}}))}}", curve={height=-12pt}, from=1-3, to=4-5]
		\arrow["{{\exists! H'}}"', from=3-3, to=4-5]
	\end{tikzcd}\]
	Furthermore we have the following commutative diagram 
	\[\begin{tikzcd}
		{(\mathcal{V}\wedge\Delta^1)_+\sqcup_{\mathcal{V}}(\mathcal{V}\wedge\Delta^1)_-} & {(\mathcal{W}\wedge\Delta^1)_+\sqcup_{\mathcal{W}}(\mathcal{W}\wedge\Delta^1)_-} \\
		{\Sigma\mathcal{V}} & {\Sigma\mathcal{W}}
		\arrow[from=1-1, to=1-2]
		\arrow["\Sigma\gamma", from=2-1, to=2-2]
		\arrow["\sim"', from=1-1, to=2-1]
		\arrow["\sim", from=1-2, to=2-2]
	\end{tikzcd}\]
	where the top morphism is induced by $\gamma$ and the vertical weak equivalences are defined by collapsing the lower cones. Therefore $H\circ\Sigma\delta= H'$. It follows that $\{\alpha, \beta, \gamma\}\circ \Sigma\delta\subseteq \{\alpha, \beta, \gamma\circ\delta\} $. The other statements can be proved by similar arguments.\\
	\end{proof}
	
	\begin{prop}[{cf. \cite[Lemma 1.4]{Toda+1963}}] 
		\label{Prop2.3.3}
		We consider the following sequence of composable morphisms in $\mathrm{sPre}(S)_{\ast}$
	\[\begin{tikzcd}
		{\mathcal{V}} & {\mathcal{W}} & {\mathcal{X}} & {\mathcal{Y}} & {\mathcal{Z}}
		\arrow["\gamma", from=1-2, to=1-3]
		\arrow["\beta", from=1-3, to=1-4]
		\arrow["\alpha", from=1-4, to=1-5]
		\arrow["\delta", from=1-1, to=1-2]
	\end{tikzcd}\]
where $\alpha\circ\beta, \beta\circ\gamma$ and $\gamma\circ\delta$ are nullhomotopic. Then we have \begin{align*}
			\{\alpha, \beta, \gamma\}\circ \Sigma \delta= -(\alpha\circ \{\beta, \gamma, \delta\}).	\end{align*}
			\end{prop} 
			\
			
\begin{proof} We may assume that $ \mathcal{X},\mathcal{Y} $ and $\mathcal{Z}$ are $\mathbb{A}^1$-fibrant. As in the proof of Lemma 2.3.2 each element of $\{\alpha, \beta, \gamma\}\circ \Sigma \delta$ is of the form $H'\circ \rho_{\mathcal{V}}^{-1}$ where $H' $ is induced by the commutative diagram 
	\[\begin{tikzcd}
		{\mathcal{V}} && {(\mathcal{V}\wedge\Delta^1)_+} \\
		{} \\
		{(\mathcal{V}\wedge\Delta^1)_-} && {(\mathcal{V}\wedge\Delta^1)_+\sqcup_{\mathcal{V}}(\mathcal{V}\wedge\Delta^1)_-} \\
		&&&& {\mathcal{Z}}
		\arrow["{d^1}", from=1-1, to=1-3]
		\arrow["{d^1}"', from=1-1, to=3-1]
		\arrow[from=3-1, to=3-3]
		\arrow[from=1-3, to=3-3]
		\arrow["{C(A)\circ C(\gamma\circ\delta)}"', curve={height=18pt}, from=3-1, to=4-5]
		\arrow["{\alpha\circ C(B\circ(\delta\wedge\mathrm{id}_{\Delta^1_{+}}))}", curve={height=-12pt}, from=1-3, to=4-5]
		\arrow["{\exists! H'}"', from=3-3, to=4-5]
	\end{tikzcd}\] 
	where $A$ is a simplicial nullhomotopy of $\alpha\circ\beta$, $B$ is a nullhomotopy of $\beta\circ \gamma$ and $\rho_{\mathcal{V}}: (\mathcal{V}\wedge\Delta^1)_+\sqcup_{\mathcal{V}}(\mathcal{V}\wedge\Delta^1)_-\rightarrow \Sigma\mathcal{V}$ is defined by collapsing the lower cone to a point. Now we apply the geometric realization functor. Let $H: |\mathcal{W}|\wedge |\Delta^1/ \partial\Delta^1|\rightarrow |\mathcal{Z}|$ be the element which we get by applying the classical construction of Toda brackets to $|A|$ and $|B|$. Then $|H'\circ \rho_{\mathcal{V}}^{-1}|$ is equal to $H\circ|\Sigma\delta|$ in $\mathcal{H}o_{\Delta}(S)$. In the next step we use the homotopy given in the proof of \cite[Lemma 1.4]{Toda+1963}.\\
	
	Let $C$ be a simplicial nullhomotopy of $\gamma\circ\delta$. We define a morphism $G: |\mathcal{V}|\wedge |\Delta^1/ \partial\Delta^1|\wedge |\Delta^1|_{+}\rightarrow |\mathcal{Z}| $ by the formula 
	\begin{align*}
		v\wedge t\wedge u\mapsto
		\begin{cases}
			|A|(U)(|C|(U)(v, 1-2t),(2-2u)(1-2t)) & \text{if} \  0\leq t\leq \frac{1}{2} \ \text{and} \ \frac{1}{2}\leq u\leq 1 \\
			|A|(U)(|C|(U)(v, 2u(1-2t)), 1-2t) & \text{if} \ 0\leq t \leq \frac{1}{2} \ \text{and} \ 0\leq u\leq \frac{1}{2}\\
			|\alpha|(U)(|B|(U)(|\delta|(U)(v), 2t-1)) & \text{if} \ \frac{1}{2}\leq t \leq 1\\			
		\end{cases}
	\end{align*}
	for every $U\in \mathcal{S}\mathrm{m}_{S}$, $v\in|\mathcal{V}|(U) $, $t\in |\Delta^1/ \partial\Delta^1|$ and $u\in |\Delta^1|$. Then $G_0$ is the morphism $H\circ|\Sigma\delta|$ and $G_1: |\mathcal{V}|\wedge |\Delta^1/ \partial\Delta^1|\rightarrow |\mathcal{Z}| $ is the morphism \begin{align*}
		v\wedge t\mapsto
		\begin{cases}
			|\alpha\circ\beta|(U)(|C|(U)(v, 1-2t)) & \text{if} \  0\leq t\leq \frac{1}{2}\\
			|\alpha|(U)(|B|(U)(|\delta|(U)(v), 2t-1)) & \text{if} \ \frac{1}{2}\leq t \leq 1\\			
		\end{cases}
	\end{align*}
	for every $U\in \mathcal{S}\mathrm{m}_{S}$, $v\in|\mathcal{V}|(U) $ and $t\in |\Delta^1/ \partial\Delta^1|$.\\
	
	Now we define a morphism $K: |\mathcal{V}|\wedge |\Delta^1/ \partial\Delta^1|\rightarrow |\mathcal{Y}|$ by \begin{align*}
		v\wedge t\mapsto
		\begin{cases}
			|\beta|(U)(|C|(U)(v, 1-2t)) & \text{if} \  0\leq t\leq \frac{1}{2}\\
			|B|(U)(|\delta|(U)(v), 2t-1) & \text{if} \ \frac{1}{2}\leq t \leq 1\\			
		\end{cases}
	\end{align*}
	for every $U\in \mathcal{S}\mathrm{m}_{S}$, $v\in|\mathcal{V}|(U) $ and $t\in |\Delta^1/ \partial\Delta^1|$. We see that in particular $G_1$ is $|\alpha|\circ K$.\\
	
	If we apply the classical construction for Toda brackets to $|B|$ and $|C|$, then we obtain a morphism $K': |\mathcal{V}|\wedge |\Delta^1/ \partial\Delta^1|\rightarrow |\mathcal{Y}| $ which is exactly $-K$ in $\mathcal{H}o_{\Delta}(S)$. Thus $|H'\circ \rho_{\mathcal{V}}^{-1}|=H\circ|\Sigma\delta|$ equals to $|\alpha|\circ (-K')=-(|\alpha|\circ K')$ in $\mathcal{H}o_{\Delta}(S)$.\\
	
	If we apply the usual construction for motivic Toda brackets to the nullhomotopies $B$ and $C$, then we get an element such that the value of the derived geometric realization functor for this element is $K'$ in $\mathcal{H}o_{\Delta}(S)$. Hence we have  $H'\circ \rho_{\mathcal{V}}^{-1}\in -(\alpha\circ \{\beta, \gamma, \delta\})$. Therefore the inclusion $\{\alpha, \beta, \gamma\}\circ \Sigma \delta \subseteq -(\alpha\circ \{\beta, \gamma, \delta\})$ holds.\\
	
	Next we prove the converse. Let $B'$ be a simplicial nullhomotopy for $\beta\circ \gamma$ and $C'$ a simplicial nullhomotopy of $\gamma\circ \delta$. Using $B$ and $C$ we get a morphism $\widehat{H}: (\mathcal{V}\wedge\Delta^1)_+\sqcup_{\mathcal{V}}(\mathcal{V}\wedge\Delta^1)_-\rightarrow \mathcal{Y} $ by the commutative diagram \[\begin{tikzcd}
		{\mathcal{V}} && {(\mathcal{V}\wedge\Delta^1)_+} \\
		{} \\
		{(\mathcal{V}\wedge\Delta^1)_-} && {(\mathcal{V}\wedge\Delta^1)_+\sqcup_{\mathcal{V}}(\mathcal{V}\wedge\Delta^1)_-} \\
		&&&& {\mathcal{Y}} & \cdot
		\arrow["{{d^1}}", from=1-1, to=1-3]
		\arrow["{{d^1}}"', from=1-1, to=3-1]
		\arrow[from=3-1, to=3-3]
		\arrow[from=1-3, to=3-3]
		\arrow["{{C(B')\circ C(\delta)}}"', curve={height=18pt}, from=3-1, to=4-5]
		\arrow["{{\beta\circ C(C')}}", curve={height=-12pt}, from=1-3, to=4-5]
		\arrow["{{\exists! \widehat{H}}}"', from=3-3, to=4-5]
	\end{tikzcd}\]
	Then $\alpha\circ\widehat{H}\circ \rho_{\mathcal{V}}^{-1}$ is contained in $\alpha\circ \{\beta, \gamma, \delta\}$. Here we apply again the geometric realization functor. If we apply the classical construction of Toda brackets to $|B'|$ and $|C'|$, then we get a morphism $\tilde{H}: |\mathcal{V}|\wedge |\Delta^1/ \partial\Delta^1|\rightarrow |\mathcal{Y}|$ such that $\tilde{H} $ equals to $|\widehat{H}\circ \rho_{\mathcal{V}}^{-1}|$ in $\mathcal{H}o_{\Delta}(S)$. In particular,  we have $|\alpha|\circ |\widehat{H}\circ \rho_{\mathcal{V}}^{-1}|= |\alpha|\circ \tilde{H}$. Recall that \begin{align*}
		\mathcal{H}_{\bullet}(S)(\Sigma\mathcal{V}, \mathcal{Z})\cong \mathcal{H}o_{\Delta}(S)(|\mathcal{V}|\wedge|\Delta^1/ \partial\Delta^1|,|\mathcal{Z}|)
	\end{align*}
	is a group isomorphism. Let $\sigma:|\Delta^1/ \partial\Delta^1|\rightarrow |\Delta^1/ \partial\Delta^1|$ be the map which sends $t$ to $1-t$. Then 
	$-(|\alpha|\circ \tilde{H})$ is represented by the morphism $ |\alpha|\circ \tilde{H}\circ \mathrm{id}_{|\mathcal{V}|}\wedge \sigma$.
	Again we consider the homotopy $G'$ defined similarly as before
	$G': |\mathcal{V}|\wedge |\Delta^1/ \partial\Delta^1|\wedge |\Delta^1|_{+}\rightarrow |\mathcal{Z}| $ by the formula 
	\begin{align*}
		v\wedge t\wedge u\mapsto
		\begin{cases}
			|A|(U)(|C'|(U)(v, 1-2t),(2-2u)(1-2t)) & \text{if} \  0\leq t\leq \frac{1}{2} \ \text{and} \ \frac{1}{2}\leq u\leq 1 \\
			|A|(U)(|C'|(U)(v, 2u(1-2t)), 1-2t) & \text{if} \ 0\leq t \leq \frac{1}{2} \ \text{and} \ 0\leq u\leq \frac{1}{2}\\
			|\alpha|(U)(|B'|(U)(|\delta|(U)(v), 2t-1)) & \text{if} \ \frac{1}{2}\leq t \leq 1\\			
		\end{cases}
	\end{align*}
	for every $U\in \mathcal{S}\mathrm{m}_{S}$, $v\in|\mathcal{V}|(U) $, $t\in |\Delta^1/ \partial\Delta^1|$ and $u\in |\Delta^1|$.
	We see that $G'_{1}$ is $|\alpha|\circ \tilde{H}\circ \mathrm{id}_{|\mathcal{V}|}\wedge \sigma$.\\
	
	Using the simplicial nullhomotopies $A $ and $B$ we get a morphism $H''$ by the following commutative diagram \[\begin{tikzcd}
		{\mathcal{V}} && {(\mathcal{V}\wedge\Delta^1)_+} \\
		{} \\
		{(\mathcal{V}\wedge\Delta^1)_-} && {(\mathcal{V}\wedge\Delta^1)_+\sqcup_{\mathcal{V}}(\mathcal{V}\wedge\Delta^1)_-} \\
		&&&& {\mathcal{Z}} & \cdot
		\arrow["{{d^1}}", from=1-1, to=1-3]
		\arrow["{{d^1}}"', from=1-1, to=3-1]
		\arrow[from=3-1, to=3-3]
		\arrow[from=1-3, to=3-3]
		\arrow["{{C(A)\circ C(\gamma\circ\delta)}}"', curve={height=18pt}, from=3-1, to=4-5]
		\arrow["{{\alpha\circ C(B'\circ(\delta\wedge\mathrm{id}_{\Delta^1_{+}}))}}", curve={height=-12pt}, from=1-3, to=4-5]
		\arrow["{{\exists! H''}}"', from=3-3, to=4-5]
	\end{tikzcd}\]
	Let $\tilde{\tilde{H}}=: |\mathcal{W}|\wedge |\Delta^1/ \partial\Delta^1|\rightarrow |\mathcal{Z}|$ be the element which we get by applying the classical construction of Toda brackets to $|A|$ and $|B'|$. It follows as before that $|H''\circ \rho_{\mathcal{V}}^{-1}|$ is equal to $\tilde{\tilde{H}}\circ|\Sigma\delta|$ in $\mathcal{H}o_{\Delta}(S)$. Especially for the homotopy $G'$ we get that $G'_0 $ is $\tilde{\tilde{H}}\circ|\Sigma\delta|$.
	Thus we have now $ |H''\circ \rho_{\mathcal{V}}^{-1}|=\tilde{\tilde{H}}\circ|\Sigma\delta|= G'_1= -(|\alpha|\circ \tilde{H})=- (|\alpha|\circ |\widehat{H}\circ \rho_{\mathcal{V}}^{-1}|)$ in $\mathcal{H}o_{\Delta}(S)$. It follows that $H''\circ \rho_{\mathcal{V}}^{-1}=  -(\alpha\circ \widehat{H}\circ \rho_{\mathcal{V}}^{-1})$. Therefore we obtain the inclusion $ -(\alpha\circ \{\beta, \gamma, \delta\})\subseteq \{\alpha, \beta, \gamma\}\circ \Sigma \delta$.\\
	\end{proof}
	
	For the next proposition we introduce the following notation. Let $\mathcal{S}\in \mathrm{sPre}(S)_{\ast} $. We set $C(\mathcal{S}):= \mathcal{S}\wedge \Delta^1$ where $\Delta^1$ is based at 1. Then we set $C'(\mathcal{S}):= \mathcal{S}\wedge \Delta^1$ where $\Delta^1$ is based at 0. Then we denote the pushout of the diagram \[\begin{tikzcd}
		{\mathcal{S}} & {C(\mathcal{S})} \\
		{C'(\mathcal{S})} & {\Sigma'\mathcal{S}}
		\arrow["{\mathrm{id}\times{0}}", from=1-1, to=1-2]
		\arrow["{\mathrm{id}\times{1}}"', from=1-1, to=2-1]
		\arrow[from=2-1, to=2-2]
		\arrow[from=1-2, to=2-2]
	\end{tikzcd}\]
	by $\Sigma'\mathcal{S} $. Note that $\Sigma'(-)$ is a functor. By collapsing the cone $C'(\mathcal{S}) $ to a point we get a weak equivalence $\sigma_{\mathcal{S}}$ from  $\Sigma'\mathcal{S} $ to $\Sigma\mathcal{S}$. In particular $\Sigma'\mathcal{S} $ is a cogroup object in $\mathcal{H}_{\bullet}(S)$ and $\sigma_{\mathcal{S}}$ induces a group isomorphism $\mathcal{H}_{\bullet}(S)(\Sigma\mathcal{S}, \mathcal{T} ) \cong \mathcal{H}_{\bullet}(S)(\Sigma'\mathcal{S}, \mathcal{T})$ for any object $ \mathcal{T}\in \mathrm{sPre}(S)_{\ast} $. Let $\theta$ and $ \theta'$ be elements in $ \mathrm{sPre}(S)_{\ast}(\mathcal{S}, \mathcal{T})$. Let $j_{\mathcal{S}}: C(\mathcal{S})\rightarrow \Sigma\mathcal{S}$ and $j'_{\mathcal{S}}: C'(\mathcal{S})\rightarrow \Sigma\mathcal{S}$ be the canonical projections. Then we can define a morphism from $\Sigma'\mathcal{S}\rightarrow$ to $ \mathcal{T} $ by the commutative diagram 
	\[\begin{tikzcd}
		{\mathcal{S}} & {C(\mathcal{S})} \\
		{C'(\mathcal{S})} & {\mathcal{T}} & \cdot
		\arrow["{{\mathrm{id}\times{0}}}", from=1-1, to=1-2]
		\arrow["{{\mathrm{id}\times{1}}}"', from=1-1, to=2-1]
		\arrow["{{\theta\circ j_{\mathcal{S}}}}"', from=2-1, to=2-2]
		\arrow["{{\theta'\circ j'_{\mathcal{S}}}}", from=1-2, to=2-2]
	\end{tikzcd}\]
	In particular the homotopy class of this morphism is equal to the morphism $(\theta+\theta')\circ \sigma_{\mathcal{S}}$ in $\mathcal{H}_{\bullet}(S)$. Therefore we denote this morphism also by $ \theta\circ \sigma_{\mathcal{S}}+ \theta'\circ \sigma_{\mathcal{S}}$.\\
	
	\begin{prop}[{cf. \cite[Lemma 1.6]{Toda+1963}}]
		\label{prop2.3.4}
		We consider the following sequence of composable morphisms \[\begin{tikzcd}
		{\mathcal{W}} & {\mathcal{X}} & {\mathcal{Y}} & {\mathcal{Z}}
		\arrow["\gamma", from=1-1, to=1-2]
		\arrow["\beta", from=1-2, to=1-3]
		\arrow["\alpha", from=1-3, to=1-4]
	\end{tikzcd}\]
	such that $\alpha\circ\beta$ and $\beta\circ\gamma$ are nullhomotopic. Then the following hold:
	\end{prop}
	\begin{itemize}
		\item[1.] \textit{Let $\gamma': \mathcal{W}\rightarrow \mathcal{X} $ be another morphism such that $\beta\circ\gamma' $ is nullhomotopic. Then we have $\{\alpha, \beta, \gamma\circ\sigma_{\mathcal{W'}}\}+\{\alpha, \beta, \gamma'\circ\sigma_{\mathcal{W'}}\}\supseteq \{\alpha, \beta, \gamma\circ\sigma_{\mathcal{W'}}+\gamma'\circ\sigma_{\mathcal{W'}}\}$ if $\mathcal{W}=\Sigma \mathcal{W'} $ for some $\mathcal{W'} \in \mathrm{sPre}(S)_{\ast}$.}
		\item[2.] \textit{Let $\beta': \mathcal{X}\rightarrow \mathcal{Y} $ be another morphism such that $\beta'\circ\gamma $ and $\alpha\circ\beta'$ are nullhomotopic. Then we have $\{\alpha, \beta\circ\sigma_{\mathcal{X'}}, \Sigma'\gamma\}+\{\alpha, \beta'\circ\sigma_{\mathcal{X'}}, \Sigma'\gamma\}=\{\alpha, \beta\circ\sigma_{\mathcal{X'}}+\beta'\circ\sigma_{\mathcal{X'}}, \Sigma'\gamma\}$ if $\mathcal{W}=\Sigma \mathcal{W'} $, $\mathcal{X}=\Sigma \mathcal{X'} $ and $\gamma=\Sigma \gamma'$ for some $\mathcal{W'}, \mathcal{X'}\in \mathrm{sPre}(S)_{\ast}$ and $\gamma'\in \mathrm{sPre}(S)_{\ast}(\mathcal{W'}, \mathcal{X'} )$.\\}
	\end{itemize}
	
\begin{proof} We prove first statement 1. Without loss of generality we assume that $\mathcal{Y}$ and $\mathcal{Z}$ are $\mathbb{A}^1$-fibrant. We consider now the sequences
	\[\begin{tikzcd}
		{\Sigma'\mathcal{W'}} & {\mathcal{X}} & {\mathcal{Y}} & {\mathcal{Z}}
		\arrow[from=1-1, to=1-2]
		\arrow["\beta", from=1-2, to=1-3]
		\arrow["\alpha", from=1-3, to=1-4]
	\end{tikzcd}\]
	where the first morphism is $\gamma\circ\sigma_{\mathcal{W'}}+\gamma'\circ\sigma_{\mathcal{W'}}$, $\gamma\circ\sigma_{\mathcal{W'}}$ or $\gamma'\circ\sigma_{\mathcal{W'}}$. The group $\mathcal{H}_{\bullet}(S)(\Sigma\Sigma'\mathcal{W'}, \mathcal{Z})$ is abelian. By Proposition~\ref{prop:proposition 2.3.1} the Toda bracket $\{\alpha, \beta, \gamma\circ\sigma_{\mathcal{W'}}+\gamma'\circ\sigma_{\mathcal{W'}}\} $ is then a coset of the subgroup $\mathcal{H}_{\bullet}(S)(\Sigma\mathcal{X}, \mathcal{Z})\circ\Sigma(\gamma\circ\sigma_{\mathcal{W'}}+\gamma'\circ\sigma_{\mathcal{W'}})+\alpha\circ \mathcal{H}_{\bullet}(S)(\Sigma\Sigma'\mathcal{W'}, \mathcal{Y}) $. On the other hand $\{\alpha, \beta, \gamma\circ\sigma_{\mathcal{W'}}\}+\{\alpha, \beta, \gamma'\circ\sigma_{\mathcal{W'}}\} $ is a coset of the subgroup $\mathcal{H}_{\bullet}(S)(\Sigma\mathcal{X}, \mathcal{Z})\circ\Sigma(\gamma\circ\sigma_{\mathcal{W'}})+\mathcal{H}_{\bullet}(S)(\Sigma\mathcal{X}, \mathcal{Z})\circ\Sigma(\gamma'\circ\sigma_{\mathcal{W'}})+\alpha\circ \mathcal{H}_{\bullet}(S)(\Sigma\Sigma'\mathcal{W'}, \mathcal{Y}) $. And we have the inclusion $\mathcal{H}_{\bullet}(S)(\Sigma\mathcal{X}, \mathcal{Z})\circ\Sigma(\gamma\circ\sigma_{\mathcal{W'}})+\mathcal{H}_{\bullet}(S)(\Sigma\mathcal{X}, \mathcal{Z})\circ\Sigma(\gamma'\circ\sigma_{\mathcal{W'}})+\alpha\circ \mathcal{H}_{\bullet}(S)(\Sigma\Sigma'\mathcal{W'}, \mathcal{Y})\supseteq \mathcal{H}_{\bullet}(S)(\Sigma\mathcal{X}, \mathcal{Z})\circ\Sigma(\gamma\circ\sigma_{\mathcal{W'}}+\gamma'\circ\sigma_{\mathcal{W'}})+\alpha\circ \mathcal{H}_{\bullet}(S)(\Sigma\Sigma'\mathcal{W'}, \mathcal{Y}) $. Therefore it suffices to show that $\{\alpha, \beta, \gamma\circ\sigma_{\mathcal{W'}}\}+\{\alpha, \beta, \gamma'\circ\sigma_{\mathcal{W'}}\} $ and $\{\alpha, \beta, \gamma\circ\sigma_{\mathcal{W'}}+\gamma'\circ\sigma_{\mathcal{W'}}\} $ have a common element.\\
	
	Let $A$ be a simplicial nullhomotopy of $\alpha\circ\beta$, $B$ a simplicial nullhomotopy of $\beta\circ\gamma$ and $B'$ a simplicial nullhomotopy of $\beta\circ\gamma'$. Using $A$ and $B $ we get an element $H$ in $\{\alpha, \beta, \gamma\}$. Using $A$ and $B'$ we get an element $H' $ in $\{\alpha, \beta, \gamma'\}$. We can get a nullhomotopy for $ \beta\circ( \gamma\circ\sigma_{\mathcal{W'}}+\gamma'\circ\sigma_{\mathcal{W'}})$ from $B$ and $B'$ in the following way. We know that $\Sigma'\mathcal{W'}\wedge\Delta^1_+$ is the pushout of the diagram 
	\[\begin{tikzcd}
		{\mathcal{W'}\wedge\Delta^1_+} && {C(\mathcal{W'})\wedge\Delta^1_+} \\
		\\
		{C'(\mathcal{W'})\wedge\Delta^1_+} && {\Sigma'\mathcal{W'}\wedge\Delta^1_+} & \cdot
		\arrow["{{(\mathrm{id}\times{0})\wedge\Delta^1_+}}", from=1-1, to=1-3]
		\arrow["{{(\mathrm{id}\times{1})\wedge\Delta^1_+}}"', from=1-1, to=3-1]
		\arrow[from=3-1, to=3-3]
		\arrow[from=1-3, to=3-3]
	\end{tikzcd}\]
	Then the commutative diagram 
	\[\begin{tikzcd}
		{\mathcal{W'}\wedge\Delta^1_+} && {C(\mathcal{W'})\wedge\Delta^1_+} \\
		\\
		{C'(\mathcal{W'})\wedge\Delta^1_+} && {\mathcal{Y}}
		\arrow["{(\mathrm{id}\times{0})\wedge\Delta^1_+}", from=1-1, to=1-3]
		\arrow["{(\mathrm{id}\times{1})\wedge\Delta^1_+}"', from=1-1, to=3-1]
		\arrow["{B\circ (j_{\mathcal{W'}}\wedge\Delta^1_+)}"', from=3-1, to=3-3]
		\arrow["{B'\circ (j_{\mathcal{W'}}\wedge\Delta^1_+)}", from=1-3, to=3-3]
	\end{tikzcd}\]
	induces a nullhomotopy $B'':\Sigma'\mathcal{W'}\wedge\Delta^1_+\rightarrow \mathcal{Y} $ of $\beta\circ( \gamma\circ\sigma_{\mathcal{W'}}+\gamma'\circ\sigma_{\mathcal{W'}})$.\\
	
	Now we apply the geometric realization functor. First note that $|\sigma_{\mathcal{W'}}|$ has a natural homotopy inverse $\widetilde{\sigma_{\mathcal{W'}}}:|\Sigma\mathcal{W'}| \rightarrow|\Sigma'\mathcal{W'}|$ which is defined by the formula
	\begin{align*}
		w'\wedge t\mapsto
		\begin{cases}
			w'\wedge 2t \in |C'(\mathcal{W'})|(U) & \text{if} \  0\leq t\leq \frac{1}{2} \\
			w'\wedge (2t-1) \in |C(\mathcal{W'})|(U)& \text{if} \ \frac{1}{2} \leq t \leq 1	
		\end{cases}
	\end{align*}
	for every $U\in \mathcal{S}\mathrm{m}_{S}$, $w'\in|\mathcal{W'}|(U) $ and $t\in |\Delta^1/ \partial\Delta^1|$. Therefore $|H\circ\Sigma_{\sigma_{\mathcal{W'}}}+ H'\circ\Sigma_{\sigma_{\mathcal{W'}}}|\circ(\widetilde{\sigma_{\mathcal{W'}}}\wedge|\Delta^1/ \partial\Delta^1|)$ is $|H|+|H'|$. As before if we apply the classical construction for Toda brackets to $|A|$ and $|B|$, we get a morphism $\tilde{H}: |\Sigma\mathcal{W'}|\wedge|\Delta^1/ \partial\Delta^1|\rightarrow |\mathcal{Z}|$ which is $|H|$ in $\mathcal{H}o_{\Delta}(S)$. If we apply the same construction to $|A|$ and $|B'|$, we get a morphism $\tilde{\tilde{H}}$ which is $|H'|$ in $\mathcal{H}o_{\Delta}(S)$. In particular, $|B''|\circ (\widetilde{\sigma_{\mathcal{W'}}}\wedge|\Delta^1|_+ ):|\Sigma\mathcal{W'}|\wedge|\Delta^1|_+\rightarrow|\mathcal{Y}| $ is the morphism
	\begin{align*}
		w'\wedge t\wedge u\mapsto
		\begin{cases}
			|B|(U)(w'\wedge 2t\wedge u) & \text{if} \  0\leq t\leq \frac{1}{2} \\
			|B'|(U)(w'\wedge (2t-1)\wedge u) & \text{if} \ \frac{1}{2} \leq t \leq 1	
		\end{cases}
	\end{align*}
	for every $U\in \mathcal{S}\mathrm{m}_{S}$, $w'\in|\mathcal{W'}|(U) $, $t\in |\Delta^1/ \partial\Delta^1|$ and $u\in |\Delta^1|_+$. It is a nullhomotopy for $|\beta\circ\gamma|+|\beta\circ\gamma'|$. Now we apply the classical construction for Toda brackets to $|A|$ and $|B''|\circ (\widetilde{\sigma_{\mathcal{W'}}}\wedge|\Delta^1|_+ )$ and we get the following morphism $G'$ from $|\Sigma\mathcal{W'}|\wedge|\Delta^1/\partial \Delta^1| $ to $|\mathcal{Z}|$ 
	\begin{align*}
		w'\wedge t\wedge u\mapsto
		\begin{cases}
			|A|(U)((|\gamma|(U)(w'\wedge2t))\wedge 1-2u) & \text{if} \  0\leq t\leq \frac{1}{2} \ \text{and} \ 0\leq u\leq \frac{1}{2} \\
			|\alpha|(U)(|B|(U)(w'\wedge2t\wedge 2u-1)) & \text{if} \ 0\leq t \leq \frac{1}{2} \ \text{and} \ \frac{1}{2}\leq u\leq 1\\
			|A|(U)((|\gamma'|(U)(w'\wedge2t-1))\wedge 1-2u) & \text{if} \  \frac{1}{2}\leq t\leq 1 \ \text{and} \ 0\leq u\leq \frac{1}{2} \\	
			|\alpha|(U)(|B'|(U)(w'\wedge2t-1\wedge 2u-1)) & \text{if} \ \frac{1}{2}\leq t \leq 1 \ \text{and} \ \frac{1}{2}\leq u\leq 1		
		\end{cases}
	\end{align*}
	for every $U\in \mathcal{S}\mathrm{m}_{S}$, $w'\in|\mathcal{W'}|(U) $, $t\in |\Delta^1/ \partial\Delta^1|$ and $u\in |\Delta^1/ \partial\Delta^1|$. But this morphism $G'$ is just $\tilde{H}+\tilde{\tilde{H}}$ where we use the first suspension coordinates for the addition.\\
	
	On the other hand we get an element $G\in \{\alpha, \beta, \gamma\circ\sigma_{\mathcal{W'}}+\gamma'\circ\sigma_{\mathcal{W'}}\} $ using $A$ and $B''$. If we apply the classical construction of Toda brackets to $|A|$ and $|B''|$, then we get a morphism $\tilde{G}: |\Sigma'\mathcal{W'}|\wedge|\Delta^1/ \partial\Delta^1|\rightarrow |\mathcal{Z}|$ which is equal to $|G|$ in $\mathcal{H}o_{\Delta}(S)$. Moreover we have that $\tilde{G}\circ (\widetilde{\sigma_{\mathcal{W'}}}\wedge|\Delta^1/ \partial\Delta^1|) $ is $G'=\tilde{H}+\tilde{\tilde{H}}$. It follows that \begin{align*}
		|G| \circ (\widetilde{\sigma_{\mathcal{W'}}}\wedge|\Delta^1/ \partial\Delta^1|)= \tilde{G}\circ (\widetilde{\sigma_{\mathcal{W'}}}\wedge|\Delta^1/ \partial\Delta^1|)\\= \tilde{H}+\tilde{\tilde{H}}\\=|H\circ\Sigma_{\sigma_{\mathcal{W'}}}+ H'\circ\Sigma_{\sigma_{\mathcal{W'}}}|\circ(\widetilde{\sigma_{\mathcal{W'}}}\wedge|\Delta^1/ \partial\Delta^1|). \end{align*}
	Thus we get $|G|=|H\circ\Sigma_{\sigma_{\mathcal{W'}}}+ H'\circ\Sigma_{\sigma_{\mathcal{W'}}}|$. It implies $G=H\circ\Sigma_{\sigma_{\mathcal{W'}}}+ H'\circ\Sigma_{\sigma_{\mathcal{W'}}}$ in $\mathcal{H}_{\bullet}(S)$. Note that $ H\circ\Sigma_{\sigma_{\mathcal{W'}}}\in \{\alpha, \beta, \gamma\circ\sigma_{\mathcal{W'}}\}$ and $H'\circ\Sigma_{\sigma_{\mathcal{W'}}}\in\{\alpha, \beta, \gamma'\circ\sigma_{\mathcal{W'}}\} $ by Lemma~\ref{lem:lemma 2.3.2}. Therefore $\{\alpha, \beta, \gamma\circ\sigma_{\mathcal{W'}}\}+\{\alpha, \beta, \gamma'\circ\sigma_{\mathcal{W'}}\} $ and $\{\alpha, \beta, \gamma\circ\sigma_{\mathcal{W'}}+\gamma'\circ\sigma_{\mathcal{W'}}\} $ have a common element. The proof of the remaining relation is similar.\\
	\end{proof}
	
	Let $\Sigma\mathcal{S}$ and $\mathcal{T}$ be in $\mathrm{sPre}(S)_{\ast}$. We may assume that $\mathcal{T} $ is $\mathbb{A}^1$-fibrant. Let $\tau $ and $\tau'$ be two elements in $\mathcal{H}_{\bullet}(S)(\Sigma\mathcal{S}, \mathcal{T})$. Then we abuse slightly the notation in the next corollary and denote any morphism from $\Sigma\mathcal{S} $ to $ \mathcal{T}$ in 
	$\mathrm{sPre}(S)_{\ast}$ which represents the sum $\tau+\tau'$ in $\mathcal{H}_{\bullet}(S)$ also by $\tau+\tau'$.\\
	
	\begin{cor} 
		\label{cor2.3.5}
	We consider the following sequence of composable morphisms \[\begin{tikzcd}
		{\mathcal{W}} & {\mathcal{X}} & {\mathcal{Y}} & {\mathcal{Z}}
		\arrow["\gamma", from=1-1, to=1-2]
		\arrow["\beta", from=1-2, to=1-3]
		\arrow["\alpha", from=1-3, to=1-4]
	\end{tikzcd}\]
	such that $\alpha\circ\beta$ and $\beta\circ\gamma$ are nullhomotopic. Then the following hold:
	\end{cor}
	\begin{itemize}
		\item[1.] \textit{Let $\gamma': \mathcal{W}\rightarrow \mathcal{X} $ be another morphism such that $\beta\circ\gamma' $ is nullhomotopic. Then we have $\{\alpha, \beta, \gamma\}+\{\alpha, \beta, \gamma'\}\supseteq \{\alpha, \beta, \gamma+\gamma'\}$ if $\mathcal{W}=\Sigma \mathcal{W'} $ for some $\mathcal{W'} \in \mathrm{sPre}(S)_{\ast}$ and $\mathcal{X} $ is $\mathbb{A}^1$-fibrant.}
		\item[2.] \textit{Let $\beta': \mathcal{X}\rightarrow \mathcal{Y} $ be another morphism such that $\beta'\circ\gamma $ and $\alpha\circ\beta'$ are nullhomotopic. Then we have $\{\alpha, \beta, \Sigma\gamma'\}+\{\alpha, \beta', \Sigma\gamma'\}=\{\alpha, \beta+\beta', \Sigma\gamma'\}$ if $\mathcal{W}=\Sigma \mathcal{W'} $, $\mathcal{X}=\Sigma \mathcal{X'} $ and $\gamma=\Sigma \gamma'$ for some $\mathcal{W'}, \mathcal{X'}\in \mathrm{sPre}(S)_{\ast}$ and $\gamma'\in \mathrm{sPre}(S)_{\ast}(\mathcal{W'}, \mathcal{X'} )$ and if $\mathcal{Y}$ is $\mathbb{A}^1$-fibrant.\\}
	\end{itemize}
	
\begin{proof} In $\mathcal{H}_{\bullet}(S)$ the weak equivalence $\Sigma\sigma_{\mathcal{W'}}$ induces a group isomorphism $ (\Sigma\sigma_{\mathcal{W'}})^{\ast}:\mathcal{H}_{\bullet}(S)(\Sigma\Sigma\mathcal{W'}, \mathcal{Z})\rightarrow \mathcal{H}_{\bullet}(S)(\Sigma\Sigma'\mathcal{W'}, \mathcal{Z})$. By Proposition 2.3.1 $\{\alpha, \beta, \gamma\}$ is the coset of the subgroup $\mathcal{H}_{\bullet}(S)(\Sigma\mathcal{X}, \mathcal{Z})\circ\Sigma\gamma +\alpha\circ \mathcal{H}_{\bullet}(S)(\Sigma\Sigma\mathcal{W'}, \mathcal{Y}) $. Therefore $\{\alpha, \beta, \gamma\}\circ\Sigma_{\sigma_{\mathcal{W'}}} $ is a coset of the subgroup $\mathcal{H}_{\bullet}(S)(\Sigma\mathcal{X}, \mathcal{Z})\circ\Sigma\gamma\circ\Sigma_{\sigma_{\mathcal{W'}}} +\alpha\circ \mathcal{H}_{\bullet}(S)(\Sigma\Sigma\mathcal{W'}, \mathcal{Y})\circ \Sigma_{\sigma_{\mathcal{W'}}}$. Since $ (\Sigma\sigma_{\mathcal{W'}})^{\ast}$ is an isomorphism, $\mathcal{H}_{\bullet}(S)(\Sigma\Sigma\mathcal{W'}, \mathcal{Y})\circ \Sigma_{\sigma_{\mathcal{W'}}}$ is the whole group $\mathcal{H}_{\bullet}(S)(\Sigma\Sigma'\mathcal{W'}, \mathcal{Z}) $.  On the other hand $\{\alpha, \beta, \gamma\circ\sigma_{\mathcal{W'}} \} $ is a coset of the same subgroup. By Lemma~\ref{lem:lemma 2.3.2} we also have the inclusion $\{\alpha, \beta, \gamma\circ\sigma_{\mathcal{W'}} \}\supseteq \{\alpha, \beta, \gamma\}\circ \Sigma\sigma_{\mathcal{W'}} $. Thus we get $\{\alpha, \beta, \gamma\circ\sigma_{\mathcal{W'}} \}=\{\alpha, \beta, \gamma\}\circ \Sigma\sigma_{\mathcal{W'}} $. By the same arguments we can show that $\{\alpha, \beta, \gamma'\circ\sigma_{\mathcal{W'}} \}=\{\alpha, \beta, \gamma'\}\circ \Sigma\sigma_{\mathcal{W'}} $ and $\{\alpha, \beta, (\gamma+\gamma')\circ\sigma_{\mathcal{W'}} \}=\{\alpha, \beta, \gamma+\gamma'\}\circ \Sigma\sigma_{\mathcal{W'}} $. It follows from Proposition~\ref{prop2.3.4} that we have the inclusion $\{\alpha, \beta, \gamma\circ\sigma_{\mathcal{W'}}\}+\{\alpha, \beta, \gamma'\circ\sigma_{\mathcal{W'}}\}\supseteq \{\alpha, \beta, \gamma\circ\sigma_{\mathcal{W'}}+\gamma'\circ\sigma_{\mathcal{W'}}\}$. Since $(\Sigma\sigma_{\mathcal{W'}})^{\ast}$ is an isomorphism, we obtain $\{\alpha, \beta, \gamma\}+\{\alpha, \beta, \gamma'\}\supseteq \{\alpha, \beta, \gamma+\gamma'\}$. The second statement can be proved in the same way.\\
	\end{proof}
	
	\begin{prop}
		\label{prop2.3.6} Let \[\begin{tikzcd}
		{\mathcal{W}} & {\mathcal{X}} & {\mathcal{Y}} & {\mathcal{Z}}
		\arrow["\gamma", from=1-1, to=1-2]
		\arrow["\beta", from=1-2, to=1-3]
		\arrow["\alpha", from=1-3, to=1-4]
	\end{tikzcd}\]
be a sequence of composable morphisms such that $\alpha\circ\beta$ and $\beta\circ\gamma$ are nullhomotopic. We assume that $\mathcal{Z}$ is $\mathbb{A}^1$-fibrant. Let $\alpha'$ be another morphism from $\mathcal{Y} $ to $\mathcal{Z}$. Then we have $\{\alpha, \Sigma\beta', \Sigma\gamma'\}+\{\alpha', \Sigma\beta', \Sigma\gamma'\}\supseteq \{\alpha+\alpha', \Sigma\beta', \Sigma\gamma'\}$ if $\mathcal{W}=\Sigma \mathcal{W'} $, $\mathcal{X}=\Sigma \mathcal{X'} $, $\mathcal{Y}=\Sigma \mathcal{Y'} $, $\gamma=\Sigma \gamma'$ and $\beta=\Sigma \beta'$ for some $\mathcal{W'}, \mathcal{X'},\mathcal{Y'} \in \mathrm{sPre}(S)_{\ast}$ and $\gamma'\in \mathrm{sPre}(S)_{\ast}(\mathcal{W'}, \mathcal{X'} )$ and $\beta'\in \mathrm{sPre}(S)_{\ast}(\mathcal{X'}, \mathcal{Y'} )$. \\
\end{prop}
	
\begin{proof} We take a functorial fibrant replacement $R(-)$ together with acyclic cofibrations $f_{\mathcal{S}}: \mathcal{S}\rightarrow R(\mathcal{S}) $ for every $\mathcal{S}\in\mathrm{sPre}(S)_{\ast} $. Then the following two diagrams 
	\[\begin{tikzcd}
		{\Sigma\mathcal{Y'}} & {\mathcal{Z}} \\
		{R(\Sigma\mathcal{Y'})}
		\arrow["{f_{\Sigma\mathcal{Y'}}}"', from=1-1, to=2-1]
		\arrow["\alpha", from=1-1, to=1-2]
		\arrow["a"{description}, dashed, from=2-1, to=1-2]
	\end{tikzcd}\]
	and
	\[\begin{tikzcd}
		{\Sigma\mathcal{Y'}} & {\mathcal{Z}} \\
		{R(\Sigma\mathcal{Y'})}
		\arrow["{f_{\Sigma\mathcal{Y'}}}"', from=1-1, to=2-1]
		\arrow["{\alpha'}", from=1-1, to=1-2]
		\arrow["{a'}"{description}, dashed, from=2-1, to=1-2]
	\end{tikzcd}\]
	have lifts. Furthermore we also have the commutative diagrams
	\[\begin{tikzcd}
		{\Sigma'\mathcal{W'}} & {\Sigma'\mathcal{X'}} \\
		{\Sigma\mathcal{W'}} & {\Sigma\mathcal{X'}} && {R(\Sigma\mathcal{Y'})} & {\mathcal{Z}}
		\arrow["{\Sigma'\gamma'}", from=1-1, to=1-2]
		\arrow["{\sigma_{\mathcal{W'}}}"', from=1-1, to=2-1]
		\arrow["{\Sigma\gamma'}", from=2-1, to=2-2]
		\arrow["{\sigma_{\mathcal{X'}}}"', from=1-2, to=2-2]
		\arrow["{f_{\Sigma\mathcal{Y'}}\circ\Sigma\beta'}", from=2-2, to=2-4]
		\arrow["a", from=2-4, to=2-5]
	\end{tikzcd}\] and 
	\[\begin{tikzcd}
		{\Sigma'\mathcal{W'}} & {\Sigma'\mathcal{X'}} \\
		{\Sigma\mathcal{W'}} & {\Sigma\mathcal{X'}} && {R(\Sigma\mathcal{Y'})} & {\mathcal{Z}} & \cdot
		\arrow["{{\Sigma'\gamma'}}", from=1-1, to=1-2]
		\arrow["{{\sigma_{\mathcal{W'}}}}"', from=1-1, to=2-1]
		\arrow["{{\Sigma\gamma'}}", from=2-1, to=2-2]
		\arrow["{{\sigma_{\mathcal{X'}}}}"', from=1-2, to=2-2]
		\arrow["{{f_{\Sigma\mathcal{Y'}}\circ\Sigma\beta'}}", from=2-2, to=2-4]
		\arrow["{{a'}}", from=2-4, to=2-5]
	\end{tikzcd}\]
	We first show that $\{a, f_{\Sigma\mathcal{Y'}}\circ\Sigma\beta'\circ\sigma_{\mathcal{X'}} , \Sigma'\gamma'\}+\{a', f_{\Sigma\mathcal{Y'}}\circ\Sigma\beta'\circ\sigma_{\mathcal{X'}}, \Sigma'\gamma'\}\supseteq \{a+a', f_{\Sigma\mathcal{Y'}}\circ\Sigma\beta'\circ\sigma_{\mathcal{X'}}, \Sigma'\gamma'\}$.\\
	
	Let $B$ be a simplicial nullhomotopy of $f_{\Sigma\mathcal{Y'}}\circ\Sigma\beta'\circ \Sigma\gamma' $. Let $A$ be a simplicial nullhomotopy of $a\circ f_{\Sigma\mathcal{Y'}}\circ\Sigma\beta'$ and $A'$ a simplicial nullhomotopy of  $a'\circ f_{\Sigma\mathcal{Y'}}\circ\Sigma\beta'$. Using $A$ and $B$ we get an element $H$ in $\{a, f_{\Sigma\mathcal{Y'}}\circ\Sigma\beta', \Sigma\gamma'\}$. With $A'$ and $B$ we get an element $H'$ in $\{a', f_{\Sigma\mathcal{Y'}}\circ\Sigma\beta', \Sigma\gamma'\}$. As in the proof of Proposition 2.3.4 the motivic space $\Sigma'\mathcal{X'}\wedge\Delta^1_+$ is the pushout of the diagram 
	\[\begin{tikzcd}
		{\mathcal{X'}\wedge\Delta^1_+} && {C(\mathcal{X'})\wedge\Delta^1_+} \\
		\\
		{C'(\mathcal{X'})\wedge\Delta^1_+} && {\Sigma'\mathcal{X'}\wedge\Delta^1_+} & \cdot
		\arrow["{{(\mathrm{id}\times{0})\wedge\Delta^1_+}}", from=1-1, to=1-3]
		\arrow["{{(\mathrm{id}\times{1})\wedge\Delta^1_+}}"', from=1-1, to=3-1]
		\arrow[from=3-1, to=3-3]
		\arrow[from=1-3, to=3-3]
	\end{tikzcd}\]
	Especially, the commutative diagram 
	\[\begin{tikzcd}
		{\mathcal{X'}\wedge\Delta^1_+} && {C(\mathcal{X'})\wedge\Delta^1_+} \\
		\\
		{C'(\mathcal{X'})\wedge\Delta^1_+} && {\mathcal{Z}}
		\arrow["{(\mathrm{id}\times{0})\wedge\Delta^1_+}", from=1-1, to=1-3]
		\arrow["{(\mathrm{id}\times{1})\wedge\Delta^1_+}"', from=1-1, to=3-1]
		\arrow["{A\circ (j_{\mathcal{X'}}\wedge\Delta^1_+)}"', from=3-1, to=3-3]
		\arrow["{A'\circ (j_{\mathcal{X'}}\wedge\Delta^1_+)}", from=1-3, to=3-3]
	\end{tikzcd}\]
	induces a nullhomotopy $A'':\Sigma'\mathcal{X'}\wedge\Delta^1_+\rightarrow \mathcal{Z} $ of $a\circ f_{\Sigma\mathcal{Y'}}\circ\Sigma\beta'\circ\sigma_{\mathcal{X'}}+a'\circ f_{\Sigma\mathcal{Y'}}\circ\Sigma\beta'\circ\sigma_{\mathcal{X'}}$. Therefore we also get an element $G$ in $\{a+a', f_{\Sigma\mathcal{Y'}}\circ\Sigma\beta'\circ\sigma_{\mathcal{X'}}, \Sigma'\gamma'\} $ using the nullhomotopies $A''$ and $B\circ \sigma_{W'}$.\\
	
	In the next step we apply again the geometric realization functor. If we apply the classical construction for Toda brackets to $|A|$ and $|B|$, then we obtain a morphism $\tilde{H}: |\Sigma\mathcal{W'}|\wedge|\Delta^1/ \partial\Delta^1|\rightarrow |\mathcal{Z}|$ which is $|H|$ in $\mathcal{H}o_{\Delta}(S)$. If we apply the same construction to $|A'|$ and $|B|$,  we get a morphism $\tilde{\tilde{H}}$ which is $|H'|$ in $\mathcal{H}o_{\Delta}(S)$. Let $\widetilde{\sigma_{\mathcal{W'}}}:|\Sigma\mathcal{W'}| \rightarrow|\Sigma'\mathcal{W'}|$ be defined as in the proof of Proposition 2.3.4. Analogously, $|\sigma_{\mathcal{X'}}|$ also has a natural homotopy inverse $\widetilde{\sigma_{\mathcal{X'}}}:|\Sigma\mathcal{X'}| \rightarrow|\Sigma'\mathcal{X'}|$ defined by the formula
	\begin{align*}
		x'\wedge t\mapsto
		\begin{cases}
			x'\wedge 2t \in |C'(\mathcal{X'})|(U) & \text{if} \  0\leq t\leq \frac{1}{2} \\
			x'\wedge (2t-1) \in |C(\mathcal{X'})|(U)& \text{if} \ \frac{1}{2} \leq t \leq 1	
		\end{cases}
	\end{align*}
	for every $U\in \mathcal{S}\mathrm{m}_{S}$, $x'\in|\mathcal{X'}|(U) $ and $t\in |\Delta^1/ \partial\Delta^1|$. Note that $|A''|\circ (\widetilde{\sigma_{\mathcal{X'}}}\wedge|\Delta^1|_+)$ is a nullhomotopy of $|a\circ f_{\Sigma\mathcal{Y'}}\circ\Sigma\beta'+a'\circ f_{\Sigma\mathcal{Y'}}\circ\Sigma\beta'|$. By applying the classical construction for Toda brackets to $|B|$ and $|A''|\circ (\widetilde{\sigma_{\mathcal{X'}}}\wedge|\Delta^1|_+)$ we get a morphism $G':|\Sigma\mathcal{W'}|\wedge |\Delta^1/ \partial\Delta^1|$. It follows from similar arguments as in the proof of Proposition 2.3.4 that $G'$ is $\tilde{H}+\tilde{\tilde{H}} $. On the other hand  $|G|\circ(\widetilde{\sigma_{\mathcal{W'}}}\wedge|\Delta^1/ \partial\Delta^1|) $ is equal to $G'$ in $\mathcal{H}o_{\Delta}(S)$. Again by similar arguments as for Proposition 2.3.4 we deduce that $\{a, f_{\Sigma\mathcal{Y'}}\circ\Sigma\beta'\circ\sigma_{\mathcal{X'}} , \Sigma'\gamma'\}+\{a', f_{\Sigma\mathcal{Y'}}\circ\Sigma\beta'\circ\sigma_{\mathcal{X'}}, \Sigma'\gamma'\}\supseteq \{a+a', f_{\Sigma\mathcal{Y'}}\circ\Sigma\beta'\circ\sigma_{\mathcal{X'}}, \Sigma'\gamma'\}$.\\
	
	By Proposition~\ref{prop:proposition 2.3.1} the Toda bracket $\{a, f_{\Sigma\mathcal{Y'}}\circ\Sigma\beta'\circ\sigma_{\mathcal{X'}} , \Sigma'\gamma'\}$ is a coset of the subgroup $\mathcal{H}_{\bullet}(S)(\Sigma\Sigma'\mathcal{X'}, \mathcal{Z})\circ\Sigma\Sigma'\gamma' +a\circ \mathcal{H}_{\bullet}(S)(\Sigma\Sigma'\mathcal{W'}, R(\Sigma\mathcal{Y'}))$. On the other hand $\{a, f_{\Sigma\mathcal{Y'}}\circ\Sigma\beta' , \Sigma\gamma'\} $ is a coset of \begin{align*}
		\mathcal{H}_{\bullet}(S)(\Sigma\Sigma\mathcal{X'}, \mathcal{Z})\circ\Sigma\Sigma\gamma' +a\circ \mathcal{H}_{\bullet}(S)(\Sigma\Sigma\mathcal{W'}, R(\Sigma\mathcal{Y'})) .\end{align*} Hence $\{a, f_{\Sigma\mathcal{Y'}}\circ\Sigma\beta' , \Sigma\gamma'\}\circ \Sigma\sigma_{\mathcal{W'}} $ is a coset of the subgroup \begin{align*} \mathcal{H}_{\bullet}(S)(\Sigma\Sigma\mathcal{X'}, \mathcal{Z})\circ\Sigma\Sigma\gamma'\circ\Sigma\sigma_{\mathcal{W'}} +a\circ \mathcal{H}_{\bullet}(S)(\Sigma\Sigma\mathcal{W'}, R(\Sigma\mathcal{Y'}))\circ \Sigma\sigma_{\mathcal{W'}}.\end{align*} Since $\Sigma\gamma'\circ\sigma_{\mathcal{W'}}=\sigma_{\mathcal{X'}}\circ \Sigma'\gamma' $ and $(\Sigma\sigma_{\mathcal{W'}})^{\ast}$ is an isomorphism, $\{a, f_{\Sigma\mathcal{Y'}}\circ\Sigma\beta' , \Sigma\gamma'\}\circ \Sigma\sigma_{\mathcal{W'}} $ and $\{a, f_{\Sigma\mathcal{Y'}}\circ\Sigma\beta'\circ\sigma_{\mathcal{X'}} , \Sigma'\gamma'\}$ are cosets of the same subgroup. By construction $\{a, f_{\Sigma\mathcal{Y'}}\circ\Sigma\beta' , \Sigma\gamma'\}\circ \Sigma\sigma_{\mathcal{W'}} $ and $\{a, f_{\Sigma\mathcal{Y'}}\circ\Sigma\beta'\circ\sigma_{\mathcal{X'}} , \Sigma'\gamma'\}$ have common elements, thus they are equal.\\
	
	By similar arguments we also have 
	$\{a', f_{\Sigma\mathcal{Y'}}\circ\Sigma\beta' , \Sigma\gamma'\}\circ \Sigma\sigma_{\mathcal{W'}}=\{a', f_{\Sigma\mathcal{Y'}}\circ\Sigma\beta'\circ\sigma_{\mathcal{X'}} , \Sigma'\gamma'\}$ and $\{a+a', f_{\Sigma\mathcal{Y'}}\circ\Sigma\beta' , \Sigma\gamma'\}\circ \Sigma\sigma_{\mathcal{W'}}=\{a+a', f_{\Sigma\mathcal{Y'}}\circ\Sigma\beta'\circ\sigma_{\mathcal{X'}} , \Sigma'\gamma'\}$. Now we see that $\{a, f_{\Sigma\mathcal{Y'}}\circ\Sigma\beta' , \Sigma\gamma'\}+\{a', f_{\Sigma\mathcal{Y'}}\circ\Sigma\beta' , \Sigma\gamma'\}\supseteq \{a+a', f_{\Sigma\mathcal{Y'}}\circ\Sigma\beta' , \Sigma\gamma'\}$. Then by Proposition~\ref{prop:propostion 2.2.10} and Defintion~\ref{def:definition 2.2.11} we obtain $\{\alpha, \Sigma\beta', \Sigma\gamma'\}+\{\alpha', \Sigma\beta', \Sigma\gamma'\}\supseteq \{\alpha+\alpha', \Sigma\beta', \Sigma\gamma'\}$.\\
	\end{proof}
	
	\begin{definition}
		\label{def:def2.3.7}
		 Let \[\begin{tikzcd}
		{\mathcal{W}} & {\mathcal{X}} & {\mathcal{Y}} & {\mathcal{Z}}
		\arrow["\gamma", from=1-1, to=1-2]
		\arrow["\beta", from=1-2, to=1-3]
		\arrow["\alpha", from=1-3, to=1-4]
	\end{tikzcd}\]
	be a sequence of composable morphisms such that $\mathcal{Y} $ and $\mathcal{Z}$ are $\mathbb{A}^1$-fibrant. We denote by $\mathcal{Y}\sqcup_{\beta}C(\mathcal{X}) $ the pushout of the diagram
	\[\begin{tikzcd}
		{\mathcal{X}} & {C(\mathcal{X})} \\
		{\mathcal{Y}} && \cdot
		\arrow["\beta"', from=1-1, to=2-1]
		\arrow["{{\mathrm{id}\times {0}}}", from=1-1, to=1-2]
	\end{tikzcd}\]
	A morphism $\bar{\alpha}: \mathcal{Y}\sqcup_{\beta}C(\mathcal{X})\rightarrow \mathcal{Z} $ will be called an extension of $\alpha$ if the restriction $\bar{\alpha}|_{\mathcal{Y}} $ is $\alpha$. A morphism $\tilde{\gamma}:C(\mathcal{W})_{+}\sqcup_{\mathcal{W}}C(\mathcal{W})_{-} \rightarrow \mathcal{Y}\sqcup_{\beta}C(\mathcal{X}) $ is a coextension of $\gamma$, if $\tilde{\gamma} $ satisfies the condition
	\begin{align*}
		\tilde{\gamma}(U)(w\wedge t)=
		\begin{cases}
			\gamma(U)(w)\wedge t  & \text{if} \ w\wedge t\in C(\mathcal{W})_{-}(U)  \\
			\in \mathcal{Y}(U) & \text{if} \  \ w\wedge t\in C(\mathcal{W})_{+}(U)	
		\end{cases}
	\end{align*}
	for every $U\in \mathcal{S}\mathrm{m}_{S}$ and every $w\wedge t \in C(\mathcal{W})_{+}\sqcup_{\mathcal{W}}C(\mathcal{W})_{-}(U)$.\\
	\end{definition}
	
	\begin{remark}
		\label{remark2.3.8}
		 It is easy to see that an extension $\bar{\alpha}$ exists if and only if $\alpha\circ\beta$ is nullhomotopic. A coextension $\tilde{\gamma}$ of $\gamma$ exists if and only if  $\beta\circ\gamma$ is nullhomotopic.\\
		 \end{remark}
	
\begin{lem}
	\label{lemma 2.3.9}
	Assume that $\alpha\circ\beta$ and $\beta\circ\gamma$ are nullhomotopic. Then the set $\{\bar{\alpha}\circ \tilde{\gamma}; \bar{\alpha}\ extension\ of\ \alpha, \ \tilde{\gamma}\ coextension\ of\ \gamma \}\circ \rho_{\mathcal{W}}^{-1} $ coincides with the Toda bracket $\{\alpha, \beta, \gamma\}$, where $\rho_{\mathcal{W}}: C(\mathcal{W})_{+}\sqcup_{\mathcal{W}}C(\mathcal{W})_{-}\rightarrow \Sigma\mathcal{W}$ is defined by collapsing the lower cone to a point.\\
	\end{lem}
	
\begin{proof} Let $\bar{\alpha}: \mathcal{Y}\sqcup_{\beta}C(\mathcal{X})\rightarrow \mathcal{Z}$ be an arbitrary extension of $\alpha$ and $\tilde{\gamma}:C(\mathcal{W})_{+}\sqcup_{\mathcal{W}}C(\mathcal{W})_{-} \rightarrow \mathcal{Y}\sqcup_{\beta}C(\mathcal{X})$ an arbitrary coextension of $\gamma$. We denote the canonical morphism from $C(\mathcal{X}) $ to $ \mathcal{Y}\sqcup_{\beta}C(\mathcal{X})$ by $\kappa_{C(\mathcal{X})}$. Then $\bar{\alpha}\circ\kappa_{C(\mathcal{X})}$ is a nullhomotopy of $\alpha\circ\beta$. Let $i_{C(\mathcal{W})_{+}}$ denote the canonical inclusion of $C(\mathcal{W}) $ into the top cone of $ C(\mathcal{W})_{+}\sqcup_{\mathcal{W}}C(\mathcal{W})_{-}$. Then $\tilde{\gamma}\circ i_{C(\mathcal{W})_{+}} $ factors through $\mathcal{Y}$, i.e., there is a morphism $B:C(\mathcal{W})\rightarrow \mathcal{Y}  $ such that the diagram\[\begin{tikzcd}
		{C(\mathcal{W})} && {C(\mathcal{W})_{+}\sqcup_{\mathcal{W}}C(\mathcal{W})_{-}} \\
		{\mathcal{Y}} && {\mathcal{Y}\sqcup_{\beta}C(\mathcal{X})}
		\arrow["B"', from=1-1, to=2-1]
		\arrow["{i_{C(\mathcal{W})_{+}}}", from=1-1, to=1-3]
		\arrow["{\tilde{\gamma}}", from=1-3, to=2-3]
		\arrow[hook, from=2-1, to=2-3]
	\end{tikzcd}\]commutes. Then $B$ is a nullhomotopy of $\beta\circ\gamma$. Now we see that $\bar{\alpha}\circ \tilde{\gamma}\circ \rho_{\mathcal{W}}^{-1}$ is just the element in $\{\alpha, \beta, \gamma\}$ obtained by the nullhomotopies $\bar{\alpha}\circ\kappa_{C(\mathcal{X})}$ and $B$.\\
	
	Conversely, let $A$ be a simplicial nullhomotopy of $\alpha\circ\beta$ and $B$ a simplicial nullhomotopy of $\beta\circ\gamma$. Then $A$ induces an extension of $\alpha$ by the commutative diagram
	\[\begin{tikzcd}
		{\mathcal{X}} & {C(\mathcal{X})} \\
		{\mathcal{Y}} & {\mathcal{Y}\sqcup_{\beta}C(\mathcal{X})} \\
		&& {\mathcal{Z}} & \cdot
		\arrow["{{\mathrm{id}\times {0}}}", from=1-1, to=1-2]
		\arrow["\beta"', from=1-1, to=2-1]
		\arrow[from=2-1, to=2-2]
		\arrow[from=1-2, to=2-2]
		\arrow["\alpha", curve={height=12pt}, from=2-1, to=3-3]
		\arrow["CA", curve={height=-12pt}, from=1-2, to=3-3]
		\arrow["{{\exists! \bar{\alpha}}}", from=2-2, to=3-3]
	\end{tikzcd}\]
	Similarly, $B$ induces a coextension $\tilde{\gamma}$ by the following diagram 
	\[\begin{tikzcd}
		{\mathcal{W}} & {C(\mathcal{W})_+} \\
		{C(\mathcal{W})_-} & {C(\mathcal{W})_{+}\sqcup_{\mathcal{W}}C(\mathcal{W})_{-}} \\
		&& {\mathcal{Y}\sqcup_{\beta}C(\mathcal{X})} & \cdot
		\arrow["{{\mathrm{id}\times {0}}}", from=1-1, to=1-2]
		\arrow["{{\mathrm{id}\times {0}}}"', from=1-1, to=2-1]
		\arrow[from=2-1, to=2-2]
		\arrow[from=1-2, to=2-2]
		\arrow["{{C(\gamma)}}", curve={height=12pt}, from=2-1, to=3-3]
		\arrow["CB", curve={height=-12pt}, from=1-2, to=3-3]
		\arrow["{{\exists! \tilde{\gamma}}}", from=2-2, to=3-3]
	\end{tikzcd}\]
	It follows again that $\bar{\alpha}\circ \tilde{\gamma}\circ \rho_{\mathcal{W}}^{-1}$ is the element in $\{\alpha, \beta, \gamma\}$ obtained using the nullhomotopies $A$ and $B$.\\
	\end{proof}
	
	\begin{prop}[{cf. \cite[Proposition 1.8]{Toda+1963}}] 
		\label{proposition 2.3.10}
		Let $\alpha,\beta$ and $\gamma$ be the same as in the previous lemma. Let $\tilde{\beta}:C(\mathcal{X})_{+}\sqcup_{\mathcal{X}}C(\mathcal{X})_{-} \rightarrow \mathcal{Z}\sqcup_{\alpha}C(\mathcal{Y}) $ be a coextension of $\beta$. Then the set of all compositions $\tilde{\beta}\circ C(\gamma)_+\sqcup_{\mathrm{id}}C(\gamma)_-\circ \rho_{\mathcal{W}}^{-1} $ coincides with $-i_{\ast}\{\alpha, \beta, \gamma\}$, where $i$ is the canonical inclusion of $\mathcal{Z}$ into $\mathcal{Z}\sqcup_{\alpha}C(\mathcal{Y})$.\\
		\end{prop}
	
\begin{proof} We apply here again the geometric realization functor and then use Toda's arguments. First note that $|\rho_{\mathcal{X}}|$ has a natural homotopy inverse $\widetilde{\rho_{\mathcal{X}}}:|\mathcal{X}|\wedge |\Delta^1/ \partial\Delta^1|\rightarrow |C(\mathcal{X})_{+}\sqcup_{\mathcal{X}}C(\mathcal{X})_{-}|$ which is defined as follows 
	\begin{align*}
		x\wedge t\mapsto
		\begin{cases}
			x\wedge (1-2t) \in |C(\mathcal{X})_{-}|(U) & \text{if} \  0\leq t\leq \frac{1}{2} \\
			x\wedge (2t-1) \in |C(\mathcal{X})_+|(U)& \text{if} \ \frac{1}{2} \leq t \leq 1	
		\end{cases}
	\end{align*}
	for every $U\in \mathcal{S}\mathrm{m}_{S}$, $x\in|\mathcal{X}|(U) $ and $t\in |\Delta^1/ \partial\Delta^1|$. Analogously, $|\rho_{\mathcal{W}}|$ has also a natural homotopy inverse $\widetilde{\rho_{\mathcal{W}}}$ defined in the same way.\\
	
	Let $i_{C(\mathcal{X})_{+}}$ be the canonical inclusion of $C(\mathcal{X}) $ into the top cone of $ C(\mathcal{X})_{+}\sqcup_{\mathcal{X}}C(\mathcal{X})_{-}$. Then $ \tilde{\beta}\circ i_{C(\mathcal{X})_{+}} $ factors through $\mathcal{Z}$
	\[\begin{tikzcd}
		{C(\mathcal{X})} && {C(\mathcal{X})_{+}\sqcup_{\mathcal{X}}C(\mathcal{X)_{-}}} \\
		{\mathcal{Z}} && {\mathcal{Z}\sqcup_{\alpha}C(\mathcal{Y})}
		\arrow["\exists !A"', from=1-1, to=2-1]
		\arrow["{i_{C(\mathcal{X})_{+}}}", from=1-1, to=1-3]
		\arrow["{\tilde{\beta}}", from=1-3, to=2-3]
		\arrow[hook, from=2-1, to=2-3]
	\end{tikzcd}\]
	In particular, $A$ is a nullhomotopy of $\alpha\circ\beta$. Let $B$ be a simplicial nullhomotopy of $\beta\circ\gamma$. Using $A$ and $B$ we get an element $H$ in $\{\alpha, \beta, \gamma\}$. If we apply the classical construction for Toda brackets to $|A|$ and $|B|$, we obtain a morphism $H': |\mathcal{W}|\wedge |\Delta^1/ \partial\Delta^1|$ which is equal to $|H|$ in $\mathcal{H}o_{\Delta}(S)$.\\
	
	In $\mathcal{H}o_{\Delta}(S)$ we also have $|\tilde{\beta}\circ C(\gamma)_+\sqcup_{\mathrm{id}}C(\gamma)_-\circ \rho_{\mathcal{W}}^{-1} |=|\tilde{\beta}|\circ\widetilde{\rho_{\mathcal{X}}}\circ (|\gamma|\wedge|\Delta^1/ \partial\Delta^1|) $. Now we define a homotopy $G:|\mathcal{W}|\wedge |\Delta^1/ \partial\Delta^1|\wedge|\Delta^1|_+ \rightarrow |\mathcal{Z}\sqcup_{\alpha}C(\mathcal{Y})|$ by the formula
	\begin{align*}
		w\wedge t\wedge s\mapsto
		\begin{cases}
			|B|(U)(w\wedge (1-2t)s)\wedge(1-s)(1-2t)  & \text{if} \  0\leq t\leq \frac{1}{2} \\
			|A|(U)(|\gamma|(U)(w)\wedge (2t-1)) & \text{if} \ \frac{1}{2} \leq t \leq 1	
		\end{cases}
	\end{align*}
	for every $U\in \mathcal{S}\mathrm{m}_{S}$, $w\in|\mathcal{W}|(U) $, $t\in |\Delta^1/ \partial\Delta^1|$ and $s\in |\Delta^1|_+ $. Then $G(-,1)$ is equal to $|i|\circ(-H')$. On the other hand $G(-,1)$ is the morphism $|\tilde{\beta}|\circ\widetilde{\rho_{\mathcal{X}}}\circ (|\gamma|\wedge|\Delta^1/ \partial\Delta^1|)$. Hence we get $|\tilde{\beta}\circ C(\gamma)_+\sqcup_{\mathrm{id}}C(\gamma)_-\circ \rho_{\mathcal{W}}^{-1} |=|\tilde{\beta}|\circ\widetilde{\rho_{\mathcal{X}}}\circ (|\gamma|\wedge|\Delta^1/ \partial\Delta^1|)=|i|\circ(-H')=|i\circ(-H)|$. It follows that $\tilde{\beta}\circ C(\gamma)_+\sqcup_{\mathrm{id}}C(\gamma)_-\circ \rho_{\mathcal{W}}^{-1}\in-i_{\ast}\{\alpha, \beta, \gamma\} $.\\
	
	Conversely, let $A'$ be a simplicial nullhomotopy of $\alpha\circ\beta$ and $B'$ a simplicial nullhomotopy of $\beta\circ\gamma$. Using $A'$ and $B'$ we get an element $\hat{H}$ in $\{\alpha, \beta, \gamma\}$. Moreover $A'$ induces a coextension $\tilde{\beta'}$ by the following diagram 
	\[\begin{tikzcd}
		{\mathcal{X}} & {C(\mathcal{X})_+} \\
		{C(\mathcal{X})_-} & {C(\mathcal{X})_{+}\sqcup_{\mathcal{X}}C(\mathcal{X})_{-}} \\
		&& {\mathcal{Z}\sqcup_{\alpha}C(\mathcal{Y})} & \cdot
		\arrow["{{\mathrm{id}\times {0}}}", from=1-1, to=1-2]
		\arrow["{{\mathrm{id}\times {0}}}"', from=1-1, to=2-1]
		\arrow[from=2-1, to=2-2]
		\arrow[from=1-2, to=2-2]
		\arrow["{{C(\beta)}}", curve={height=12pt}, from=2-1, to=3-3]
		\arrow["CA", curve={height=-12pt}, from=1-2, to=3-3]
		\arrow["{{\exists! \tilde{\beta'}}}", from=2-2, to=3-3]
	\end{tikzcd}\]
	Then by exactly the same arguments as before we can show that $-i\circ\hat{H}=\tilde{\beta'}\circ C(\gamma)_+\sqcup_{\mathrm{id}}C(\gamma)_-\circ \rho_{\mathcal{W}}^{-1}$ in $\mathcal{H}_{\bullet}(S)$.
	\end{proof}
	
	\newpage
	
	\section{Examples}\label{example}
\
\

	In this section we would like to give some examples of motivic Toda brackets. In order to construct the examples we need some results from \cite{ASENS_2012_4_45_4_511_0}, so we first recall some basic facts about pointed naive $\mathbb{A}^1$-homotopies. We work in this section over the base $ \mathrm{Spec}S$ where $S$ is either a field or the ring of integers $ \mathbb{Z}$.
	\
	\
	\
	
	\subsection{Rational functions and naive $\mathbb{A}^1$-homotopies}
	\
	\begin{definition}
		\label{definition3.1.1}
		Let $\mathcal{X}$ and $\mathcal{Y}$ be two pointed motivic spaces in $\mathrm{sPre}(S)_{\ast}$. Let $f$ and $g$ be two pointed morphisms from $\mathcal{X}$ to $\mathcal{Y}$. A naive pointed $\mathbb{A}^1$-homotopy is a morphism $F:\mathcal{X}\wedge\mathbb{A}^1_+\rightarrow  \mathcal{Y} $ such that $F|_{\mathcal{X}\times\{0\}}$ is $f$ and $F|_{\mathcal{X}\times\{1\}}$ is $g$. We define the set $[\mathcal{X}, \mathcal{Y}]^{\mathrm{N}}$ of pointed naive homotopy classes of morphisms from $\mathcal{X}$ to $\mathcal{Y}$ as the quotient of the set of pointed morphisms by the equivalence relation generated by pointed naive $\mathbb{A}^1$-homotopies.\\
		\end{definition}
	
	If there is a pointed $\mathbb{A}^1$-homotopy from $f$ to $g$, then $f$ is equal to $g$ in $\mathcal{H}_{\bullet}(S)$. Therefore there is a canonical map\begin{align*}
		[\mathcal{X}, \mathcal{Y}]^{\mathrm{N}}\rightarrow\mathcal{H}_{\bullet}(S)(\mathcal{X}, \mathcal{Y}) \ .
	\end{align*}
	In general this map is far from being a bijection. Examples where this map is not a bijection can be found in \cite[Section 4]{BALWE2015335}. Let $S$ for now be a field $k$. We equip the projective line $\mathbb{P}^1_{k}=\mathrm{Proj}k[T_0, T_1]$ over $k$ with the base point $\infty=[1:0]$. We are interested in the set $ [\mathbb{P}^1_{k}, \mathbb{P}^1_{k}]^{\mathrm{N}}$. A morphism from $\mathbb{P}^1_{k}$ to $\mathbb{P}^1_{k}$ in $\mathrm{sPre}(k)_{\ast}$ is uniquely determined by a pointed scheme endomorphism of $\mathbb{P}^1_{k}$, therefore we can restrict ourselves to scheme morphisms. In the following we recall how to represent a pointed scheme endomorphism of $\mathbb{P}^1_{k}$.\\
	
	Let $f:\mathbb{P}^1_{k}\rightarrow \mathbb{P}^1_{k}$ be a pointed endomorphism. We consider the line bundle $\mathcal{L}:=f^{\ast}\mathcal{O}_{\mathbb{P}^1_{k}}(1)$. Let $\widetilde{\mathcal{L}(\mathrm{D}_{+}(T_0))}$ denote the quasi-coherent module generated by $\mathcal{L}(\mathrm{D}_{+}(T_0))$ on $\mathrm{D}_{+}(T_0)\subseteq\mathbb{P}^1_{k} $ and $\widetilde{\mathcal{L}(\mathrm{D}_{+}(T_1))}$ the quasi-coherent module generated by $\mathcal{L}(\mathrm{D}_{+}(T_1))$ on $\mathrm{D}_{+}(T_1)\subseteq\mathbb{P}^1_{k} $. Particularly, there are natural isomorphisms $\widetilde{\mathcal{L}(\mathrm{D}_{+}(T_0))}\cong \mathcal{L}|_{\mathrm{D}_{+}(T_0)} $, $\widetilde{\mathcal{L}(\mathrm{D}_{+}(T_1))}\cong \mathcal{L}|_{\mathrm{D}_{+}(T_1)}$ and $\widetilde{\mathcal{L}(\mathrm{D}_{+}(T_0 T_1))}\cong \mathcal{L}|_{\mathrm{D}_{+}(T_0T_1)}$. It follows that $\mathcal{L}(\mathrm{D}_{+}(T_0))$ is a free module over $k[\frac{T_1}{T_0}]$ of rank 1 and $\mathcal{L}(\mathrm{D}_{+}(T_1))$ is a free module over $k[\frac{T_0}{T_1}]$ of rank 1.\\
	
	Furthermore $\mathcal{L}(\mathrm{D}_{+}(T_0 T_1))$ is also a free module over $k[\frac{T_0}{T_1},\frac{T_1}{T_0} ]$ of rank 1. Hence we can choose a generator $x_0$ of $\mathcal{L}(\mathrm{D}_{+}(T_0))$ and a generator $x_1$ of $\mathcal{L}(\mathrm{D}_{+}(T_1))$. Then $x_0|_{\mathrm{D}_{+}(T_0 T_1)}$ is equal to $c(\frac{T_0}{T_1})^nx_1|_{\mathrm{D}_{+}(T_0 T_1)}$ for some $c\in k^{\times}$ and $n\geq 0$, since the units in $k[\frac{T_0}{T_1},\frac{T_1}{T_0} ]$ are of the form $\{c(\frac{T_0}{T_1})^n; c\in k^{\times}, n\in \mathbb{Z} \}$. In our special case $n$ must be non-negative, because otherwise the global section of $\mathcal{L}$ would be 0. This cannot be true as we have $f(\infty)=\infty$ and $\mathcal{L}$ has two generating global sections coming from $\mathcal{O}_{\mathbb{P}^1_{k}}(1) $. Using this we can define a bejection $\phi$ from the global section of $ \mathcal{L}$ to the set $k[T_0, T_1]_{n} $ of all homogeneous polynomials of degree $n$ in $T_0$ and $T_1$. By choosing different generators we get bijections which differ from $\phi$ only by multiplication with a unit in $k$.\\
	
	Since $\mathcal{O}_{\mathbb{P}^1_{k}}(1) $ is generated by the two global sections $T_0$ and $T_1$, $\mathcal{L}$ is also generated by two global sections $s_0$ and $s_1$ which correspond to $T_0$ and $T_1$ under $f$. Under the bijection $\phi $ these two global sections correspond to two homogeneous polynomials of degree $n$. We write $\phi(s_0)$ as\begin{align*} a_nT_0^n+a_{n-1}T_0^{n-1}T_1+\cdot\cdot\cdot+a_0T_1^n
	\end{align*}
	and $\phi(s_1)$ as 
	\begin{align*} b_nT_0^n+b_{n-1}T_0^{n-1}T_1+\cdot\cdot\cdot+b_0T_1^n \ .
	\end{align*}
	Now since $f(\infty)=\infty$ and $\mathcal{O}_{\mathbb{P}^1_{k}}(1)_{f(\infty)}\otimes_{\mathcal{O}_{\mathbb{P}^1_{k},f(\infty)}}\mathcal{O}_{\mathbb{P}^1_{k},\infty}\cong \mathcal{L}_{\infty} $, we must have $a_n\neq 0$ and $b_n=0$. Hence we can work in coordinates $X:= \frac{T_0}{T_1} $. Dividing the two homogeneous polynomials above by $a_nT_1^n$, we get the following two polynomials of the form 
	\begin{align*} f=X^n+a_{n-1}'X^{n-1}+\cdot\cdot\cdot+a_0'
	\end{align*}
	and  
	\begin{align*} g=b_{n-1}'X^{n-1}+\cdot\cdot\cdot+b_0' \ .
	\end{align*}
	
	\
	
	\begin{definition}
		\label{def3.1.2}
		 Let $R$ be a commutative ring. Let $f$ and $g$ be two polynomials in $R[X]$ of the form
	\begin{align*} f=\alpha_nX^n+\alpha_{n-1}X^{n-1}+\cdot\cdot\cdot+\alpha_0\end{align*}
	and	\begin{align*}
		g=\beta_mX^m+\beta_{m-1}X^{m-1}+\cdot\cdot\cdot+\beta_0 \ .
	\end{align*}
	We do not require here that $\alpha_n,  \beta_m\neq 0$, so we call $n$ (respectively m) the formal degree of $f$ (respectively g). Then we define the resultant $res_{n,m}(f,g)$ to be the determinant of the matrix
	\begin{center}	
		$\begin{pmatrix}
			\alpha_n&0&\cdots&0 &\beta_{m} &0&\cdots&0\\
			\alpha_{n-1}&\alpha_{n}&\cdots&0 &\beta_{m-1}&\beta_{m}&\cdots&0\\
			\alpha_{n-2}&\alpha_{n-1}&\ddots&0 &\beta_{m-2}&\beta_{m-1}&\ddots&0\\
			\vdots&\vdots&\ddots&\alpha_{n}& \vdots&\vdots&\ddots&\beta_{m}\\
			\alpha_{0}&\alpha_{1}&\cdots&\vdots&\beta_{0}&\beta_{1}&\cdots&\vdots\\
			0&\alpha_{0}&\ddots&\vdots&0&\beta_{0}&\ddots&\vdots\\
			\vdots&\vdots&\ddots&\alpha_{1}& \vdots&\vdots&\ddots&\beta_{1}\\
			0&0&\cdots&\alpha_{0}&0&0&\cdots&\beta_{0}\\
		\end{pmatrix}
		$
	\end{center}
	\end{definition}
	\
	
	\
	
	\begin{lem}[{\cite[Bemerkung 4.4.4, Satz 4.4.6]{bosch2009algebra}}]
		\label{Lemma 3.1.3}
			\end{lem}
		\begin{itemize}
		\item[1.] \textit{$res_{n,m}(f,g)=(-1)^{mn}res_{m,n}(g,f)$}
		\item[2.]\textit{$res_{n,m}(af,bg)=a^nb^mres_{n,m}(f,g) $ for constants $a,b\in R$}
		\item [3.] \textit{If $n+m\geq 1$, then there exist polynomials $p,q\in R[X]$ with $deg(p)<m$, $deg(q)<n$ and $res_{n,m}(f,g)=pf+qg $.\\}
	\end{itemize}

		\begin{lem}[{\cite[Korollar  4.4.9]{bosch2009algebra}}] \label{Lemma 3.1.4}   Let $f,g\in R[X]$ be two polynomials such that $f$ has formal degree $n$ and $g$ has formal degree $m$. Let $R'$ be a commutative ring such that $R\subseteq R'$ and $R'$ contains the zeros of $f$ and $g$. Then $f$ and $g $ can be written as \begin{align*} f=\alpha\prod_{i=1}^{n}(X-\alpha_i)\\
			g=\beta\prod_{j=1}^{m}(X-\beta_j)
		\end{align*}
		with $\alpha,\beta\in R$ and $\alpha_i, \beta_j$ in $R'$. It follows that \begin{align*} 
			res_{n,m}(f,g)=\alpha^m\prod_{i=1}^{n}g(\alpha_i)=\alpha^m\beta^n\prod_{\substack{1\leq i\leq n\\ 1\leq j\leq m}}(\alpha_i-\beta_j) \ .
		\end{align*}\\
		\end{lem}
	
	An easy consequence from this lemma is the following corollary.\\
	
	\begin{cor}Let $f\in R[X]$ be a polynomial of the form $f=\alpha_nX^n+\alpha_{n-1}X^{n-1}+\cdot\cdot\cdot+\alpha_0\ $. Then we define its reciprocal polynomial $f^{\ast}$ to be the polynomial $\alpha_n+\alpha_{n-1}X+\cdot\cdot\cdot+\alpha_0X^n=X^nf(X^{-1})$. We have then \begin{align*}res_{n,m}(f,g)=res_{n,m}(f^{\ast},g^{\ast}).
	\end{align*}
	\end{cor}
	\
	
	The condition that $\phi(s_0)$ and $\phi(s_1)$ are generating global sections is equivalent to the condition that the ideals  $(f=X^n+a_{n-1}'X^{n-1}+\cdot\cdot\cdot+a_0', g=b_{n-1}'X^{n-1}+\cdot\cdot\cdot+b_0' ), (f^{\ast}, g^{\ast})\subseteq k[X] $ satisfy  
	\begin{align*}
		(f,g)=k[X], \ (f^{\ast}, g^{\ast})=k[X] \ .
	\end{align*}
	By Lemma 3.1.3 and Corollary 3.1.5 this is in turn equivalent to the condition that $res_{n,n}(f,g)$ is invertible in $k$.\\
	
	Conversely, two polynomials in $k[X]$ of the form \begin{align*} f=\alpha_nX^n+\alpha_{n-1}X^{n-1}+\cdot\cdot\cdot+\alpha_0\\
		g=\beta_mX^m+\beta_{m-1}X^{m-1}+\cdot\cdot\cdot+\beta_0
	\end{align*}
	such that $res_{n,m}(f,g)\in k^{\times}$ determine a pointed scheme endomorphism of $\mathbb{P}^1_{k}$ by the next proposition.\\
	
	\begin{prop}[{\cite[Proposition 5.1.31]{liu2006algebraic}}]Let $\mathcal{Y}=\mathrm{Proj}A[T_0,\cdots , T_d]$ be a projective space over a commutative ring $A$, and let $\mathcal{X}$ be a scheme over $A$.
		\end{prop}
	\begin{itemize}
		\item[1.] \textit{Let $f:\mathcal{X}\rightarrow \mathcal{Y}$ be a morphism of $A$-schemes. Then $f^{\ast}\mathcal{O}_{\mathcal{Y}}(1)$ is an invertible sheaf on $\mathcal{X}$, generated by $d+1$ of its global sections.}
		\item[2.] \textit{Conversely, for any invertible sheaf $\mathcal{L}$ on $\mathcal{X}$ generated by $d+1$ global sections $s_0, \cdots , s_d $, there exists a morphism $f:\mathcal{X}\rightarrow \mathcal{Y}$ such that $ \mathcal{L} \cong f^{\ast}\mathcal{O}_{\mathcal{Y}}(1)$ and $f^{\ast}T_i=s_i $ via this morphism. If we fix this set of generating global sections, the morphism $f$ is unique.\\} 
	\end{itemize}
	
	Now since $f$ and $g$ do not have common  prime factors, the two corresponding homogeneous polynomials 
	\begin{align*} T_0^n+a_{n-1}'T_0^{n-1}T_1+\cdot\cdot\cdot+a_0'T_1^n
	\end{align*}
	and  
	\begin{align*} b_{n-1}'T_0^{n-1}T_1+\cdot\cdot\cdot+b_0'T_1^n
	\end{align*}
	are generating global sections for $\mathcal{O}_{\mathbb{P}^1_{k}}(n) $. By the previous proposition we get an endomorphism of $\mathbb{P}^1_{k}$ which sends $\infty$ to itself.\\
	
	Altogether we obtain the following characterization of pointed endomorphisms of $\mathbb{P}^1_{k}$ by rational functions.\\
	
	\begin{prop}\label{Proposition 3.1.7} Any pointed endomorphism of $\mathbb{P}^1_{k}$ is given uniquely by a pair of polynomials $(f,g)\in k[X]$, where
		\end{prop}
		\begin{itemize}
			\item f is monic of degree n,
			\item g is of degree strictly less than n,
			\item $res_{n,n}(f,g)$ is invertible in $k$.\\
	\end{itemize}
	
	We abuse the notation and denote such a pair in the following by $\frac{f}{g}$.\\
	
	Next we would like to characterize pointed naive $\mathbb{A}^1$-homotopies. A morphism of pointed presheaves $\mathbb{P}^1_{k}\wedge\mathbb{A}^1_+\rightarrow \mathbb{P}^1_{k} $ is determined uniquely by a scheme morphism from $\mathbb{P}^1_{k}\times_{k}\mathbb{A}^1_{k}=\mathrm{Proj}k[T][T_0, T_1] $ to $\mathbb{P}^1_{k} $ which sends all homogeneous prime ideals of the form $(\rho, T_1)$ of $k[T][T_0, T_1]  $ to $\infty$ where $\rho$ runs over the prime polynomials in $k[T]$. This condition is equivalent to that the composition \begin{align*} 
		\mathrm{Spec}k\times_{k}\mathbb{A}^1_{k}\rightarrow\mathbb{P}^1_{k}\times_{k}\mathbb{A}^1_{k}\rightarrow\mathbb{P}^1_{k}
	\end{align*} 
	is the constant morphism. Therefore again we can restrict ourselves to scheme morphisms.\\
	
	Now note that the units in $k[T][\frac{T_0}{T_1},\frac{T_1}{T_0} ] $ are also of the form $\{c(\frac{T_0}{T_1})^n; c\in k^{\times}, n\in \mathbb{Z} \}$. Hence we can apply the previous arguments for pointed endomorphims of $\mathbb{P}^1_{k}$ to morphisms $f: \mathrm{Proj}k[T][T_0, T_1] \rightarrow \mathbb{P}^1_{k} $ with $f((T_1))=\infty$. In particular we get again a bijection $\phi$ from the global section of $ f^{\ast}\mathcal{O}_{\mathbb{P}^1_{k}}(1)$ to the set $k[T][T_0, T_1]_{n} $ of all homogeneous polynomials of degree $n$ in $T_0$ and $T_1$ over $k[T]$. By choosing different generators we get bijections which differ from $\phi$ only by multiplication with a unit in $k$. We know that $\mathcal{O}_{\mathbb{P}^1_{k}}(1) $ is generated by the two global sections $T_0$ and $T_1$, thus $f^{\ast}\mathcal{O}_{\mathbb{P}^1_{k}}(1)$ is generated by two global sections $s_0$ and $s_1$ which correspond to $T_0$ and $T_1$ under $f$. Under the bijection $\phi $ these two global sections correspond to two homogeneous polynomials of degree $n$. We write $\phi(s_0)$ as\begin{align*} a_nT_0^n+a_{n-1}T_0^{n-1}T_1+\cdot\cdot\cdot+a_0T_1^n
	\end{align*}
	and $\phi(s_1)$ as 
	\begin{align*} b_nT_0^n+b_{n-1}T_0^{n-1}T_1+\cdot\cdot\cdot+b_0T_1^n
	\end{align*} with $a_i, b_i\in k[T]$.\\
	
	Since $f(\rho,T_1)$ is also $\infty$, we must have $a_n\in k^{\times}$ and $b_n=0$. Again we can work in coordinates $X:=\frac{T_0}{T_1}$. Since $a_n\in k^{\times}$, we can divide by $a_nT_1^n$ and get the following two polynomials of the form 
	\begin{align*} f=X^n+a_{n-1}'X^{n-1}+\cdot\cdot\cdot+a_0'
	\end{align*}
	and  
	\begin{align*} g=b_{n-1}'X^{n-1}+\cdot\cdot\cdot+b_0' \ .
	\end{align*}
	As before $s_0$ and $s_1$ are generating global sections if and only if the ideal $(f,g)$ is equal to the ring $k[T][X]$. It is again equivalent to that $res_{n,n}(f,g) $ is invertible in $k[T]$. Together with Proposition 3.1.6 we obtain the following characterization of pointed naive $\mathbb{A}^1$-homotopies of $\mathbb{P}^1_{k}$.\\
	
		\begin{prop}
			\label{Proposition 3.1.8}\ Any pointed $\mathbb{A}^1$-homotopy of $\mathbb{P}^1_{k}$ is given uniquely by a pair of polynomials $\frac{f}{g}$ with $f,g\in k[T][X]$, where
			\end{prop}
		\begin{itemize}
			\item f is monic of degree n,
			\item g is of degree strictly less than n,
			\item $res_{n,n}(f,g)$ is invertible in $k[T]$.\\
	\end{itemize}
	
	Next we consider $\mathbb{P}^1_{\mathbb{Z}}=\mathrm{Proj}\mathbb{Z}[T_0, T_1]$ where it is equipped with a morphism $ \infty: \mathrm{Spec}\mathbb{Z}\rightarrow \mathbb{P}^1_{\mathbb{Z}}$ in $\mathcal{S}\mathrm{m}_{\mathbb{Z}} $. It is given by $\mathbb{Z}[\frac{T_1}{T_0}]\rightarrow \mathbb{Z} ; \frac{T_1}{T_0} \mapsto 0 $. A pointed endomorphism of $\mathbb{P}^1_{\mathbb{Z}}$ is a scheme morphism $f:\mathbb{P}^1_{\mathbb{Z}}\rightarrow \mathbb{P}^1_{\mathbb{Z}} $ such that the diagram 
	\[\begin{tikzcd}
		& {\mathrm{Spec}\mathbb{Z}} \\
		{\mathbb{P}^1_{\mathbb{Z}}} && {\mathbb{P}^1_{\mathbb{Z}}}
		\arrow["\infty"', from=1-2, to=2-1]
		\arrow["\infty", from=1-2, to=2-3]
		\arrow["f", from=2-1, to=2-3]
	\end{tikzcd}\]
	commutes. This condition is equivalent to $f(p,T_1)=(T_1)$ where $p$ is either 0 or runs over the prime numbers in $\mathbb{Z}$. The same arguments for pointed endomorphisms of $\mathbb{P}^1_{k}$ work also for $\mathbb{P}^1_{\mathbb{Z}}$. Again we work in coordinates $X:=\frac{T_0}{T_1} $. Together with Proposition 3.1.6 we get a characterization of pointed endomorphisms of $\mathbb{P}^1_{\mathbb{Z}}$.\\
	
		\begin{prop}
			\label{Proposition 3.1.9}
				Any pointed endomorphism of $\mathbb{P}^1_{\mathbb{Z}}$ is given uniquely by a pair of polynomials $\frac{f}{g}$ with $f,g\in \mathbb{Z}[X]$, where
				\end{prop}
		\begin{itemize}
			\item f is monic of degree n, 
			\item g is of degree strictly less than n,
			\item $res_{n,n}(f,g)$ is invertible in $\mathbb{Z}$.\\
	\end{itemize}
	
	Finally we can also consider pointed naive $\mathbb{A}^1$-homotopies of $\mathbb{P}^1_{\mathbb{Z}}$. A pointed naive $\mathbb{A}^1$-homotopy can be viewed as a scheme morphism $f:\mathbb{P}^1_{\mathbb{Z}}\times_{\mathbb{Z}}\mathbb{A}^1_{\mathbb{Z}}=\mathrm{Proj}\mathbb{Z}[T][T_0, T_1]  \rightarrow \mathbb{P}^1_{\mathbb{Z}} $ such that the composition \begin{align*} 
		\mathrm{Spec}\mathbb{Z}\times_{\mathbb{Z}}\mathbb{A}^1_{\mathbb{Z}}\rightarrow\mathbb{P}^1_{\mathbb{Z}}\times_{\mathbb{Z}}\mathbb{A}^1_{\mathbb{Z}}\rightarrow\mathbb{P}^1_{\mathbb{Z}}
	\end{align*} 
	factors through the structure morphism $ \mathrm{Spec}\mathbb{Z}\rightarrow \mathbb{P}^1_{\mathbb{Z}}$. This is  equivalent to $f((\rho,T_1))=(T_1)$ where $(T_1) $ is the homogeneous prime ideal in $\mathbb{Z}[T][T_0, T_1] $ generated by $T_1$ and $\rho$ runs over the prime ideals of $\mathbb{Z}[T]$. As before we can apply the arguments for pointed $\mathbb{A}^1$-homotopies of $\mathbb{P}^1_{k}$ to pointed $\mathbb{A}^1$-homotopies of $\mathbb{P}^1_{\mathbb{Z}}$. We take here $X:=\frac{T_0}{T_1}$, then we obtain the characterization below.\\
	
	\begin{prop}
		\label{Proposition 3.1.10} Any pointed $\mathbb{A}^1$-homotopy of $\mathbb{P}^1_{\mathbb{Z}}$ is given uniquely by a pair of polynomials $\frac{f}{g}$ with $f,g\in \mathbb{Z}[T][X]$, where
		\end{prop}
		\begin{itemize}
			\item f is monic of degree n,
			\item g is of degree strictly less than n,
			\item $res_{n,n}(f,g)$ is invertible in $\mathbb{Z}[T]$.\\
	\end{itemize}
	
	In \cite{ASENS_2012_4_45_4_511_0} Cazanave gives the set $ [\mathbb{P}^1_{k}, \mathbb{P}^1_{k}]^{\mathrm{N}}$ a monoid structure. Actually, his method works also over $\mathbb{Z}$, so we introduce the monoid structure for $ [\mathbb{P}^1_{\mathbb{Z}}, \mathbb{P}^1_{\mathbb{Z}}]^{\mathrm{N}}$. Let $\frac{f}{g}$ be a pair of polynomials which determines a pointed endomorphism of $\mathbb{P}^1_{\mathbb{Z}}$ such that $deg(f)=n$. Then by Lemma 3.1.3 there exist polynomials $p,q\in \mathbb{Z}[X]$ with $deg(p)<n-1$ and $deg(q)<n$ such that $1=pf+qg$, since $res_{n,n}(f,g)$ is invertible in $\mathbb{Z}$. Furthermore $p$ and $q$ are unique.  \\
	
		\begin{definition}
			\label{Definition 3.1.11} Let $\frac{f_1}{g_1}, \frac{f_2}{g_2}$ be two pairs of polynomials which determine pointed endomorphisms of $\mathbb{P}^1_{\mathbb{Z}}$ with  $deg(f_1)=n_1$ and $deg(f_2)=n_2 $. Then there are unique polynomials  $p_1,q_1, p_2, q_2\in \mathbb{Z}[X]$ with $deg(p_1)<n_1-1, deg(q_1)<n_1, deg(p_2)<n_2-1, deg(q_2)<n_2$ such that $1=p_1f_1+q_1g_1$ and $1=p_2f_2+q_2g_2$. We define polynomials $f_3, g_3, p_3$ and $q_3$ by setting 
	\begin{center}	
		$\begin{pmatrix}
			f_3&-q_3\\
			g_3&p_3
		\end{pmatrix}:= \begin{pmatrix}
			f_1&-q_1\\
			g_1&p_1
		\end{pmatrix}\cdot\begin{pmatrix}
			f_2&-q_2\\
			g_2&p_2
		\end{pmatrix} $
	\end{center}
	The matrices $\begin{pmatrix}
		f_1&-q_1\\
		g_1&p_1
	\end{pmatrix} $
	and $\begin{pmatrix}
		f_2&-q_2\\
		g_2&p_2
	\end{pmatrix}$ are in $\mathrm{SL}_2(\mathbb{Z}[X])$, hence it is also true for $\begin{pmatrix}
		f_3&-q_3\\
		g_3&p_3
	\end{pmatrix}$. By definition $f_3=f_1f_2-q_1g_2$ is monic of degree $n_1+n_2$ and $g_3=g_1f_2+p_1g_2$ is of degree strictly less than $n_1+n_2$. Therefore $\frac{f_3}{g_3} $ defines a pointed endomorphism of $\mathbb{P}^1_{\mathbb{Z}}$ by Proposition 3.1.8. We define the sum $\frac{f_1}{g_1}\oplus^{\mathrm{N}}\frac{f_2}{g_2} $ to be the pair of polynomials $\frac{f_3}{g_3} $. The neutral element for this addition is the pair of polynomials $\frac{1}{0}$ which represents the constant morphism.\\[1cm]
	\end{definition}
	
	\subsection{Cogroup structure on $\mathbb{P}^1_{\mathbb{Z}}$}
	\

	In this section we would like to study the cogroup structure on $\mathbb{P}^1_{\mathbb{Z}}$ in some detail. We can equip the presheaf $\mathbb{P}^1_{\mathbb{Z}}$ with three base points which are given by the following scheme morphisms:\begin{align*}
		\infty: \mathrm{Spec}\mathbb{Z}\rightarrow \mathbb{P}^1_{\mathbb{Z}} 
	\end{align*} 
	induced by $\mathbb{Z}[\frac{T_1}{T_0}]\rightarrow \mathbb{Z} ; \frac{T_1}{T_0} \mapsto 0 $ ;
	\begin{align*}
		0: \mathrm{Spec}\mathbb{Z}\rightarrow \mathbb{P}^1_{\mathbb{Z}} 
	\end{align*} 
	induced by $\mathbb{Z}[\frac{T_0}{T_1}]\rightarrow \mathbb{Z} ; \frac{T_0}{T_1} \mapsto 0 $
	\
	and 
	\begin{align*}
		1: \mathrm{Spec}\mathbb{Z}\rightarrow \mathbb{P}^1_{\mathbb{Z}} 
	\end{align*} 
	induced by $\mathbb{Z}[\frac{T_1}{T_0},\frac{T_0}{T_1} ]\rightarrow \mathbb{Z} ; \frac{T_1}{T_0} \mapsto 1,\ \frac{T_0}{T_1}\mapsto 1$ .\\
	
	We first recall the standard elementary distinguished sqaure for $\mathbb{P}^1_{\mathbb{Z}}$
	\[\begin{tikzcd}
		{\mathbb{G}_{m}} & {\mathbb{A}^1_{\mathbb{Z}}} \\
		{\mathbb{A}^1_{\mathbb{Z}}} & {\mathbb{P}^1_{\mathbb{Z}}} & {\cdot}
		\arrow["{{t_0}}", from=1-1, to=1-2]
		\arrow["{{t_{\infty}}}"', from=1-1, to=2-1]
		\arrow["{{j_{\infty}}}"', from=2-1, to=2-2]
		\arrow["{{j_0}}", from=1-2, to=2-2]
	\end{tikzcd}\]
	The morphism $j_0$ is induced by the canonical ring isomorphism $\mathbb{Z}[T]\cong\mathbb{Z}[\frac{T_0}{T_1}]$ and $j_{\infty}$ is induced by $\mathbb{Z}[T]\cong\mathbb{Z}[\frac{T_1}{T_0}]$. The open immersion $t_0$ is defined by $\mathbb{Z}[T]\rightarrow\mathbb{Z}[T, T^{-1}];\ T\mapsto T $ and $t_1$ is defined by $\mathbb{Z}[T]\rightarrow\mathbb{Z}[T, T^{-1}];\ T\mapsto T^{-1} $. In particular the canonical morphism from the pushout of the diagram above to $\mathbb{P}^1_{\mathbb{Z}}$ is a motivic weak equivalence. If we equip $\mathbb{G}_{m}, \mathbb{A}^1_{\mathbb{Z}} $ and $\mathbb{P}^1_{\mathbb{Z}} $ with the base point 1, this weak equivalence becomes a weak equivalence of pointed motivic spaces. Note that the pushout of the diagram is also the pushout of the following diagram 
	\[\begin{tikzcd}
		{(\mathbb{G}_{m},1)\vee(\mathbb{G}_{m},1)} & {(\mathbb{G}_{m},1)} \\
		{(\mathbb{A}^1_{\mathbb{Z}},1)\vee (\mathbb{A}^1_{\mathbb{Z}},1)} && \cdot
		\arrow["{{(\mathrm{id},\mathrm{id} )}}", from=1-1, to=1-2]
		\arrow["{{t_{\infty}\vee t_0}}"', from=1-1, to=2-1]
	\end{tikzcd}\]
	Let $I$ be the pointed space $\Delta^1{}_{1}\vee_{0}\Delta^1$. Then $I$ is a simplicial model of the interval admitting a mid-point. We denote the glueing point of $I $ by $\frac{1}{2}$. There is a canonical weak equivalence from $I$ to $\Delta^1$ by projecting $\Delta^1{}_{1}$ to the point 0.  We consider now the comparison maps between pushouts from \cite[Lemma 4.1]{levine2010slices} (In the left diagram $\mathbb{G}_{m} $ and $\mathbb{A}^1_{\mathbb{Z}}$ are equipped with the base point 1)
	\[\begin{tikzcd}
		{\mathbb{G}_{m}\vee \mathbb{G}_{m}} & {\mathbb{G}_{m}} && {0_+\wedge\mathbb{G}_{m}\vee 1_+\wedge\mathbb{G}_{m}} & {I_+\wedge\mathbb{G}_{m}} \\
		{\mathbb{A}^1_{\mathbb{Z}}\vee \mathbb{A}^1_{\mathbb{Z}}} & {} & {} & {0_+\wedge\mathbb{A}^1_{\mathbb{Z}}\vee 1_+\wedge\mathbb{A}^1_{\mathbb{Z}}} \\
		&&& {} \\
		&&& {} \\
		&&& {0_+\wedge\mathbb{G}_{m}\vee 1_+\wedge\mathbb{G}_{m}} & {I_+\wedge\mathbb{G}_{m}} \\
		&&& \ast
		\arrow["{t_{\infty}\vee t_0}"', from=1-1, to=2-1]
		\arrow["{(\mathrm{id},\mathrm{id} )}", from=1-1, to=1-2]
		\arrow["{t_{\infty}\vee t_0}", from=1-4, to=2-4]
		\arrow["{(l_0,l_1)}", from=1-4, to=1-5]
		\arrow[from=2-3, to=2-2]
		\arrow[from=3-4, to=4-4]
		\arrow["{(l_0,l_1)}", from=5-4, to=5-5]
		\arrow[from=5-4, to=6-4]
	\end{tikzcd}\]
	The first map is induced by the canonical projections and the second by collapsing $\mathbb{A}^1_{\mathbb{Z}}$ to a point. Both comparison maps are motivic weak equivalences. We denote the pushout of the middle diagram by $\mathcal{X}$. There is a canonical weak equivalence from the pushout of the diagram 
	\[\begin{tikzcd}
		{0_+\wedge\mathbb{G}_{m}\vee 1_+\wedge\mathbb{G}_{m}} & {I_+\wedge\mathbb{G}_{m}} \\
		\ast
		\arrow[from=1-1, to=1-2]
		\arrow[from=1-1, to=2-1]
	\end{tikzcd}\]
	to $S^1\wedge \mathbb{G}_{m} $ which is induced by the weak equivalence from $I$ to $\Delta^1$. We can equip $\mathcal{X}$ with the base point $\infty$ which comes from the point $0: \mathrm{Spec}\mathbb{Z}\rightarrow 0_+\wedge\mathbb{A}^1_{\mathbb{Z}}$. It is defined by $\mathbb{Z}[T]\rightarrow\mathbb{Z}; T\mapsto 0 $. Now we equip ${\mathbb{A}^1_{\mathbb{Z}}\vee \mathbb{A}^1_{\mathbb{Z}}} $ also with the base point \[\begin{tikzcd}
		{\mathrm{Spec}\mathbb{Z}} & {\mathbb{A}^1_{\mathbb{Z}}} & {\mathbb{A}^1_{\mathbb{Z}}\vee \mathbb{A}^1_{\mathbb{Z}}}
		\arrow["0", from=1-1, to=1-2]
		\arrow["{i_1}", hook, from=1-2, to=1-3]
	\end{tikzcd}\]
	where $i_1$ is the inclusion into the first copy of $\mathbb{A}^1_{\mathbb{Z}} $. Then the  canonical morphism $\mathbb{A}^1_{\mathbb{Z}}\vee \mathbb{A}^1_{\mathbb{Z}}\rightarrow \mathbb{P}^1_{\mathbb{Z}} $ becomes a morphism of the pointed motivic spaces $(\mathbb{A}^1_{\mathbb{Z}}\vee \mathbb{A}^1_{\mathbb{Z}},0) \rightarrow (\mathbb{P}^1_{\mathbb{Z}},\infty) $. Hence the first comparison map above between pushouts induces also a weak equivalence of pointed motivic spaces $(\mathcal{X},\infty)\simeq (\mathbb{P}^1_{\mathbb{Z}},\infty)$. The second comparison map induces also a weak equivalence $(\mathcal{X},\infty)\simeq S^1\wedge \mathbb{G}_{m}$, since $0_+\wedge\mathbb{A}^1_{\mathbb{Z}} $ is sent to the point $\ast$. Thus we get here an \hypertarget{iso}{isomorphism} $\alpha:(\mathbb{P}^1_{\mathbb{Z}},\infty)\rightarrow S^1\wedge \mathbb{G}_{m} $ in $\mathcal{H}_{\bullet}(\mathbb{Z})$ and $(\mathbb{P}^1_{\mathbb{Z}},\infty)$ inherits a cogroup structure from $S^1\wedge \mathbb{G}_{m}$ via $\alpha$.\\
	
	In \cite[Lemma B.4]{ASENS_2012_4_45_4_511_0} Cazanave gives a co-diagonal morphism for $(\mathbb{P}^1_{k},\infty)$. Actually his method works also over the base $\mathrm{Spec}\mathbb{Z}$. In the following we write down this morphism in detail. Again we consider the commutative diagram 
	\[\begin{tikzcd}
		{\mathbb{G}_{m}\vee\mathbb{G}_{m}} & {\mathbb{G}_{m}} \\
		{\mathbb{A}^1_{\mathbb{Z}}\vee \mathbb{A}^1_{\mathbb{Z}}} & {\mathbb{P}^1_{\mathbb{Z}}} & \cdot
		\arrow["{{t_{\infty}\vee t_0}}"', from=1-1, to=2-1]
		\arrow["{{j_{\infty}\vee j_0}}"', from=2-1, to=2-2]
		\arrow["{{\mathrm{id}\vee\mathrm{id}}}", from=1-1, to=1-2]
		\arrow[from=1-2, to=2-2]
	\end{tikzcd}\]
	We denote by $(\mathbb{P}^1_{\mathbb{Z}},\mathbb{G}_{m})$ the cofiber of the inclusion $\mathbb{G}_{m}\hookrightarrow \mathbb{P}^1_{\mathbb{Z}}$ above. Similarly, we denote by $(\mathbb{A}^1_{\mathbb{Z}}, \mathbb{G}_{m})_{t_{\infty}} $ the cofiber of the inclusion $t_{\infty} $ and by $(\mathbb{A}^1_{\mathbb{Z}}, \mathbb{G}_{m})_{t_{0}} $ the cofiber of $t_0$. Then from the first comparison map 
	\[\begin{tikzcd}
		{\mathbb{G}_{m}\vee\mathbb{G}_{m}} & {\mathbb{G}_{m}} && {0_+\wedge\mathbb{G}_{m}\vee 1_+\wedge\mathbb{G}_{m}} & {I_+\wedge\mathbb{G}_{m}} \\
		{\mathbb{A}^1_{\mathbb{Z}}\vee \mathbb{A}^1_{\mathbb{Z}}} & {} & {} & {	0_+\wedge\mathbb{A}^1_{\mathbb{Z}}\vee 1_+\wedge\mathbb{A}^1_{\mathbb{Z}}}
		\arrow["{\mathrm{id}\vee\mathrm{id}}", from=1-1, to=1-2]
		\arrow["{t_{\infty}\vee t_0}"', from=1-1, to=2-1]
		\arrow[from=1-4, to=1-5]
		\arrow["{t_{\infty}\vee t_0}"', from=1-4, to=2-4]
		\arrow[from=2-3, to=2-2]
	\end{tikzcd}\]
	we obtain a motivic weak equivalence
	\begin{align*}
		\mathcal{X}/(I_+\wedge\mathbb{G}_{m})=	(\mathbb{A}^1_{\mathbb{Z}}, \mathbb{G}_{m})_{t_{\infty}} \vee (\mathbb{A}^1_{\mathbb{Z}}, \mathbb{G}_{m})_{t_{0}}\rightarrow(\mathbb{P}^1_{\mathbb{Z}},\mathbb{G}_{m})
	\end{align*}
	which is induced by $j_{\infty}\vee j_0$.\\
	
	From the elementary distinguished square 
	$\mathbb{P}^1_{\mathbb{Z}}$
	\[\begin{tikzcd}
		{\mathbb{G}_{m}} & {\mathbb{A}^1_{\mathbb{Z}}} \\
		{\mathbb{A}^1_{\mathbb{Z}}} & {\mathbb{P}^1_{\mathbb{Z}}}
		\arrow["{t_0}", from=1-1, to=1-2]
		\arrow["{t_{\infty}}"', from=1-1, to=2-1]
		\arrow["{j_{\infty}}"', from=2-1, to=2-2]
		\arrow["{j_0}", from=1-2, to=2-2]
	\end{tikzcd}\]
	we get two motivic weak equivalences
	\begin{align*}
		(\mathbb{A}^1_{\mathbb{Z}}, \mathbb{G}_{m})_{t_{\infty}} 	\rightarrow(\mathbb{P}^1_{\mathbb{Z}},\mathbb{A}^1_{\mathbb{Z}})_{j_0}
	\end{align*}
	and 
	\begin{align*}
		(\mathbb{A}^1_{\mathbb{Z}}, \mathbb{G}_{m})_{t_{0}} 	\rightarrow(\mathbb{P}^1_{\mathbb{Z}},\mathbb{A}^1_{\mathbb{Z}})_{j_{\infty}} \ .
	\end{align*}
	\
	
	Next we recall that there is a path $j: \mathbb{A}^1_{\mathbb{Z} }\rightarrow \mathbb{P}^1_{\mathbb{Z}}$ from $\infty $ to 0. In \cite[Appendix B]{ASENS_2012_4_45_4_511_0} Cazanave call this path the canonical path from $\infty $ to 0. In `` homogeneous coordinates'' it is given by $[1-T: T]$. Moreover we can also give the precise definition of $j$. We first define an automorphism $\psi $ of $\mathbb{P}^1_{\mathbb{Z}}$ which is induced by the ring isomorphism
	\begin{align*}
		\mathbb{Z}[T_0,T_1]\rightarrow \mathbb{Z}[T_0,T_1]; T_0\mapsto T_0-T_1, T_1\mapsto T_1 \ .
	\end{align*}
	Recall that $j_\infty: \mathbb{A}^1_{\mathbb{Z} }\hookrightarrow \mathbb{P}^1_{\mathbb{Z}}$ is the open embedding into $\mathrm{D}_+(T_0)$. Then we define the path $j$ to be the composition $\psi\circ j_\infty$. Therefore this path is an open embedding. We denote the cofiber of $j$ just simply by $\mathbb{P}^1_{\mathbb{Z}}/\mathbb{A}^1_{\mathbb{Z}}$. The canonical projection $\theta:\mathbb{P}^1_{\mathbb{Z}}\rightarrow  \mathbb{P}^1_{\mathbb{Z}}/\mathbb{A}^1_{\mathbb{Z}}$ is a weak equivalence and it induces two pointed weak equivalences $\theta_0:(\mathbb{P}^1_{\mathbb{Z}},0)\rightarrow \mathbb{P}^1_{\mathbb{Z}}/\mathbb{A}^1_{\mathbb{Z}}$ and $\theta_\infty:(\mathbb{P}^1_{\mathbb{Z}},\infty)\rightarrow \mathbb{P}^1_{\mathbb{Z}}/\mathbb{A}^1_{\mathbb{Z}}$.\\
	
	Finally we can write down the co-diagonal morphism for $(\mathbb{P}^1_{\mathbb{Z}},\infty)$ given by Cazanave:
	\[\begin{tikzcd}
		{(\mathbb{P}^1_{\mathbb{Z}},\infty)} & {(\mathbb{P}^1_{\mathbb{Z}},\mathbb{G}_{m})} & {	(\mathbb{A}^1_{\mathbb{Z}}, \mathbb{G}_{m})_{t_{\infty}} \vee (\mathbb{A}^1_{\mathbb{Z}}, \mathbb{G}_{m})_{t_{0}}} \\
		&& {(\mathbb{P}^1_{\mathbb{Z}},\mathbb{A}^1_{\mathbb{Z}})_{j_0}\vee(\mathbb{P}^1_{\mathbb{Z}},\mathbb{A}^1_{\mathbb{Z}})_{j_{\infty}}} \\
		&& {(\mathbb{P}^1_{\mathbb{Z}},0)\vee(\mathbb{P}^1_{\mathbb{Z}},\infty)} \\
		&& {\mathbb{P}^1_{\mathbb{Z}}/\mathbb{A}^1_{\mathbb{Z}}\vee (\mathbb{P}^1_{\mathbb{Z}},\infty) } \\
		&& {(\mathbb{P}^1_{\mathbb{Z}},\infty)\vee (\mathbb{P}^1_{\mathbb{Z}},\infty)}
		\arrow[from=1-1, to=1-2]
		\arrow["\sim"', from=1-3, to=1-2]
		\arrow["\sim", from=1-3, to=2-3]
		\arrow["\sim"', from=3-3, to=2-3]
		\arrow["{\theta_{0}\vee \mathrm{id}}", from=3-3, to=4-3]
		\arrow["\sim"', from=3-3, to=4-3]
		\arrow["{\theta_{\infty}\vee \mathrm{id}}"', from=5-3, to=4-3]
		\arrow["\sim", from=5-3, to=4-3]
		\arrow[from=1-3, to=1-2]
		\arrow["{\mathrm{induced \ by }\ j_\infty\vee j_0}"', from=1-3, to=2-3]
	\end{tikzcd}\]
	The weak equivalences in the diagram are indicated by $\sim$. We emphasize that we equip here $(\mathbb{P}^1_{\mathbb{Z}},\mathbb{G}_{m})$ with the base point $\infty$ coming from the corresponding base point of $\mathbb{P}^1_{\mathbb{Z}}$. Analogously, we equip $(\mathbb{A}^1_{\mathbb{Z}}, \mathbb{G}_{m})_{t_{\infty}} \vee (\mathbb{A}^1_{\mathbb{Z}}, \mathbb{G}_{m})_{t_{0}}$ with the base point \[\begin{tikzcd}
		{\mathrm{Spec}\mathbb{Z}} & {\mathbb{A}^1_{\mathbb{Z}}} & {(\mathbb{A}^1_{\mathbb{Z}}, \mathbb{G}_{m})_{t_{\infty}}} & {(\mathbb{A}^1_{\mathbb{Z}}, \mathbb{G}_{m})_{t_{\infty}} \vee (\mathbb{A}^1_{\mathbb{Z}}, \mathbb{G}_{m})_{t_{0}}}
		\arrow["0", from=1-1, to=1-2]
		\arrow[from=1-2, to=1-3]
		\arrow[hook, from=1-3, to=1-4]
	\end{tikzcd}\]
	and we call this base point also 0. Furthermore we equip $ (\mathbb{P}^1_{\mathbb{Z}},\mathbb{A}^1_{\mathbb{Z}})_{j_0}\vee(\mathbb{P}^1_{\mathbb{Z}},\mathbb{A}^1_{\mathbb{Z}})_{j_{\infty}}$ with the base point 
	\[\begin{tikzcd}
		{\mathrm{Spec}\mathbb{Z}} & {\mathbb{P}^1_{\mathbb{Z}}} & {(\mathbb{P}^1_{\mathbb{Z}},\mathbb{A}^1_{\mathbb{Z}})_{j_0}} & {(\mathbb{P}^1_{\mathbb{Z}},\mathbb{A}^1_{\mathbb{Z}})_{j_0}\vee(\mathbb{P}^1_{\mathbb{Z}},\mathbb{A}^1_{\mathbb{Z}})_{j_{\infty}}}
		\arrow["\infty", from=1-1, to=1-2]
		\arrow[from=1-2, to=1-3]
		\arrow[hook, from=1-3, to=1-4]
	\end{tikzcd}\] and we call it $\infty$, too. Then we equip 
	$(\mathbb{P}^1_{\mathbb{Z}},0)\vee(\mathbb{P}^1_{\mathbb{Z}},\infty)$ also with the base point 
	\[\begin{tikzcd}
		{\mathrm{Spec}\mathbb{Z}} & {(\mathbb{P}^1_{\mathbb{Z}}, 0)} & {(\mathbb{P}^1_{\mathbb{Z}},0)\vee(\mathbb{P}^1_{\mathbb{Z}},\infty)}
		\arrow[hook, from=1-2, to=1-3]
		\arrow["\infty", from=1-1, to=1-2]
	\end{tikzcd}\] and denote it also by $\infty$. The motivic spaces $\mathcal{Z}/\mathbb{A}^1_{\mathbb{Z}}\vee (\mathbb{P}^1_{\mathbb{Z}},\infty)$ and $(\mathbb{P}^1_{\mathbb{Z}},\infty)\vee (\mathbb{P}^1_{\mathbb{Z}},\infty)$ are equipped with the canonical base points.
	Using these base points and the diagram above we get a pointed co-diagonal morphism $\tilde{\nabla}:(\mathbb{P}^1_{\mathbb{Z}},\infty)\rightarrow (\mathbb{P}^1_{\mathbb{Z}},\infty)\vee (\mathbb{P}^1_{\mathbb{Z}},\infty) $ in $\mathcal{H}_{\bullet}(\mathbb{Z})$.\\
	
	On the other hand we also obtain a pointed co-diagonal morphism $\nabla$ for $(\mathbb{P}^1_{\mathbb{Z}},\infty)$ via the \hyperlink{iso} {isomorphism} $ \alpha: (\mathbb{P}^1_{\mathbb{Z}},\infty)\cong S^1\wedge\mathbb{G}_{m} $. Now we also write down this morphism explicitly. We recall that there is weak equivalence $I\rightarrow \Delta^1$ where $I$ is  $\Delta^1{}_{1}\vee_{0}\Delta^1$ and this weak equivalence is induced by projecting $\Delta^1{}_{1}$ to the point $0\in \Delta^1$. We denote the ``mid-point'' of $I$ by $\frac{1}{2}$. Let $\partial I$ denote the boundary of $I$, then this weak equivalence induces a weak equivalence $\mu: I/\partial I\rightarrow \Delta^1/\partial\Delta^1 $. Particularly, the weak equivalence $\mu\wedge \mathrm{id}:I/\partial I\wedge \mathbb{G}_{m}\rightarrow S^1\wedge\mathbb{G}_{m} $ is an isomorphism of cogroup objects in $\mathcal{H}_{\bullet}(\mathbb{Z})$. Moreover there is also a morphism from $ I/\partial I$ to $S^1\vee S^1$ by sending $\Delta^1{}_{1}$ to the first copy of $S^1$ and ${}_{0}\Delta^1$ to the second copy of $S^1$. Now we have the diagram  \[\begin{tikzcd}
		{(\mathbb{P}^1_{\mathbb{Z}},\infty)} & {(\mathcal{X},\infty)} & {I/\partial I\wedge \mathbb{G}_{m}} \\
		&& {(S^1\vee S^1)\wedge\mathbb{G}_{m}} \\
		&& {(I/\partial I\vee I/\partial I)\wedge\mathbb{G}_{m}} \\
		&& {(I/\partial I\wedge \mathbb{G}_{m})\vee (I/\partial I\wedge \mathbb{G}_{m})} \\
		&& {(\mathcal{X},\infty)\vee (\mathcal{X},\infty)} \\
		&& {(\mathbb{P}^1_{\mathbb{Z}},\infty)\vee (\mathbb{P}^1_{\mathbb{Z}},\infty)} & \cdot
		\arrow["\sim"', from=1-2, to=1-1]
		\arrow["\sim", from=1-2, to=1-3]
		\arrow[from=1-3, to=2-3]
		\arrow["{{(\mu \vee\mu )\wedge \mathrm{id}}}"', from=3-3, to=2-3]
		\arrow["{{=}}", from=3-3, to=4-3]
		\arrow["\sim"', from=5-3, to=4-3]
		\arrow["\sim", from=5-3, to=6-3]
		\arrow["\sim", from=3-3, to=2-3]
	\end{tikzcd}\]
	The weak equivalences are indicated again by $\sim$. Therefore we get from this diagram a co-diagonal morphism $\nabla:(\mathbb{P}^1_{\mathbb{Z}},\infty)\rightarrow (\mathbb{P}^1_{\mathbb{Z}},\infty)\vee (\mathbb{P}^1_{\mathbb{Z}},\infty) $ in $\mathcal{H}_{\bullet}(\mathbb{Z})$. In \cite{ASENS_2012_4_45_4_511_0} Cazanave did not show that his co-diagonal $\tilde{\nabla}$ coincides with $\nabla$, so we prove it in this paper.\\
	
	\begin{prop}\label{Proposition 3.2.1}
	The two co-diagonal morphisms $\tilde {\nabla} $ and $\nabla$ are the same.\\
	\end{prop}
	
\begin{proof} We have to show that the diagram  
	\[\begin{tikzcd}
		{(\mathbb{P}^1_{\mathbb{Z}},\infty)} && {(\mathbb{P}^1_{\mathbb{Z}},\infty)\vee (\mathbb{P}^1_{\mathbb{Z}},\infty)} \\
		{ S^1\wedge\mathbb{G}_{m}} && {( S^1\wedge\mathbb{G}_{m})\vee ( S^1\wedge\mathbb{G}_{m})}
		\arrow["{\tilde{\nabla}}", from=1-1, to=1-3]
		\arrow["\alpha"', from=1-1, to=2-1]
		\arrow["\alpha\vee\alpha", from=1-3, to=2-3]
		\arrow[from=2-1, to=2-3]
	\end{tikzcd}\]
	commutes in $\mathcal{H}_{\bullet}(\mathbb{Z})$, where $ S^1\wedge\mathbb{G}_{m}\rightarrow ( S^1\wedge\mathbb{G}_{m})\vee ( S^1\wedge\mathbb{G}_{m})$ is the morphism in $\mathcal{H}_{\bullet}(\mathbb{Z})$ induced by the sequence 
	\[\begin{tikzcd}
		{S^1\wedge\mathbb{G}_{m}} & {I/\partial I\wedge \mathbb{G}_{m}} & {(S^1\vee S^1)\wedge\mathbb{G}_{m}} & \cdot
		\arrow["{{\mu\wedge \mathrm{id}}}"', from=1-2, to=1-1]
		\arrow[from=1-2, to=1-3]
	\end{tikzcd}\]
	We can rewrite the square as follows:
	\[\begin{tikzcd}
		{(\mathbb{P}^1_{\mathbb{Z}},\infty)} & {(\mathbb{P}^1_{\mathbb{Z}},\mathbb{G}_{m})} & {	(\mathbb{A}^1_{\mathbb{Z}}, \mathbb{G}_{m})_{t_{\infty}} \vee (\mathbb{A}^1_{\mathbb{Z}}, \mathbb{G}_{m})_{t_{0}}} \\
		{S^1\wedge\mathbb{G}_{m}} && {(\mathbb{P}^1_{\mathbb{Z}},\mathbb{A}^1_{\mathbb{Z}})_{j_0}\vee(\mathbb{P}^1_{\mathbb{Z}},\mathbb{A}^1_{\mathbb{Z}})_{j_{\infty}}} \\
		&& {(\mathbb{P}^1_{\mathbb{Z}},0)\vee(\mathbb{P}^1_{\mathbb{Z}},\infty)} \\
		&& {\mathbb{P}^1_{\mathbb{Z}}/\mathbb{A}^1_{\mathbb{Z}}\vee (\mathbb{P}^1_{\mathbb{Z}},\infty) } \\
		{( S^1\wedge\mathbb{G}_{m})\vee ( S^1\wedge\mathbb{G}_{m})} && {(\mathbb{P}^1_{\mathbb{Z}},\infty)\vee (\mathbb{P}^1_{\mathbb{Z}},\infty)}
		\arrow[from=1-1, to=1-2]
		\arrow["\sim"', from=1-3, to=1-2]
		\arrow["\sim", from=1-3, to=2-3]
		\arrow["\sim"', from=3-3, to=2-3]
		\arrow["{{{\theta_{0}\vee \mathrm{id}}}}", from=3-3, to=4-3]
		\arrow["\sim"', from=3-3, to=4-3]
		\arrow["{{{\theta_{\infty}\vee \mathrm{id}}}}"', from=5-3, to=4-3]
		\arrow["\sim", from=5-3, to=4-3]
		\arrow[from=1-3, to=1-2]
		\arrow["{{{\mathrm{induced \ by }\ j_\infty\vee j_0}}}"', from=1-3, to=2-3]
		\arrow["\alpha"', from=1-1, to=2-1]
		\arrow[from=2-1, to=5-1]
		\arrow["\alpha\vee\alpha"', from=5-3, to=5-1]
		\arrow["{{\textcircled{1}}}"', from=1-2, to=5-1]
		\arrow["\sim", from=1-2, to=5-1]
		\arrow["{{\textcircled{2}}}"', from=3-3, to=5-1]
		\arrow["\sim", from=3-3, to=5-1]
	\end{tikzcd}\]
	We first explain what the weak equivalence indicated by $\textcircled{1}$ is. For this we have to consider the diagram below again 
	\[\begin{tikzcd}
		{\mathbb{G}_{m}\vee \mathbb{G}_{m}} & {\mathbb{G}_{m}} && {0_+\wedge\mathbb{G}_{m}\vee 1_+\wedge\mathbb{G}_{m}} & {I_+\wedge\mathbb{G}_{m}} \\
		{\mathbb{A}^1_{\mathbb{Z}}\vee \mathbb{A}^1_{\mathbb{Z}}} & {} & {} & {0_+\wedge\mathbb{A}^1_{\mathbb{Z}}\vee 1_+\wedge\mathbb{A}^1_{\mathbb{Z}}} \\
		&&& {} \\
		&&& {} \\
		&&& {0_+\wedge\mathbb{G}_{m}\vee 1_+\wedge\mathbb{G}_{m}} & {I_+\wedge\mathbb{G}_{m}} \\
		&&& \ast
		\arrow["{t_{\infty}\vee t_0}"', from=1-1, to=2-1]
		\arrow["{(\mathrm{id},\mathrm{id} )}", from=1-1, to=1-2]
		\arrow["{t_{\infty}\vee t_0}", from=1-4, to=2-4]
		\arrow["{(l_0,l_1)}", from=1-4, to=1-5]
		\arrow[from=2-3, to=2-2]
		\arrow[from=3-4, to=4-4]
		\arrow["{(l_0,l_1)}", from=5-4, to=5-5]
		\arrow[from=5-4, to=6-4]
	\end{tikzcd}\]
	Recall that we denoted the pushout of the middle diagram by $\mathcal{X}$. We have here the inclusions $\{\frac{1}{2}\}_+\wedge\mathbb{G}_{m} \hookrightarrow\mathcal{X} $ and $ \{\frac{1}{2}\}_+\wedge\mathbb{G}_{m}\hookrightarrow I/\partial I\wedge \mathbb{G}_{m}$. Via the first comparison map the inclusion $\{\frac{1}{2}\}_+\wedge\mathbb{G}_{m} \hookrightarrow\mathcal{X} $ corresponds to the inclusion $\mathbb{G}_{m}\hookrightarrow \mathbb{P}^1_{\mathbb{Z}}$. Via the second comparison map the inclusion $\{\frac{1}{2}\}_+\wedge\mathbb{G}_{m} \hookrightarrow\mathcal{X} $ corresponds to $ \{\frac{1}{2}\}_+\wedge\mathbb{G}_{m}\hookrightarrow I/\partial I\wedge \mathbb{G}_{m}$. Thus we get a sequence of pointed motivic weak equivalences 
	\[\begin{tikzcd}
		{((\mathbb{P}^1_{\mathbb{Z}},\mathbb{G}_{m}),\infty)} & {((\mathcal{X}, \{\frac{1}{2}\}_+\wedge\mathbb{G}_{m}),\infty)} & {((I/\partial I\wedge \mathbb{G}_{m}, \{\frac{1}{2}\}_+\wedge\mathbb{G}_{m}),\ast)}
		\arrow["\sim"', from=1-2, to=1-1]
		\arrow["\sim", from=1-2, to=1-3]
	\end{tikzcd}\] where $(\mathbb{P}^1_{\mathbb{Z}},\mathbb{G}_{m}), (\mathcal{X}, \{\frac{1}{2}\}_+\wedge\mathbb{G}_{m})  $ and $(I/\partial I\wedge \mathbb{G}_{m}, \{\frac{1}{2}\}_+\wedge\mathbb{G}_{m})$ are the cofibers of the corresponding inclusions and $\ast$ is the canonical base point of $(I/\partial I\wedge \mathbb{G}_{m}, \{\frac{1}{2}\}_+\wedge\mathbb{G}_{m})$. Note that $(I/\partial I\wedge \mathbb{G}_{m}, \{\frac{1}{2}\}_+\wedge\mathbb{G}_{m})$ is just $( S^1\wedge\mathbb{G}_{m})\vee ( S^1\wedge\mathbb{G}_{m}) $. Via this zig-zag of weak equivalences we get the morphism indicated by $\textcircled{1}$ which is an isomorphism in $\mathcal{H}_{\bullet}(\mathbb{Z})$.\\
	
	Next we explain what the weak equivalence indicated by $\textcircled{2}$ is. If we equip $\mathcal{X}$ with the base point 
	\[\begin{tikzcd}
		{\mathrm{Spec}\mathbb{Z}} & {1_+\wedge\mathbb{A}^1_{\mathbb{Z}}} & {\mathcal{X}} & \cdot
		\arrow["0", from=1-1, to=1-2]
		\arrow[from=1-2, to=1-3]
	\end{tikzcd}\]
	We denote this base point also by 0. Then the first comparison map induces a pointed weak equivalence $(\mathcal{X},0)\rightarrow (\mathbb{P}^1_{\mathbb{Z}},0)$ and the second comparison map induces a pointed weak equivalence $(\mathcal{X},0)\rightarrow S^1\wedge\mathbb{G}_{m} $. Therefore we also get an isomorphism $\tilde{\alpha}: (\mathbb{P}^1_{\mathbb{Z}},0)\cong S^1\wedge\mathbb{G}_{m} $. Hence the morphism indicated by $\textcircled{2}$ is $\tilde{\alpha}\vee \alpha$.\\
	
	In order to show that the square in the beginning is commutative we have to show that the following three diagrams 
	\[\begin{tikzcd}
		{(\mathbb{P}^1_{\mathbb{Z}},\infty)} & {(\mathbb{P}^1_{\mathbb{Z}},\mathbb{G}_{m})} \\
		{S^1\wedge\mathbb{G}_{m}} & {( S^1\wedge\mathbb{G}_{m})\vee ( S^1\wedge\mathbb{G}_{m})} & {,}
		\arrow[from=1-1, to=1-2]
		\arrow["\alpha"', from=1-1, to=2-1]
		\arrow[from=2-1, to=2-2]
		\arrow["{{\textcircled{1}}}", from=1-2, to=2-2]
	\end{tikzcd}\]
	\[\begin{tikzcd}
		{(\mathbb{P}^1_{\mathbb{Z}},\mathbb{G}_{m})} && {(\mathbb{A}^1_{\mathbb{Z}}, \mathbb{G}_{m})_{t_{\infty}} \vee (\mathbb{A}^1_{\mathbb{Z}}, \mathbb{G}_{m})_{t_{0}}} \\
		&& {(\mathbb{P}^1_{\mathbb{Z}},\mathbb{A}^1_{\mathbb{Z}})_{j_0}\vee(\mathbb{P}^1_{\mathbb{Z}},\mathbb{A}^1_{\mathbb{Z}})_{j_{\infty}}} \\
		{( S^1\wedge\mathbb{G}_{m})\vee ( S^1\wedge\mathbb{G}_{m})} && {(\mathbb{P}^1_{\mathbb{Z}},0)\vee(\mathbb{P}^1_{\mathbb{Z}},\infty)}
		\arrow["\sim"', from=1-3, to=1-1]
		\arrow["\sim"', from=3-3, to=2-3]
		\arrow["{{\textcircled{1}}}"', from=1-1, to=3-1]
		\arrow["{{\tilde{\alpha}\vee \alpha}}", from=3-3, to=3-1]
		\arrow["{{{\mathrm{induced \ by }\ j_\infty\vee j_0}}}", from=1-3, to=2-3]
		\arrow["\sim"', from=1-3, to=2-3]
		\arrow["\sim", from=1-1, to=3-1]
	\end{tikzcd}\]
	\hypertarget{third}{and}
\[\begin{tikzcd}
		{(\mathbb{P}^1_{\mathbb{Z}},0)\vee(\mathbb{P}^1_{\mathbb{Z}},\infty)} & {\mathbb{P}^1_{\mathbb{Z}}/\mathbb{A}^1_{\mathbb{Z}}\vee (\mathbb{P}^1_{\mathbb{Z}},\infty)} \\
		{( S^1\wedge\mathbb{G}_{m})\vee ( S^1\wedge\mathbb{G}_{m})} & {(\mathbb{P}^1_{\mathbb{Z}},\infty)\vee (\mathbb{P}^1_{\mathbb{Z}},\infty)}
		\arrow["{{\tilde{\alpha}\vee \alpha}}"', from=1-1, to=2-1]
		\arrow["{{{\theta_{0}\vee \mathrm{id}}}}", from=1-1, to=1-2]
		\arrow["{{{\theta_{\infty}\vee \mathrm{id}}}}"', from=2-2, to=1-2]
		\arrow["{{\alpha\vee \alpha}}", from=2-2, to=2-1]
		\arrow["\sim"', from=1-1, to=1-2]
		\arrow["\sim", from=2-2, to=1-2]
	\end{tikzcd}\]
	commute in the pointed homotopy category.\\
	
	We start with the first diagram. If we use the definition of the morphisms involved in this diagram, we can write it as follows :
	\[\begin{tikzcd}
		{(\mathbb{P}^1_{\mathbb{Z}},\infty)} && {(\mathbb{P}^1_{\mathbb{Z}},\mathbb{G}_{m})} \\
		{(\mathcal{X},\infty)} && {(\mathcal{X}, \{\frac{1}{2}\}_+\wedge\mathbb{G}_{m}) } \\
		{I/\partial I\wedge \mathbb{G}_{m}} \\
		{S^1\wedge\mathbb{G}_{m}} & {I/\partial I\wedge \mathbb{G}_{m}} & {( S^1\wedge\mathbb{G}_{m})\vee ( S^1\wedge\mathbb{G}_{m})}
		\arrow[from=2-1, to=3-1]
		\arrow["{\mu\wedge\mathrm{id}}"', from=3-1, to=4-1]
		\arrow[from=4-2, to=4-3]
		\arrow[from=2-3, to=4-3]
		\arrow[from=1-1, to=1-3]
		\arrow[from=3-1, to=4-1]
		\arrow["{(\mu\wedge\mathrm{id})^{-1}}", from=4-1, to=4-2]
		\arrow[from=1-1, to=2-1]
		\arrow[from=1-3, to=2-3]
	\end{tikzcd}\]
	Now we see that this diagram is commutative.\\
	
	Next we consider the second diagram. Note that $\mathcal{X}/(I_+\wedge\mathbb{G}_{m})=	(\mathbb{A}^1_{\mathbb{Z}}, \mathbb{G}_{m})_{t_{\infty}} \vee (\mathbb{A}^1_{\mathbb{Z}}, \mathbb{G}_{m})_{t_{0}} $ and the diagram 
	\[\begin{tikzcd}
		{\mathcal{X}/(I_+\wedge\mathbb{G}_{m})} & {(\mathbb{P}^1_{\mathbb{Z}},\mathbb{G}_{m})} \\
		{\mathcal{X}/(\{\frac{1}{2}\}_+\wedge\mathbb{G}_{m})}
		\arrow[from=2-1, to=1-1]
		\arrow[from=1-1, to=1-2]
		\arrow[from=2-1, to=1-2]
	\end{tikzcd}\]
	commutes where $ \mathcal{X}/(\{\frac{1}{2}\}_+\wedge\mathbb{G}_{m})\rightarrow \mathcal{X}/(I_+\wedge\mathbb{G}_{m})$ is just the canonical morphism. There are also two morphisms $\phi_1, \phi_2:I_+\wedge\mathbb{G}_{m}\rightarrow\mathcal{X}/(\{\frac{1}{2}\}_+\wedge\mathbb{G}_{m}) $. The morphism $\phi_1$ is induced by the commutative diagram 
	\[\begin{tikzcd}
		{\{\frac{1}{2}\}_+\wedge\mathbb{G}_{m}} && {(\Delta^1{}_{1})_+\wedge\mathbb{G}_{m}} \\
		{({}_{0}\Delta^1)_++\wedge\mathbb{G}_{m}} && {\mathcal{X}/(\{\frac{1}{2}\}_+\wedge\mathbb{G}_{m})} & \cdot
		\arrow[hook, from=1-1, to=2-1]
		\arrow["{{\text{canonical morphism}}}"', from=2-1, to=2-3]
		\arrow[hook, from=1-1, to=1-3]
		\arrow["\ast"', from=1-3, to=2-3]
	\end{tikzcd}\]
	The morphism $\phi_2$ is induced by 
	\[\begin{tikzcd}
		{\{\frac{1}{2}\}_+\wedge\mathbb{G}_{m}} & {(\Delta^1{}_{1})_+\wedge\mathbb{G}_{m}} \\
		{({}_{0}\Delta^1)_++\wedge\mathbb{G}_{m}} & {\mathcal{X}/(\{\frac{1}{2}\}_+\wedge\mathbb{G}_{m})} && \cdot
		\arrow[hook, from=1-1, to=2-1]
		\arrow[hook, from=1-1, to=1-2]
		\arrow["\ast", from=2-1, to=2-2]
		\arrow["{{\text{canonical morphism}}}", from=1-2, to=2-2]
	\end{tikzcd}\]
	In both diagrams $\ast$ is the constant morphism. Using $\phi_1$ and $\phi_2$ we get a pointed morphism from $(\mathcal{X},0)\vee (\mathcal{X},\infty)$ to $\mathcal{X}/(\{\frac{1}{2}\}_+\wedge\mathbb{G}_{m})$ which is defined by the two commutative diagrams 
	\[\begin{tikzcd}
		{0_+\wedge\mathbb{G}_{m}\vee 1_+\wedge\mathbb{G}_{m}} &&& {I_+\wedge\mathbb{G}_{m}} \\
		\\
		{0_+\wedge\mathbb{A}^1_{\mathbb{Z}}\vee 1_+\wedge\mathbb{A}^1_{\mathbb{Z}}} &&& {\mathcal{X}/(\{\frac{1}{2}\}_+\wedge\mathbb{G}_{m})}
		\arrow["{t_{\infty}\vee t_0}"', from=1-1, to=3-1]
		\arrow["{\phi_2}", from=1-4, to=3-4]
		\arrow["{(l_0,l_1)}", hook, from=1-1, to=1-4]
		\arrow["{\text{canonical morphism}\vee\ast}"', from=3-1, to=3-4]
	\end{tikzcd}\]
	and 
	\[\begin{tikzcd}
		{0_+\wedge\mathbb{G}_{m}\vee 1_+\wedge\mathbb{G}_{m}} &&& {I_+\wedge\mathbb{G}_{m}} \\
		\\
		{0_+\wedge\mathbb{A}^1_{\mathbb{Z}}\vee 1_+\wedge\mathbb{A}^1_{\mathbb{Z}}} &&& {\mathcal{X}/(\{\frac{1}{2}\}_+\wedge\mathbb{G}_{m})} & \cdot
		\arrow["{{{t_{\infty}\vee t_0}}}"', from=1-1, to=3-1]
		\arrow["{{{\phi_1}}}", from=1-4, to=3-4]
		\arrow["{{{(l_0,l_1)}}}", hook, from=1-1, to=1-4]
		\arrow["{{{\ast\vee\text{canonical morphism}}}}"', from=3-1, to=3-4]
	\end{tikzcd}\]
	\
	
	The first diagram induces a morphism $ (\mathcal{X},0)\rightarrow \mathcal{X}/(\{\frac{1}{2}\}_+\wedge\mathbb{G}_{m})$ 
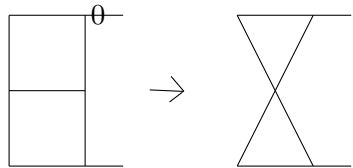
\begin{center}
		\begin{tikzpicture}
			\draw[draw=black, thin, solid] (-1.00,2.00) -- (-0.50,1.00);
			\draw[draw=black, thin, solid] (0.00,2.00) -- (-0.50,1.00);
			\draw[draw=black, thin, solid] (-1.00,2.00) -- (0.50,2.00);
			\draw[draw=black, thin, solid] (-0.50,1.00) -- (-1.00,0.00);
			\draw[draw=black, thin, solid] (-0.50,1.00) -- (0.00,0.00);
			\draw[draw=black, thin, solid] (-1.00,0.00) -- (0.50,0.00);
			\draw[draw=black, thin, solid] (-4.00,2.00) -- (-4.00,0.00);
			\draw[draw=black, thin, solid] (-3.00,2.00) -- (-3.00,0.00);
			\draw[draw=black, thin, solid] (-4.00,2.00) -- (-2.50,2.00);
			\draw[draw=black, thin, solid] (-4.00,0.00) -- (-2.50,0.00);
			\draw[draw=black, thin, solid] (-4.00,1.00) -- (-3.00,1.00);
			\draw[draw=black, thin, solid] (-2.15,1.01) -- (-1.70,0.99);
			\draw[draw=black, thin, solid] (-1.70,0.99) -- (-1.87,1.19);
			\draw[draw=black, thin, solid] (-1.72,0.98) -- (-1.89,0.85);
			\node[black, anchor=south west] at (-3.06,1.75) {$0$};
	\end{tikzpicture}\end{center}
	which sends the ``upper half'' of $\mathcal{X} $ to the base point of $ \mathcal{X}/(\{\frac{1}{2}\}_+\wedge\mathbb{G}_{m})$ and for the ``lower half'' it is just the canonical morphism to $ \mathcal{X}/(\{\frac{1}{2}\}_+\wedge\mathbb{G}_{m})$.\\
	
	The second diagram induces a morphism $(\mathcal{X},\infty)\rightarrow \mathcal{X}/(\{\frac{1}{2}\}_+\wedge\mathbb{G}_{m})$
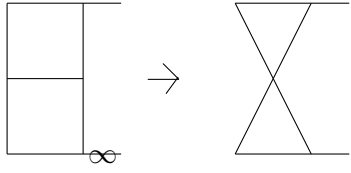
	\begin{center}
		\begin{tikzpicture}
			\draw[draw=black, thin, solid] (-1.00,2.00) -- (-0.50,1.00);
			\draw[draw=black, thin, solid] (0.00,2.00) -- (-0.50,1.00);
			\draw[draw=black, thin, solid] (-0.50,1.00) -- (-1.00,0.00);
			\draw[draw=black, thin, solid] (-0.50,1.00) -- (0.00,0.00);
			\draw[draw=black, thin, solid] (-1.00,2.00) -- (0.50,2.00);
			\draw[draw=black, thin, solid] (-1.00,0.00) -- (0.50,0.00);
			\draw[draw=black, thin, solid] (-4.00,2.00) -- (-4.00,0.00);
			\draw[draw=black, thin, solid] (-3.00,2.00) -- (-3.00,0.00);
			\draw[draw=black, thin, solid] (-4.00,2.00) -- (-2.50,2.00);
			\draw[draw=black, thin, solid] (-4.00,1.00) -- (-3.00,1.00);
			\draw[draw=black, thin, solid] (-4.00,0.00) -- (-2.50,0.00);
			\draw[draw=black, thin, solid] (-2.15,0.99) -- (-1.74,0.99);
			\draw[draw=black, thin, solid] (-1.74,0.99) -- (-1.95,1.20);
			\draw[draw=black, thin, solid] (-1.75,1.02) -- (-1.91,0.80);
			\node[black, anchor=south west] at (-3.07,-0.25) {$\infty$};
		\end{tikzpicture}
	\end{center}
	which sends the ``lower half'' to the base point of $ \mathcal{X}/(\{\frac{1}{2}\}_+\wedge\mathbb{G}_{m})$ and for the ``upper half'' it is the canonical morphism to $ \mathcal{X}/(\{\frac{1}{2}\}_+\wedge\mathbb{G}_{m})$.\\
	
	Therefore it suffices to show that the following diagram is commutative 
	\[\begin{tikzcd}
		{(\mathbb{P}^1_{\mathbb{Z}},\mathbb{G}_{m})} & {\mathcal{X}/(\{\frac{1}{2}\}_+\wedge\mathbb{G}_{m})} & {(\mathbb{A}^1_{\mathbb{Z}}, \mathbb{G}_{m})_{t_{\infty}} \vee (\mathbb{A}^1_{\mathbb{Z}}, \mathbb{G}_{m})_{t_{0}}=\mathcal{X}/(I_+\wedge\mathbb{G}_{m})} \\
		{\mathcal{X}/(\{\frac{1}{2}\}_+\wedge\mathbb{G}_{m})} && {(\mathbb{P}^1_{\mathbb{Z}},\mathbb{A}^1_{\mathbb{Z}})_{j_0}\vee(\mathbb{P}^1_{\mathbb{Z}},\mathbb{A}^1_{\mathbb{Z}})_{j_{\infty}}} \\
		{( S^1\wedge\mathbb{G}_{m})\vee ( S^1\wedge\mathbb{G}_{m})} & {(\mathcal{X},0)\vee (\mathcal{X},\infty)} & {(\mathbb{P}^1_{\mathbb{Z}},0)\vee(\mathbb{P}^1_{\mathbb{Z}},\infty)}
		\arrow["\sim"', from=1-2, to=1-1]
		\arrow["\sim", from=1-2, to=1-3]
		\arrow[from=1-1, to=2-1]
		\arrow[from=2-1, to=3-1]
		\arrow["\sim"', from=3-2, to=3-1]
		\arrow[from=3-3, to=2-3]
		\arrow["{{\text{induced by}\  j_{\infty}\vee j_0}}"', from=1-3, to=2-3]
		\arrow["\sim", from=3-2, to=3-3]
		\arrow["{{{\textcircled{3}}}}"', from=3-2, to=1-2]
		\arrow["\sim", from=3-2, to=1-2]
	\end{tikzcd}\]
	where the morphism indicated by $\textcircled{3}$ is just the morphism we defined above. We will see soon that it is also a weak equivalence. Note that this is a pointed morphism if we take for $\mathcal{X}/(\{\frac{1}{2}\}_+\wedge\mathbb{G}_{m})$ the base point $\infty$ and for $(\mathcal{X},0)\vee (\mathcal{X},\infty)$ the base point $\infty\in (\mathcal{X},0)$.\\
	
	At first we consider the left square. Note that the composition 
	\[\begin{tikzcd}
		{\mathcal{X}/(\{\frac{1}{2}\}_+\wedge\mathbb{G}_{m})} & {(\mathbb{P}^1_{\mathbb{Z}},\mathbb{G}_{m})} & {\mathcal{X}/(\{\frac{1}{2}\}_+\wedge\mathbb{G}_{m})}
		\arrow["\sim", from=1-1, to=1-2]
		\arrow[from=1-2, to=1-3]
	\end{tikzcd}\]
	is just the identity. Now we apply the geometric realization functor defined in section 2. The diagram\[\begin{tikzcd}
		& {|\mathcal{X}/(\{\frac{1}{2}\}_+\wedge\mathbb{G}_{m})|} \\
		\\
		{|( S^1\wedge\mathbb{G}_{m})\vee ( S^1\wedge\mathbb{G}_{m})|} & {|(\mathcal{X},0)\vee (\mathcal{X},\infty)|}
		\arrow["\sim"', from=3-2, to=3-1]
		\arrow["{{{\textcircled{3}}}}"', from=3-2, to=1-2]
		\arrow["\sim"', from=1-2, to=3-1]
	\end{tikzcd}\] 
	commutes up to homotopy in $\mathrm{Pre}_{\Delta}(\mathbb{Z})_{\ast}$. Therefore \[\begin{tikzcd}
		& {\mathcal{X}/(\{\frac{1}{2}\}_+\wedge\mathbb{G}_{m})} \\
		\\
		{( S^1\wedge\mathbb{G}_{m})\vee ( S^1\wedge\mathbb{G}_{m})} & {(\mathcal{X},0)\vee (\mathcal{X},\infty)}
		\arrow["\sim"', from=3-2, to=3-1]
		\arrow["\sim"', from=1-2, to=3-1]
		\arrow["{{{\textcircled{3}}}}"', from=3-2, to=1-2]
	\end{tikzcd}\] 
	commutes in $\mathcal{H}_{\bullet}(\mathbb{Z})$. In particular it follows that $\textcircled{3} $  is a weak equivalence. The commutativity of the right square follows directly from the definition of the morphisms.\\
	
	Finally we show the commutativity of the third diagram. 
	Recall that $\mathbb{P}^1_{\mathbb{Z}}/\mathbb{A}^1_{\mathbb{Z}}$ is the cofiber of the path $j:\mathbb{A}^1_{\mathbb{Z}}\rightarrow \mathbb{P}^1_{\mathbb{Z}}$ from $\infty$ to $0$. The path $j$ is the composition $\psi\circ j_\infty$ where $\psi$ is induced by the ring isomorphism \begin{align*}
		\mathbb{Z}[T_0,T_1]\rightarrow \mathbb{Z}[T_0,T_1]; T_0\mapsto T_0-T_1, T_1\mapsto T_1
	\end{align*}
	In particular $j$ induces an isomorphism from $\mathbb{A}^1_{\mathbb{Z}} $ to the open subset $\mathrm{D}_+(T_0+T_1)$ of $\mathbb{P}^1_{\mathbb{Z}}$. Therefore the cofiber $\mathbb{P}^1_{\mathbb{Z}}/\mathbb{A}^1_{\mathbb{Z}} $ is just the cofiber of the inclusion $\mathrm{D}_+(T_0+T_1)\hookrightarrow \mathbb{P}^1_{\mathbb{Z}}$. Hence we can replace in the third diagram $\mathbb{P}^1_{\mathbb{Z}}/\mathbb{A}^1_{\mathbb{Z}} $ by 
	$\mathbb{P}^1_{\mathbb{Z}}/\mathrm{D}_+(T_0+T_1) $.\\
	
	Next we consider the following pushout diagram 
	\[\begin{tikzcd}
		{0\times\mathrm{D}_+((T_0+T_1)T_0T_1)\sqcup 1\times\mathrm{D}_+((T_0+T_1)T_0T_1)} & {I\times\mathrm{D}_+((T_0+T_1)T_0T_1)} \\
		{0\times\mathrm{D}_+((T_0+T_1)T_0)\sqcup 1\times\mathrm{D}_+((T_0+T_1)T_1)}
		\arrow["{\text{inclusion}\sqcup\text{inclusion}}"', hook, from=1-1, to=2-1]
		\arrow[hook, from=1-1, to=1-2]
	\end{tikzcd}\] where we denote the pushout by \hypertarget{pushoutA}{$A$}. It is clear that there is a canonical weak equivalence from $A$ to $\mathrm{D}_+((T_0+T_1)T_0)$. Then there is a comparison map 
	\[\begin{tikzcd}
		{0\times\mathrm{D}_+((T_0+T_1)T_0T_1)\sqcup 1\times\mathrm{D}_+((T_0+T_1)T_0T_1)} & {I\times\mathrm{D}_+((T_0+T_1)T_0T_1)} \\
		{0\times\mathrm{D}_+((T_0+T_1)T_0)\sqcup 1\times\mathrm{D}_+((T_0+T_1)T_1)} \\
		{} \\
		{} \\
		{0\times\mathbb{G}_{m}\sqcup 1\times\mathbb{G}_{m}} & {I \times\mathbb{G}_{m}} \\
		{0\times \mathbb{A}^1_{\mathbb{Z}}\sqcup 1\times\mathbb{A}^1_{\mathbb{Z}}}
		\arrow["{\text{inclusion}\sqcup\text{inclusion}}"', hook, from=1-1, to=2-1]
		\arrow[hook, from=1-1, to=1-2]
		\arrow[from=3-1, to=4-1]
		\arrow[from=5-1, to=5-2]
		\arrow["{t_\infty\sqcup t_0}"', from=5-1, to=6-1]
	\end{tikzcd}\]
	where we denote the pushout of the second diagram by $\tilde{\mathcal{X}}$. Note that $\tilde{\mathcal{X}}/I\times\{1\}$ is the motivic space $\mathcal{X}$, where 1 is the base point of $\mathbb{G}_{m}$. Naturally there is a weak equivalence from $\tilde{\mathcal{X}}$ to $\mathbb{P}^1_{\mathbb{Z}}$ just as for $\mathcal{X}$. Furthermore there is also a canonical weak equivalence from $\tilde{\mathcal{X}}$ to $S^1\wedge\mathbb{G}_{m}$. It is the composition $\tilde{\mathcal{X}}\rightarrow\mathcal{X}\rightarrow S^1\wedge\mathbb{G}_{m}$.\\
	
	As for $\mathcal{X}$ we can equip $\tilde{\mathcal{X}}$ with the base points $0$ and $\infty$. Moreover we have the following commutative diagrams 
	\[\begin{tikzcd}
		& {(\mathbb{P}^1_{\mathbb{Z}},0)} \\
		{(\tilde{\mathcal{X}},0)} & {(\mathcal{X},0)} \\
		& {S^1\wedge\mathbb{G}_{m}}
		\arrow[from=2-2, to=1-2]
		\arrow["\sim", from=2-1, to=2-2]
		\arrow[from=2-1, to=1-2]
		\arrow[from=2-2, to=3-2]
		\arrow[from=2-1, to=3-2]
	\end{tikzcd}\]
	and 
	\[\begin{tikzcd}
		& {(\mathbb{P}^1_{\mathbb{Z}},\infty)} \\
		{(\tilde{\mathcal{X}},\infty)} & {(\mathcal{X},\infty)} \\
		& {S^1\wedge\mathbb{G}_{m}} & \cdot
		\arrow[from=2-2, to=1-2]
		\arrow["\sim", from=2-1, to=2-2]
		\arrow[from=2-1, to=1-2]
		\arrow[from=2-2, to=3-2]
		\arrow[from=2-1, to=3-2]
	\end{tikzcd}\]
	Thus we only have to show that the diagram 
	\[\begin{tikzcd}
		{(\mathbb{P}^1_{\mathbb{Z}},0)\vee(\mathbb{P}^1_{\mathbb{Z}},\infty)} &&& {\mathbb{P}^1_{\mathbb{Z}}/\mathbb{A}^1_{\mathbb{Z}}\vee (\mathbb{P}^1_{\mathbb{Z}},\infty)} \\
		\\
		{(\tilde{\mathcal{X}},0)\vee(\tilde{\mathcal{X}},\infty)} \\
		\\
		{( S^1\wedge\mathbb{G}_{m})\vee ( S^1\wedge\mathbb{G}_{m})} && {(\tilde{\mathcal{X}},\infty)\vee(\tilde{\mathcal{X}},\infty)} & {(\mathbb{P}^1_{\mathbb{Z}},\infty)\vee(\mathbb{P}^1_{\mathbb{Z}},\infty)}
		\arrow["\sim", from=3-1, to=1-1]
		\arrow["\sim"', from=3-1, to=5-1]
		\arrow["{{{{{\theta_{0}\vee \mathrm{id}}}}}}", from=1-1, to=1-4]
		\arrow["\sim"', from=1-1, to=1-4]
		\arrow["\sim", from=5-3, to=5-1]
		\arrow["{{{{{ \sim}}}}}"', from=5-3, to=5-4]
		\arrow["{{{{{\theta_{\infty}\vee \mathrm{id}}}}}}"', from=5-4, to=1-4]
		\arrow["\sim", from=5-4, to=1-4]
	\end{tikzcd}\]
	is commutative.\\
	
	The comparison map \[\begin{tikzcd}
		{0\times\mathrm{D}_+((T_0+T_1)T_0T_1)\sqcup 1\times\mathrm{D}_+((T_0+T_1)T_0T_1)} & {I\times\mathrm{D}_+((T_0+T_1)T_0T_1)} \\
		{0\times\mathrm{D}_+((T_0+T_1)T_0)\sqcup 1\times\mathrm{D}_+((T_0+T_1)T_1)} \\
		{} \\
		{} \\
		{0\times\mathbb{G}_{m}\sqcup 1\times\mathbb{G}_{m}} & {I \times\mathbb{G}_{m}} \\
		{0\times \mathbb{A}^1_{\mathbb{Z}}\sqcup 1\times\mathbb{A}^1_{\mathbb{Z}}}
		\arrow["{\text{inclusion}\sqcup\text{inclusion}}"', hook, from=1-1, to=2-1]
		\arrow[hook, from=1-1, to=1-2]
		\arrow[from=3-1, to=4-1]
		\arrow[from=5-1, to=5-2]
		\arrow["{t_\infty\sqcup t_0}"', from=5-1, to=6-1]
	\end{tikzcd}\]
	induces an inclusion $\hyperlink{pushoutA}{A}\hookrightarrow\tilde{\mathcal{X}}$. Hence we get a canonical weak equivalence $\tilde{\mathcal{X}}/A\rightarrow \mathbb{P}^1_{\mathbb{Z}}/\mathrm{D}_+(T_0+T_1)$. The canonical projection $\tilde{\mathcal{X}}\rightarrow\tilde{\mathcal{X}}/A $ induces two pointed weak equivalnces $\tilde{\theta}_\infty:(\tilde{\mathcal{X}},\infty)\rightarrow \tilde{\mathcal{X}}/A $ and $\tilde{\theta}_0:(\tilde{\mathcal{X}},0)\rightarrow \tilde{\mathcal{X}}/A $.\\
	
	The canonical morphism $\tilde{\mathcal{X}}\rightarrow \mathbb{P}^1_{\mathbb{Z}}$ induces  pointed weak equivalences $\nu_\infty:(\tilde{\mathcal{X}},\infty)\rightarrow (\mathbb{P}^1_{\mathbb{Z}},\infty)$ and $\nu_0:(\tilde{\mathcal{X}},0)\rightarrow (\mathbb{P}^1_{\mathbb{Z}},0)$. Now we can consider the following diagram 
	\[\begin{tikzcd}
		{(\mathbb{P}^1_{\mathbb{Z}},0)\vee(\mathbb{P}^1_{\mathbb{Z}},\infty)} && {\mathbb{P}^1_{\mathbb{Z}}/\mathrm{D}_+(T_0+T_1)\vee (\mathbb{P}^1_{\mathbb{Z}},\infty)} \\
		\\
		{(\tilde{\mathcal{X}},0)\vee(\tilde{\mathcal{X}},\infty)} & {\tilde{\mathcal{X}}/A\vee (\tilde{\mathcal{X}},\infty)} \\
		\\
		{( S^1\wedge\mathbb{G}_{m})\vee ( S^1\wedge\mathbb{G}_{m})} & {(\tilde{\mathcal{X}},\infty)\vee(\tilde{\mathcal{X}},\infty)} & {(\mathbb{P}^1_{\mathbb{Z}},\infty)\vee(\mathbb{P}^1_{\mathbb{Z}},\infty)}
		\arrow["\sim", from=3-1, to=1-1]
		\arrow["\sim"', from=3-1, to=5-1]
		\arrow["\sim", from=3-1, to=3-2]
		\arrow["{{{{{{{{{\theta_{0}\vee \mathrm{id}}}}}}}}}}", from=1-1, to=1-3]
		\arrow["\sim"', from=3-2, to=1-3]
		\arrow["\sim"', from=1-1, to=1-3]
		\arrow["{{{\theta_{\infty}\vee \mathrm{id}}}}", from=5-3, to=1-3]
		\arrow["\sim", from=5-2, to=5-3]
		\arrow["\sim"', from=5-2, to=5-1]
		\arrow["\sim"', from=5-3, to=1-3]
		\arrow["\sim"', from=5-2, to=3-2]
		\arrow["{{{\tilde{\theta}_0\vee\mathrm{id}}}}"', from=3-1, to=3-2]
		\arrow["{{{\tilde{\theta}_\infty\vee\mathrm{id}}}}", from=5-2, to=3-2]
		\arrow["{{\nu_0\vee\nu_\infty}}"', from=3-1, to=1-1]
		\arrow["{{\nu_\infty\vee\nu_\infty}}"', from=5-2, to=5-3]
	\end{tikzcd}\]
	where the outer diagram is our \hyperlink{third}{third diagram}.\\
	
	By construction the diagrams
	\[\begin{tikzcd}
		{(\mathbb{P}^1_{\mathbb{Z}},0)\vee(\mathbb{P}^1_{\mathbb{Z}},\infty)} && {\mathbb{P}^1_{\mathbb{Z}}/\mathrm{D}_+(T_0+T_1)\vee (\mathbb{P}^1_{\mathbb{Z}},\infty)} \\
		\\
		{(\tilde{\mathcal{X}},0)\vee(\tilde{\mathcal{X}},\infty)} && {\tilde{\mathcal{X}}/A\vee (\tilde{\mathcal{X}},\infty)}
		\arrow["{{{{{\theta_{0}\vee \mathrm{id}}}}}}", from=1-1, to=1-3]
		\arrow["{{\nu_0\vee\nu_\infty}}", from=3-1, to=1-1]
		\arrow["\sim", from=3-3, to=1-3]
		\arrow["\sim"', from=3-1, to=1-1]
		\arrow["\sim"', from=1-1, to=1-3]
		\arrow["{{\tilde{\theta}_0\vee\mathrm{id}}}", from=3-1, to=3-3]
		\arrow["\sim"', from=3-1, to=3-3]
	\end{tikzcd}\]
	and 
	\[\begin{tikzcd}
		{(\mathbb{P}^1_{\mathbb{Z}},\infty)\vee(\mathbb{P}^1_{\mathbb{Z}},\infty)} && {\mathbb{P}^1_{\mathbb{Z}}/\mathrm{D}_+(T_0+T_1)\vee (\mathbb{P}^1_{\mathbb{Z}},\infty)} \\
		\\
		{(\tilde{\mathcal{X}},\infty)\vee(\tilde{\mathcal{X}},\infty)} && {\tilde{\mathcal{X}}/A\vee (\tilde{\mathcal{X}},\infty)}
		\arrow["{{{{{\theta_{\infty}\vee \mathrm{id}}}}}}", from=1-1, to=1-3]
		\arrow["{{\nu_\infty\vee\nu_\infty}}", from=3-1, to=1-1]
		\arrow["\sim", from=3-3, to=1-3]
		\arrow["{{\tilde{\theta}_\infty\vee\mathrm{id}}}", from=3-1, to=3-3]
		\arrow["\sim"', from=1-1, to=1-3]
		\arrow["\sim"', from=3-1, to=1-1]
		\arrow["\sim"', from=3-1, to=3-3]
	\end{tikzcd}\]
	are commutative.\\
	
	In the final step we have to show that the diagram 
	\[\begin{tikzcd}
		{(\tilde{\mathcal{X}},0)\vee(\tilde{\mathcal{X}},\infty)} & {\tilde{\mathcal{X}}/A\vee (\tilde{\mathcal{X}},\infty)} \\
		\\
		{( S^1\wedge\mathbb{G}_{m})\vee ( S^1\wedge\mathbb{G}_{m})} & {(\tilde{\mathcal{X}},\infty)\vee(\tilde{\mathcal{X}},\infty)}
		\arrow["\sim"', from=1-1, to=3-1]
		\arrow["\sim", from=1-1, to=1-2]
		\arrow["\sim"', from=3-2, to=3-1]
		\arrow["\sim"', from=3-2, to=1-2]
		\arrow["{{{\tilde{\theta}_0\vee\mathrm{id}}}}"', from=1-1, to=1-2]
		\arrow["{{{\tilde{\theta}_\infty\vee\mathrm{id}}}}", from=3-2, to=1-2]
	\end{tikzcd}\]
	commutes. For this we want to construct a morphism from $\tilde{\mathcal{X}}/A$ to $ S^1\wedge\mathbb{G}_{m} $ in the pointed homotopy category.\\
	
	Taking the composition \[\begin{tikzcd}
		{I\times \{1\}} & {\tilde{\mathcal{X}}} & {\tilde{\mathcal{X}}/A}
		\arrow[hook, from=1-1, to=1-2]
		\arrow[from=1-2, to=1-3]
	\end{tikzcd}\]where $1$ is the canonical base point of $\mathbb{G}_{m}$ we can form the cofiber $(\tilde{\mathcal{X}}/A)/ I\times \{1\}$ and we show that the projection \[\begin{tikzcd}
		{\tilde{\mathcal{X}}/A} & {(\tilde{\mathcal{X}}/A)/ I\times \{1\}}
		\arrow[from=1-1, to=1-2]
	\end{tikzcd}\]
	is a sectionwise weak equivalence of pointed spaces. Let $U\in\mathcal{S}m_{\mathbb{Z}} $. Recall that  $1\in \mathbb{G}_{m}(U) $ is the morphism \[\begin{tikzcd}
		U & {\mathrm{Spec}\mathbb{Z}} & {\mathbb{G}_{m}}
		\arrow[from=1-1, to=1-2]
		\arrow["1", from=1-2, to=1-3]
	\end{tikzcd}\]
	where $U\rightarrow\mathrm{Spec}\mathbb{Z}$ is the structure morphism. If this morphism factors through $\mathrm{D}_+((T_0+T_1)T_0T_1) $, then the image of the composition 
	\[\begin{tikzcd}
		{(I\times \{1\})(U)} & {\tilde{\mathcal{X}}(U)} & {(\tilde{\mathcal{X}}/A)(U)}
		\arrow[hook, from=1-1, to=1-2]
		\arrow[from=1-2, to=1-3]
	\end{tikzcd}\]
	is the base point of $(\tilde{\mathcal{X}}/A)(U)$. If \[\begin{tikzcd}
		U & {\mathrm{Spec}\mathbb{Z}} & {\mathbb{G}_{m}}
		\arrow[from=1-1, to=1-2]
		\arrow["1", from=1-2, to=1-3]
	\end{tikzcd}\]
	does not factor through $\mathrm{D}_+((T_0+T_1)T_0T_1)$, then the set-theoretic image of this morphism contains the prime ideal $(T_0-T_1,2)$. Recall that the set-theoretic image of $1:\mathrm{Spec}\mathbb{Z}\rightarrow\mathbb{G}_{m}$ consists of the prime ideals of the form $(T_0-T_1) $ and $(T_0-T_1,p)$ where $p$ runs over the prime numbers. The prime ideals $(T_0-T_1) $ and $(T_0-T_1,p)$ for $p\neq2$ lie in $\mathrm{D}_+(T_0+T_1)$. Only $(T_0-T_1,2)$ is not contained in $\mathrm{D}_+(T_0+T_1)$, so if \[\begin{tikzcd}
		U & {\mathrm{Spec}\mathbb{Z}} & {\mathbb{G}_{m}}
		\arrow[from=1-1, to=1-2]
		\arrow["1", from=1-2, to=1-3]
	\end{tikzcd}\]
	does not factor through $\mathrm{D}_+((T_0+T_1)T_0T_1)$, then the set-theoretic image of this morphism must contain $(T_0-T_1,2)$. It follows that $A(U)$ and $(I\times \{1\})(U)$ are disjoint. Therefore \[\begin{tikzcd}
		{(I\times \{1\})(U)} & {\tilde{\mathcal{X}}(U)} & {(\tilde{\mathcal{X}}/A)(U)}
		\arrow[hook, from=1-1, to=1-2]
		\arrow[from=1-2, to=1-3]
	\end{tikzcd}\]
	is injective and \[\begin{tikzcd}
		{(\tilde{\mathcal{X}}/A)(U)} & {((\tilde{\mathcal{X}}/A)/ I\times \{1\})(U)}
		\arrow[from=1-1, to=1-2]
	\end{tikzcd}\] is a weak equivalence of pointed simplicial sets.\\
	
	Let $\mathcal{C}(\mathbb{A}^1_{\mathbb{Z}}) $ be $\Delta^1\wedge \mathbb{A}^1_{\mathbb{Z}}$ where $\Delta^1$ is based at 1 and $\mathbb{A}^1_{\mathbb{Z}}$ is based at 1. Let $\mathcal{C}'(\mathbb{A}^1_{\mathbb{Z}}) $ be $\Delta^1\wedge \mathbb{A}^1_{\mathbb{Z}}$ where $\Delta^1$ is based at 0 and $\mathbb{A}^1_{\mathbb{Z}}$ is based at 1. Then we consider the following diagram
	\[\begin{tikzcd}
		{0_+\wedge\mathbb{A}^1_{\mathbb{Z}}\vee 1_+\wedge\mathbb{A}^1_{\mathbb{Z}}} & {(\tilde{\mathcal{X}}/A)/ I\times \{1\}} \\
		{\mathcal{C}'(\mathbb{A}^1_{\mathbb{Z}})\vee\mathcal{C}(\mathbb{A}^1_{\mathbb{Z}}) }
		\arrow[from=1-1, to=1-2]
		\arrow[hook, from=1-1, to=2-1]
	\end{tikzcd}\]
	where we denote the pushout of this diagram by \hypertarget{ha}{$\hat{\mathcal{X}}$}. The canonical inclusion $\tilde{\mathcal{X}}\rightarrow\hat{\mathcal{X}} $ is a weak equivalence. We can illustrate $\hat{\mathcal{X}}$ as follows: 
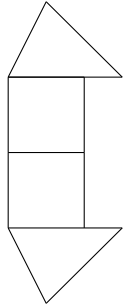
\begin{center}
		\begin{tikzpicture}
			\draw[draw=black, thin, solid] (-2.00,2.00) -- (-2.00,1.00);
			\draw[draw=black, thin, solid] (-2.00,1.00) -- (-1.00,1.00);
			\draw[draw=black, thin, solid] (-2.00,2.00) -- (-1.00,2.00);
			\draw[draw=black, thin, solid] (-2.00,1.00) -- (-2.00,0.00);
			\draw[draw=black, thin, solid] (-1.00,1.00) -- (-1.00,0.00);
			\draw[draw=black, thin, solid] (-1.00,2.00) -- (-1.00,1.00);
			\draw[draw=black, thin, solid] (-2.00,0.00) -- (-0.50,0.00);
			\draw[draw=black, thin, solid] (-2.00,2.00) -- (-0.50,2.00);
			\draw[draw=black, thin, solid] (-1.50,3.00) -- (-2.00,2.00);
			\draw[draw=black, thin, solid] (-1.50,3.00) -- (-0.50,2.00);
			\draw[draw=black, thin, solid] (-2.00,0.00) -- (-1.50,-1.00);
			\draw[draw=black, thin, solid] (-0.50,0.00) -- (-1.50,-1.00);
		\end{tikzpicture}
	\end{center}
	Now we can apply the geometric realization functor defined in chapter 2. In particular we can construct a pointed weak equivalence $| S^1\wedge\mathbb{G}_{m}|\rightarrow|\hat{\mathcal{X}}|$. First note that we can consider the following pushout
	\[\begin{tikzcd}
		{0_+\wedge\mathbb{G}_{m}\vee 1_+\wedge\mathbb{G}_{m}} & {I_+\wedge\mathbb{G}_{m}} \\
		{\mathcal{C}'(\mathbb{G}_{m})\vee\mathcal{C}(\mathbb{G}_{m}) }
		\arrow[from=1-1, to=1-2]
		\arrow[hook, from=1-1, to=2-1]
	\end{tikzcd}\]
	where the reduced cones $\mathcal{C}(\mathbb{G}_{m}) $ and $\mathcal{C}'(\mathbb{G}_{m})$ are just defined as for $\mathbb{A}^1_{\mathbb{Z}}$. We denote this pushout by $\mathcal{S}$. Moreover we also have the comparison map 
	\[\begin{tikzcd}
		{\mathcal{C}'(\mathbb{G}_{m})\vee\mathcal{C}(\mathbb{G}_{m}) } & {0_+\wedge\mathbb{G}_{m}\vee 1_+\wedge\mathbb{G}_{m}} & {I_+\wedge\mathbb{G}_{m}} \\
		{\mathcal{C}'(\mathbb{A}^1_{\mathbb{Z}})\vee\mathcal{C}(\mathbb{A}^1_{\mathbb{Z}})} & {{0_+\wedge\mathbb{A}^1_{\mathbb{Z}}\vee 1_+\wedge\mathbb{A}^1_{\mathbb{Z}}}} & {(\tilde{\mathcal{X}}/A)/ I\times \{1\}}
		\arrow[from=1-2, to=1-3]
		\arrow[hook', from=1-2, to=1-1]
		\arrow[hook, from=1-2, to=2-2]
		\arrow[from=1-3, to=2-3]
		\arrow[from=2-2, to=2-3]
		\arrow[from=1-1, to=2-1]
		\arrow[from=2-2, to=2-1]
	\end{tikzcd}\]
	which induces a pointed morphism $\rho:\mathcal{S}\rightarrow \hyperlink{ha}{\hat{\mathcal{X}}}$. Now it is easy to see that there is a pointed weak equivalence $\phi:| S^1\wedge\mathbb{G}_{m}|\rightarrow|\mathcal{S}| $ by stretching  $|S^1\wedge\mathbb{G}_{m}|$. We illustrate this morphism as follows:
	
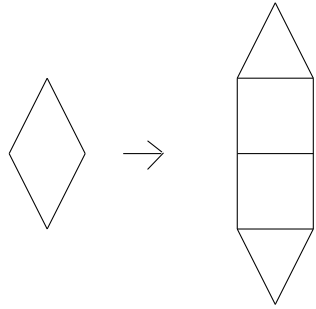
\begin{center}
		\begin{tikzpicture}
			\draw[draw=black, thin, solid] (-2.50,2.00) -- (-3.00,1.00);
			\draw[draw=black, thin, solid] (-2.50,2.00) -- (-2.00,1.00);
			\draw[draw=black, thin, solid] (-3.00,1.00) -- (-2.50,0.00);
			\draw[draw=black, thin, solid] (-2.00,1.00) -- (-2.50,0.00);
			\draw[draw=black, thin, solid] (0.00,2.00) -- (0.00,1.00);
			\draw[draw=black, thin, solid] (0.00,2.00) -- (1.00,2.00);
			\draw[draw=black, thin, solid] (1.00,2.00) -- (1.00,0.00);
			\draw[draw=black, thin, solid] (0.00,1.00) -- (0.00,0.00);
			\draw[draw=black, thin, solid] (0.00,0.00) -- (1.00,0.00);
			\draw[draw=black, thin, solid] (0.00,1.00) -- (1.00,1.00);
			\draw[draw=black, thin, solid] (0.50,3.00) -- (0.00,2.00);
			\draw[draw=black, thin, solid] (0.50,3.00) -- (1.00,2.00);
			\draw[draw=black, thin, solid] (0.00,0.00) -- (0.50,-1.00);
			\draw[draw=black, thin, solid] (1.00,0.00) -- (0.50,-1.00);
			\draw[draw=black, thin, solid] (-1.50,1.00) -- (-1.00,1.00);
			\draw[draw=black, thin, solid] (-0.99,1.02) -- (-1.18,1.17);
			\draw[draw=black, thin, solid] (-0.97,1.01) -- (-1.18,0.79);
		\end{tikzpicture}
	\end{center}
	Altogether we get a pointed morphism $|\rho|\circ \phi:|S^1\wedge\mathbb{G}_{m}|\rightarrow \hat{\mathcal{X}}$. We consider now the following diagram 
	\[\begin{tikzcd}
		{|(\tilde{\mathcal{X}},0)|} &&& {|\tilde{\mathcal{X}}/A|} \\
		\\
		{| S^1\wedge\mathbb{G}_{m}|} & {|\mathcal{S}|} && {|\hat{\mathcal{X}}|}
		\arrow["\sim"', from=1-1, to=3-1]
		\arrow["{|\tilde{\theta}_0|}", from=1-1, to=1-4]
		\arrow["\phi"', from=3-1, to=3-2]
		\arrow["{|\rho|}"', from=3-2, to=3-4]
		\arrow["{\text{canonical morphism}}", from=1-4, to=3-4]
		\arrow["\sim"', from=1-4, to=3-4]
		\arrow["\sim"', from=1-1, to=1-4]
	\end{tikzcd}\]
	which commutes in the pointed homotopy category $\mathcal{H}o_{\Delta}(\mathbb{Z})$. We explain now why this diagram commutes. For this we first look at the morphism
	\[\begin{tikzcd}
		{(0,1]\times\{0\}} & {|\mathcal{C}(\mathbb{A}^1_{\mathbb{Z}})|} & {|\hat{\mathcal{X}}|} & 
		\arrow[from=1-1, to=1-2]
		\arrow[from=1-2, to=1-3]
	\end{tikzcd}\]
	where $0$ in the brackets is the base point 0 of $\mathbb{A}^1_{\mathbb{Z}}$. Let $|\hat{\mathcal{X}}|/|((0,1]\times\{0\}|)$ be the cofiber of this morphism. We show that the projection $|\hat{\mathcal{X}}|\rightarrow |\hat{\mathcal{X}}/((0,1]\times\{0\})|$ is a sectionwise weak equivalence. Let $U\in\mathcal{S}m_{\mathbb{Z}}$. Then $0$ in $\mathbb{A}^1_{\mathbb{Z}}(U)$ is the composition 
	\[\begin{tikzcd}
		U & {\mathrm{Spec}\mathbb{Z}} & {\mathbb{A}^1_{\mathbb{Z}}}
		\arrow[from=1-1, to=1-2]
		\arrow["0", from=1-2, to=1-3]
	\end{tikzcd}\]
	where $U\rightarrow\mathrm{Spec}\mathbb{Z}$ is the structure morphism. This morphism does not factor through \[\begin{tikzcd}
		U & {\mathrm{Spec}\mathbb{Z}} & {\mathbb{G}_{m}} & {,}
		\arrow[from=1-1, to=1-2]
		\arrow["1", from=1-2, to=1-3]
	\end{tikzcd}\] therefore \[\begin{tikzcd}
		{((0,1]\times\{0\})(U)} & {|\mathcal{C}(\mathbb{A}^1_{\mathbb{Z}})|(U)} & {|\hat{\mathcal{X}}(U)|} 
		\arrow[from=1-1, to=1-2]
		\arrow[from=1-2, to=1-3]
	\end{tikzcd}\] is injective. Since collapsing a half-open interval induces a weak equivalence, the projection $|\hat{\mathcal{X}}|\rightarrow |\hat{\mathcal{X}}/((0,1]\times\{0\})|$ is a pointed sectionwise weak equivalence.\\
	
In the next step we can also look at the morphism 
	\[\begin{tikzcd}
		{[0,1)\times\{0\}} & {|\mathcal{C}'(\mathbb{A}^1_{\mathbb{Z}})|} & {|\hat{\mathcal{X}}/((0,1]\times\{0\})|} & \cdot
		\arrow[from=1-1, to=1-2]
		\arrow[from=1-2, to=1-3]
	\end{tikzcd}\]
	Using the same arguments as before we can see that also the projection of  $|\hat{\mathcal{X}}|$ to the cofiber of this morphism is a pointed sectionwise weak equivalence. Now we denote the cofiber by $\mathcal{Z}$.\\
	
	We can deform the composition 
	\[\begin{tikzcd}
		{|(\tilde{\mathcal{X}},0)|} & {| S^1\wedge\mathbb{G}_{m}|} & {|\mathcal{S}|} & {|\hat{\mathcal{X}}|} & {\mathcal{Z}}
		\arrow[from=1-1, to=1-2]
		\arrow["\phi", from=1-2, to=1-3]
		\arrow["{{|\rho|}}", from=1-3, to=1-4]
		\arrow[from=1-4, to=1-5]
	\end{tikzcd}\]
	to \[\begin{tikzcd}
		{|(\tilde{\mathcal{X}},0)|} & {|\tilde{\mathcal{X}}/A|} &&& {|\hat{\mathcal{X}}|} & {\mathcal{Z}}
		\arrow["{{|\tilde{\theta}_0|}}", from=1-1, to=1-2]
		\arrow["{{\text{canonical morphism}}}", from=1-2, to=1-5]
		\arrow[from=1-5, to=1-6]
	\end{tikzcd}\] by stretching, too. In particluar, the stretching gives us a pointed homotopy. In addition it also follows that $|\rho|$ is a weak equivalence. Analogously, the diagram 
	\[\begin{tikzcd}
		{|(\tilde{\mathcal{X}},\infty)|} &&& {|\tilde{\mathcal{X}}/A|} \\
		\\
		{| S^1\wedge\mathbb{G}_{m}|} & {|\mathcal{S}|} && {|\hat{\mathcal{X}}|}
		\arrow["\sim"', from=1-1, to=3-1]
		\arrow["{|\tilde{\theta}_\infty|}", from=1-1, to=1-4]
		\arrow["\phi"', from=3-1, to=3-2]
		\arrow["{|\rho|}"', from=3-2, to=3-4]
		\arrow["{\text{canonical morphism}}", from=1-4, to=3-4]
		\arrow["\sim"', from=1-4, to=3-4]
		\arrow["\sim"', from=1-1, to=1-4]
	\end{tikzcd}\]
	commutes in the pointed homotopy category $\mathcal{H}o_{\Delta}(\mathbb{Z})$ by the same arguments. Since the derived geometric realization functor is an equivalence of categories, there is a unique isomorphism $\epsilon: S^1\wedge\mathbb{G}_{m}\rightarrow \hat{\mathcal{X}}$ such that $|\epsilon|$ is $|\rho|\circ \phi$. Then the diagram 
	\[\begin{tikzcd}
		{(\tilde{\mathcal{X}},0)\vee(\tilde{\mathcal{X}},\infty)} && {\tilde{\mathcal{X}}/A\vee (\tilde{\mathcal{X}},\infty)} \\
		& {\hat{\mathcal{X}}\vee(\tilde{\mathcal{X}},\infty)} \\
		{( S^1\wedge\mathbb{G}_{m})\vee ( S^1\wedge\mathbb{G}_{m})} && {(\tilde{\mathcal{X}},\infty)\vee(\tilde{\mathcal{X}},\infty)}
		\arrow["\sim"', from=1-1, to=3-1]
		\arrow["{{{\tilde{\theta}_\infty\vee\mathrm{id}}}}", from=3-3, to=1-3]
		\arrow["{{{\tilde{\theta}_0\vee\mathrm{id}}}}", from=1-1, to=1-3]
		\arrow["\sim", from=1-3, to=2-2]
		\arrow["{\epsilon^{-1}\vee\mathrm{id}}"', from=2-2, to=3-1]
		\arrow["\sim"', from=3-3, to=3-1]
	\end{tikzcd}\]
	commutes in $\mathcal{H}o_{\bullet}(\mathbb{Z})$. Now we see that for 
	\[\begin{tikzcd}
		{(\mathbb{P}^1_{\mathbb{Z}},0)\vee(\mathbb{P}^1_{\mathbb{Z}},\infty)} && {\mathbb{P}^1_{\mathbb{Z}}/\mathrm{D}_+(T_0+T_1)\vee (\mathbb{P}^1_{\mathbb{Z}},\infty)} \\
		\\
		{(\tilde{\mathcal{X}},0)\vee(\tilde{\mathcal{X}},\infty)} & {\tilde{\mathcal{X}}/A\vee (\tilde{\mathcal{X}},\infty)} \\
		\\
		{( S^1\wedge\mathbb{G}_{m})\vee ( S^1\wedge\mathbb{G}_{m})} & {(\tilde{\mathcal{X}},\infty)\vee(\tilde{\mathcal{X}},\infty)} & {(\mathbb{P}^1_{\mathbb{Z}},\infty)\vee(\mathbb{P}^1_{\mathbb{Z}},\infty)}
		\arrow["\sim", from=3-1, to=1-1]
		\arrow["\sim"', from=3-1, to=5-1]
		\arrow["\sim", from=3-1, to=3-2]
		\arrow["{{{{{{{{{\theta_{0}\vee \mathrm{id}}}}}}}}}}", from=1-1, to=1-3]
		\arrow["\sim"', from=3-2, to=1-3]
		\arrow["\sim"', from=1-1, to=1-3]
		\arrow["{{{\theta_{\infty}\vee \mathrm{id}}}}", from=5-3, to=1-3]
		\arrow["\sim", from=5-2, to=5-3]
		\arrow["\sim"', from=5-2, to=5-1]
		\arrow["\sim"', from=5-3, to=1-3]
		\arrow["\sim"', from=5-2, to=3-2]
		\arrow["{{{\tilde{\theta}_0\vee\mathrm{id}}}}"', from=3-1, to=3-2]
		\arrow["{{{\tilde{\theta}_\infty\vee\mathrm{id}}}}", from=5-2, to=3-2]
		\arrow["{{\nu_0\vee\nu_\infty}}"', from=3-1, to=1-1]
		\arrow["{{\nu_\infty\vee\nu_\infty}}"', from=5-2, to=5-3]
	\end{tikzcd}\]
	all three inner diagrams commute. Since all involved morphisms are isomorphisms in $\mathcal{H}o_{\bullet}(\mathbb{Z})$, we can conclude that the outer diagram is commutative.\\[1cm]
	\end{proof}
	
	\subsection{Change of base points }
	\
	
	There is a unique automorphism $\Phi$ of $\mathbb{P}^1_{\mathbb{Z}}$ which interchanges the base points $1$ and $\infty $ and sends $0$ to itself. It is induced by the ring isomorphism\begin{align*}
		\mathbb{Z}[T_0,T_1]\rightarrow \mathbb{Z}[T_0,T_1];\  T_0\mapsto T_0, T_1\mapsto T_0-T_1\ \ .
	\end{align*}
	Note that we have $\Phi\circ \Phi=\mathrm{id} $.\\
	
	Recall that we equip $ (\mathbb{P}^1_{\mathbb{Z}},\infty)$ with a cogroup structure via the following zig-zag of pointed weak equivalences
	\[\begin{tikzcd}
		{(\mathbb{P}^1_{\mathbb{Z}},\infty)} & {(\tilde{\mathcal{X}},\infty)} & {S^1\wedge\mathbb{G}_{m}} & \cdot
		\arrow["\sim", from=1-2, to=1-3]
		\arrow["\sim"', from=1-2, to=1-1]
	\end{tikzcd}\]
	Now we would like to construct a similar zig-zag of pointed weak equivalences for $ (\mathbb{P}^1_{\mathbb{Z}},1)$. It follows from the definition of $\Phi$ that we have $\Phi(\mathrm{D}_+(T_0))= \mathrm{D}_+(T_0)$, $\Phi(\mathrm{D}_+(T_1))= \mathrm{D}_+(T_0-T_1) $ and $\Phi(\mathrm{D}_+(T_0T_1))= \mathrm{D}_+((T_0-T_1)T_0) $. The morphism $\infty: \mathrm{Spec}\mathbb{Z}\rightarrow \mathrm{D}_+(T_0)$ factors through $\mathrm{D}_+((T_0-T_1)T_0)$: \[\begin{tikzcd}
		{\mathrm{Spec}\mathbb{Z}} && {\mathrm{D}_+(T_0)} \\
		& {\mathrm{D}_+((T_0-T_1)T_0)}
		\arrow["\infty", from=1-1, to=1-3]
		\arrow[from=1-1, to=2-2]
		\arrow[hook, from=2-2, to=1-3]
	\end{tikzcd}\]
	Hence we also denote the morphism ${\mathrm{Spec}\mathbb{Z}}\rightarrow \mathrm{D}_+((T_0-T_1)T_0)$ by $\infty$.\\
	
	We consider now the pushout diagram 
	\[\begin{tikzcd}
		{0\times\mathrm{D}_+((T_0-T_1)T_0)\sqcup 1\times \mathrm{D}_+((T_0-T_1)T_0)} & {I\times\mathrm{D}_+((T_0-T_1)T_0)} \\
		{0\times\mathrm{D}_+(T_0)\sqcup 1\times\mathrm{D}_+(T_0-T_1)} && \cdot
		\arrow[from=1-1, to=1-2]
		\arrow["{{{\text{inclusion}\sqcup\text{inclusion}}}}"', hook, from=1-1, to=2-1]
	\end{tikzcd}\]
	
	We denote the pushout of the diagram by $\mathcal{Y}$. We can also equip $\mathcal{Y}$ with the base point 
	\[\begin{tikzcd}
		{\mathrm{Spec}\mathbb{Z}} & {\mathrm{D}_+(T_0T_1)} & {0\times\mathrm{D}_+(T_0)}
		\arrow["1", from=1-1, to=1-2]
		\arrow[hook, from=1-2, to=1-3]
	\end{tikzcd}\]
	which we also denote by 1.\\
	
	In the next step we can consider the comparison maps
	between pushout diagrams
	\[\begin{tikzcd}
		{0\times\mathrm{D}_+((T_0-T_1)T_0)\sqcup 1\times\mathrm{D}_+((T_0-T_1)T_0)} & {I\times\mathrm{D}_+((T_0-T_1)T_0)} \\
		{0\times \mathrm{D}_+(T_0)\sqcup 1\times\mathrm{D}_+(T_0-T_1)} \\
		{} \\
		{} \\
		{\mathrm{D}_+((T_0-T_1)T_0)\sqcup \mathrm{D}_+((T_0-T_1)T_0)} & {\mathrm{D}_+((T_0-T_1)T_0)} \\
		{\mathrm{D}_+(T_0)\sqcup\mathrm{D}_+(T_0-T_1)}
		\arrow[from=1-1, to=1-2]
		\arrow["{{{{{\text{inclusion}\sqcup \text{inclusion}}}}}}"', hook, from=1-1, to=2-1]
		\arrow[from=3-1, to=4-1]
		\arrow[from=5-1, to=5-2]
		\arrow["{{{{{\text{inclusion}\sqcup \text{inclusion}}}}}}"', hook, from=5-1, to=6-1]
	\end{tikzcd}\]
	which induces a pointed weak equivalence $(\mathcal{Y},1)\rightarrow (\mathbb{P}^1_{\mathbb{Z}},1)$. Since the spaces $\mathrm{D}_+(T_0)$, $\mathrm{D}_+(T_0-T_1) $ and $\mathrm{D}_+((T_0-T_1)T_0)$ all contain the base point $\infty$, there is an inclusion $I\times\{\infty\}\hookrightarrow I\times \mathrm{D}_+((T_0-T_1)T_0)\hookrightarrow \mathcal{Y}$. The induced projection $\mathcal{Y}\rightarrow \mathcal{Y}/I\times\{\infty\} $ is a weak equivalence. By collapsing $\mathrm{D}_+(T_0)$ and $\mathrm{D}_+(T_0-T_1) $ to a point we get furthermore a weak equivalence\begin{align*}
		\mathcal{Y}/I\times\{\infty\}\rightarrow S^1\wedge\mathrm{D}_+((T_0-T_1)T_0)
	\end{align*}
	such that the composition  $(\mathcal{Y},1)\rightarrow\mathcal{Y}/I\times\{\infty\}\rightarrow S^1\wedge\mathrm{D}_+((T_0-T_1)T_0) $ is a pointed weak equivalence. \\
	
	Now the automorphism $\Phi$ induces a comparison map 
	\[\begin{tikzcd}
		{0\times\mathbb{G}_{m}\sqcup 1\times\mathbb{G}_{m}} & {I\times\mathbb{G}_{m}} \\
		{0\times\mathbb{A}^1_{\mathbb{Z}}\sqcup 1\times \mathbb{A}^1_{\mathbb{Z}}} \\
		{} \\
		{} \\
		{0\times\mathrm{D}_+((T_0-T_1)T_0)\sqcup 1\times\mathrm{D}_+((T_0-T_1)T_0)} & {I\times\mathrm{D}_+((T_0-T_1)T_0)} \\
		{0\times \mathrm{D}_+(T_0)\sqcup 1\times\mathrm{D}_+(T_0-T_1)}
		\arrow[from=5-1, to=5-2]
		\arrow["{{{{\text{inclusion}\sqcup \text{inclusion}}}}}"', hook, from=5-1, to=6-1]
		\arrow[from=1-1, to=1-2]
		\arrow[from=1-1, to=2-1]
		\arrow[from=3-1, to=4-1]
	\end{tikzcd}\]
	which in turn induces a pointed weak equivalence $\tilde{\Phi}: (\tilde{\mathcal{X}},\infty)\rightarrow (\mathcal{Y},1)$. Similarly, the automorphism $\Phi$ also induces a comparison map 
	\[\begin{tikzcd}
		{\mathbb{G}_{m}\sqcup\mathbb{G}_{m}} & {\mathbb{G}_{m}} \\
		{\mathbb{A}^1_{\mathbb{Z}}\sqcup \mathbb{A}^1_{\mathbb{Z}}} \\
		{} \\
		{} \\
		{\mathrm{D}_+((T_0-T_1)T_0)\sqcup \mathrm{D}_+((T_0-T_1)T_0)} & {\mathrm{D}_+((T_0-T_1)T_0)} \\
		{\mathrm{D}_+(T_0)\sqcup \mathrm{D}_+(T_0-T_1)}
		\arrow[from=5-1, to=5-2]
		\arrow["{{{{\text{inclusion}\sqcup \text{inclusion}}}}}"', hook, from=5-1, to=6-1]
		\arrow[from=1-1, to=1-2]
		\arrow[from=1-1, to=2-1]
		\arrow[from=3-1, to=4-1]
	\end{tikzcd}\]
	The induced morphism between the pushouts is just $\Phi$. Altogether we  obtain the commutative diagram 
	\[\begin{tikzcd}
		{S^1\wedge\mathbb{G}_{m}} & {(\tilde{\mathcal{X}},\infty)} & {(\mathbb{P}^1_{\mathbb{Z}},\infty)} \\
		{S^1\wedge\mathrm{D}_+((T_0-T_1)T_0)} & {(\mathcal{Y},1)} & {(\mathbb{P}^1_{\mathbb{Z}},1)} & \cdot
		\arrow["\sim", from=1-2, to=1-3]
		\arrow["{\tilde{\Phi}}"', from=1-2, to=2-2]
		\arrow["\Phi", from=1-3, to=2-3]
		\arrow["\sim", from=2-2, to=2-3]
		\arrow["{\mathrm{id}\wedge\Phi|_{\mathbb{G}_{m}}}"', from=1-1, to=2-1]
		\arrow["\sim"', from=2-2, to=2-1]
		\arrow["\sim"', from=1-2, to=1-1]
	\end{tikzcd}\]
	\
	
	In addition we can also consider the following comparison map which is induced by inclusions 
	\[\begin{tikzcd}
		{0\times\mathrm{D}_+((T_0-T_1)T_1T_0)\sqcup 1\times\mathrm{D}_+((T_0-T_1)T_1T_0)} & {I\times\mathrm{D}_+((T_0-T_1)T_1T_0)} \\
		{0\times\mathrm{D}_+(T_0T_1)\sqcup 1\times \mathrm{D}_+((T_0-T_1)T_1)} \\
		{} \\
		{} \\
		{0\times\mathrm{D}_+((T_0-T_1)T_0)\sqcup 1\times\mathrm{D}_+((T_0-T_1)T_0)} & {I\times\mathrm{D}_+((T_0-T_1)T_0)} \\
		{0\times \mathrm{D}_+(T_0)\sqcup 1\times\mathrm{D}_+(T_0-T_1)}
		\arrow[from=5-1, to=5-2]
		\arrow["{{{{{\text{inclusion}\sqcup \text{inclusion}}}}}}"', hook, from=5-1, to=6-1]
		\arrow[from=1-1, to=1-2]
		\arrow[from=1-1, to=2-1]
		\arrow[from=3-1, to=4-1]
	\end{tikzcd}\]where we denote the pushout of the first diagram by $B$ and $B$ is canonically weakly equivalent to $\mathrm{D}_+(T_1) $. We equip the cofiber \hypertarget{yb}{$\mathcal{Y}/B$} with the canonical base point. In particular the projection $(\mathcal{Y},1)\rightarrow \mathcal{Y}/B$ is a pointed weak equivalence. Moreover there is a canonical a weak equivalence $ \mathcal{Y}/B \rightarrow \mathbb{P}^1_{\mathbb{Z}}/\mathrm{D}_+(T_1)$ of pointed spaces such that the diagram 
	\[\begin{tikzcd}
		{(\mathcal{Y},1)} & {(\mathbb{P}^1_{\mathbb{Z}},1)} \\
		{\mathcal{Y}/B} & {\mathbb{P}^1_{\mathbb{Z}}/\mathrm{D}_+(T_1)}
		\arrow["\sim", from=1-1, to=1-2]
		\arrow["\sim"', from=1-1, to=2-1]
		\arrow["\sim"', from=2-1, to=2-2]
		\arrow["\sim", from=1-2, to=2-2]
	\end{tikzcd}\]
	commutes. We denote the composition \[\begin{tikzcd}
		{(\mathbb{P}^1_{\mathbb{Z}},\infty)} & {(\mathbb{P}^1_{\mathbb{Z}},1)} & {\mathbb{P}^1_{\mathbb{Z}}/\mathrm{D}_+(T_1)}
		\arrow[from=1-2, to=1-3]
		\arrow["\Phi", from=1-1, to=1-2]
	\end{tikzcd}\]by $\bar{\Phi}$. Altogether we get the \hypertarget{dia3.3}{commutative diagram}
	\[\begin{tikzcd}
		{S^1\wedge\mathbb{G}_{m}} & {(\tilde{\mathcal{X}},\infty)} & {(\mathbb{P}^1_{\mathbb{Z}},\infty)} \\
		{S^1\wedge\mathrm{D}_+((T_0-T_1)T_0)} & {(\mathcal{Y},1)} & {\mathbb{P}^1_{\mathbb{Z}}/\mathrm{D}_+(T_1)} & \cdot
		\arrow["\sim", from=1-2, to=1-3]
		\arrow["{{\tilde{\Phi}}}"', from=1-2, to=2-2]
		\arrow["{\bar{\Phi}}", from=1-3, to=2-3]
		\arrow["\sim", from=2-2, to=2-3]
		\arrow["{{\mathrm{id}\wedge\Phi|_{\mathbb{G}_{m}}}}"', from=1-1, to=2-1]
		\arrow["\sim"', from=2-2, to=2-1]
		\arrow["\sim"', from=1-2, to=1-1]
	\end{tikzcd}\] \\[1cm]
	
	\subsection{An unstable Hopf relation}
	\
	
	In this section we again work over the base $\mathrm{Spec}\mathbb{Z}$. In chapter 1 on page 4 we introduced motivic spheres. We recall here the definition. Let $s, w\geq 0$ be integers. We define $S^{s+(w)}$ to be the pointed simplicial presheaf $S^{s}\wedge \mathbb{G}_{m}^w $ where $S^{s}$ is the smash product $\underbrace{S^1\wedge ... \wedge S^1}_{s \  times}$ of the simplicial circle $S^1= \Delta^1 /\partial\Delta^1$ and $\mathbb{G}_{m} $ is based at 1. We call $s$ the degree and $w$ the weight of $S^{s+(w)}$. Suspension from the right with $\mathbb{G}_{m} $ increases the weight $(w)$ by 1. Suspension from the right by the simplicial circle $S^1$ increases the degree $s$ by 1. Let $\mathcal{E}$ be an arbitrary pointed motivic space in $ \mathrm{sPre}(\mathbb{Z})_{\ast}$. Then we set $\pi_{s+(w)}\mathcal{E}$ to be the group $\mathcal{H}o_{\bullet}(\mathbb{Z})(S^{s+(w)}, \mathcal{E}) $ for $s>0$ and $w\geq 0$.\\
	
	Next we recall the definition of the Hopf map $\eta: \mathbb{A}^2_{\mathbb{Z}}-\{0\}\rightarrow \mathbb{P}^1_{\mathbb{Z}}$. It is the canonical map $(T_0, T_1)\mapsto [T_0: T_1]$. The reduced join $\mathbb{G}_{m}\ast \mathbb{G}_{m}$ is defined to be the quotient of $\mathbb{G}_{m}\times \mathbb{G}_{m}\times \Delta^1$ by the relations $(x,y,0)=(x,y',0), (x,y,1)=(x',y,1)$ and $(1,1,t)=(1,1,s)$ for any $t,s\in \Delta^1 $. Note that $\mathbb{G}_{m}$ is a sheaf of abelian groups. In particular we can consider the pointed map \begin{align*}
		\mu'_{\mathbb{G}_{m}}: \mathbb{G}_{m}\times \mathbb{G}_{m}\rightarrow \mathbb{G}_{m},\ (g,h)\mapsto g^{-1}h \ \ \ \cdot
	\end{align*}
	This morphism induces a pointed morphism \begin{align*}
		\eta_{\mathbb{G}_{m}}: \mathbb{G}_{m}\ast \mathbb{G}_{m}\rightarrow S^1\wedge \mathbb{G}_{m}
	\end{align*}
	which is called the algebraic Hopf map. Note that $\mathbb{A}^2_{\mathbb{Z}}-\{0\}$ is a \hypertarget{caniso}{homotopy pushout} of the diagram 
	\[\begin{tikzcd}
		{\mathbb{G}_{m}} & {\mathbb{G}_{m}\times \mathbb{G}_{m}} & {\mathbb{G}_{m}} & \cdot
		\arrow["{pr_1}", from=1-2, to=1-3]
		\arrow["{pr_2}"', from=1-2, to=1-1]
	\end{tikzcd}\]
	Additionally, $\mathbb{G}_{m}\ast \mathbb{G}_{m}$ is also a homotopy pushout of this diagram. In particular, we can show that the Hopf map $\eta$ is canonically $\mathbb{A}^1$-weakly equivalent to $\eta_{\mathbb{G}_{m}}$. The canonical projection from $\mathbb{G}_{m}\ast \mathbb{G}_{m} $ to $S^1\wedge\mathbb{G}_{m}\wedge \mathbb{G}_{m} $ is a motivic weak equivalence. We also call the composition in $\mathcal{H}o_{\bullet}(\mathbb{Z})$
	\[\begin{tikzcd}
		{S^1\wedge\mathbb{G}_{m}\wedge \mathbb{G}_{m}} & {\mathbb{G}_{m}\ast \mathbb{G}_{m}} & {S^1\wedge \mathbb{G}_{m}}
		\arrow["{\eta_{\mathbb{G}_{m}}}", from=1-2, to=1-3]
		\arrow[from=1-1, to=1-2]
	\end{tikzcd}\]
	the Hopf map.\\
	
	\begin{lem}\label{Lemma 3.4.1} Let $w\geq0$ be a natural number. The group $\pi_{1+(w)}S^{1+(2)}$ is commutative.\\
		\end{lem}
	
\begin{proof} The motivic sphere $S^{1+(2)}$ is $\mathbb{A}^1$-weakly equivalent to $\mathbb{A}^2_{\mathbb{Z}}-\{0\}$. Let $\mathrm{SL}_2$ be the special linear group scheme $\mathrm{Spec}\mathbb{Z}[T_{11},T_{12},T_{21},T_{22}]/(\mathrm{det}-1)$. The projection onto the last column $\mathrm{SL}_2\rightarrow \mathbb{A}^2_{\mathbb{Z}}-\{0\} $ is an $\mathbb{A}^1$-weak equivalence \cite[Example 2.12(3)]{motihopf}. Therefore we can equip $S^{1+(2)}$ with a group structure in $\mathcal{H}o_{\bullet}(\mathbb{Z})$. Using the Eckmann-Hilton argument we can show that $\pi_{1+(w)}S^{1+(2)}$ is commutative.\\
	\end{proof}
	
Morphisms from motivic spheres to motivic spheres are indexed by the bidegree of the target. For example we have a morphism $\phi_{s_2+(w_2)}: S^{s_1+(w_1)}\rightarrow S^{s_2+(w_2)} $. Then suspension from the right yields suspended morphisms $\phi_{s_2+s+(w_2+w)}:S^{s_1+s+(w_1+w)}\rightarrow S^{s_2+s+(w_2+w)} $ for $s>0$ and $w>0$. The Hopf map might be denoted by $\eta_{1+(1)}$. Suspension from the right yields suspended Hopf maps $\eta_{s+(w)}$
	for all $s>0$ and $w>0$. Let $n$ be an arbitrary integer. We define the power map\begin{align*}
		P_n: \mathbb{G}_{m}\rightarrow \mathbb{G}_{m}, \ x\mapsto x^n \ \ \ \cdot
	\end{align*}
	For $ n=-1$ we set $\epsilon:= P_{-1}$. For $ n=2$ we set $q_{(1)}:= P_{2}$. Furthermore we define the hyperbolic plane $h_{1+(1)}$ to be $1_{1+(1)}-\epsilon_{1+(1)}$ where $1_{1+(1)}$ is just the identity morphism for $S^{1+(1)}$. We would like to study the relation between $q_{1+(1)} $ and $1_{1+(1)}-\epsilon_{1+(1)}$.\\
	
	Via the zig-zag of pointed weak equivalences 
	\[\begin{tikzcd}
		{(\mathbb{P}^1_{\mathbb{Z}},\infty)} & {(\tilde{\mathcal{X}},\infty)} & {S^1\wedge\mathbb{G}_{m}}
		\arrow["\sim", from=1-2, to=1-3]
		\arrow["\sim"', from=1-2, to=1-1]
	\end{tikzcd}\]
	the morphism $q_{1+(1)}$ corresponds to the pointed endomorphism \begin{align*}
		(\mathbb{P}^1_{\mathbb{Z}},\infty)\rightarrow(\mathbb{P}^1_{\mathbb{Z}},\infty), [T_0:T_1]\rightarrow[T_0^2: T_1^2]\ \ \ \cdot
	\end{align*}
	\
	
	\begin{prop}
		\label{Proposition 3.4.2}
		Let $\tau: \mathbb{P}^1_{\mathbb{Z}}\rightarrow\mathbb{P}^1_{\mathbb{Z}}$ be the automorphism induced by $[T_0:T_1]\mapsto [T_1:T_0]$. Then under the zig-zag of pointed weak equivalences
	\[\begin{tikzcd}
		{(\mathbb{P}^1_{\mathbb{Z}},\infty)} & {(\tilde{\mathcal{X}},\infty)} & {S^1\wedge\mathbb{G}_{m}}
		\arrow["\sim", from=1-2, to=1-3]
		\arrow["\sim"', from=1-2, to=1-1]
	\end{tikzcd}\]
	the morphism $-\epsilon_{1+(1)}$ corresponds to the pointed automorphism $\Phi\circ\tau\circ\Phi$ of $(\mathbb{P}^1_{\mathbb{Z}},\infty) $.\\
	\end{prop}
	
\begin{proof} In section 3.3 we have the \hyperlink{dia3.3}{commutative diagram}
	\[\begin{tikzcd}
		{S^1\wedge\mathbb{G}_{m}} & {(\tilde{\mathcal{X}},\infty)} & {(\mathbb{P}^1_{\mathbb{Z}},\infty)} \\
		{S^1\wedge\mathrm{D}_+((T_0-T_1)T_0)} & {(\mathcal{Y},1)} & {\mathbb{P}^1_{\mathbb{Z}}/\mathrm{D}_+(T_1)} & \cdot
		\arrow["\sim", from=1-2, to=1-3]
		\arrow["{{\tilde{\Phi}}}"', from=1-2, to=2-2]
		\arrow["{\bar{\Phi}}", from=1-3, to=2-3]
		\arrow["\sim", from=2-2, to=2-3]
		\arrow["{(\Phi|_{\mathbb{G}_{m}})_{1+(1)}}"', from=1-1, to=2-1]
		\arrow["\sim"', from=2-2, to=2-1]
		\arrow["\sim"', from=1-2, to=1-1]
	\end{tikzcd}\]
	Therefore via the isomorphism $(\Phi|_{\mathbb{G}_{m}})_{1+(1)} $ the morphism $\epsilon_{1+(1)}$ corresponds to $(\Phi|_{\mathbb{G}_{m}}\circ\epsilon\circ (\Phi|_{\mathbb{G}_{m}})^{-1})_{1+(1)} $. The morphism $\Phi|_{\mathbb{G}_{m}}\circ\epsilon\circ (\Phi|_{\mathbb{G}_{m}})^{-1}$ is induced by the ring homomorphism\begin{align*}
		\mathbb{Z}[T_0,T_1]_{((T_0-T_1)T_0)}\rightarrow \mathbb{Z}[T_0,T_1]_{((T_0-T_1)T_0)}
	\end{align*} which interchanges $\dfrac{(T_0-T_1)^2}{(T_0-T_1)T_0}$ with $ \dfrac{T_0^2}{(T_0-T_1)T_0}$.\\
	
	On the other hand the automorphism $\Phi\circ\tau\circ\Phi$ of $\mathbb{P}^1_{\mathbb{Z}}$ is induced by the ring isomorphism\begin{align*}
		\mathbb{Z}[T_0,T_1]\rightarrow \mathbb{Z}[T_0,T_1]; T_0\mapsto T_0-T_1,\ T_1\mapsto -T_1 \ .
	\end{align*}
	The restriction $\Phi\circ\tau\circ\Phi|_{\mathrm{D}_+((T_0-T_1)T_0)}$ is $\Phi|_{\mathbb{G}_{m}}\circ\epsilon\circ (\Phi|_{\mathbb{G}_{m}})^{-1}$. Moreover the automorphism $\Phi\circ\tau\circ\Phi$ interchanges $\mathrm{D}_+(T_0) $ with $\mathrm{D}_+(T_0-T_1)$ and sends $\mathrm{D}_+(T_1) $ to itself. Hence $\Phi\circ\tau\circ\Phi $ also induces a morphism $\overline{\Phi\circ\tau\circ\Phi}:\mathbb{P}^1_{\mathbb{Z}}/\mathrm{D}_+(T_1) \rightarrow \mathbb{P}^1_{\mathbb{Z}}/\mathrm{D}_+(T_1)$. Now recall that in section 3.3 we also construcuted the pointed space \hyperlink{yb}{$\mathcal{Y}/B$}. Analogously, we can consider the comparison map between pushout diagrams induced by inclusions
	\[\begin{tikzcd}
		{0\times\mathrm{D}_+((T_0-T_1)T_1T_0)\sqcup 1\times\mathrm{D}_+((T_0-T_1)T_1T_0)} & {I\times\mathrm{D}_+((T_0-T_1)T_1T_0)} \\
		{0\times\mathrm{D}_+((T_0-T_1)T_1)\sqcup 1\times \mathrm{D}_+(T_0T_1)} \\
		{} \\
		{} \\
		{0\times\mathrm{D}_+((T_0-T_1)T_0)\sqcup 1\times\mathrm{D}_+((T_0-T_1)T_0)} & {I\times\mathrm{D}_+((T_0-T_1)T_0)} \\
		{0\times \mathrm{D}_+(T_0-T_1)\sqcup 1\times\mathrm{D}_+(T_0)}
		\arrow[from=5-1, to=5-2]
		\arrow["{{{{{{\text{inclusion}\sqcup \text{inclusion}}}}}}}"', hook, from=5-1, to=6-1]
		\arrow[from=1-1, to=1-2]
		\arrow[from=1-1, to=2-1]
		\arrow[from=3-1, to=4-1]
	\end{tikzcd}\]where we denote the pushout of the first diagram by $B'$ and the second by $\mathcal{Y}'$. Then we take the cofiber $\mathcal{Y}'/B'$. There is again a canonical weak equivalence $ \mathcal{Y}'/B'\rightarrow \mathbb{P}^1_{\mathbb{Z}}/\mathrm{D}_+(T_1) $. Since the automorphism $\Phi\circ\tau\circ\Phi$ interchanges $\mathrm{D}_+(T_0) $ with $\mathrm{D}_+(T_0-T_1)$ and keeps $\mathrm{D}_+(T_1) $ invariant, it induces an isomorphism $\mathcal{Y}/B\cong \mathcal{Y}'/B' $. We have then the commutative diagram 
	\[\begin{tikzcd}
		{\mathcal{Y}/B} & {\mathbb{P}^1_{\mathbb{Z}}/\mathrm{D}_+(T_1)} \\
		{\mathcal{Y}'/B' } & {\mathbb{P}^1_{\mathbb{Z}}/\mathrm{D}_+(T_1)} && \cdot
		\arrow["{\overline{\Phi\circ\tau\circ\Phi}}", from=1-2, to=2-2]
		\arrow["\sim", from=1-1, to=1-2]
		\arrow["\sim"', from=2-1, to=2-2]
		\arrow["\cong"', from=1-1, to=2-1]
	\end{tikzcd}\]
	\
	
	Next we apply the geometric realization functor. There is an isomorphism $s$ from $|\mathcal{Y}'/B'| $ to $|\mathcal{Y}/B| $ which is induced by the swap morphism $|I|\rightarrow |I|$. In particular the following diagram 
	\[\begin{tikzcd}
		{|\mathcal{Y}/B|} & {|\mathbb{P}^1_{\mathbb{Z}}/\mathrm{D}_+(T_1)|} \\
		{|\mathcal{Y}'/B'|} & {|\mathbb{P}^1_{\mathbb{Z}}/\mathrm{D}_+(T_1)|} \\
		{|\mathcal{Y}/B|}
		\arrow["{|\overline{\Phi\circ\tau\circ\Phi}|}", from=1-2, to=2-2]
		\arrow["\sim", from=1-1, to=1-2]
		\arrow["\sim"', from=2-1, to=2-2]
		\arrow["\cong"', from=1-1, to=2-1]
		\arrow["s"', from=2-1, to=3-1]
		\arrow["\sim"', from=3-1, to=2-2]
	\end{tikzcd}\]
	is commutative. Moreover the swap morphism induces also an isomorphism $s': |S^1\wedge\mathrm{D}_+((T_0-T_1)T_0)|\rightarrow |S^1\wedge\mathrm{D}_+((T_0-T_1)T_0)| $ which is the inverse morphism for the cogroup object $|S^1\wedge\mathrm{D}_+((T_0-T_1)T_0)| $. We can equip $\mathcal{Y}'$ with the base point 
	\[\begin{tikzcd}
		{\mathrm{Spec}\mathbb{Z}} & {0\times\mathrm{D}_+(T_0-T_1)} & {\mathcal{Y}'} & \cdot
		\arrow["0", from=1-1, to=1-2]
		\arrow[hook, from=1-2, to=1-3]
	\end{tikzcd}\]
	Then it follows from the construction of 
	$\mathcal{Y}'$ that there is a pointed weak equivalence $(\mathcal{Y}',0)\rightarrow S^1\wedge\mathrm{D}_+((T_0-T_1)T_0)$. At the end we have now the diagram 
	\[\begin{tikzcd}
		{|\mathcal{Y}/B|} & {|(\mathcal{Y},1)|} & {|S^1\wedge\mathrm{D}_+((T_0-T_1)T_0)|} \\
		{|\mathcal{Y}'/B'|} & {|(\mathcal{Y}',0)|} & {|S^1\wedge\mathrm{D}_+((T_0-T_1)T_0)|} \\
		{|\mathcal{Y}/B|} & {|(\mathcal{Y},1)|} & {|S^1\wedge\mathrm{D}_+((T_0-T_1)T_0)|}
		\arrow[from=1-2, to=1-3]
		\arrow[from=1-2, to=1-1]
		\arrow["{|(\Phi|_{\mathbb{G}_{m}}\circ\epsilon\circ (\Phi|_{\mathbb{G}_{m}})^{-1})_{1+(1)}|}", from=1-3, to=2-3]
		\arrow[from=1-1, to=2-1]
		\arrow["s"', from=2-1, to=3-1]
		\arrow["{s'}", from=2-3, to=3-3]
		\arrow[from=3-2, to=3-3]
		\arrow[from=3-2, to=3-1]
		\arrow[from=2-2, to=2-1]
		\arrow[from=2-2, to=2-3]
		\arrow[from=1-2, to=2-2]
	\end{tikzcd}\]
	where $|\mathcal{Y}/B|\rightarrow |\mathcal{Y}'/B'|$ and $|(\mathcal{Y},1)|\rightarrow |(\mathcal{Y}',0)|$ are induced by $\Phi\circ\tau\circ\Phi$. The first part of the previous diagram 
	\[\begin{tikzcd}
		{|\mathcal{Y}/B|} & {|(\mathcal{Y},1)|} & {|S^1\wedge\mathrm{D}_+((T_0-T_1)T_0)|} \\
		{|\mathcal{Y}'/B'|} & {|(\mathcal{Y}',0)|} & {|S^1\wedge\mathrm{D}_+((T_0-T_1)T_0)|}
		\arrow["{|(\Phi|_{\mathbb{G}_{m}}\circ\epsilon\circ (\Phi|_{\mathbb{G}_{m}})^{-1})_{1+(1)}|}", from=1-3, to=2-3]
		\arrow[from=1-2, to=1-3]
		\arrow[from=1-2, to=1-1]
		\arrow[from=2-2, to=2-1]
		\arrow[from=2-2, to=2-3]
		\arrow[from=1-1, to=2-1]
		\arrow[from=1-2, to=2-2]
	\end{tikzcd}\]
	is commutative. We would like to show that the second part commutes, too. First, we can also equip $\mathcal{Y}$ with the base point 
	\[\begin{tikzcd}
		{\mathrm{Spec}\mathbb{Z}} & {1\times\mathrm{D}_+(T_0-T_1)} & {\mathcal{Y}} & \cdot
		\arrow["0", from=1-1, to=1-2]
		\arrow[hook, from=1-2, to=1-3]
	\end{tikzcd}\]
	Then the swap morphism induces an isomorphism $s'': |(\mathcal{Y}',0)|\rightarrow |(\mathcal{Y},0)|$. We also have pointed weak equivalences $|(\mathcal{Y},0)|\rightarrow |\mathcal{Y}/B|$ and $|(\mathcal{Y},0)|\rightarrow |S^1\wedge\mathrm{D}_+((T_0-T_1)T_0)|$. Now we have the diagram 
	\[\begin{tikzcd}
		{|\mathcal{Y}'/B'|} & {|(\mathcal{Y}',0)|} & {|S^1\wedge\mathrm{D}_+((T_0-T_1)T_0)|} \\
		& {|(\mathcal{Y},0)|} & {|S^1\wedge\mathrm{D}_+((T_0-T_1)T_0)|} \\
		{|\mathcal{Y}/B|} && {|(\mathcal{Y},1)|} & \cdot
		\arrow[from=1-2, to=1-3]
		\arrow[from=1-2, to=1-1]
		\arrow["s"', from=1-1, to=3-1]
		\arrow["{{s''}}", from=1-2, to=2-2]
		\arrow["\sim"', from=2-2, to=3-1]
		\arrow[from=3-3, to=3-1]
		\arrow[from=3-3, to=2-3]
		\arrow["{{s'}}", from=1-3, to=2-3]
		\arrow["\sim", from=2-2, to=2-3]
	\end{tikzcd}\]
	\
	
	It follows from the definition of the morphisms $s$, $s'$ and $s''$ that the diagrams
	\[\begin{tikzcd}
		{|\mathcal{Y}'/B'|} & {|(\mathcal{Y}',0)|} \\
		\\
		{|\mathcal{Y}/B|} & {|(\mathcal{Y},0)|}
		\arrow[from=1-2, to=1-1]
		\arrow["s"', from=1-1, to=3-1]
		\arrow[from=3-2, to=3-1]
		\arrow["{s''}", from=1-2, to=3-2]
	\end{tikzcd}\]
	and
	\[\begin{tikzcd}
		{|(\mathcal{Y}',0)|} & {|S^1\wedge\mathrm{D}_+((T_0-T_1)T_0)|} \\
		{|(\mathcal{Y},0)|} & {|S^1\wedge\mathrm{D}_+((T_0-T_1)T_0)|}
		\arrow[from=1-1, to=1-2]
		\arrow["{{s''}}", from=1-1, to=2-1]
		\arrow["{{s'}}", from=1-2, to=2-2]
		\arrow[from=2-1, to=2-2]
	\end{tikzcd}\]
	are commutative.\\
	
	In Proposition~\ref{Proposition 3.2.1} we proved that the diagram 
	\[\begin{tikzcd}
		{(\tilde{\mathcal{X}}, 0)} && {S^1\wedge\mathbb{G}_{m}} \\
		{\tilde{\mathcal{X}}/A} && {(\tilde{\mathcal{X}}, \infty)}
		\arrow["\sim"', from=1-1, to=2-1]
		\arrow["\sim"', from=2-3, to=2-1]
		\arrow["\sim", from=1-1, to=1-3]
		\arrow["\sim"', from=2-3, to=1-3]
	\end{tikzcd}\]
	commutes. Now we can use exactly the same methods to show that 
	\[\begin{tikzcd}
		{(\mathcal{Y},0)} && {S^1\wedge\mathrm{D}_+((T_0-T_1)T_0)} \\
		{\mathcal{Y}/B} && {(\mathcal{Y},1)}
		\arrow["\sim", from=2-3, to=2-1]
		\arrow["\sim"', from=2-3, to=1-3]
		\arrow["\sim"', from=1-1, to=2-1]
		\arrow["\sim", from=1-1, to=1-3]
	\end{tikzcd}\]
	is commutative, too. Therefore all three inner diagrams in 
	\[\begin{tikzcd}
		{|\mathcal{Y}'/B'|} & {|(\mathcal{Y}',0)|} & {|S^1\wedge\mathrm{D}_+((T_0-T_1)T_0)|} \\
		& {|(\mathcal{Y},0)|} & {|S^1\wedge\mathrm{D}_+((T_0-T_1)T_0)|} \\
		{|\mathcal{Y}/B|} && {|(\mathcal{Y},1)|}
		\arrow[from=1-2, to=1-3]
		\arrow[from=1-2, to=1-1]
		\arrow["s"', from=1-1, to=3-1]
		\arrow["{s''}", from=1-2, to=2-2]
		\arrow["\sim"', from=2-2, to=3-1]
		\arrow[from=3-3, to=3-1]
		\arrow[from=3-3, to=2-3]
		\arrow["{s'}", from=1-3, to=2-3]
		\arrow["\sim", from=2-2, to=2-3]
	\end{tikzcd}\]
	commutes. Since all involved morphisms are weak equivalences, it follows that also the outer diagram is commutative. In particular we have the following commutative diagram 
	\[\begin{tikzcd}
		{|\mathbb{P}^1_{\mathbb{Z}}/\mathrm{D}_+(T_1)|} & {|\mathcal{Y}/B|} & {|(\mathcal{Y},1)|} & {|S^1\wedge\mathrm{D}_+((T_0-T_1)T_0)|} \\
		{|\mathbb{P}^1_{\mathbb{Z}}/\mathrm{D}_+(T_1)|} & {|\mathcal{Y}'/B'|} && {|S^1\wedge\mathrm{D}_+((T_0-T_1)T_0)|} \\
		& {|\mathcal{Y}/B|} & {|(\mathcal{Y},1)|} & {|S^1\wedge\mathrm{D}_+((T_0-T_1)T_0)|}
		\arrow[from=1-3, to=1-4]
		\arrow[from=1-3, to=1-2]
		\arrow["{{|(\Phi|_{\mathbb{G}_{m}}\circ\epsilon\circ (\Phi|_{\mathbb{G}_{m}})^{-1})_{1+(1)}|}}", from=1-4, to=2-4]
		\arrow[from=1-2, to=2-2]
		\arrow["s"', from=2-2, to=3-2]
		\arrow["{{s'}}", from=2-4, to=3-4]
		\arrow[from=3-3, to=3-4]
		\arrow[from=3-3, to=3-2]
		\arrow["{|\overline{\Phi\circ\tau\circ\Phi}|}"', from=1-1, to=2-1]
		\arrow[from=3-2, to=2-1]
		\arrow[from=1-2, to=1-1]
		\arrow[from=2-2, to=2-1]
	\end{tikzcd}\] which implies that via the zig-zag of weak equivalences 
	\[\begin{tikzcd}
		{\mathbb{P}^1_{\mathbb{Z}}/\mathrm{D}_+(T_1)} & {(\mathcal{Y},1)} & {S^1\wedge\mathrm{D}_+((T_0-T_1)T_0)}
		\arrow[from=1-2, to=1-3]
		\arrow[from=1-2, to=1-1]
	\end{tikzcd}\]
	the morphism $-(\Phi|_{\mathbb{G}_{m}}\circ\epsilon\circ (\Phi|_{\mathbb{G}_{m}})^{-1})_{1+(1)}$ corresponds to $\overline{\Phi\circ\tau\circ\Phi} $.\\
	
	In the next step we consider again the commutative diagram 
	\[\begin{tikzcd}
		{S^1\wedge\mathbb{G}_{m}} & {(\tilde{\mathcal{X}},\infty)} & {(\mathbb{P}^1_{\mathbb{Z}},\infty)} \\
		{S^1\wedge\mathrm{D}_+((T_0-T_1)T_0)} & {(\mathcal{Y},1)} & {\mathbb{P}^1_{\mathbb{Z}}/\mathrm{D}_+(T_1)} & \cdot
		\arrow["\sim", from=1-2, to=1-3]
		\arrow["{{\tilde{\Phi}}}"', from=1-2, to=2-2]
		\arrow["{\bar{\Phi}}", from=1-3, to=2-3]
		\arrow["\sim", from=2-2, to=2-3]
		\arrow["{(\Phi|_{\mathbb{G}_{m}})_{1+(1)}}"', from=1-1, to=2-1]
		\arrow["\sim"', from=2-2, to=2-1]
		\arrow["\sim"', from=1-2, to=1-1]
	\end{tikzcd}\]
	Via the isomorphism $(\Phi|_{\mathbb{G}_{m}})_{1+(1)} $ the morphism $-\epsilon_{1+(1)}$ corresponds to $-(\Phi|_{\mathbb{G}_{m}}\circ\epsilon\circ (\Phi|_{\mathbb{G}_{m}})^{-1})_{1+(1)} $ and under the zig-zag of weak equivalences 
	\[\begin{tikzcd}
		{\mathbb{P}^1_{\mathbb{Z}}/\mathrm{D}_+(T_1)} & {(\mathcal{Y},1)} & {S^1\wedge\mathrm{D}_+((T_0-T_1)T_0)}
		\arrow["\sim", from=1-2, to=1-3]
		\arrow["\sim"', from=1-2, to=1-1]
	\end{tikzcd}\]
	$-(\Phi|_{\mathbb{G}_{m}}\circ\epsilon\circ (\Phi|_{\mathbb{G}_{m}})^{-1})_{1+(1)}$ corresponds to $\overline{\Phi\circ\tau\circ\Phi} $. Therefore we only need to determine which pointed endomorphism of $(\mathbb{P}^1_{\mathbb{Z}},\infty)$ equals to $\bar{\Phi}^{-1}\circ(\overline{\Phi\circ\tau\circ\Phi})\circ \bar{\Phi}$ in $\mathcal{H}o_{\bullet}(\mathbb{Z})$, because this is then the morphism which corresponds to $-\epsilon_{1+(1)}$ under the zig-zag of weak equivalences
	\[\begin{tikzcd}
		{S^1\wedge\mathbb{G}_{m}} & {(\tilde{\mathcal{X}},\infty)} & {(\mathbb{P}^1_{\mathbb{Z}},\infty)} & \cdot
		\arrow["\sim", from=1-2, to=1-3]
		\arrow["\sim"', from=1-2, to=1-1]
	\end{tikzcd}\]\\
	
	We claim that the diagram 
	\[\begin{tikzcd}
		{(\mathbb{P}^1_{\mathbb{Z}},\infty)} && {(\mathbb{P}^1_{\mathbb{Z}},\infty)} \\
		{\mathbb{P}^1_{\mathbb{Z}}/\mathrm{D}_+(T_1)} && {\mathbb{P}^1_{\mathbb{Z}}/\mathrm{D}_+(T_1)}
		\arrow["\Phi\circ\tau\circ\Phi", from=1-1, to=1-3]
		\arrow["{\overline{\Phi\circ\tau\circ\Phi}}"', from=2-1, to=2-3]
		\arrow["{\bar{\Phi}}"', from=1-1, to=2-1]
		\arrow["{\bar{\Phi}}", from=1-3, to=2-3]
	\end{tikzcd}\]
	commutes in $\mathcal{H}o_{\bullet}(\mathbb{Z})$. The composition $\overline{\Phi\circ\tau\circ\Phi}\circ \bar{\Phi} $ is just 
	\[\begin{tikzcd}
		{(\mathbb{P}^1_{\mathbb{Z}},\infty)} & {(\mathbb{P}^1_{\mathbb{Z}},0)} & {(\mathbb{P}^1_{\mathbb{Z}},0)} & {\mathbb{P}^1_{\mathbb{Z}}/\mathrm{D}_+(T_1)}
		\arrow["\tau", from=1-1, to=1-2]
		\arrow["\Phi", from=1-2, to=1-3]
		\arrow[from=1-3, to=1-4]
	\end{tikzcd}\]
	where $(\mathbb{P}^1_{\mathbb{Z}},0)\rightarrow\mathbb{P}^1_{\mathbb{Z}}/\mathrm{D}_+(T_1)$ is the canonical projection. Similarly, $\bar{\Phi}\circ\Phi\circ\tau\circ\Phi $ is 
	\[\begin{tikzcd}
		{(\mathbb{P}^1_{\mathbb{Z}},\infty)} & {(\mathbb{P}^1_{\mathbb{Z}},1)} & {(\mathbb{P}^1_{\mathbb{Z}},1)} & {\mathbb{P}^1_{\mathbb{Z}}/\mathrm{D}_+(T_1)}
		\arrow["\Phi", from=1-1, to=1-2]
		\arrow["\tau", from=1-2, to=1-3]
		\arrow[from=1-3, to=1-4]
	\end{tikzcd}\]
	where $(\mathbb{P}^1_{\mathbb{Z}},0)\rightarrow\mathbb{P}^1_{\mathbb{Z}}/\mathrm{D}_+(T_1)$ is the canonical projection. First we would like to find a sequence of naive $ \mathbb{A}^1$-homotopies $H:\mathbb{P}^1_{\mathbb{Z}}\times_{\mathrm{Spec}\mathbb{Z}}\mathbb{A}^1_{\mathbb{Z}} \rightarrow \mathbb{P}^1_{\mathbb{Z}}$ from $\Phi\circ\tau$ to $\tau\circ\Phi$ such that every composition 
	\[\begin{tikzcd}
		{\mathrm{Spec}\mathbb{Z}\times_{\mathrm{Spec}\mathbb{Z}}\mathbb{A}^1_{\mathbb{Z}}} && {\mathbb{P}^1_{\mathbb{Z}}\times_{\mathrm{Spec}\mathbb{Z}}\mathbb{A}^1_{\mathbb{Z}}} & {\mathbb{P}^1_{\mathbb{Z}}}
		\arrow["{\infty\times\mathrm{id}_{\mathbb{A}^1_{\mathbb{Z}}}}", from=1-1, to=1-3]
		\arrow["H", from=1-3, to=1-4]
	\end{tikzcd}\]
	factors through $\mathrm{D}_+(T_1)$. This condition is equivalent to the condition that $H((T_1,\rho))$ are all contained in $\mathrm{D}_+(T_1)$ where $(T_1,\rho)$ is the homogeneous prime ideal of $\mathbb{Z}[T][T_0,T_1] $ generated by $\rho$ and $T_1$ and $\rho$ runs over the prime ideals of $\mathbb{Z}[T]$. Then such a sequence of naive $ \mathbb{A}^1$-homotopies induces a sequence of pointed naive $ \mathbb{A}^1$-homotopies $\mathrm{Spec}\mathbb{Z}\wedge(\mathbb{A}^1_{\mathbb{Z}})_+\rightarrow\mathbb{P}^1_{\mathbb{Z}}/\mathrm{D}_+(T_1) $ from $\overline{\Phi\circ\tau\circ\Phi}\circ \bar{\Phi}$ to $\bar{\Phi}\circ\Phi\circ\tau\circ\Phi$.\\
	
	The scheme morphism $\tau\circ \Phi$ is induced by the ring isomorphism \begin{align*}
		\mathbb{Z}[T_0, T_1]\rightarrow \mathbb{Z}[T_0, T_1]; \  T_0\mapsto T_0-T_1,\ T_1\mapsto T_0 
	\end{align*} and $\Phi\circ \tau$ is induced by the ring isomorphism \begin{align*}
		\mathbb{Z}[T_0, T_1]\rightarrow \mathbb{Z}[T_0, T_1]; \  T_0\mapsto T_1,\ T_1\mapsto T_1-T_0 \ \ \cdot
	\end{align*}
	We first have a naive $\mathbb{A}^1$-homotopy $H_1:\mathbb{P}^1_{\mathbb{Z}}\times_{\mathrm{Spec}\mathbb{Z}}\mathbb{A}^1_{\mathbb{Z}} \rightarrow \mathbb{P}^1_{\mathbb{Z}} $ induced by the ring homomorphism\begin{align*}
		\mathbb{Z}[T_0, T_1]\rightarrow \mathbb{Z}[T][T_0, T_1];\  T_0\mapsto TT_0-T_1,\ T_1\mapsto T_0 \ \ \cdot
	\end{align*}
	It follows from the definition of $H_1$ that $H_1((T_1,\rho))$ are all contained in $\mathrm{D}_+(T_1)$ where $\rho$ runs over the prime ideals of $\mathbb{Z}[T]$. Hence $\tau\circ \Phi$ is $\mathbb{A}^1$-homotopic to the scheme endomorphism of $\mathbb{P}^1_{\mathbb{Z}}$ defined by $[T_0:T_1]\mapsto [-T_1:T_0]=[T_1:-T_0]$. Next we can give an $\mathbb{A}^1$-homotopy $H_2$ from this morphism to $\Phi\circ \tau$. It is given by \begin{align*}
		\mathbb{Z}[T_0, T_1]\rightarrow \mathbb{Z}[T][T_0, T_1];\  T_0\mapsto T_1,\ T_1\mapsto TT_1-T_0 \ \ \cdot
	\end{align*}
	Again we have $H_2((T_1,\rho))\in \mathrm{D}_+(T_1)$ for all prime ideals $\rho$ of $\mathbb{Z}[T]$. Thus we obtain a sequence of pointed naive $ \mathbb{A}^1$-homotopies $\mathrm{Spec}\mathbb{Z}\wedge(\mathbb{A}^1_{\mathbb{Z}})_+\rightarrow\mathbb{P}^1_{\mathbb{Z}}/\mathrm{D}_+(T_1) $ from $\overline{\Phi\circ\tau\circ\Phi}\circ \bar{\Phi}$ to $\bar{\Phi}\circ\Phi\circ\tau\circ\Phi$.\\
	\end{proof}
	
	\begin{prop}\label{Proposition 3.4.3} The morphism $q_{1+(1)}$ is equal to $1_{1+(1)}-\epsilon_{1+(1)}$ in $\mathcal{H}o_{\bullet}(\mathbb{Z})$.\\
	\end{prop}
	
\begin{proof} Recall that $q_{(1)}$ is the pointed morphism \begin{align*}
		\mathbb{G}_{m}\rightarrow \mathbb{G}_{m}, \ x\mapsto x^2\ \ \ \cdot
	\end{align*}
	It is easy to see that via the zig-zag of pointed weak equivalences
	\[\begin{tikzcd}
		{(\mathbb{P}^1_{\mathbb{Z}},\infty)} & {(\tilde{\mathcal{X}},\infty)} & {S^1\wedge\mathbb{G}_{m}}
		\arrow["\sim", from=1-2, to=1-3]
		\arrow["\sim"', from=1-2, to=1-1]
	\end{tikzcd}\]the morphism $q_{1+(1)}$ corresponds to the pointed endomorphism of $(\mathbb{P}^1_{\mathbb{Z}},\infty)$ which is geven by $[T_0:T_1]\mapsto[T_0^2:T_1^2] $. By Proposition 3.4.2 the morphism $-\epsilon_{1+(1)}$ corresponds to $\Phi\circ\tau\circ\Phi$. Since we equipped $(\mathbb{P}^1_{\mathbb{Z}},\infty)$ with a cogroup structure using the zig-zag above, $1_{1+(1)}-\epsilon_{1+(1)}$ corresponds to $\mathrm{id}_{\mathbb{P}^1_{\mathbb{Z}}}+\Phi\circ\tau\circ\Phi$.\\
	
	By Proposition~\ref{Proposition 3.1.9} the morphism given by $[T_0:T_1]\mapsto[T_0^2:T_1^2] $ is represented by the pair of polynomials $\frac{X^2}{1}$ where $X$ is $\frac{T_0}{T_1}$. Simiarly, $\Phi\circ\tau\circ\Phi$ is represented by $\frac{X-1}{-1}$ and $\mathrm{id}_{\mathbb{P}^1_{\mathbb{Z}}}$ is represented by $\frac{X}{1}$.\\
	
	Cazanave gives the set $ [\mathbb{P}^1_{k}, \mathbb{P}^1_{k}]^{\mathrm{N}}$ of pointed $\mathbb{A}^1$-homotopy classes of scheme morphisms a monoid structure, where $\mathbb{P}^1_{k}$ is equipped with the base point $\infty$. We denote the addition for this monoid structure by $\oplus^{\mathrm{N}}$. Via the same zig-zag of weak equivalences as for $(\mathbb{P}^1_{\mathbb{Z}},\infty)$ we can equip $\mathbb{P}^1_{k}$ with a cogroup structure. Then we have a group structure on $\mathcal{H}o_{\bullet}(k)(\mathbb{P}^1_{k}, \mathbb{P}^1_{k})$. We denote addition for this group structure by $\oplus^{\mathbb{A}^1}$. In \cite[Appendix B]{ASENS_2012_4_45_4_511_0} Cazanave shows that $\frac{X}{a}\oplus^{\mathrm{N}}g$ is equal to $ \frac{X}{a}\oplus^{\mathbb{A}^1}g$ for any units $a\in k$ and $g$ a pair of polynomials which represents a pointed endomorphism of $(\mathbb{P}^1_{k},\infty) $. Actually, his methods works also over $\mathrm{Spec}\mathbb{Z}$ for units $a\in \mathbb{Z} $ because his proof relies on the homotopy purity theorem and does not use any specific facts about fields. Therefore we also have $ \frac{X}{1}\oplus^{\mathrm{N}}\frac{X-1}{-1}=\frac{X}{1}\oplus^{\mathbb{A}^1}\frac{X-1}{-1}$ over $ \mathrm{Spec}\mathbb{Z}$. By Definition~\ref{Definition 3.1.11} the sum $ \frac{X}{1}\oplus^{\mathrm{N}}\frac{X-1}{-1}$ is equal to $\frac{X^2-X+1}{X-1}$. In the following we give a sequence of pointed $\mathbb{A}^1$-homotopies between $\frac{X^2}{1}$ and $\frac{X^2-X+1}{X-1} $. We characterized pointed $\mathbb{A}^1$-homotopies in Proposition~\ref{Proposition 3.1.10}.\\
	
	At first we have the pointed  $\mathbb{A}^1$-homotopy \begin{align*}
		\frac{X^2}{TX+1}
	\end{align*} from $\frac{X^2}{1}$ to $\frac{X^2}{X+1}$. Then \begin{align*}
		\frac{X^2+2TX+2T}{X+1}
	\end{align*} is a pointed  $\mathbb{A}^1$-homotopy from $\frac{X^2}{X+1}$ to $ \frac{X^2+2X+2}{X+1}$. Next \begin{align*}
		\frac{X^2+2TX+2T}{X+(2T-1)}
	\end{align*}is a pointed  $\mathbb{A}^1$-homotopy from $\frac{X^2+2X+2}{X+1}$ to $ \frac{X^2}{X-1}$. Finally \begin{align*}
		\frac{X^2-TX+T}{X-1}
	\end{align*}is a pointed  $\mathbb{A}^1$-homotopy from $\frac{X^2}{X-1}$ to $ \frac{X^2-X+1}{X-1}$.
	\end{proof}
	\
	
If we suspend from the right with $\mathbb{G}_{m}$, we get the following corollary.\\
	
	\begin{cor}
		\label{Corollary 3.4.4}
		Let $w>0$ be a natural number. Then the morphism $q_{1+(w)}: S^{1+(w)}\rightarrow S^{1+(w)}$ coincides with $(1-\epsilon)_{1+(w)}$ in $\mathcal{H}o_{\bullet}(\mathbb{Z})$.\\
		\end{cor}
	
	Now we are interested in the case $w=2$.\\
	
	\begin{prop}\label{Proposition 3.4.5}
		In $\mathcal{H}o_{\bullet}(\mathbb{Z})$ the morphism $q_{1+(2)}$ is equal to $1_{1+(1)}\wedge q_{(1)}$.\\
		\end{prop}
	
\begin{proof} The morphism $q_{1+(2)}$ is given by\begin{align*}
		S^1\wedge\mathbb{G}_{m}^2\rightarrow	S^1\wedge\mathbb{G}_{m}^2; \ t\wedge x\wedge y\mapsto t\wedge x^2\wedge y \ \ \cdot
	\end{align*}
	And the morphism $1_{1+(1)}\wedge q_{(1)}$ is given by \begin{align*}
		S^1\wedge\mathbb{G}_{m}^2\rightarrow	S^1\wedge\mathbb{G}_{m}^2; \ t\wedge x\wedge y\mapsto t\wedge x\wedge y^2\ \ \cdot
	\end{align*}
	Next we consider the space $\mathbb{A}^2_{\mathbb{Z}}-\{0\}$. It is the pushout of \[\begin{tikzcd}
		{\mathbb{A}^1_{\mathbb{Z}}\times\mathbb{G}_{m}} & {\mathbb{G}_{m}\times \mathbb{G}_{m}} & {\mathbb{G}_{m}\times\mathbb{A}^1_{\mathbb{Z}}} & {,}
		\arrow[hook, from=1-2, to=1-3]
		\arrow[hook', from=1-2, to=1-1]
	\end{tikzcd}\]so we can equip it with the base point $(1,1)$ coming from ${\mathbb{G}_{m}\times \mathbb{G}_{m}}$.
	Recall that there is a \hyperlink{caniso}{canonical zig-zag} of pointed weak equivalences from $(\mathbb{A}^2_{\mathbb{Z}}-\{0\},(1,1))$ to $ \mathbb{G}_{m}\ast \mathbb{G}_{m}$ and the projection $\mathbb{G}_{m}\ast \mathbb{G}_{m}\rightarrow 	S^1\wedge\mathbb{G}_{m}^2 $ is a weak equivalence. Via this isomorphism $q_{1+(2)}$ corresponds to a pointed morphism $(\mathbb{A}^2_{\mathbb{Z}}-\{0\},(1,1))\rightarrow(\mathbb{A}^2_{\mathbb{Z}}-\{0\},(1,1)) $ which is induced by the ring homomorphism \begin{align*}
		\mathbb{Z}[T_0,T_1]\rightarrow	\mathbb{Z}[T_0,T_1]; \ T_0\mapsto T_0^2,\  T_1\mapsto T_1 \ \  (\ast)\ \ \cdot
	\end{align*}
	Similarly, $1_{1+(1)}\wedge q_{(1)}$ corresponds to the morphism induced by \begin{align*}
		\mathbb{Z}[T_0,T_1]\rightarrow	\mathbb{Z}[T_0,T_1]; \ T_0\mapsto T_0,\  T_1\mapsto T_1^2\ \ (\ast\ast) \ \ \cdot
	\end{align*}
	We would like to show that these two pointed endomorphisms of $\mathbb{A}^2_{\mathbb{Z}}-\{0\}$ are pointed $\mathbb{A}^1$-homotopic. Actually, it is enough to show that they are $\mathbb{A}^1$-homotopic. We will prove in Lemma~\ref{Lemma 3.4.6} that the canonical map $\mathcal{H}o_{\bullet}(\mathbb{Z})((\mathbb{A}^2_{\mathbb{Z}}-\{0\},(1,1)), (\mathbb{A}^2_{\mathbb{Z}}-\{0\},(1,1))) \rightarrow \mathcal{H}o(\mathbb{Z})(\mathbb{A}^2_{\mathbb{Z}}-\{0\}, \mathbb{A}^2_{\mathbb{Z}}-\{0\})$ is injective. Now we give an explicit sequence of $\mathbb{A}^1$-homotopies between the two morphisms $(\ast)$ and $(\ast\ast)$.\\
	
	The first one is given by the ring homomorphism  \begin{align*}
		\mathbb{Z}[T_0,T_1]\rightarrow	\mathbb{Z}[T_0,T_1,T]; \ T_0\mapsto (T_0+TT_1)^2,\  T_1\mapsto T_1\ \ \cdot
	\end{align*}
	This ring homomorphism induces a morphism $\mathbb{A}^2_{\mathbb{Z}}-\{0\}\times_{\mathbb{Z}}\mathbb{A}^1_{\mathbb{Z}}\rightarrow \mathbb{A}^2_{\mathbb{Z}}-\{0\}$ which is an $\mathbb{A}^1$-homotopy from $(\ast)$ to the morphism induced by 
	\begin{align*}
		\mathbb{Z}[T_0,T_1]\rightarrow	\mathbb{Z}[T_0,T_1]; \ T_0\mapsto (T_0+T_1)^2,\  T_1\mapsto T_1^2 \ \ \cdot
	\end{align*}
	The second one is induced by \begin{align*}
		\mathbb{Z}[T_0,T_1]\rightarrow	\mathbb{Z}[T_0,T_1,T]; \ T_0\mapsto (T_0+T_1)^2,\  T_1\mapsto TT_1+(T-1)T_0\ \ \cdot
	\end{align*}
	The third one is given by 
	\begin{align*}
		\mathbb{Z}[T_0,T_1]\rightarrow	\mathbb{Z}[T_0,T_1,T]; \ T_0\mapsto (TT_0+T_1)^2,\  T_1\mapsto -T_0\ \ \cdot
	\end{align*}
	The fourth one is given by \begin{align*}
		\mathbb{Z}[T_0,T_1]\rightarrow	\mathbb{Z}[T_0,T_1,T]; \ T_0\mapsto TT_0+T_1^2,\  T_1\mapsto -T_0\ \ \cdot
	\end{align*}
	The fifth one is given by 
	\begin{align*}
		\mathbb{Z}[T_0,T_1]\rightarrow	\mathbb{Z}[T_0,T_1,T]; \ T_0\mapsto T_0+TT_1^2,\  T_1\mapsto -T_0+(1-T)T_1^2\ \ \cdot
	\end{align*}
	And the last one is given by \begin{align*}
		\mathbb{Z}[T_0,T_1]\rightarrow	\mathbb{Z}[T_0,T_1,T]; \ T_0\mapsto T_0,\  T_1\mapsto -TT_0+T_1^2\ \ \cdot
	\end{align*}
	The last ring homomorphism induces an $\mathbb{A}^1$-homotopy between $(\ast\ast)$ and the morphism induced by 
	\begin{align*}
		\mathbb{Z}[T_0,T_1]\rightarrow	\mathbb{Z}[T_0,T_1]; \ T_0\mapsto T_0,\  T_1\mapsto -T_0+T_1^2\ \ \cdot
	\end{align*}
	Therefore the two morphisms $q_{1+(2)}$ and $1_{1+(1)}\wedge q_{(1)}$ coincide.\\
	\end{proof}
	
	\begin{lem}\label{Lemma 3.4.6}
		The canonical map  $\mathcal{H}o_{\bullet}(\mathbb{Z})((\mathbb{A}^2_{\mathbb{Z}}-\{0\},(1,1)), (\mathbb{A}^2_{\mathbb{Z}}-\{0\},(1,1))) \rightarrow \mathcal{H}o(\mathbb{Z})(\mathbb{A}^2_{\mathbb{Z}}-\{0\}, \mathbb{A}^2_{\mathbb{Z}}-\{0\})$ is injective.\\
		\end{lem}
	
\begin{proof} In Lemma 3.4.1 we saw that the projection onto the last column $\mathrm{SL}_2\rightarrow \mathbb{A}^2_{\mathbb{Z}}-\{0\}$ is a motivic weak equivalence. We can equip $\mathrm{SL}_2 $ with the base point induced by the ring homomorphism \begin{align*}
		\mathbb{Z}[T_{11},T_{12},T_{21},T_{22}]/(\mathrm{det}-1)\rightarrow \mathbb{Z}; \  T_{11}\mapsto 1,\ T_{12}\mapsto 1,\ T_{21}\mapsto 0,  T_{22}\mapsto 1
	\end{align*}and we denote this base point by $\begin{pmatrix}
		1&1\\
		0&1
	\end{pmatrix}$. Then the projection onto the last column is a pointed weak equivalence from $(\mathrm{SL}_2, \begin{pmatrix}
		1&1\\
		0&1
	\end{pmatrix})$ to $(\mathbb{A}^2_{\mathbb{Z}}-\{0\},(1,1))$. Therefore we only need to prove the corresponding claim for $(\mathrm{SL}_2, \begin{pmatrix}
		1&1\\
		0&1
	\end{pmatrix})$.\\
	
	Now we can also equip $\mathrm{SL}_2$ with the base point $ \begin{pmatrix}
		1&-1\\
		0&1
	\end{pmatrix}$ induced by \begin{align*}
		\mathbb{Z}[T_{11},T_{12},T_{21},T_{22}]/(\mathrm{det}-1)\rightarrow \mathbb{Z}; \  T_{11}\mapsto 1,\ T_{12}\mapsto -1,\ T_{21}\mapsto 0,  T_{22}\mapsto 1
	\end{align*} and the base point $\begin{pmatrix}
		1&0\\
		0&1
	\end{pmatrix}$ induced by \begin{align*}
		\mathbb{Z}[T_{11},T_{12},T_{21},T_{22}]/(\mathrm{det}-1)\rightarrow \mathbb{Z}; \  T_{11}\mapsto 1,\ T_{12}\mapsto 0,\ T_{21}\mapsto 0,  T_{22}\mapsto 1 \ \ \cdot
	\end{align*}
	Let $\cdot$ denote the multiplication of $\mathrm{SL}_2$. Then we have $\begin{pmatrix}
		1&1\\
		0&1
	\end{pmatrix}\cdot\begin{pmatrix}
		1&-1\\
		0&1
	\end{pmatrix}=\begin{pmatrix}
		1&0\\
		0&1
	\end{pmatrix}.$ By \cite[Proposition 2.2.1]{Asok2015TheSS} there is an $\mathbb{A}^1$-localization functor $L_{\mathbb{A}^1}$ of the category $\mathrm{sPre}(\mathbb{Z})$ together with a natural transformation $\theta$ which commutes with finite limits. Hence $L_{\mathbb{A}^1}(\mathrm{SL}_2)$ is a also a group object in $\mathrm{sPre}(\mathbb{Z})$.\\
	
	We set $a:=\theta(\mathrm{SL}_2)\circ \begin{pmatrix}
		1&1\\
		0&1
	\end{pmatrix} $, $b:=\theta(\mathrm{SL}_2)\circ \begin{pmatrix}
		1&-1\\
		0&1
	\end{pmatrix} $ and $e:=\theta(\mathrm{SL}_2)\circ \begin{pmatrix}
		1&0\\
		0&1
	\end{pmatrix} $. Then $e$ is the natural element for the group operation on $L_{\mathbb{A}^1}(\mathrm{SL}_2)$. We consider in the following $L_{\mathbb{A}^1}(\mathrm{SL}_2)$ equipped with the base point $a$. As for pointed topological spaces, there is an action of the group $\pi_1(L_{\mathbb{A}^1}(\mathrm{SL}_2))$  on $\mathcal{H}o_{\bullet}(\mathbb{Z})((L_{\mathbb{A}^1}(\mathrm{SL}_2),a), (L_{\mathbb{A}^1}(\mathrm{SL}_2),a))$. Furthermore the canonical map $\mathcal{H}o_{\bullet}(\mathbb{Z})((L_{\mathbb{A}^1}(\mathrm{SL}_2),a), (L_{\mathbb{A}^1}(\mathrm{SL}_2),a))\rightarrow \mathcal{H}o_{\bullet}(\mathbb{Z})(L_{\mathbb{A}^1}(\mathrm{SL}_2), L_{\mathbb{A}^1}(\mathrm{SL}_2))$ induces an injection from the quotient set of this action to $\mathcal{H}o_{\bullet}(\mathbb{Z})(L_{\mathbb{A}^1}(\mathrm{SL}_2), L_{\mathbb{A}^1}(\mathrm{SL}_2))$. We show now that this action is trivial. Let $\alpha: \Delta^1\rightarrow L_{\mathbb{A}^1}(\mathrm{SL}_2)$ be a loop of $a$ and $f: (L_{\mathbb{A}^1}(\mathrm{SL}_2),a)\rightarrow (L_{\mathbb{A}^1}(\mathrm{SL}_2),a) $ be a pointed morphism. Then we consider the following diagram
\[\begin{tikzcd}
	{\{a\}\times\Delta^1\cup L_{\mathbb{A}^1}(\mathrm{SL}_2)\times\{0\}} && {L_{\mathbb{A}^1}(\mathrm{SL}_2)} \\
	{L_{\mathbb{A}^1}(\mathrm{SL}_2)\times\Delta^1} &&& \cdot
	\arrow[from=1-1, to=2-1]
	\arrow["{{\alpha\cup f}}", from=1-1, to=1-3]
\end{tikzcd}\]
	This diagram has a lift $H: L_{\mathbb{A}^1}(\mathrm{SL}_2)\times\Delta^1\rightarrow L_{\mathbb{A}^1}(\mathrm{SL}_2)$ defined by $ (x,t)\mapsto f(x)\cdot b\cdot \alpha(t)$. In particular we have that  $H(-,1)$ is equal to $f$. Therefore the action is trivial. It follows that the canonical map \begin{align*}\mathcal{H}o_{\bullet}(\mathbb{Z})((L_{\mathbb{A}^1}(\mathrm{SL}_2),a), (L_{\mathbb{A}^1}(\mathrm{SL}_2),a))\rightarrow \mathcal{H}o_{\bullet}(\mathbb{Z})(L_{\mathbb{A}^1}(\mathrm{SL}_2), L_{\mathbb{A}^1}(\mathrm{SL}_2))\end{align*} is injective. Hence the corresponding map for $(\mathrm{SL}_2, \begin{pmatrix}
		1&1\\
		0&1
	\end{pmatrix} )$ is injective, too.\\
	\end{proof}
	
	Let $\tau: \mathbb{G}_{m}\wedge\mathbb{G}_{m}\rightarrow  \mathbb{G}_{m}\wedge\mathbb{G}_{m} $ be the morphism defined by $x\wedge y\mapsto y\wedge x$.\\
	
	\begin{lem}
		\label{Lemma 3.4.7} The relation $\epsilon_{1+(2)}=1_{1+(1)}\wedge\epsilon_{(1)}$ holds.\\
		\end{lem}
	
\begin{proof} The morphism $(\mathrm{id}_{S^1}\wedge \tau)\circ (1_{1+(1)}\wedge q_{(1)})$ is given by $t\wedge x\wedge y\mapsto t\wedge y^2\wedge x$. Therefore it is equal to $q_{1+(2)}\circ(\mathrm{id}_{S^1}\wedge \tau) $. Next we have \begin{align*}
		q_{1+(2)}\circ(\mathrm{id}_{S^1}\wedge \tau)=(1_{1+(2)}-\epsilon_{1+(2)})\circ(\mathrm{id}_{S^1}\wedge \tau)\\= \mathrm{id}_{S^1}\wedge \tau-\epsilon_{1+(2)}\circ (\mathrm{id}_{S^1}\wedge \tau)\\=\mathrm{id}_{S^1}\wedge \tau-(\mathrm{id}_{S^1}\wedge \tau)\circ (1_{1+(1)}\wedge \epsilon_{(1)})\\= (\mathrm{id}_{S^1}\wedge \tau)\circ(1_{1+(2)}-1_{1+(1)}\wedge \epsilon_{(1)}).
	\end{align*}
	Since $\mathrm{id}_{S^1}\wedge \tau$ is an isomorphism in the pointed homotopy category, we get $ 1_{1+(1)}\wedge q_{(1)}= 1_{1+(2)}-1_{1+(1)}\wedge \epsilon_{(1)}$. By Lemma 3.4.5 we have that $ 1_{1+(1)}\wedge q_{(1)} $ is equal to $q_{1+(2)}$. Hence we obtain $1_{1+(2)}-\epsilon_{1+(2)}=q_{1+(2)}=1_{1+(2)}-1_{1+(1)}\wedge \epsilon_{(1)}$. It follows that $\epsilon_{1+(2)}=1_{1+(1)}\wedge \epsilon_{(1)} $.\\
	\end{proof}
	
	\begin{cor}\label{Corollary 3.4.8} Under the canonical isomorphism from $\mathbb{A}^2_{\mathbb{Z}}-\{0\} $ to $S^{1+(2)} $ the morphism $\mathbb{A}^2_{\mathbb{Z}}-\{0\} \rightarrow \mathbb{A}^2_{\mathbb{Z}}-\{0\}, (T_0, T_1)\mapsto (T_0^2, T_1^2)$ corresponds to $q_{1+(2)}\circ ( 1_{1+(1)}\wedge q_{(1)})=q_{1+(2)}\circ q_{1+(2)}=(1-\epsilon)_{1+(2)}\circ (1-\epsilon)_{1+(2)}=2(1-\epsilon)_{1+(2)}$ in the commutative group $\pi_{1+(2)}S^{1+(2)}$.\\
		\end{cor}
	
\begin{proof} It is clear that the morphism $\mathbb{A}^2_{\mathbb{Z}}-\{0\} \rightarrow \mathbb{A}^2_{\mathbb{Z}}-\{0\}, (T_0, T_1)\mapsto (T_0^2, T_1^2)$ corresponds to $q_{1+(2)}\circ ( 1_{1+(1)}\wedge q_{(1)})$.  Furthermore we have that \begin{align*}
		(1-\epsilon)_{1+(2)}\circ (1-\epsilon)_{1+(2)}= (1-\epsilon)_{1+(2)}\circ (1_{1+(2)}-1_{1+(1)}\wedge \epsilon_{(1)})\\= 1_{1+(2)}-1_{1+(1)}\wedge \epsilon_{(1)}-\epsilon_{1+(2)}+\epsilon_{1+(2)}\circ (1_{1+(1)}\wedge \epsilon_{(1)})\\=1_{1+(2)}-1_{1+(1)}\wedge \epsilon_{(1)}-\epsilon_{1+(2)}+\epsilon_{1+(2)}\circ\epsilon_{1+(2)}\\= 1_{1+(2)}-\epsilon_{1+(2)}-\epsilon_{1+(2)}+1_{1+(2)}\\=2(1-\epsilon)_{1+(2)}.
	\end{align*}
\end{proof}
	
	\begin{prop}\label{Proposition 3.4.9} 
		The elements $\eta_{1+(2)}\circ h_{1+(3)}$ and $h_{1+(2)}\circ\eta_{1+(2)}$ are $\mathbb{A}^1$-nullhomotopic.\\\end{prop}
	
\begin{proof} The smash product $\eta_{1+(1)}\wedge q_{(1)} $ can be expressed in two different ways. First we have that \begin{align*}
		\eta_{1+(1)}\wedge q_{(1)}= (\eta_{1+(1)}\wedge \mathrm{id}_{\mathbb{G}_{m}})\circ  1_{1+(2)}\wedge q_{(1)}\\=\eta_{1+(2)}\circ 1_{1+(2)}\wedge q_{(1)}.
	\end{align*}
	By Proposition 3.4.5 we get $1_{1+(1)}\wedge q_{(1)}\wedge \mathrm{id}_{\mathbb{G}_{m}}=q_{1+(2)}\wedge \mathrm{id}_{\mathbb{G}_{m}}=q_{1+(3)}$. Moreover we have $(1_{1+(1)}\wedge \tau)\circ 1_{1+(2)}\wedge q_{(1)}\circ (1_{1+(1)}\wedge \tau)=1_{1+(1)}\wedge q_{(1)}\wedge \mathrm{id}_{\mathbb{G}_{m}}$. Therefore the equation $1_{1+(2)}\wedge q_{(1)}=(1_{1+(1)}\wedge \tau)\circ q_{1+(3)}\circ(1_{1+(1)}\wedge \tau)= q_{1+(3)}$ holds.\\
	
	On the other hand we have that \begin{align*}
		\eta_{1+(1)}\wedge q_{(1)}=  1_{1+(1)}\wedge q_{(1)}\circ(\eta_{1+(1)}\wedge \mathrm{id}_{\mathbb{G}_{m}}) \\=q_{1+(2)}\wedge \eta_{1+(2)} .
	\end{align*}
	Hence we get the equation $\eta_{1+(2)}\circ q_{1+(3)}=q_{1+(2)}\circ \eta_{1+(2)} $.\\
	
	Furthermore we have the commutative diagram 
	\[\begin{tikzcd}
		{\mathbb{A}^2_{\mathbb{Z}}-\{0\}} &&& {\mathbb{A}^2_{\mathbb{Z}}-\{0\}} \\
		{\mathbb{P}^1_{\mathbb{Z}}} &&& {\mathbb{P}^1_{\mathbb{Z}}}
		\arrow["{(T_0,T_1)\mapsto(T_0^2,T_1^2)}", from=1-1, to=1-4]
		\arrow["{[T_0:T_1]\mapsto[T_0^2:T_1^2]}", from=2-1, to=2-4]
		\arrow["\eta"', from=1-1, to=2-1]
		\arrow["\eta", from=1-4, to=2-4]
	\end{tikzcd}\]
	where $\eta$ is the geometric Hopf map. This commutative diagram implies the equation $q_{1+(1)}\circ \eta_{1+(1)}= \eta_{1+(1)}\circ q_{1+(2)}\circ q_{1+(2)}$.  Since $q_{1+(1)}$ is equal to $h_{1+(1)}$, we get the equations \begin{align*}
		0=(\eta_{1+(1)}\circ q_{1+(2)}\circ q_{1+(2)})\wedge \mathrm{id}_{\mathbb{G}_{m}}-(q_{1+(1)}\circ \eta_{1+(1)})\wedge \mathrm{id}_{\mathbb{G}_{m}}\\=\eta_{1+(2)}\circ h_{1+(3)}\circ h_{1+(3)}-\eta_{1+(2)}\circ h_{1+(3)}\\=\eta_{1+(2)}\circ( h_{1+(3)}\circ h_{1+(3)}-h_{1+(3)}).
	\end{align*}
	It follows from Corollary 3.4.8 that $h_{1+(3)}\circ h_{1+(3)}-h_{1+(3)}= h_{1+(3)}$. Thus we get $\eta_{1+(2)}\circ h_{1+(3)}=0 $. Finally, we get $h_{1+(2)}\circ \eta_{1+(2)}= \eta_{1+(2)}\circ h_{1+(3)}\circ h_{1+(3)}$ from the equation $q_{1+(1)}\circ \eta_{1+(1)}= \eta_{1+(1)}\circ q_{1+(2)}\circ q_{1+(2)}$. Since $\eta_{1+(2)}\circ h_{1+(3)}=0 $, we obtain $h_{1+(2)}\circ\eta_{1+(2)}=0 $.\\
	\end{proof}
	
	Next we consider the diagram 
	\[\begin{tikzcd}
		{S^{1+(3)}} & {S^{1+(3)}} & {S^{1+(2)}} && {S^{1+(2)}} \\
		&& {L_{\mathbb{A}^1}(S^{1+(2)})} && {L_{\mathbb{A}^1}(S^{1+(2)})}
		\arrow["{h_{1+(3)}}", from=1-1, to=1-2]
		\arrow["{\eta_{1+(2)}}", from=1-2, to=1-3]
		\arrow["{h_{1+(2)}}", from=1-3, to=1-5]
		\arrow["{\theta(S^{1+(2)})}", from=1-3, to=2-3]
		\arrow["{\theta(S^{1+(2)})}", from=1-5, to=2-5]
		\arrow["{L_{\mathbb{A}^1}(h_{1+(2)})}", from=2-3, to=2-5]
		\arrow["\beta"{description}, dashed, from=1-2, to=2-3]
	\end{tikzcd}\]
	in $\mathcal{H}o_{\bullet}(\mathbb{Z})$, where $\beta$ ia an arbitrary morphism of pointed motivic spaces which represents the composition $\theta(S^{1+(2)})\circ \eta_{1+(2)} $. By Proposition 3.4.9 the Toda bracket of the sequence 
	\[\begin{tikzcd}
		{S^{1+(3)}} & {S^{1+(3)}} & {L_{\mathbb{A}^1}(S^{1+(2)})} && {L_{\mathbb{A}^1}(S^{1+(2)})}
		\arrow["{h_{1+(3)}}", from=1-1, to=1-2]
		\arrow["\beta", from=1-2, to=1-3]
		\arrow["{L_{\mathbb{A}^1}(h_{1+(2)})}", from=1-3, to=1-5]
	\end{tikzcd}\]
	is defined and we denote it by $\{h_{1+(2)}, \eta_{1+(2)},h_{1+(3)} \}$.\\
	
	Similarly, we can consider the diagram 
	\[\begin{tikzcd}
		{S^{1+(4)}} & {S^{1+(3)}} & {S^{1+(3)}} && {S^{1+(2)}} \\
		&& {L_{\mathbb{A}^1}(S^{1+(3)})} && {L_{\mathbb{A}^1}(S^{1+(2)})}
		\arrow["{\eta_{1+(3)}}", from=1-1, to=1-2]
		\arrow["{h_{1+(3)}}", from=1-2, to=1-3]
		\arrow["{\eta_{1+(2)}}", from=1-3, to=1-5]
		\arrow["{\theta(S^{1+(3)})}", from=1-3, to=2-3]
		\arrow["{\theta(S^{1+(2)})}", from=1-5, to=2-5]
		\arrow["{L_{\mathbb{A}^1}(\eta_{1+(2)})}", from=2-3, to=2-5]
		\arrow["\gamma"{description}, dashed, from=1-2, to=2-3]
	\end{tikzcd}\]
	where $\gamma$ ia an arbitrary morphism of pointed motivic spaces which represents the composition $\theta(S^{1+(3)})\circ h_{1+(3)} $. By Proposition 3.4.9 the Toda bracket of the sequence 
	\[\begin{tikzcd}
		{S^{1+(4)}} & {S^{1+(3)}} & {L_{\mathbb{A}^1}(S^{1+(3)})} && {L_{\mathbb{A}^1}(S^{1+(2)})}
		\arrow["{\eta_{1+(3)}}", from=1-1, to=1-2]
		\arrow["\gamma", from=1-2, to=1-3]
		\arrow["{L_{\mathbb{A}^1}(\eta_{1+(2)})}", from=1-3, to=1-5]
	\end{tikzcd}\]
	is also defined and we denote it by $\{\eta_{1+(2)}, h_{1+(3)}, \eta_{1+(3)} \}$.\\
	
	\begin{prop}\label{Proposition 3.4.10} The Toda brackets \begin{align*}
			\{h_{1+(2)}, \eta_{1+(2)},h_{1+(3)} \}	\end{align*} and \begin{align*}\{\eta_{1+(2)}, h_{1+(3)}, \eta_{1+(3)}\}  \end{align*} are not trivial, in the sense that they do not contain the homotopy classes of constant morphisms.\\
			\end{prop}
			
\begin{proof} For the proof we need to use the pointed $\mathbb{A}^1$-local flasque model structure (see \cite{isaksen2004flasque}) and the complex realization functor $\mathbf{R}_{\mathbb{Z}}$ (see ~\ref{complex}). By \cite[Lemma 6.3]{isaksen2004flasque} motivic spheres $S^{s+(w)}$ are cofibrant in the pointed $\mathbb{A}^1$-local flasque model structure. Therefore the $\mathbb{A}^1$-localizations $L_{\mathbb{A}^1}(S^{s+(w)}) $ are also cofibrant in this model structure. The complex realization functor is a left Quillen functor for the pointed $\mathbb{A}^1$-local flasque model structure and preserves therefore weak equivalences between cofibrant objects.\\
	
	For motivic spheres we have that $\mathbf{R}_{\mathbb{Z}}(S^{s+(w)}) $ is the topological sphere $S^{s+w}$. Let $\eta_{\mathrm{top}}: S^3\rightarrow S^2$ be the topological Hopf map. Then the left derived functor $\mathbb{L}\mathbf{R}_{\mathbb{Z}}$ sends $ \eta_{1+(2)}$ to $\Sigma\eta_{\mathrm{top}}$. Since $h_{1+(3)}=q_{1+(3)}$, the functor $\mathbb{L}\mathbf{R}_{\mathbb{Z}}$ sends $ h_{1+(3)}$ to the homotopy class of the pointed continuous map $S^4\rightarrow S^4;\ x_1\wedge x_2\wedge x_3\wedge x_4\mapsto x_1\wedge x_2^2\wedge x_3\wedge x_4 $ which equals to $2\mathrm{id}_{S^4}$. Analogously, we have $\mathbb{L}\mathbf{R}_{\mathbb{Z}}(h_{1+(2)})= 2\mathrm{id}_{S^3} $.\\
	
	If we apply $\mathbb{L}\mathbf{R}_{\mathbb{Z}}$ to the diagram
	\[\begin{tikzcd}
		{S^{1+(3)}} & {S^{1+(3)}} & {S^{1+(2)}} && {S^{1+(2)}} \\
		&& {L_{\mathbb{A}^1}(S^{1+(2)})} && {L_{\mathbb{A}^1}(S^{1+(2)})}
		\arrow["{h_{1+(3)}}", from=1-1, to=1-2]
		\arrow["{\eta_{1+(2)}}", from=1-2, to=1-3]
		\arrow["{h_{1+(2)}}", from=1-3, to=1-5]
		\arrow["{\theta(S^{1+(2)})}", from=1-3, to=2-3]
		\arrow["{\theta(S^{1+(2)})}", from=1-5, to=2-5]
		\arrow["{L_{\mathbb{A}^1}(h_{1+(2)})}", from=2-3, to=2-5]
		\arrow["\beta"{description}, dashed, from=1-2, to=2-3]
	\end{tikzcd}\]
	then we get the diagram 
	\[\begin{tikzcd}
		{S^{4}} & {S^{4}} & {S^{3}} && {S^{3}} \\
		&& {\mathbb{L}\mathbf{R}_{\mathbb{Z}}(L_{\mathbb{A}^1}(S^{1+(2)}))} && {\mathbb{L}\mathbf{R}_{\mathbb{Z}}(L_{\mathbb{A}^1}(S^{1+(2)}))}
		\arrow["{2\mathrm{id}_{S^4}}", from=1-1, to=1-2]
		\arrow["{\Sigma\eta_{\mathrm{top}}}", from=1-2, to=1-3]
		\arrow["{2\mathrm{id}_{S^3}}", from=1-3, to=1-5]
		\arrow["{\mathbb{L}\mathbf{R}_{\mathbb{Z}}({\theta(S^{1+(2)})})}", from=1-3, to=2-3]
		\arrow["{\mathbb{L}\mathbf{R}_{\mathbb{Z}}({\theta(S^{1+(2)})})}", from=1-5, to=2-5]
		\arrow["{\mathbb{L}\mathbf{R}_{\mathbb{Z}}({L_{\mathbb{A}^1}(h_{1+(2)})})}", from=2-3, to=2-5]
		\arrow["{\mathbb{L}\mathbf{R}_{\mathbb{Z}}(\beta)}"{description}, dashed, from=1-2, to=2-3]
	\end{tikzcd}\] in the pointed homotopy category of topological spaces, where $\mathbb{L}\mathbf{R}_{\mathbb{Z}}({\theta(S^{1+(2)})})$ is a weak equivalence of topological spaces. In particular we have that \begin{align*}
		\mathbb{L}\mathbf{R}_{\mathbb{Z}}({\theta(S^{1+(2)})})\circ \{2\mathrm{id}_{S^3}, \Sigma\eta_{\mathrm{top}}, 2\mathrm{id}_{S^4} \}\\= \{\mathbb{L}\mathbf{R}_{\mathbb{Z}}({L_{\mathbb{A}^1}(h_{1+(2)})}), \mathbb{L}\mathbf{R}_{\mathbb{Z}}(\beta), 2\mathrm{id}_{S^4} \} .
	\end{align*}
	By \cite[Example 2, p.84]{Toda+1963} the Toda bracket $ \{2\mathrm{id}_{S^3}, \Sigma\eta_{\mathrm{top}}, 2\mathrm{id}_{S^4} \}$ is not trivial. Actually every element in the Toda bracket $ \{2\mathrm{id}_{S^3}, \Sigma\eta_{\mathrm{top}}, 2\mathrm{id}_{S^4} \}$ generates the group $\pi_{5}S^3\cong \mathbb{Z}/2\mathbb{Z}$. Since $\mathbb{L}\mathbf{R}_{\mathbb{Z}}({\theta(S^{1+(2)})})$ is a weak equivalence, the Toda bracket $\{\mathbb{L}\mathbf{R}_{\mathbb{Z}}({L_{\mathbb{A}^1}(h_{1+(2)})}), \mathbb{L}\mathbf{R}_{\mathbb{Z}}(\beta), 2\mathrm{id}_{S^4} \}$ does not contain $0$ either. Note that $\mathbb{L}\mathbf{R}_{\mathbb{Z}}(\{h_{1+(2)}, \eta_{1+(2)},h_{1+(3)} \})$ is contained in the Toda bracket $\{\mathbb{L}\mathbf{R}_{\mathbb{Z}}({L_{\mathbb{A}^1}(h_{1+(2)})}), \mathbb{L}\mathbf{R}_{\mathbb{Z}}(\beta), 2\mathrm{id}_{S^4} \} $. Therefore the Toda bracket $\{h_{1+(2)}, \eta_{1+(2)},h_{1+(3)} \}$ is not trivial.\\
	
	Similarly we can apply $\mathbb{L}\mathbf{R}_{\mathbb{Z}}$ to the diagram 
	\[\begin{tikzcd}
		{S^{1+(4)}} & {S^{1+(3)}} & {S^{1+(3)}} && {S^{1+(2)}} \\
		&& {L_{\mathbb{A}^1}(S^{1+(3)})} && {L_{\mathbb{A}^1}(S^{1+(2)})}
		\arrow["{\eta_{1+(3)}}", from=1-1, to=1-2]
		\arrow["{h_{1+(3)}}", from=1-2, to=1-3]
		\arrow["{\eta_{1+(2)}}", from=1-3, to=1-5]
		\arrow["{\theta(S^{1+(3)})}", from=1-3, to=2-3]
		\arrow["{\theta(S^{1+(2)})}", from=1-5, to=2-5]
		\arrow["{L_{\mathbb{A}^1}(\eta_{1+(2)})}", from=2-3, to=2-5]
		\arrow["\gamma"{description}, dashed, from=1-2, to=2-3]
	\end{tikzcd}\]
	and get the diagram
	\[\begin{tikzcd}
		{S^{5}} & {S^{4}} & {S^{4}} && {S^{3}} \\
		&& {L_{\mathbb{A}^1}(S^{1+(3)})} && {L_{\mathbb{A}^1}(S^{1+(2)})} && \cdot
		\arrow["{\Sigma^2\eta_{\mathrm{top}}}", from=1-1, to=1-2]
		\arrow["{2\mathrm{id}_{S^4}}", from=1-2, to=1-3]
		\arrow["{\Sigma\eta_{\mathrm{top}}}", from=1-3, to=1-5]
		\arrow["{\mathbb{L}\mathbf{R}_{\mathbb{Z}}(\theta(S^{1+(3)}))}", from=1-3, to=2-3]
		\arrow["{\mathbb{L}\mathbf{R}_{\mathbb{Z}}(\theta(S^{1+(2)}))}", from=1-5, to=2-5]
		\arrow["{\mathbb{L}\mathbf{R}_{\mathbb{Z}}(L_{\mathbb{A}^1}(\eta_{1+(2)}))}", from=2-3, to=2-5]
		\arrow["{\mathbb{L}\mathbf{R}_{\mathbb{Z}}(\gamma)}"{description}, dashed, from=1-2, to=2-3]
	\end{tikzcd}\]
	By \cite[Proposition 5.6]{Toda+1963} the Toda bracket $\{\Sigma\eta_{\mathrm{top}}, 2\mathrm{id}_{S^4}, \Sigma^2\eta_{\mathrm{top}}\}$ is not trivial. Any element $\nu'$ in this Toda bracket generates the $2$-primary component of $\pi_{6}S^3$. Hence the Toda bracket $\{\eta_{1+(2)}, h_{1+(3)}, \eta_{1+(3)} \}$ is not trivial.\\
	\end{proof}
	
	Let $k$ be a field of characteristic zero. Analogously, the Toda bracket $\{\eta_{1+(2)}, h_{1+(3)}, \eta_{1+(3)} \}$ is also defined over $k$. Using the same argument as for the base $\mathbb{Z}$ we can show that this Toda bracket is not trivial. Let $\nu'$ be an arbitrary element of $\{\eta_{1+(2)}, h_{1+(3)}, \eta_{1+(3)} \}$. The left derived complex realization functor sends this element to the element of the same name defined in \cite[Proposition 5.6]{Toda+1963}. By \cite[Proposition 4.14]{10.1112/jtopol/jtt046} the group $\pi_{2+(4)}\mathrm{SL}_3$ is equal to $\mathbb{Z}/6\mathbb{Z}$ over $k$. Let $j$ denote the inclusion $\mathrm{SL}_2\hookrightarrow \mathrm{SL}_3$. Since $S^{1+(2)}$ is isomorphic to $\mathrm{SL}_2 $, $\nu'$ can be viewed as an element of $\pi_{2+(4)}\mathrm{SL}_2$. We claim that the element $j_{\ast}(\nu')$ generates the 2-primary component of $\pi_{2+(4)}\mathrm{SL}_3$.\\
	
	\begin{prop} \label{Proposition 3.4.11} The element $j_{\ast}(\nu')$ generates the 2-primary component of $\pi_{2+(4)}\mathrm{SL}_3$ over any field $k$ of characteristic zero.\\
		\end{prop}
	
\begin{proof} Let $i$ be the inclusion of the topological groups $\mathrm{SL}(2,\mathbb{C})\hookrightarrow \mathrm{SL}(3,\mathbb{C})$. Using the complex realization we get the following commutative diagram of abelian groups 
	\[\begin{tikzcd}
		{\pi_{2+(4)}\mathrm{SL}_2} & {\pi_{2+(4)}\mathrm{SL}_3} \\
		{\pi_{6}\mathrm{SL}(2,\mathbb{C})} & {\pi_{6}\mathrm{SL}(3,\mathbb{C})} & \cdot
		\arrow["{j_{\ast}}", from=1-1, to=1-2]
		\arrow[from=1-1, to=2-1]
		\arrow[from=1-2, to=2-2]
		\arrow["{i_{\ast}}", from=2-1, to=2-2]
	\end{tikzcd}\]
	By \cite[Theorem 5.5]{10.1112/jtopol/jtt046} the homomorphism $\pi_{2+(4)}\mathrm{SL}_3\rightarrow \pi_{6}\mathrm{SL}(3,\mathbb{C})$ induced by the complex realization is an isomorphism over any field $k$ of characteristic zero. Since $ \mathrm{SL}(n,\mathbb{C})$ is homotopy equivalent to $\mathrm{SU}(n)$, we can rewrite the previous diagram in the following way 
	\[\begin{tikzcd}
		{\pi_{2+(4)}\mathrm{SL}_2} & {\pi_{2+(4)}\mathrm{SL}_3} \\
		{\pi_{6}\mathrm{SU}(2)} & {\pi_{6}\mathrm{SU}(3)}
		\arrow["{j_{\ast}}", from=1-1, to=1-2]
		\arrow[from=1-1, to=2-1]
		\arrow["\cong", from=1-2, to=2-2]
		\arrow["{i_{\ast}}", from=2-1, to=2-2]
	\end{tikzcd}\]
	where we abuse the notation and denote also the inclusion $\mathrm{SU}(2)\hookrightarrow \mathrm{SU}(3)$ by $i$. By \cite[Theorem 4.1]{10.1215/kjm/1250524818} the 2-primary component of $\pi_{6}\mathrm{SU}(3)$ is generated by $i_{\ast}(\nu')$. Therefore $j_{\ast}(\nu')$ generates the 2-primary component of $\pi_{2+(4)}\mathrm{SL}_3$ over any field $k$ of characteristic zero.\\
	\end{proof}
At the end of the paper we construct another motivic Toda bracket over $\mathrm{Spec}\mathbb{Z}$ which complex realization is trivial; nevertheless this Toda bracket itself is not trivial.\\

Let $\Delta_{(2)}: \mathbb{G}_{m}\rightarrow \mathbb{G}_{m}\wedge \mathbb{G}_{m} $ be the diagonal morphism $x\mapsto x\wedge x $.\\

\begin{prop}
	The element $\Delta_{1+(3)}\circ h_{1+(2)} $ is $\mathbb{A}^1$-nullhomotopic.\\
	\end{prop}

\begin{proof}	
	Since $h_{1+(2)}$ is by definition equal to $1_{1+(2)}-\epsilon_{1+(2)}$, we have that $\Delta_{1+(3)}\circ h_{1+(2)}= \Delta_{1+(3)}- \Delta_{1+(3)}\circ \epsilon_{1+(2)} $. In the following we show that $\Delta_{1+(3)}\circ \epsilon_{1+(2)}$ is equal to $ \Delta_{1+(3)}$. The morphism $\Delta_{1+(3)}\circ \epsilon_{1+(2)}$ is given by  $ t\wedge x\wedge y\mapsto t\wedge x^{-1}\wedge x^{-1}\wedge y$. Therefore  $\Delta_{1+(3)}\circ \epsilon_{1+(2)}$ equals to $\epsilon_{1+(3)} \circ 1_{1+(1)}\wedge\epsilon_{(1)}\wedge \mathrm{id}_{\mathbb{G}_{m}}\circ \Delta_{1+(3)} $. By Lemma~\ref{Lemma 3.4.7} we get $\epsilon_{1+(3)}=1_{1+(1)}\wedge\epsilon_{(1)}\wedge \mathrm{id}_{\mathbb{G}_{m}}$. Hence we have that \begin{align*}
		\epsilon_{1+(3)} \circ 1_{1+(1)}\wedge\epsilon_{(1)}\wedge \mathrm{id}_{\mathbb{G}_{m}}=	\epsilon_{1+(3)}\circ 	\epsilon_{1+(3)}=1_{1+(3)}.
	\end{align*}
	It follows that $\Delta_{1+(3)}\circ \epsilon_{1+(2)}$  is equal to $\Delta_{1+(3)}$.\\
	\end{proof}
	
By Proposition~\ref{Proposition 3.4.9} we also have that $h_{1+(2)}\circ\eta_{1+(2)} $ is $\mathbb{A}^1$-nullhomotopic. In particular the Toda bracket $\{\Delta_{1+(3)},h_{1+(2)}, \eta_{1+(2)} \}$ is defined. The complex realization of the morphism $\Delta_{1+(3)} $ is a pointed continuous map from $S^3$ to $S^4$; hence it is nullhomotopic. The left derived complex realization functor sends the Toda bracket $\{\Delta_{1+(3)},h_{1+(2)}, \eta_{1+(2)} \}$ to $\{0, 2\mathrm{id}_{S^3}, \Sigma\eta_{top} \}$. The topological Toda bracket  $\{0, 2\mathrm{id}_{S^3}, \Sigma\eta_{top} \}$ is trivial. Therefore the complex realization of $\{\Delta_{1+(3)},h_{1+(2)}, \eta_{1+(2)} \}$ is trivial, too. Although the complex realization of $\{\Delta_{1+(3)},h_{1+(2)}, \eta_{1+(2)} \}$ is trivial, we can show that this Toda bracket itself is not trivial.\\

\begin{prop}
	The Toda bracket $\{\Delta_{1+(3)},h_{1+(2)}, \eta_{1+(2)} \}$ does not contain 0.	
	\end{prop}
	
\begin{proof}	
The Toda bracket $\{\Delta_{1+(3)},h_{1+(2)}, \eta_{1+(2)} \}$ is again defined over any field $k$ of characteristic zero. In the following we first work over such a base $k$. By Lemma~\ref{Prop2.3.3} we get the equation\begin{align*}
		\{\Delta_{1+(3)},h_{1+(2)}, \eta_{1+(2)} \} \circ h_{2+(3)} = -(\Delta_{1+(3)}\circ \{h_{1+(2)}, \eta_{1+(2)},  h_{1+(3)} \} ).
	\end{align*}	
We want to study 	$\Delta_{1+(3)}\circ \{h_{1+(2)}, \eta_{1+(2)},  h_{1+(3)} \} $. We have to consider now the stable motivic catagory $\mathcal{SH}(k)$ (see \cite{morelintro}). The category $\mathcal{SH}(k)$ is naturally a triangulated category, therefore Toda brackets are defined in this category. Furthermore there is a suspension spectrum functor \begin{align*}
	\Sigma^{\infty}_{\mathbb{P}^1}: \mathcal{H}o_{\bullet}(k)\rightarrow \mathcal{SH}(k)
\end{align*}
We set $h:=\Sigma^{\infty}_{\mathbb{P}^1}(h_{1+(1)}) $ and $\eta:= \Sigma^{\infty}_{\mathbb{P}^1}(\eta_{1+(1)})$. Then the functor $\Sigma^{\infty}_{\mathbb{P}^1}$ sends $\{h_{1+(2)}, \eta_{1+(2)},  h_{1+(3)} \} $ to the Toda bracket $<h, \eta, h>$ in $\mathcal{SH}(k)$.\\

Let $[-1]: S^0\rightarrow \mathbb{G}_{m}$ be the pointed morphism defined by sending $\ast\in S^0=\{\ast\}_+$ to $-1\in \mathbb{G}_{m}$. In \cite[Proposition 3.5]{motihopf} Dugger and Isaksen show that $\rho:=\Sigma^{\infty}_{\mathbb{P}^1}([-1])$ is equal to $\Sigma^{\infty}_{\mathbb{P}^1}(\Delta_{(2)})$. In the following we will denote the suspension spectrum $\Sigma^{\infty}_{\mathbb{P}^1}(\mathcal{E})$ of a pointed motivic space simply by $\mathcal{E} $. The suspension spectrum of $S^0$ is called the sphere spectrum.\\

We can equip $\mathcal{SH}(k)$ with a smash product $\wedge$ which makes $\mathcal{SH}(k)$ into a tensor triangulated category. Both $S^1$ and $\mathbb{G}_{m} $ are $\wedge$-invertible. We define the bigraded homotopy groups $\pi_{s+(w)}\mathbf{1}:= \mathcal{SH}(k)((S^1)^m\wedge\mathbb{G}_{m}^w, S^0) $ for all $s, w\in \mathbb{Z}$. In particular, we have that $h\in \pi_{0+(0)}\mathbf{1}, \eta\in \pi_{0+(1)}\mathbf{1}$ and $\rho\in  \pi_{0+(-1)}\mathbf{1}$. Furthermore the Toda bracket $<h, \eta, h>$ is contained in $\pi_{1+(1)}\mathbf{1}$. \\

By work of Morel \cite[Theorem 6.4.1]{morelintro} we have an isomorphism \begin{align*}
	K^{MW}_{-\ast}(k)\rightarrow \underset{w\in \mathbb{Z}}{\oplus}\pi_{0+(w)}\mathbf{1}
\end{align*}
of graded rings, where $K^{MW}_{\ast}(k)$ is the Milnor-Witt K-theory of the field $k$ (see \cite[Definition 6.4]{Isaksen2018MotivicSH}). Under this isomorphism the morphisms $h,\eta$ and $\rho$ are sent to the elements in $K^{MW}_{-\ast}$ of the same names. The Milnor K-theory $K^{M}_{\ast} (k)$ is defined to be $K^{MW}_{\ast}(k)/(\eta) $. Let $\eta_{top}: S^3\rightarrow S^2$ be the topological Hopf map. It is an element of $\pi_{1+(0)}\mathbf{1}$. In \cite[(1.1)]{10.4007/annals.2019.189.1.1} we can find a short exact sequence \[\begin{tikzcd}
	0 & {K^{M}_{2} (k)/24} & {\pi_{1+(0)}\mathbf{1}} & {k^{\times}/2\oplus \mathbb{Z}/2} & 0 & \cdot
	\arrow[from=1-4, to=1-5]
	\arrow[from=1-3, to=1-4]
	\arrow[from=1-2, to=1-3]
	\arrow[from=1-1, to=1-2]
\end{tikzcd}\]
The kernel $K^{M}_{2}(k)/24 $ is generated by the second motivic Hopf map $\nu \in \pi_{1+(2)}\mathbf{1}$ (see \cite[Definition 4.7]{motihopf}), in the sense that its elements are of the form $\alpha\nu$ for $\alpha\in\pi_{0+(-2)}\mathbf{1} $.	
The second factor of $k^{\times}/2\oplus \mathbb{Z}/2 $ is generated by the image of $\eta_{top}$ and the first factor is generated by $\eta\eta_{top}$, in the sense that its elements are of the form $\alpha\eta\eta_{top}$ for $\alpha\in\pi_{0+(-1)}\mathbf{1} $. These generators are subject to the relations $24\nu=0$ and $12\nu= \eta^2\eta_{top}$.\\

By \cite[Proposition 4.1]{Rondigs2019RemarksOM} the Toda bracket $<h, \eta, h>$ is of the form $\eta\eta_{top}+ 2K^{M}_{1}(k)/24$, where $2K^{M}_{1}(k)/24$ is the indeterminacy. The suspension functor $\Sigma^{\infty}_{\mathbb{P}^1}$ sends  the set $\Delta_{1+(3)}\circ \{h_{1+(2)}, \eta_{1+(2)},  h_{1+(3)} \}$ to the set $\rho\cdot(\eta\eta_{top}+ 2K^{M}_{1}(k)/24)\subseteq\pi_{1+(0)}\mathbf{1}$. It suffices to show that $\rho\cdot(\eta\eta_{top}+ 2K^{M}_{1}(k)/24)$ does not contain 0. The surjection $\pi_{1+(0)}\mathbf{1}\rightarrow k^{\times}/2\oplus \mathbb{Z}/2$ in the short exact sequence above sends the elements of $\rho\cdot(\eta\eta_{top}+ 2K^{M}_{1}(k)/24)$ to $\rho\cdot\eta\eta_{top} $. Therefore if $-1$ is not a quadratic root in $k$, then $\rho\cdot\eta\eta_{top} $ is not equal to 0. This is in particular the case if $k=\mathbb{Q}$. It follows that for $k=\mathbb{Q}$ the set  $\Delta_{1+(3)}\circ \{h_{1+(2)}, \eta_{1+(2)},  h_{1+(3)} \}$ does not contain 0, hence the Toda bracket $\{\Delta_{1+(3)},h_{1+(2)}, \eta_{1+(2)} \}$ is not trivial over $\mathbb{Q}$.\\

Finally we would like to show that $\{\Delta_{1+(3)},h_{1+(2)}, \eta_{1+(2)} \}$ is not trivial over $\mathbb{Z}$. For this we use base change arguments. Let $f:\mathrm{Spec}\mathbb{Q}\rightarrow \mathrm{Spec}\mathbb{Z}$ be the canonical morphism. Then there is a functor \begin{align*}
	f_{\ast} : \mathrm{sPre}(\mathbb{Q})\rightarrow \mathrm{sPre}(\mathbb{Z})
\end{align*} which is induced by
\begin{align*}
	\mathcal{S}\mathrm{m}_{\mathbb{Z}}\rightarrow \mathcal{S}\mathrm{m}_{\mathbb{Q}}; \ X\mapsto X\times_{\mathrm{Spec}\mathbb{Z}}\mathrm{Spec}\mathbb{Q}
\end{align*}
for all  $X\in\mathcal{S}\mathrm{m}_{\mathbb{Z}} $. The functor $f_{\ast}$ admits a left adjoint $f^{\ast}:\mathrm{sPre}(\mathbb{Z})\rightarrow \mathrm{sPre}(\mathbb{Q})$ with the property that it maps the sheaf represented by $ X\in\mathcal{S}\mathrm{m}_{\mathbb{Z}}$ to the sheaf represented by $X\times_{\mathrm{Spec}\mathbb{Z}}\mathrm{Spec}\mathbb{Q}$. The explicit construction can be found in (see\cite{jardine2015local}  p.108). By \cite[Corollary 5.24]{jardine2015local} and \cite[Proposition 3.3.18]{hirschhorn2003model} the adjoint functors \begin{align*}
f^{\ast}:\mathrm{sPre}(\mathbb{Z})\rightleftarrows \mathrm{sPre}(\mathbb{Q}): f_{\ast}
\end{align*}
form a Quillen adjunction for the $\mathbb{A}^1$-local injective models. Since both functors $f^{\ast}$ and $f_{\ast}$ preserve terminal objects, we get also the pointed version of this Quillen adjunction  \begin{align*}
	f^{\ast}:\mathrm{sPre}(\mathbb{Z})_{\ast}\rightleftarrows \mathrm{sPre}(\mathbb{Q})_{\ast}: f_{\ast} \ \ \ .
\end{align*}
We denote motivic spheres over $\mathbb{Z}$ by $S^{s+(w)}_{\mathbb{Z}}$ and motivic spheres over $\mathbb{Q}$ by $S^{s+(w)}_{\mathbb{Q}}$. It follows from the construction of $f^{\ast}$ that $f^{\ast}(S^{s+(w)}_{\mathbb{Z}})=S^{s+(w)}_{\mathbb{Q}}$. Furthermore the left derived functor $\mathbb{L}f^{\ast} $ sends the morphisms $\Delta_{1+(3)},h_{1+(2)}$ and $\eta_{1+(2)}$ in $\mathcal{H}_{\bullet}(\mathbb{Z})$ to the morphisms of the same names in $\mathcal{H}_{\bullet}(\mathbb{Q})$. Therefore $\mathbb{L}f^{\ast} $ sends the Toda bracket $\{\Delta_{1+(3)},h_{1+(2)}, \eta_{1+(2)} \}$ over $\mathbb{Z}$ to the corresponding Toda bracket $\{\Delta_{1+(3)},h_{1+(2)}, \eta_{1+(2)} \} $ over $\mathbb{Q}$. We already know that the Toda bracket $\{\Delta_{1+(3)},h_{1+(2)}, \eta_{1+(2)} \} $ over $\mathbb{Q}$ does not contain 0, hence the Toda bracket $\{\Delta_{1+(3)},h_{1+(2)}, \eta_{1+(2)} \}$ over $\mathbb{Z}$ is also not trivial.
\end{proof}

	 \newpage
	
	\section{Appendix}
	\
	\
	
	\subsection{$\Delta$-generated spaces}\label{delta}
	\
	
This section contains facts about $\Delta$-generated spaces which are used in the paper. The notion of $\Delta$-generated spaces was originally proposed by Jeff Smith as a nice category of spaces for homotopy category. However, Jeff Smith never published his ideas and there are only few references on this notion. In the following we will follow the unpublished notes by Daniel Dugger \cite{delta}.\\

Let $\mathcal{T}op$ denote the category of all topological spaces and continous maps and $\Delta$ be the full subcategory of $\mathcal{T}op$ consisting of the topological simplices $\Delta^n$.\\

\begin{definition}[{\cite[Definition 1.2]{delta}}]
	\label{Definition A.1.1}
 A topological space $X$ is called $\Delta$-generated if it has the property that a subset $S\subseteq X$ is open if and only if $f^{-1}(S)$ is open for every continuous map $f:Z\rightarrow X$ with $Z\in\Delta$. Let $\mathcal{T}op_\Delta $ denote the full subcategory of $\Delta$-generated spaces.\\
 \end{definition}

\begin{prop}[{\cite[Proposition 1.3]{delta}}]
	\label{Proposition A.1.2} 
	Any object of $\Delta$ is $\Delta$-generated. Any colimit of $\Delta$-generated spaces is again $\Delta$-generated.\\
	\end{prop}

Therefore $\mathcal{T}op_\Delta $ is a cocomplete category and the colimits are the same as those in $\mathcal{T}op$. Moreover, it also follows that $\mathcal{T}op_\Delta$ contains the geometric realization of every simplicial set. We now show that this category is also complete. Let $X$ be a topological space and $(\Delta\downarrow X)$ be the overcategory. Then there is a canonical diagram $(\Delta\downarrow X)\rightarrow \mathcal{T}op$ sending every object $(f:Z\rightarrow X)$ to $Z$. The colimit of this diagram will be denoted by $\mathit{k}_\Delta(X)$. By the above proposition this colimit is again $\Delta$-generated, and there is a canonical map $\mathit{k}_\Delta(X)\rightarrow X $.\\

\begin{prop}[{\cite[Proposition 1.5]{delta}}]
	\label{Proposition A.1.3}	
	(a) $\mathit{k}_\Delta(X)\rightarrow X $ is a set-theoretic bijection.\\
	(b) $X$ is $\Delta$-generated if and only if $\mathit{k}_\Delta(X)\rightarrow X $ is a homeomorphism.\\
	(c) A space is $\Delta$-generated if and only if it is a colimit of some diagram whose objects belong to $\Delta$.\\
	(d) The functors $i:\mathcal{T}op_\Delta\rightleftarrows\mathcal{T}op:\mathit{k}_\Delta$ are an adjoint pair, where $i$ is the inclusion.\\
		\end{prop} 

Now by Proposition 4.1.3 (b) and (d) we see that $\mathcal{T}op_\Delta$ is also complete; limits are computed by first taking the limit in  $\mathcal{T}op$ and then applying the functor $\mathit{k}_\Delta(-)$.\\

One of the most important properties of the category $\mathcal{T}op_\Delta$ is that it is locally presentable.\\

\begin{prop}[{\cite[Corollary 3.7]{Fajstrup2007ACC}}]
	\label{Proposition A.1.4} The category $\mathcal{T}op_\Delta$ is locally presentable.\\
	\end{prop}

It follows that in particular every object of $\mathcal{T}op_\Delta$ is small, therefore $\mathcal{T}op_\Delta$ permits the small object argument. Furthermore by \cite[3.3]{WYLER1973225}, the category $\mathcal{T}op_\Delta$ is even cartesian closed. For $X,Y\in \mathcal{T}op_\Delta$  we write $X\otimes Y$ for the product in $\mathcal{T}op_\Delta$ and $X\times Y$ for the usual cartesian product in $\mathcal{T}op $.\\

\begin{prop}[{\cite[Proposition 1.14]{delta}}]
	\label{Proposition A.1.5} 
	The natural map $X\otimes Y\rightarrow X\times Y $ is a homeomorphism.\\
	\end{prop}

At the end of this section we also mention that every open subset of a $\Delta$-generated space is again $\Delta$-generated.\\

\begin{prop}[{\cite[Proposition 1.18]{delta}}]
	\label{Proposition A.1.6} (a) Every open subset of $\Delta^n$ is $\Delta$-generated.\\
	(b) An open subset of a $\Delta$-generated space is $\Delta$-generated. \\
	\end{prop}
	
	\
	
\subsection{Complex realization}\label{complex}
\

We would like to recall some facts on complex realization. Any scheme $X\in \mathcal{S}\mathrm{m}_{\mathbb{C}}$ gives rise to a topological space $X(\mathbb{C})$ consisting of its $\mathbb{C}$-valued points with the analytic topology. In this way we get a functor  $\mathcal{S}\mathrm{m}_{\mathbb{C}}\rightarrow \mathcal{T}op$. Every simplicial presheaf $\mathcal{Y}$ is a canonical colimit\begin{center}
	$\underset{X\times\Delta^n\rightarrow\mathcal{Y} }{\mathrm{colim}} X\times\Delta^n\cong \mathcal{Y} $
\end{center}
and then we define $\mathbf{R}_{\mathbb{C}}(\mathcal{Y}):= \underset{X\times\Delta^n\rightarrow\mathcal{Y} }{\mathrm{colim}} X(\mathbb{C})\times |\Delta^n|$. We get  a functor $\mathbf{R}_{\mathbb{C}}:\mathrm{sPre}(\mathbb{C})\rightarrow \mathcal{T}op$. Note that if $\mathcal{Y}$ is pointed, then so is $\mathbf{R}_{\mathbb{C}}(\mathcal{Y})$.  It turns out that if we use the $\mathbb{A}^1$-local projective or the $\mathbb{A}^1$-local flasque model structure, the functor $\mathbf{R}_{\mathbb{C}}$ becomes a left Quillen functor (see \cite[Theorem 5.2]{tophyper} and \cite[Theorem A.23]{Panin2009}).\\

Suppose now $R\hookrightarrow \mathbb{C}$ is a subring of $\mathbb{C}$. Let $f:\mathrm{Spec}(\mathbb{C})\rightarrow\mathrm{Spec}(R)$ denote the resulting morphism of the base scheme. Then we first get a functor $ \mathcal{S}\mathrm{m}_{R}\rightarrow \mathcal{S}\mathrm{m}_{\mathbb{C}}$ by sending $U\in\mathcal{S}\mathrm{m}_{R} $ to $U\times_{\mathrm{Spec}(R)}\mathrm{Spec}(\mathbb{C})\in \mathcal{S}\mathrm{m}_{\mathbb{C}} $. From this we get another functor $f_{\ast}: \mathrm{sPre}(\mathbb{C})\rightarrow\mathrm{sPre}(R) $. By \cite[p.108]{jardine2015local} the functor $f_{\ast}$ admits a left adjoint $f^{\ast}: \mathrm{sPre}(R)\rightarrow \mathrm{sPre}(\mathbb{C})$. Finally, we define the realization with respect to $R$ to be the composition
\begin{center}
	$\mathbf{R}_{R}=\mathbf{R}_{\mathbb{C}}\circ f^{\ast} : \mathrm{sPre}(R)\rightarrow \mathcal{T}op $.
\end{center}
It is again a left Quillen functor with respect to the $\mathbb{A}^1$-local projective or the $\mathbb{A}^1$-local flasque model structure. We are mostly interested in the case $R=\mathbb{Z}$. Analogously, we have also pointed versions of the above left Quillen functors.\newpage

		\bibliographystyle{alpha} 
			\clearpage
		\addcontentsline{toc}{section}{References}
	\bibliography{references} 	
	\nocite{*}
	\end{document}